\documentclass[a4paper, 11pt, reqno]{amsart}

\usepackage{amssymb,amsthm,amsmath,mathrsfs}
\usepackage{mathtools}
\usepackage{bbm}
\usepackage{stackengine}
\usepackage{enumitem}
\usepackage{soul}
\usepackage{thm-restate}

\usepackage[foot]{amsaddr}

\usepackage{caption,color,graphicx}
\usepackage[dvipsnames]{xcolor}
\usepackage{soul}

\usepackage{hyperref}
\hypersetup{colorlinks  = true,
            urlcolor    = blue,
            linkcolor   = red,
            citecolor   = blue}

\usepackage{verbatim}

\usepackage{environ}
\mathtoolsset{showonlyrefs,showmanualtags}

\usepackage[centering, top=3cm, bottom=2.5cm]{geometry}

\allowdisplaybreaks

\newcommand{\R}{\mathbb{R}}

\renewcommand{\d}{\mathrm{d}}
\newcommand{\proj}{\mathrm{proj}}

\newcommand{\e}{\varepsilon}

\newcommand{\vertiii}[1]{{\left\vert\kern-0.25ex\left\vert\kern-0.25ex\left\vert #1 \right\vert\kern-0.25ex\right\vert\kern-0.25ex\right\vert}}

\let\oldtocsection=\tocsection
\let\oldtocsubsection=\tocsubsection
\renewcommand{\tocsection}[2]{\hspace{0em}\oldtocsection{#1}{#2}}
\renewcommand{\tocsubsection}[2]{\hspace{1em}\oldtocsubsection{#1}{#2}}

\theoremstyle{plain}
\newtheorem*{theorem*}{Theorem}
\newtheorem{theorem}{Theorem}[section]
\newtheorem{lemma}[theorem]{Lemma}
\newtheorem{proposition}[theorem]{Proposition}
\newtheorem{corollary}[theorem]{Corollary}
\newtheorem{claim}[theorem]{Claim}

\newtheorem{manualtheorem}{Theorem}
\newenvironment{mytheorem}[1][]{%
    \renewcommand\themanualtheorem{#1}%
    \begin{manualtheorem}
}{\end{manualtheorem}}

\newtheorem{manualcorollary}{Corollary}

\newtheorem{manualhyp}{Hypothesis}
\newenvironment{hypothesis}[1][]{%
    \renewcommand\themanualhyp{#1}%
    \begin{manualhyp}
}{\end{manualhyp}}

\theoremstyle{remark}
\newtheorem{definition}[theorem]{Definition}
\newtheorem{remark}[theorem]{Remark}
\newtheorem{example}[theorem]{Example}

\newtheorem{manualstep}{\bf Step}
\newenvironment{step}[1][]{%
    \renewcommand\themanualstep{#1}%
    \begin{manualstep}\itshape
}{\end{manualstep}}



\IfPackageLoadedTF{MnSymbol}{}{}

\newlist{todolist}{itemize}{2}
\setlist[todolist]{label=$\square$}
\usepackage{pifont}

\author[Lucas Amorim]{Lucas Amorim$^{1,2}$}
\author[Matheus M.~Castro]{Matheus M. Castro$^{3}$}
\author[Benoît Saussol]{Benoît Saussol$^{4}$}
\author[Sandro Vaienti]{Sandro Vaienti$^{1,5}$}
\email{%
\href{mailto:lucas.amorim.vb@gmail.com}{lucas.amorim.vb@gmail.com},
\href{mailto:mmcastro@ime.unicamp.br}{mmcastro@ime.unicamp.br}, 
 \href{mailto:benoit.saussol@univ-amu.fr}{benoit.saussol@univ-amu.fr}}
\email{\href{mailto:vaienti@cpt.univ-mrs.fr}{vaienti@cpt.univ-mrs.fr}}

\address{$^{1}$Université de Toulon, Aix Marseille Université, CNRS, CPT, 13009 Marseille, France}
\address{$^{2}$Centro de Matemática da Universidade do Porto, Rua do Campo Alegre 687, 4169-007 Porto, Portugal}
\address{$^{3}$Departamento de Matemática --  Universidade Estadual de Campinas (UNICAMP), 13083-859, Campinas, SP, Brazil}
\address{$^{4}$Aix Marseille Univ, CNRS, I2M, Marseille, France}
\address{$^{5}$Department of Mathematics, School of Sciences, Great Bay University, Dongguan, Guangdong 523000, China}

\numberwithin{equation}{section}
\usepackage{bbm}

\usepackage[backend = biber, doi = true, url = false, isbn = false, maxbibnames = 10, giveninits=true, style = numeric]{biblatex}
\AtBeginBibliography{\small}

\newcommand{\opnorm}[1]{\left\| #1 \right\|}

\title[Thermodynamic formalism for hyperbolic random dynamical systems]{Thermodynamic formalism for hyperbolic
random dynamical systems}

\begin{document}

\begin{abstract}
We develop thermodynamic formalism for random Anosov maps and uniformly H\"older random potentials. We assume uniform fibre hyperbolicity given by deterministic invariant cone fields, a one-dimensional stable direction, and a fibrewise mixing condition whose mixing time may depend on the base point. To do so, we construct adapted projective cones for the random Perron--Frobenius cocycle and prove that the cocycle contracts the associated Hilbert projective metrics. This allows us to construct a $\mathbb P$-relative equilibrium state, prove its uniqueness, and establish quenched exponential decay of correlations.

\end{abstract}

\keywords{Thermodynamic Formalism; skew products; random dynamical systems; hyperbolic dynamical systems}
\subjclass[2020]{28D20; 37D20; 37D35; 37H05}
\maketitle
\tableofcontents

\section{Introduction}

Given a dynamical system $T:M\to M$, ergodic theory provides a way of studying the long-time behaviour of the orbits of $T$ through a probabilistic viewpoint. Once a $T$-invariant probability measure $\mu$ is fixed, every observable $g:M\to\mathbb R$ gives rise to the stationary process $X_k=g\circ T^k:(M,\mathbb P)\to\mathbb R$ where  $k\in\mathbb N$. One may then ask whether this process satisfies the usual statistical laws of probability theory, such as the law of large numbers, decay of correlations, limit theorems, and large deviations. One weakness of such an approach comes from the fact that the invariant measure is usually not canonically determined by the dynamics. Indeed, uniformly hyperbolic systems admit uncountably many invariant measures \cite[Theorem 4]{Sigmund1970}, and the problem is therefore not merely to find invariant measures, but to identify the dynamically meaningful ones and establish their statistical properties.

Inspired by ideas from statistical mechanics, thermodynamic formalism gives a systematic way to select distinguished invariant measures. Instead of choosing an invariant probability measure directly, one fixes a potential $\phi:M\to\mathbb R$ and looks for invariant measures maximising the free energy $h_\mu(T)+\int \phi\,\d\mu$ over the $T$-invariant probability measures $\mu$. Such maximisers are called equilibrium states. Thus the problem of selecting meaningful invariant measures is transformed into a variational problem combining dynamical complexity, measured by entropy, with the physical information, given by the potential $\phi$.

For uniformly hyperbolic systems, this point of view was developed in the works of Sinai, Ruelle and Bowen. Sinai introduced Gibbs measures into ergodic theory, showing that ideas from statistical mechanics naturally describe invariant measures for hyperbolic dynamics \cite{Sinai1972}. Ruelle developed the abstract thermodynamic formalism and the transfer-operator viewpoint \cite{Ruelle1978}. Bowen developed a construction based on Markov partitions, which allow one to reduce the smooth dynamics on a basic hyperbolic set to symbolic dynamics given by a subshift of finite type. Through such a reduction, Bowen proved the existence and uniqueness of equilibrium states for Hölder potentials for Axiom A diffeomorphisms \cite{Bowen1975}. Similar results, as well as refined statistical properties such as sharp decay of correlations and spectral stability, have also been obtained by transfer-operator methods acting on suitable anisotropic spaces adapted to the hyperbolic splitting; see, for instance, \cite{liverani1995decay,LSV1998,BlankKellerLiverani2002,GouezelLiverani2006,BaladiTsujii2007,YellowBook,Baladi2018}.

Thermodynamic formalism can also be extended to random dynamical systems. In this setting, the dynamics is described by a skew product
$$F:\Omega\times M\to \Omega\times M,\  F(\omega,x)=(\theta\omega,T_\omega (x)),$$
where the base dynamics $\theta$ is a $\mathbb P$-invariant map and $T_\omega:M\to M$ is the fibre map applied over $\omega$; see \cite{arnold1995random,liu2006smooth,crauel2002random}. The probability measure $\mathbb P$ is fixed, and the main object is the fibre dynamics along a $\mathbb{P}$-typical realisation of the base. Accordingly, the relevant invariant measures are those $F$-invariant probabilities whose projection to the base is $\mathbb P$. Hence, the deterministic selection principle has a relative random counterpart: for a random potential $\phi:\Omega\times M\to\mathbb R$, one seeks invariant measures maximising $h_\mu(F\mid \mathbb P)+\int\phi\,\d\mu$  among all $F$-invariant measures whose projection is $\mathbb P$, where $h_\mu(F\mid \mathbb P)$ denotes the relative metric entropy of $F$ with respect to the base measure $\mathbb P$ (see Definition \ref{def:relativeentro}). The basic theory of thermodynamic formalism for random systems was developed by Kifer and Bogenschütz \cite{Kifer92,bogenschutz1992entropy,Kifer2001,kifer2006random}. This relative viewpoint is also closely related to the relativised variational principle of Ledrappier--Walters and to Walters' compensation functions for factor maps \cite{LedrappierWalters77,Walters86}.

Thermodynamic formalism for random systems has been developed in many directions: random symbolic dynamics \cite{BogenschutzGundlach1995ETDS,GundlachKifer2000DCDS,DenkerKiferStadlbauer2008DCDS,Stadlbauer2010StochDyn,Stadlbauer2017ETDS}; smooth expanding maps \cite{Kifer92,MSU11}; non-uniformly expanding maps \cite{ArbietoMatheusOliveira03,StadlbauerSuzukiVarandas2021CMP}; and non-uniformly hyperbolic random interval maps, both closed and open and possibly with discontinuities, \cite{AtnipFroylandGonzalezTokmanVaienti2021CMP,AtnipFroylandGonzalezTokmanVaienti2023ETDS,AtnipFroylandGonzalezTokmanVaienti2024DM}. Nevertheless, for random systems generated by diffeomorphisms, results beyond the SRB setting remain comparatively scarce. For SRB-type results in deterministic and random hyperbolic systems, see \cite{Ledrappier-Young,BlumenthalYoung2019,DragicevicFroylandGonzalezTokmanVaienti2020TAMS,Alves2023,huang2016ergodic,liu2024exponential} and the references therein.

For random expanding maps and symbolic systems, thermodynamic formalism is often obtained through Birkhoff cone-contraction arguments for random transfer operators. For hyperbolic diffeomorphisms, however, the deterministic spectral theory is usually based on anisotropic Banach spaces adapted to the stable and unstable directions \cite{GouezelLiverani2006,BaladiTsujii2007,Baladi2000}. These methods are powerful, but they do not directly provide a random cone-contraction framework for transfer-operator cocycles. This is the gap addressed here.

In this paper, we work under two hyperbolicity and mixing assumptions. Hypothesis~\ref{hyp:h} requires uniform hyperbolicity along the fibres, expressed through deterministic invariant cone fields, with one-dimensional stable direction. It also assumes a fibrewise topological mixing condition formulated in terms of a random mixing time: images of local unstable manifolds are required to become dense in the fibre, but the time at which this happens may depend on the base point $\omega$. Hypothesis~\ref{hyp:h'} strengthens this condition by requiring an exponential tail estimate for the successive fibrewise mixing times.

Under Hypothesis~\ref{hyp:h}, we introduce adapted projective cones for the random Perron--Frobenius cocycle associated with the system. The cones are defined through stable leaves and unstable holonomies, and their construction is based on \cite{viana1997stochastic,liu2024exponential}. The cocycle acts on these cones by Hilbert-metric contractions, yielding a quenched spectral decomposition. This allows us to construct a $\mathbb P$-relative equilibrium state $\upsilon_\phi$ for every uniformly H\"older random potential $\phi:\Omega\times M\to\mathbb R$.

Our first main result shows that this equilibrium state is unique among all $F$-invariant probability measures whose projection to the base is $\mathbb P$. Equivalently, $\upsilon_\phi$ is the unique maximiser of
$h_\mu(F\mid\mathbb P)+\int\phi\,\d\mu$.

Our second main result is quenched exponential decay of correlations for $\upsilon_\phi$. More explicitly, for $\mathbb P$-almost every $\omega$, there exist $\Lambda\in(0,1)$ and a measurable constant $C(\omega)>0$ such that, for all $f,g\in \mathcal C^\beta(M)$ and all $n\geq1$,
\begin{align*}
\left|
\int_M f\circ T_\omega^n\, g\,\d \upsilon_\omega
-
\int_M f\,\d \upsilon_{\theta^n\omega}
\int_M g\,\d \upsilon_\omega
\right|
\leq
C(\omega)\Lambda^n
\|f\|_{\mathcal C^\beta(M)}\|g\|_{\mathcal C^\beta(M)}.
\end{align*}
Thus the decay is quenched in the sense that the estimate holds fibrewise, along almost every realisation of the base, with constants depending on $\omega$. Under Hypothesis~\ref{hyp:h}, the constant is measurable. Under the stronger Hypothesis~\ref{hyp:h'}, it is shown that the constant $C\in L^p(\Omega,\mathbb P)$ for any $p\in[1,\infty)$. Uniform fibrewise mixing, as in \cite{huang2016ergodic,liu2024exponential}, corresponds to the special case in which the mixing time is uniformly bounded.

The paper is organised as follows. In Section~2 we introduce the setting, the relative notions of entropy, pressure and equilibrium state, and state the main theorems. Section~3 presents examples satisfying the hypotheses. Section~4 recalls the required geometric results for hyperbolic random dynamical systems, including stable and unstable manifolds, local product structure and holonomies. Section~5 constructs the adapted projective cones and the corresponding Hilbert metrics. In Section~6 we define the random Perron--Frobenius operator and prove the quenched spectral decomposition, which is then used to construct the candidate equilibrium state and establish decay of correlations. Section~7 proves the variational principle and the weak Gibbs property. Section~8 proves uniqueness of the equilibrium state. Section~9 proves the main theorems.

\section{Setup and main results}

Throughout the paper, we fix a compact, connected, smooth Riemannian manifold $M$. We write $\|\cdot\|$ for the norm on $TM$ induced by the Riemannian metric, and $d$ for the corresponding distance on $M$. We also fix a compact metric space $\Omega$. We consider skew products of the form
\begin{align*}
F:\Omega\times M&\to\Omega\times M\\
\left(\omega,x\right)&\mapsto\left(\theta\left(\omega\right),T_\omega(x)\right),
\end{align*}
under the following standing assumptions:
\begin{itemize}
\item $F$ is a homeomorphism;
\item $\theta:\Omega\to\Omega$ is a homeomorphism, and $\mathbb P$ is a fixed ergodic $\theta$-invariant Borel probability measure on $\Omega$;
\item for each $\omega\in\Omega$, the map $T_\omega:M\to M$ is a $\mathcal C^{2}$ diffeomorphism; 
\item $\sup_{\omega\in \Omega}\|T_\omega\|_{\mathcal C^2} <\infty$ and $\sup_{\omega\in \Omega}\|(T_\omega)^{-1}\|_{\mathcal C^2} <\infty$. 
\end{itemize}
We refer to a skew product satisfying these assumptions as a \emph{regular random dynamical system}.

For $n\in\mathbb Z$ and $\left(\omega,x\right)\in\Omega\times M$, we define the fibre iterates by
$$
T_\omega^n\left(x\right)=\begin{cases}
    T_{\theta^{n-1}\omega}\circ\cdots\circ T_{\theta\omega}\circ T_\omega\left(x\right),&\text{if }n\in\mathbb N\\
    x,& \text{if }n=0,\\
    (T_{\theta^n\omega})^{-1}\circ \ldots \circ (T_{\theta^{-2}\omega})^{-1}\circ (T_{\theta^{-1}\omega})^{-1}(x),&\text{if }n\in \mathbb Z_{<0}
\end{cases}.
$$
Let $\proj_M:\Omega\times M\to M$ and $\proj_\Omega:\Omega\times M\to\Omega$ denote the coordinate projections:
$$ \proj_M\left(\omega,x\right)=x, \proj_\Omega\left(\omega,x\right)=\omega. $$
Then $T_\omega^n\left(x\right)=\proj_M\left(F^n\left(\omega,x\right)\right)$ for all $n\in\mathbb N$. Unless stated otherwise, all measures are Borel probability measures.

\subsection{Relative entropy and relative equilibrium states}
In this section we recall the notion of equilibrium states. We start by recalling the concept of relative entropy.

\begin{definition}[Relative metric entropy]\label{def:relativeentro}
Let $F:\Omega\times M\to\Omega\times M$ be a regular random dynamical system, and let $\mu$ be an $F$-invariant probability measure such that $(\proj_\Omega)_*\mu=\mathbb P$. We define the \emph{$\mathbb P$-relative metric entropy of $\mu$ with respect to $F$} by
$$
h_\mu(F\mid\mathbb P)
:=
\sup_{\mathcal P}
\lim_{n\to\infty}
\frac{1}{n}
\int_\Omega
H_{\mu_\omega}
\left(
\bigvee_{k=0}^{n-1}
(T_\omega^k)^{-1}\mathcal P(\theta^k\omega)
\right)
\mathbb P(\d\omega),
$$
where the supremum is taken over all finite measurable partitions $\mathcal P$ of $\Omega\times M$.

More explicitly, if $\mathcal P=\{P_1,\ldots,P_r\}$, then for each $\omega\in\Omega$ we define
$$
P_i(\omega):=\{x\in M:(\omega,x)\in P_i\} \ \text{and}\ \mathcal P(\omega):=\{P_1(\omega),\ldots,P_r(\omega)\}.
$$
Hence, $\mathcal P(\omega)$ is a finite measurable partition of $M$, up to $\mu_\omega$-null sets, and
$$
(T_\omega^k)^{-1}\mathcal P(\theta^k\omega)
=
\{(T_\omega^k)^{-1}P_i(\theta^k\omega):1\le i\le r\}.
$$

The family $\{\mu_\omega\}_{\omega\in\Omega}$ denotes the disintegration of $\mu$ over $\mathbb P$, that is, $\mu(\d\omega,\d x)=\mu_\omega(\d x)\mathbb P(\d\omega)$
see \cite[Theorem 5.1.11]{Viana2016}. For a finite partition $\mathcal Q$ we set
$$
H_\nu(\mathcal Q)
=
\sum_{Q\in\mathcal Q}
-\nu(Q)\log\nu(Q).
$$
Since $\theta$ is $\mathbb P$-ergodic, the limit exists $\mathbb P$-almost surely and agrees with the integrated value above, see \cite[Page 383]{kifer2006random}.
\end{definition}

\begin{definition}[Relative pressure and relative equilibrium states]
Let $F:\Omega\times M\to\Omega\times M$ be a regular random dynamical system and let $\phi:\Omega\times M\to\mathbb R$ be a measurable function. The \emph{$\mathbb P$-relative topological pressure} of $\phi$ with respect to $F$ is defined, via the variational principle, by
$$
P_{\mathrm{top}}\left(F,\phi\mid\mathbb P\right)
:=\sup\left\{h_\nu\left(F\mid\mathbb P\right)+\int\phi\,\mathrm d\nu:\ \nu\ \text{is }F\text{-invariant and }\left(\proj_\Omega\right)_*\nu=\mathbb P\right\}.
$$
A probability measure $\mu$ on $\Omega\times M$ is called a \emph{$\mathbb P$-relative equilibrium state} for $\phi$ if $\mu$ is $F$-invariant, $\left(\proj_\Omega\right)_*\mu=\mathbb P$, and it attains the above supremum, i.e. 
$$
h_\mu\left(F\mid\mathbb P\right)+\int\phi\,\d\mu
=P_{\mathrm{top}}\left(F,\phi\mid\mathbb P\right).
$$

\end{definition}

The following theorem is a classic result in Random Dynamical Systems (see \cite[Corollary 1.2.8 and Theorem 1.2.13]{kifer2006random}, see also \cite{Kifer2001,bogenschutz1992entropy}). Such an alternative characterisation of topological pressure will be useful in Section \ref{sec:wgibbs}.

\begin{proposition}[Spanning-set formula for the relative pressure]\label{prop:psaning}
Let $F:\Omega\times M\to\Omega\times M$ be a regular random dynamical system. Let $\phi:\Omega\times M\to\mathbb R$ be such that, for $\mathbb P$-almost every $\omega\in\Omega$, the map
$x\in M\mapsto \phi(\omega,x)\in\mathbb R$ 
is continuous and $\mathbb E\left[ \|\phi(\omega,\cdot)\|_\infty\right]<\infty$.

For $\omega\in\Omega$, $n\in\mathbb N$ and $x,y\in M$, define
\begin{align}
d_n^\omega(x,y):=\max_{0\le k<n} d(T_\omega^k(x),T_\omega^k(y))
\text{ and }
S_n\phi(\omega,x):=\sum_{k=0}^{n-1}\phi(F^k(\omega,x)).\label{eq:dnomega}\end{align}
A set $D\subset M$ is called $(\omega,n,\varepsilon)$-spanning if for every $x\in M$ there exists $y\in D$ such that $
d_n^\omega(x,y)\le\varepsilon.$
Set
$$
Z_0(n,\e, \omega)
:=\inf\left\{\sum_{y\in D} e^{S_n\phi(\omega,y)}:\ D\subset M\ \text{is }(\omega,n,\varepsilon)\text{-spanning}\right\}.
$$
Then
\begin{align*}
P_{\mathrm{top}}(F,\phi\mid\mathbb P)
&=\lim_{\varepsilon\to0}\limsup_{n\to\infty}\int_\Omega \frac1n \log Z_0(n,\e, \omega)\,\mathbb P(\d \omega) \\
&= \lim_{\varepsilon\to0}\liminf_{n\to\infty}\int_\Omega \frac1n  \log Z_0(n,\e, \omega)\,\mathbb P(\d\omega)  .
\end{align*}
Moreover, if $\mathbb P$ is $\theta$-ergodic, then the above equation remains true $\mathbb P$-almost surely without taking the integral.
\end{proposition}

\subsection{Main Theorems}

Below, we state the main hypotheses on $F$. Hypothesis \ref{hyp:h} assumes a deterministic pair of invariant cone fields on $TM$ giving uniform hyperbolicity for the fibre maps $T_\omega$, and a fibrewise topological mixing condition for $F$. Hypothesis \ref{hyp:h'} strengthens Hypothesis \ref{hyp:h} by requiring an exponential tail bound for the fibrewise mixing time, that is, an exponential large deviation estimate for the time needed for the iterates $T_\omega^n$ to achieve the fibrewise mixing property.

\begin{hypothesis}[H]\label{hyp:h}
In the above notation, we say that $F$ satisfies Hypothesis \ref{hyp:h} if the following conditions hold:
\begin{enumerate}
\item[{\bf(H1)}]
There exists a continuous deterministic family of cone fields
$$
\mathcal C=\left(\mathcal C^-(x),\mathcal C^+(x)\right)_{x\in M},\ \mathcal C^\pm(x)\subset T_xM,
$$
such that for every $x\in M$ and for $\mathbb P$-almost every $\omega\in\Omega$,
$$
DT_{\omega}^{-1}(x)\mathcal C^+(x)\subset \mathring{\mathcal C}^+\left(T_{\omega}^{-1}x\right),
DT_\omega(x)\mathcal C^-(x)\subset \mathring{\mathcal C}^-\left(T_\omega x\right).
$$
Moreover, defining
$$
E^s(\omega,x)=\bigcap_{n\in\mathbb N}DT^{-n}_{\theta^n\omega}\left(T_\omega^n(x)\right)\mathcal C^+\left(T_\omega^n(x)\right),
$$
$$
E^u(\omega,x)=\bigcap_{n\in\mathbb N}DT^{n}_{\theta^{-n}\omega}\left(T_\omega^{-n}(x)\right)\mathcal C^-\left(T_\omega^{-n}(x)\right),
$$
we assume that $$\dim E^s(\omega,x)=1\ \text{for every }x\in M\text{ and for }\mathbb P\text{-almost every }\omega\in\Omega.$$ Finally, there exists a constant $\lambda_0>0$ such that for every $n\in\mathbb N$, every $x\in M$, and for $\mathbb P$-almost every $\omega\in\Omega$,
$$
\|DT_\omega^n(x)v\|\le e^{-\lambda_0 n}\|v\|,\ \forall v\in E^s(\omega,x)
$$
and
$$
\|DT_\omega^{-n}(x)v\|\le e^{-\lambda_0 n}\|v\|,\ \forall v\in E^u(\omega,x).
$$

\item[{\bf(H2)}]
For every sufficiently small $\e>0$ and every $\delta>0$, there exists a constant $B=B(\delta,\e)\geq 1$ such that the stopping time
\begin{align}
N(\omega)&:=\inf\left\{n\in\mathbb N:
T_\omega^n(W_\e^u(\omega,x))\text{ is }\delta\text{-dense in }M\ \text{for every}\ x\in M
\right\}\label{eq:N}
\end{align}
satisfies $\mathbb P\left[N\leq B\right]>0.$ The local unstable manifold $W_\e^u(\omega,x)$ is defined in Definition \ref{def:localmanifolds}.
\end{enumerate}
\end{hypothesis}

\begin{remark}\label{rem:adaptednorm}
Assume that the exponential estimates in \textbf{(H1)} hold (possibly with an additional constant $C_0\geq 1$), namely
$$
\|DT_\omega^n(x)v\|\leq C_0 e^{-\lambda_0 n}\|v\|,\  \forall v\in E^s(\omega,x),
$$
and
$$
\|DT_\omega^{-n}(x)v\|\leq C_0 e^{-\lambda_0 n}\|v\|,\ \forall v\in E^u(\omega,x),
$$
for every $n\in\mathbb N$, every $x\in M$, and for $\mathbb P$-almost every $\omega\in\Omega$.
Then one may replace the ambient Riemannian norm $\|\cdot\|$ by an equivalent (adapted) norm $\|\cdot\|'$ on $TM$
for which the same estimates hold with $C_0=1$, that is,
$$
\|DT_\omega^n(x)v\|'\leq e^{-\lambda_0 n}\|v\|',\  \forall v\in E^s(\omega,x),
$$
and
$$
\|DT_\omega^{-n}(x)v\|'\leq e^{-\lambda_0 n}\|v\|',\  \forall v\in E^u(\omega,x).
$$
In particular, after passing to the adapted norm, the formulation of \textbf{(H1)} may be assumed with $C_0=1$
without loss of generality (see \cite[Proposition 4.2]{Shub1987}). 
\end{remark}

\begin{hypothesis}[H']\label{hyp:h'}
We say that $F$ satisfies Hypothesis \ref{hyp:h'} if $F$ satisfies Hypothesis \ref{hyp:h} and, in addition, the following holds:
\begin{itemize}
\item[{\bf(H2')}] Let $N$ be as in Hypothesis \ref{hyp:h} for the pair $(\delta,\e)$ with $\e>0$ small enough, and define
$$N_k(\omega) := \begin{cases}
    N(\omega),\ \text{if }k=1\\
    N(\theta^{N_{k-1}(\omega)}\omega) + N_{k-1}(\omega) ,\ \text{if }k\in\mathbb N\setminus \{1\}.
\end{cases} $$
\end{itemize}
Then, there exists $c=c(\delta,\e)$, $I=I(\delta,\e)>0$ and $\iota=\iota(\delta,\e)>0$ such that
$$\mathbb P[N_k(\omega)\geq c k] \leq I e^{-\iota k} \ \text{for every }k\in\mathbb N.$$
\end{hypothesis}

The class of potentials to be considered is defined as follows.

\begin{definition}[Uniform $\beta$-H\"older potential]\label{def:potential}
Let $\beta>0$. A measurable function $\phi:\Omega\times M\to\mathbb R$ is called a \emph{uniform $\beta$-H\"older potential} if $
\phi\in L^\infty\left(\Omega;\mathcal C^\beta\left(M\right)\right),$
that is, if the fibre functions  $\phi_\omega\left(x\right)=\phi\left(\omega,x\right)$ lie in $\mathcal C^\beta\left(M\right)$ for $\mathbb P$-almost every $\omega$ and
$$
\|\phi\|_{L^\infty\left(\Omega;C^\beta\left(M\right)\right)}
=\mathrm{ess}\sup_{\omega\in\Omega}\|\phi_\omega\|_{\mathcal C^\beta}<\infty,
$$
where
$$
\|f\|_{\mathcal C^\beta}=\sup_{x\in M}\left|f\left(x\right)\right|+[f]_\beta,\ 
\text{and}\ 
[f]_\beta=\sup_{\substack{x,y\in M\\ x\neq y}}\frac{\left|f\left(x\right)-f\left(y\right)\right|}{d\left(x,y\right)^\beta}.
$$
\end{definition}
 
It is also useful to define the following spaces of probabilities. 
\begin{definition}[Random invariant probabilities] Given a topological space $X$, we denote $\mathcal M_1(X)$ as the space of Borel probability measures on $X$. Moreover, given a regular random dynamical system we define $\mathcal M_1(F\mid \mathbb P)$ as the $F$-invariant probability measures which satisfy $(\proj_\Omega)_*\mu = \mathbb P$. We mention that each $\mu\in \mathcal M_1(F\mid \mathbb P)$ can be written as $\mu(\d \omega,\d x) = \mu_\omega(\d x) \mathbb P(\d \omega)$ (see \cite[Proposition 3.3]{crauel2002random} and \cite[Section 1.4]{arnold1995random}).
\end{definition}

The first main theorem concerns the existence and uniqueness of quenched equilibrium states for uniform H\"{o}lder potentials under Hypothesis \ref{hyp:h}, whereas the second main theorem states that Hypothesis \ref{hyp:h'} implies quenched exponential decay of correlations for uniform H\"{o}lder potentials.

\begin{mytheorem}[A]\label{theorem:mainthmA} Let $F:\Omega\times M\to \Omega\times M$ be a regular random dynamical system satisfying Hypothesis \ref{hyp:h} and $\phi:\Omega\times M\to \R$ be a uniform $\beta$-Hölder potential for some $\beta>0$. Then, there exists a unique $\mathbb P$-relative equilibrium state $\upsilon \in \mathcal M_1(F\mid \mathbb P)$ for $\phi$. 
\end{mytheorem}

Theorem \ref{theorem:mainthmA} is proved in Section \ref{sec:maintheorems}.

\begin{mytheorem}[B]\label{theorem:mainthmB}
Let $F:\Omega\times M\to\Omega\times M$ be a regular random dynamical system satisfying Hypothesis \ref{hyp:h}. Let $\phi:\Omega\times M\to\mathbb R$ be a uniform $\beta$-H\"older potential for some $\beta>0$, and let $\upsilon$ be the unique $\mathbb P$-relative equilibrium state for $\phi$, with disintegration $\upsilon(\d \omega,\d x) = \upsilon_\omega(\d x)\mathbb P(\d \omega)$. Then there exists $\Lambda\in(0,1)$ and a measurable function $C:\Omega\to \mathbb R$  such that for all $f,g\in \mathcal C^\beta\left(M\right)$,
$$
\left|\int_M f\circ T_{\omega}^n\, g\,\mathrm \d\upsilon_{\omega}
-\int_M f\,\mathrm \d\upsilon_{\theta^n\omega}\int_M g\,\mathrm \d\upsilon_{\omega}\right|
\le C\left(\omega\right)\Lambda^n\|f\|_{ \mathcal C^\beta}\|g\|_{\mathcal C^\beta}
$$
and
$$
\left|\int_M f\circ T_{\theta^{-n}\omega}^n\, g\,\mathrm \d\upsilon_{\theta^{-n}\omega}
-\int_M f\,\mathrm \d\upsilon_{\omega}\int_M g\,\mathrm \d\upsilon_{\theta^{-n}\omega}\right|
\le C\left(\omega\right)\Lambda^n\|f\|_{ \mathcal C^\beta}\|g\|_{\mathcal C^\beta}.
$$

If, in addition, we assume Hypothesis \ref{hyp:h'}.
Then for every $p\in[1,\infty)$, there exist $\Lambda_p\in\left(0,1\right)$ and $C_p\in L^p(\Omega,\mathbb P)$ such that for every $n\in\mathbb N$, and for all $f,g\in \mathcal C^\beta\left(M\right)$,
$$
\left|\int_M f\circ T_{\omega}^n\, g\,\mathrm \d\upsilon_{\omega}
-\int_M f\,\mathrm \d\upsilon_{\theta^n\omega}\int_M g\,\mathrm \d\upsilon_{\omega}\right|
\le C_p\left(\omega\right)\Lambda_p^n\|f\|_{ \mathcal C^\beta}\|g\|_{\mathcal C^\beta}
$$
and
$$
\left|\int_M f\circ T_{\theta^{-n}\omega}^n\, g\,\mathrm \d\upsilon_{\theta^{-n}\omega}
-\int_M f\,\mathrm \d\upsilon_{\omega}\int_M g\,\mathrm \d\upsilon_{\theta^{-n}\omega}\right|
\le C_p\left(\omega\right)\Lambda_p^n\|f\|_{ \mathcal C^\beta}\|g\|_{\mathcal C^\beta}.
$$

\end{mytheorem}

Theorem \ref{theorem:mainthmB} is proved in Section \ref{sec:maintheorems}.

\begin{remark}\label{rmk:anosovonfibres}
The topological mixing on fibres condition used in \cite{liu2024exponential} and \cite{huang2016ergodic} is stronger than \textbf{(H2')}. Indeed, in those works the mixing time is uniform in the fibre parameter: after reducing the condition to the finite collection of sets used in Hypothesis \ref{hyp:h'}, one obtains a constant $B>0$ such that $\mathbb P[N\leq B]=1$ (in the notation of Hypothesis $\mathbf{(H2)})$. Hence \textbf{(H2')} holds with $c = 2B$. Moreover, in this uniformly mixing situation, the conclusion of Theorem \ref{theorem:mainthmB} can be strengthened: the multiplicative constant in the exponential decay of correlations can be chosen uniformly bounded in $\omega$. In particular, for the SRB measure considered in \cite{liu2024exponential}, the measures $\upsilon_\omega$ are defined for every $\omega\in\Omega$, and the exponential decay of correlations for H\"older observables holds with a uniform constant. We will not prove such a strengthened result, but it follows naturally from the proof presented.
\end{remark}

We mention that our formulation of Hypotheses \ref{hyp:h} and \ref{hyp:h'} assumes that the stable direction is one-dimensional, while the unstable direction may have arbitrary dimension. Theorems \ref{theorem:mainthmA} and \ref{theorem:mainthmB} can also be proved in the dual situation where $\dim E^u(\omega,x)=1$ and $\dim E^s(\omega,x)$ is arbitrary. In that case, the same proof strategy is applied to the inverse skew product. More precisely, one considers
$$
F^{-1}(\omega,x)=\left(\theta^{-1}\omega,(T_{\theta^{-1}\omega})^{-1}(x)\right)=\left(\theta^{-1}\omega,\widehat T_\omega(x)\right),
$$
where $\widehat T_\omega=(T_{\theta^{-1}\omega})^{-1}.$
Thus the inverse dynamics is a cocycle over $\theta^{-1}$ generated by the family $\widehat T_\omega$. This exchanges the roles of the stable and unstable directions. Accordingly, condition \textbf{(H2)} must be replaced by the analogous requirement that stable manifolds of length $\varepsilon$ become $\delta$-dense under the inverse dynamics, with the corresponding exponential tail condition replacing \textbf{(H2')}. The geometric and cone arguments used below then apply in the same way after reversing the time direction.

\section{Examples}

In this section, we provide three examples of regular random dynamical systems $F$ that satisfy Hypothesis \ref{hyp:h'}.

\begin{example}\label{ex:An}
Let $N\in\mathbb N$, and let $A_1,\ldots,A_N:\mathbb T^2\to\mathbb T^2$ be toral automorphisms of the form
$$
A_i
\begin{pmatrix}
x\\
y
\end{pmatrix}
=
\begin{pmatrix}
a_i & b_i\\
c_i & d_i
\end{pmatrix}
\begin{pmatrix}
x\\
y
\end{pmatrix}
\mod 1,
$$
where $a_i,b_i,c_i,d_i\in\mathbb N$, $\det A_i=1$, and $\operatorname{tr}(A_i)>2$ for every $i\in\{1,\ldots,N\}$. Since $\det A_i=1$ and $\operatorname{tr}(A_i)>2$, each $A_i$ is a hyperbolic toral automorphism. In particular, $A_i$ is Anosov for every $i\in\{1,\ldots,N\}$.

Consider now a topologically mixing subshift of finite type
$$
\Sigma_Q=\{\omega\in\{1,\ldots,N\}^{\mathbb Z}:Q_{\omega_j,\omega_{j+1}}=1\text{ for every }j\in\mathbb Z\},
$$
where $Q$ is an $N\times N$ matrix with entries in $\{0,1\}$. Let $\theta:\Sigma_Q\to\Sigma_Q$ be the left shift, and let $\mathbb P$ be any ergodic $\theta$-invariant Borel probability measure on $\Sigma_Q$. Define the skew product
$$
F:\Sigma_Q\times\mathbb T^2\to\Sigma_Q\times\mathbb T^2,\ 
F(\omega,x)=(\theta\omega,A_{\omega_0}x).
$$
Since $\theta$ is a homeomorphism of the two-sided subshift and each $A_i$ is a toral automorphism, $F$ is a homeomorphism. Moreover, the map $\omega\mapsto A_{\omega_0}$ is locally constant.

We claim that $F$ satisfies Hypothesis \ref{hyp:h'}. First, $F$ satisfies \textbf{(H1)}. Indeed, the matrices $A_i$ are positive hyperbolic elements of $\mathrm{SL}(2,\mathbb N)$, and therefore preserve the standard unstable and stable cone fields
$$
\mathcal C^-=\{(u,v)\in\mathbb R^2:uv\geq 0\},\ 
\mathcal C^+=\{(u,v)\in\mathbb R^2:uv\leq 0\}.
$$
Since the family $\{A_1,\ldots,A_N\}$ is finite, the corresponding hyperbolicity constants may be chosen uniformly. Hence the associated stable and unstable directions satisfy the uniform cone and exponential estimates required in \textbf{(H1)} (see \cite{Liverani2009} for more details).

By \cite[Appendix A.1]{liu2024exponential}, the present random composition of positive area-preserving $2\times2$ matrices is topologically mixing on fibres, meaning that for every pair of non-empty open sets $U,V\subset\mathbb T^2$, there exists $n_0=n_0(U,V)\in\mathbb N$ such that
$$
T_\omega^n(U)\cap V\neq\varnothing
$$
for every $n\geq n_0$ and every $\omega\in\Sigma_Q$. By Remark \ref{rmk:anosovonfibres}, topological mixing on fibres is stronger than \textbf{(H2')}: for every sufficiently small $\varepsilon>0$ and every $\delta>0$, there exists a deterministic constant $B=B(\delta,\varepsilon)>0$ such that the stopping time $N$ in \textbf{(H2)} satisfies $
\mathbb P[N\leq B]=1.$
Hence the exponential tail required in \textbf{(H2')} holds automatically. Consequently, $F$ satisfies Hypothesis \ref{hyp:h'}.

Moreover, by Remark \ref{rmk:anosovonfibres}, the conclusion is stronger in this uniformly mixing case: the multiplicative constant in the quenched exponential decay of correlations can be chosen uniformly in $\omega$.
\end{example}

\begin{example}
Let $T_1,\ldots,T_k:\mathbb T^2\to\mathbb T^2$ be $\mathcal C^2$ Anosov diffeomorphisms preserving the same deterministic cone system. More precisely, assume that there exist continuous cone fields $\mathcal C^u$ and $\mathcal C^s$ on $\mathbb T^2$ such that, for every $i\in\{1,\ldots,k\}$,
$$
DT_i(x)\mathcal C^u(x)\subset\operatorname{int}\mathcal C^u(T_i x),
\ 
DT_i^{-1}(x)\mathcal C^s(x)\subset\operatorname{int}\mathcal C^s(T_i^{-1}x),
$$
and assume that the corresponding expansion and contraction estimates are uniform over the finite family. For instance, this applies when $T_2,\ldots,T_k$ are sufficiently small $\mathcal C^1$ perturbations of $T_1^{m_2},\ldots,T_1^{m_k}$ for some $m_i\in\mathbb N$, so that the same invariant cone fields persist.

Let $\Sigma_Q\subset\{1,\ldots,k\}^{\mathbb Z}$ be a topologically mixing subshift of finite type, let $\theta:\Sigma_Q\to\Sigma_Q$ be the left shift, and let $\mathbb P$ be a $\theta$-invariant equilibrium state for a Hölder potential (recall that this implies that $\mathbb P$ has full supported \cite[Theorem 1.16]{Bowen1975}). Define
$$
F:\Sigma_Q\times\mathbb T^2\to\Sigma_Q\times\mathbb T^2,
\ 
F(\omega,x)=(\theta\omega,T_{\omega_0}x).
$$
Observe that $F$ is a regular random dynamical system. We prove that $F$ satisfies Hypothesis \ref{hyp:h'}. The condition \textbf{(H1)} follows immediately from the assumptions already stated. It remains to verify \textbf{(H2')}. We do this in three steps.

\begin{step}[1]
Let $\varepsilon>0$ and $\delta>0$ be sufficiently small. We construct a finite admissible word $v$ in the shift $\Sigma_Q$ of length $\ell$ and the cylinder
\begin{align}
G=[v]_0^{\ell-1}
=
\{\omega\in\Sigma_Q:\omega_0\omega_1\cdots\omega_{\ell-1}=v\}.
\label{eq:Gdef-general}
\end{align}
Then, denoting $\tau_G(\omega):=\min\{r\ge 0:\theta^r\omega\in G\}$, we prove that
\begin{equation}
N(\omega)\le\tau_G(\omega)+\ell.
\label{eq:N-leq-tau-plus-ell-general}
\end{equation}
\end{step}

Choose an admissible periodic word $w=(a_0,\ldots,a_{m-1})$ in $\Sigma_Q$, and set
$$
T_w:=T_{a_{m-1}}\circ\cdots\circ T_{a_0}.
$$
By the common cone assumptions, $T_w$ is an Anosov diffeomorphism of $\mathbb T^2$. Since every Anosov diffeomorphism of $\mathbb T^2$ is topologically mixing, the inclination lemma \cite[Proposition 6.2.23]{Katok1995} implies that there exists $q=q(\delta,\varepsilon)\in\mathbb N$ such that, for every local unstable curve $\gamma\subset\mathbb T^2$ of length $\varepsilon$, the image $T_w^q(\gamma)$ is $\delta$-dense in $\mathbb T^2$.

Set
$$
v=\underbrace{w\cdots w}_{q\text{ times}},
\ 
|v|=qm=:\ell,
$$
and define the cylinder $G$ as in \eqref{eq:Gdef-general} and $\tau_G$ as in the Step 1 statement. Since $\mathbb P$ has full support, $\mathbb P(G)>0$. If $\theta^r\omega\in G$, then the fibre map from time $r$ to time $r+\ell$ is exactly $T_w^q$. Moreover, the maps $T_i$ send local unstable curves to local unstable curves and expand them uniformly. Hence, for every local unstable curve $\gamma$ of length $\varepsilon$, the image $T_\omega^r(\gamma)$ contains a local unstable subcurve of length $\varepsilon$. Applying $T_w^q$ to this subcurve gives a $\delta$-dense subset of $\mathbb T^2$. Therefore $T_\omega^{r+\ell}(\gamma)$ is $\delta$-dense in $\mathbb T^2$ for every local unstable curve $\gamma$ of length $\varepsilon$. Taking $r=\tau_G(\omega)$ gives \eqref{eq:N-leq-tau-plus-ell-general}.

\begin{step}[2]
For the cylinder $G$ constructed in Step 1, we prove that there exist $h\in\mathbb N$ and $p>0$ such that, whenever $H\in\mathcal B(\Sigma_Q)$ is a finite union of cylinders determined by coordinates $0,\ldots,L-1$, one has, for every $n$ large enough so that $n-L\ge h-\ell$:
\begin{equation}
\mathbb P(H\cap\theta^{-n}G)\ge p\,\mathbb P(H)\text{ which implies }\mathbb P[\theta^{-n}G\mid H] = \frac{\mathbb P[\theta^{-n}G\cap H]}{\mathbb P[H]} \geq p.
\label{eq:uniform-hit-deterministic-general}
\end{equation}
\end{step}

By \cite[Proposition 1.14]{Bowen1975}, there exist constants $K_0>0$ and $\gamma>0$ such that, if $C$ and $D$ are cylinders and $C$ is determined by coordinates $0,\ldots,L-1$, then
\begin{align}
\left|
\mathbb P(C\cap\theta^{-n}D)-\mathbb P(C)\mathbb P(D)
\right|
\le
K_0e^{-\gamma(n-L)}\mathbb P(C)\mathbb P(D)
\label{eq:bowen-mixing-general}
\end{align}
for every $n\ge L$.

Choose $n_0\in\mathbb N$ such that $K_0e^{-\gamma n_0}\le 1/2$, and set
$$
h:=\ell+n_0,
\ 
p:=\frac{\mathbb P(G)}{2}>0.
$$
Let $H\in\mathcal B(\Sigma_Q)$ be a finite union of cylinders determined by coordinates $0,\ldots,L-1$, and assume that $n-L\ge n_0$. Write $H=\bigsqcup_{a\in\mathcal A}C_a,$
where each $C_a$ is a cylinder determined by coordinates $0,\ldots,L-1$. Applying \eqref{eq:bowen-mixing-general} with $C=C_a$ and $D=G$, we get
\begin{align*}
\mathbb P(C_a\cap\theta^{-n}G)
&\ge
\left(1-K_0e^{-\gamma(n-L)}\right)\mathbb P(C_a)\mathbb P(G)\\
&\ge
\frac12\mathbb P(C_a)\mathbb P(G)
=
p\,\mathbb P(C_a).
\end{align*}
Summing over $a\in\mathcal A$, we obtain $\mathbb P(H\cap\theta^{-n}G)\ge p\,\mathbb P(H)$. This proves \eqref{eq:uniform-hit-deterministic-general}.

\begin{step}[3]
We show that there exist constants $c>0$, $K>0$, and $\kappa>0$ such that
$$
\mathbb P[N_i\ge ci]\le Ke^{-\kappa i}
$$
for every $i\ge 1$.
\end{step}

Let $G$ be as in Step 1 and $h,p$ as in Step 2. For $k\ge 0$, define $Y_k(\omega):=\mathbbm 1_G(\theta^{kh}\omega)$. Hence, $Y_k=1$ means that the orbit of $\omega$ hits $G$ at the sampled time $kh$.

We claim that, for every $J\ge 1$ and every word $(y_0,\ldots,y_{J-1})\in\{0,1\}^J$,
\begin{equation}
\mathbb P[Y_0=y_0,\ldots,Y_{J-1}=y_{J-1}]
\le
(1-p)^{\#\{0\le k\le J-1:y_k=0\}}.
\label{eq:pattern-bound-general}
\end{equation}
Indeed, for each $k\in\{0,\ldots,J-1\}$, set $H_k:=\{Y_0=y_0,\ldots,Y_{k-1}=y_{k-1}\}$. Then $H_k$ is a finite union of cylinders determined by coordinates $0,\ldots,(k-1)h+\ell-1$, and $\theta^{-kh}G$ is separated from these coordinates by a gap $h-\ell$. Hence, from \eqref{eq:uniform-hit-deterministic-general},
$$
\mathbb P[Y_k=1\mid Y_0= y_0,\ldots, Y_{k-1}=y_{k-1}]
=
\mathbb P[Y_k=1\mid H_k]
=
\mathbb P[\theta^{-kh}G\mid H_k]
\ge p,
$$
whenever $\mathbb P(H_k)>0$. Therefore
$$
\mathbb P[Y_k=0\mid Y_0=y_0,\ldots,Y_{k-1}=y_{k-1}]
\le 1-p,
$$
and trivially
$$
\mathbb P[Y_k=1\mid Y_0=y_0,\ldots,Y_{k-1}=y_{k-1}]
\le 1.
$$
The two inequalities above, multiplied for $k=0,\ldots,J-1$, imply \eqref{eq:pattern-bound-general}.

Choose $\alpha>2$ large enough so that $\rho:=e(\alpha+1)(1-p)^{\alpha-1}<1$. For $i\ge 1$, set $J_i:=\lceil\alpha i\rceil$. From \eqref{eq:pattern-bound-general}, $\alpha>2$ and the well-known bound $i!\geq(i/e)^i$ for every $i\in\mathbb N$, we have that
\begin{align}
\mathbb P\left[\sum_{k=0}^{J_i-1}Y_k<i\right]
&\le
\sum_{r=0}^{i-1}
\sum_{\substack{I\subset\{0,\ldots,J_i-1\}\\ |I|=r}}
\mathbb P\left[Y_k=1 \text{ for } k\in I,\ Y_k=0 \text{ for } k\notin I\right]\notag\\
&\le
\sum_{r=0}^{i-1}
\sum_{\substack{I\subset\{0,\ldots,J_i-1\}\\ |I|=r}}
(1-p)^{J_i-r}
=
\sum_{r=0}^{i-1}\binom{J_i}{r}(1-p)^{J_i-r}\notag\\
&\le
\sum_{r=0}^{i-1}\binom{J_i}{i}(1-p)^{J_i-r}
=
\sum_{r=0}^{i-1}\frac{J_i!}{i!(J_i-i)!}(1-p)^{J_i-r}
\le
\sum_{r=0}^{i-1}\frac{J_i^i}{i!}(1-p)^{J_i-r}\notag\\
&\le
\sum_{r=0}^{i-1}
\left(\frac{eJ_i}{i}\right)^i
(1-p)^{J_i-r}
\le
\sum_{r=0}^{i-1}
\left(e(\alpha+1)\right)^i
(1-p)^{J_i-r}\notag\\
&\le
\sum_{r=0}^{i-1}
\left(e(\alpha+1)\right)^i
(1-p)^{J_i-i+1}
\le
\sum_{r=0}^{i-1}
\left(e(\alpha+1)\right)^i
(1-p)^{(\alpha-1)i}\notag\\
&\le
i\left(e(\alpha+1)\right)^i(1-p)^{(\alpha-1)i}
\le
(i+1)\left(e(\alpha+1)(1-p)^{\alpha-1}\right)^i\notag\\
&=
(i+1)\rho^i.
\label{eq:few-sampled-hits-general}
\end{align}
Since $\rho<1$, from \eqref{eq:few-sampled-hits-general} we obtain that there exist constants $I>0$ and $\iota>0$ such that
\begin{equation}
\mathbb P\left[\sum_{k=0}^{J_i-1}Y_k<i\right]\le Ie^{-\iota i}
\label{eq:few-sampled-hits-exp-general}
\end{equation}
for every $i\ge 1$.

If $\sum_{k=0}^{J_i-1}Y_k\ge i$, let $0\le k_1<\cdots<k_i\le J_i-1$ be the first $i$ indices such that $Y_{k_r}=1$, and set $t_r=k_rh$. Then $\theta^{t_r}\omega\in G$ for every $r=1,\ldots,i$. Since $t_{r+1}-t_r\ge h\ge\ell$ for every $r$, the bound $N(\cdot)\le\tau_G(\cdot)+\ell$ can be applied successively at the times $t_1,\ldots,t_i$. Thus,
$$
N_r(\omega)\le t_r+\ell
$$
for every $r=1,\ldots,i$, which implies
\begin{align}
N_i(\omega)\le t_i+\ell\le J_ih+\ell\le(\alpha+1)hi+\ell.
\label{eq:Ni-bound-general}
\end{align}
Choose $c:=(\alpha+1)h+\ell+1$. From \eqref{eq:Ni-bound-general} we obtain
$$
\left\{\sum_{k=0}^{J_i-1}Y_k\ge i\right\}\subset\{N_i<ci\},\ \text{which implies }
\{N_i\ge ci\}\subset\left\{\sum_{k=0}^{J_i-1}Y_k<i\right\}.
$$
Using \eqref{eq:few-sampled-hits-exp-general}, we obtain $\mathbb P[N_i\ge ci]\le I e^{-\iota i}$. Thus \textbf{(H2')} holds, and therefore $F$ satisfies Hypothesis \ref{hyp:h'}.
\end{example}

\begin{example}
Let $\Omega$ be a compact metric space, let $\theta:\Omega\to\Omega$ be a homeomorphism, and let $\mathbb P$ be an ergodic $\theta$-invariant Borel probability measure. Let $T:\mathbb T^2\to\mathbb T^2$ be a mixing Anosov diffeomorphism. Define
$$
F:\Omega\times\mathbb T^2\to\Omega\times\mathbb T^2,\ 
F(\omega,x)=(\theta\omega,Tx).$$
Then $F$ is a regular random dynamical system and satisfies Hypothesis \ref{hyp:h'}.

In particular, Theorems \ref{theorem:mainthmA} and \ref{theorem:mainthmB} apply
to any uniformly H\"older random potential
$\phi:\Omega\times\mathbb T^2\to\mathbb R$. Notice that the potential plays no role
in Hypothesis \ref{hyp:h'}, which depends only on the underlying random dynamical
system. Moreover, by Remark \ref{rmk:anosovonfibres}, the random constants appearing
in Theorem \ref{theorem:mainthmB} can be chosen uniformly in $\omega$.
\end{example}

\section{Geometric theory of hyperbolic random dynamical systems}\label{sec:geom}

This section has two purposes. The first is to collect the geometric facts on fibre stable and unstable manifolds of $F$ that will be used later. The second is to fix the constants $\varepsilon$, $\varepsilon^*$, $\delta$ and a finite cover $\left\{B_M\left(\delta/4,x_i\right)\right\}_{i=1}^\ell$ appearing in the definition of the mixing time in Hypotheses \ref{hyp:h} and \ref{hyp:h'}. Throughout this section, we assume that $F:\Omega\times M\to\Omega\times M$ is a regular random dynamical system satisfying Hypothesis \ref{hyp:h}. We also fix $\lambda\in(0,\lambda_0]$, where $\lambda_0$ is the hyperbolicity exponent in \textbf{(H1)}. The results recalled below follow \cite[Section 3]{liu2024exponential} and \cite[Section 4.2]{kifer2006random} (see also \cite[Section 3.1]{huang2016ergodic}) . We start by defining the fibre local stable and unstable manifolds and recalling the result that guarantees their existence.

\begin{definition}[Fibre local stable and unstable manifolds]\label{def:localmanifolds}
For $\varepsilon>0$, $\omega\in\Omega$ and $x\in M$, we define the fibre local stable and unstable sets at $(\omega,x)$ of dynamical length $\e$ by
\begin{align*}
W_\varepsilon^s(\omega,x)
&:=\left\{y\in M:d\left(T_\omega^n (x),T_\omega^n(y)\right)\le\varepsilon\text{ for every }n\ge0\right\},\\
W_\varepsilon^u(\omega,x)
&:=\left\{y\in M:d\left( (T_\omega^n)^{-1}(x),(T_\omega^{n})^{-1}(y)\right)\le\varepsilon\text{ for every }n\ge0\right\}.
\end{align*}
Throughout the text, by abuse of terminology, we refer to $W_\varepsilon^s(\omega,x)$ and $W_\varepsilon^u(\omega,x)$ as local stable and unstable manifolds of length $\varepsilon$, respectively.
\end{definition}

\begin{proposition}[{\cite[Lemma 3.2]{liu2024exponential}}]\label{prop:localmanifolds}
There exists $\varepsilon_0>0$ such that for every $\varepsilon\in\left(0,\varepsilon_0\right]$, every $x\in M$ and $\mathbb P$-almost every $\omega\in\Omega$, 
\begin{itemize}
    \item the sets $W_\varepsilon^{u/s}\left(\omega,x\right)$  are $\mathcal C^2$ embedded discs;
    \item there exists $C_1,\nu_0>0$ such that for every $\omega\in \Omega$ fixed the maps 
    \begin{align*}\max\left\{
\sup_{\substack{v\in E^{u/s}(\omega,x)\\ \|v\|=1}}
\inf_{w\in E^{u/s}(\omega,y)}
\left\|
\frac{v}{\|v\|}-\frac{w}{\|w\|}
\right\|,
\,
\sup_{\substack{w\in E^{u/s}(\omega,y)\\ \|w\|=1}}
\inf_{v\in E^{u/s}(\omega,x)}
\left\|
\frac{v}{\|v\|}-\frac{w}{\|w\|}
\right\|
\right\}\\
\leq C_1d(x,y)^{\nu_0}.
\end{align*}
\item the maps $x\mapsto W_\e(\omega,x)$ is $C^{\nu_0}$-continuous in the Hausdorff topology.

\end{itemize}
\end{proposition}
\begin{remark}
From \textbf{(H1)}, the stable direction is one-dimensional, that is,
$\dim E^s(\omega,x)=1$ for every $x\in M$ and for $\mathbb P$-almost every $\omega\in\Omega$.
Consequently, whenever $\varepsilon>0$ is small enough so that the local manifold theorem applies,
$W_\varepsilon^s(\omega,x)$ is a one-dimensional embedded disk. No analogous restriction is imposed on
the unstable direction: $W_\varepsilon^u(\omega,x)$ has dimension $\dim E^u(\omega,x)$, which may be arbitrary.
\end{remark}

The following proposition gives the local product structure for the random system. This allows us to define the corresponding rectangles coming from the local product structure.
\begin{proposition}[{\cite[Lemma 3.3.]{liu2024exponential}}]\label{prop:productstructure}
For every $\varepsilon\in\left(0,\varepsilon_0\right]$ there exists $\delta=\delta\left(\varepsilon\right)\in\left(0,\varepsilon\right)$ such that, for every $x,y\in M$ with $d\left(x,y\right)<\delta$ and for $\mathbb P$-almost every $\omega\in\Omega$, the intersection $W_\varepsilon^s\left(\omega,x\right)\cap W_\varepsilon^u\left(\omega,y\right)$ consists of a single point, denoted by $\left[x,y\right]_\omega^\varepsilon$. The map 
$$\left(x,y\right)\mapsto\left[x,y\right]_\omega^\varepsilon$$
is continuous on $\left\{\left(x,y\right)\in\Omega\times M\times M:d\left(x,y\right)<\delta\right\}$.
\end{proposition}

\begin{definition}[Rectangles]\label{def:rectangles}
Fix $\varepsilon\in\left(0,\varepsilon_0\right]$ and set $\delta=\delta\left(\varepsilon\right)$ as in Proposition \ref{prop:productstructure}. For $\omega\in\Omega$ and $x_0\in M$, define the rectangle
$$
R_\delta\left(\omega,x_0\right)=\left\{\left[z_1,z_2\right]_\omega^\varepsilon:
z_1\in W_\delta^u\left(\omega,x_0\right),\, z_2\in W_\delta^s\left(\omega,x_0\right)\right\}.
$$
\end{definition}

To perform the analysis of Sections \ref{sec:projhilb} and \ref{sec:RPF}, it will be useful to define the geometric potential below and to record its basic properties.
\begin{definition}\label{def:jacobians}
For $\tau\in\left\{s,u\right\}$ define
$$
J_\tau\left(\omega,x\right)=\left|\det\left(\left(\mathrm DT_\omega\left(x\right)\right)\big|_{E^\tau\left(\omega,x\right)}\right)\right|,\ 
\phi^{J_\tau}\left(\omega,x\right)=-\log J_\tau\left(\omega,x\right).
$$
 
\end{definition}

\begin{proposition}[{\cite[Lemmas 3.4 and 3.5]{liu2024exponential}}]\label{prop:holderjacs}
There exist constants $C_{10}\geq 1$, $C_2>0$, $\nu_0\in\left(0,\beta\right)$, such that for any $\tau \in\{s,u\}$, for  $\mathbb P$-almost every $\omega\in\Omega$ and all $x,y\in M$,
$$
C_{10}^{-1}\le J_\tau\left(\omega,x\right)\le C_{10},
\ 
\left|J_\tau\left(\omega,x\right)-J_\tau\left(\omega,y\right)\right|\le C_2 d\left(x,y\right)^{\nu_0},
$$
and hence
$$
\left|\phi^{J_\tau}\left(\omega,x\right)-\phi^{J_\tau}\left(\omega,y\right)\right|\le C_2 d\left(x,y\right)^{\nu_0}.
$$
 
\end{proposition}

We next define the holonomy maps between nearby stable manifolds induced by unstable manifolds, and record their basic properties.
\begin{definition}[{\cite[Lemma~3.9]{liu2024exponential}}]
\label{def:holonomy}
Fix $\varepsilon\in(0,\varepsilon_0]$. For $\omega\in\Omega$, set
$\mathcal F_\omega^s=\{W_\varepsilon^s(\omega,x):x\in M\}.$
Given $\tilde\gamma_{\omega},\gamma_{\omega}\in\mathcal F_\omega^s$, we say that
$(\tilde\gamma_{\omega},\gamma_{\omega})$ is a nearby pair if the map
$$
\operatorname{hol}_\omega^u:\tilde\gamma_{\omega}\to\gamma_{\omega},
\ 
x\mapsto W_\varepsilon^u(\omega,x)\cap\gamma_{\omega},
$$
is well-defined.
\end{definition}

\begin{proposition}[{\cite[Proposition 3.11 and Equation (3.76)]{liu2024exponential}}]\label{prop:holjac}
Let $m_U$ denote the Riemannian volume measure induced on a submanifold $U$. For $\mathbb P$-almost every $\omega\in\Omega$ and every nearby pair $\left(\tilde\gamma\left(\omega\right),\gamma\left(\omega\right)\right)$, the pushforward $\left(\left(\operatorname{hol}_\omega^u\right)^{-1}\right)_*m_{\gamma\left(\omega\right)}$ is absolutely continuous with respect to $m_{\tilde\gamma\left(\omega\right)}$, and its Radon--Nikodym derivative satisfies
\begin{align*}
\operatorname{Jac}\left(\operatorname{hol}_\omega^u\right)\left(x\right)
&=\frac{\mathrm d\left[\left(\left(\operatorname{hol}_\omega^u\right)^{-1}\right)_*m_{\gamma\left(\omega\right)}\right]}{\mathrm d m_{\tilde\gamma\left(\omega\right)}}\left(x\right)=\lim_{n\to\infty}\frac{\left|\det\left(\left(\mathrm DT_\omega^{-n}\left(x\right)\right)\big|_{E^s\left(\omega,x\right)}\right)\right|}{\left|\det\left(\left(\mathrm DT_\omega^{-n}\left(\operatorname{hol}_\omega^u x\right)\right)\big|_{E^s\left(\omega,\operatorname{hol}_\omega^u x\right)}\right)\right|} \\ & = \prod_{j=1}^\infty \frac{e^{\phi^{J_s} \circ F^{-j} (\omega,\operatorname{hol}_\omega^u x ) } }{ e^{\phi^{J_s} \circ F^{-j} (\omega,x ) }}.
\end{align*}

\end{proposition}

The next proposition records how the Riemannian volume on $M$ disintegrates over rectangles, and the following definition introduces the admissible stable leaves.
\begin{proposition}[{\cite[Lemma 3.12]{liu2024exponential}}]\label{prop:holderhol}
There exist constants $a_0',C,J,\nu_0>0$ such that for $\mathbb P$-almost every $\omega\in\Omega$ and every nearby pair $\left(\tilde\gamma\left(\omega\right),\gamma\left(\omega\right)\right)$:
\begin{itemize}
\item the maps $\operatorname{hol}_\omega^u$ and $\log\operatorname{Jac}\left(\operatorname{hol}_\omega^u\right)$ are $\mathcal C^{\nu_0}$, and $J^{-1}\le\operatorname{Jac}\left(\operatorname{hol}_\omega^u\right)\le J$;
\item for all $x\in\tilde\gamma\left(\omega\right)$,
$$
\left|\log\operatorname{Jac}\left(\operatorname{hol}_\omega^u\right)\left(x\right)\right|\le a_0' d\left(x,\operatorname{hol}_\omega^u\left(x\right)\right)^{\nu_0};
$$
\item for all $x\in\tilde\gamma\left(\omega\right)$,
$$
d\left(T_\omega^{-1}x,T_\omega^{-1}\operatorname{hol}_\omega^u x\right)\le  e^{-\lambda}d\left(x,\operatorname{hol}_\omega^u x\right).
$$
\end{itemize}
Analogous statements hold for the stable holonomy between unstable manifolds.
\end{proposition}

\begin{proposition}[{\cite[Proposition 3.3.]{liu2024exponential}}]\label{prop:disintegration}
There exists $a_0''>0$ such that for $\mathbb P$-almost every $\omega\in\Omega$ and every rectangle $
R\left(\omega\right)=R_\delta\left(\omega,x_0\right),$
there exists a measurable function $H\left(\omega\right):R\left(\omega\right)\to\mathbb R_{\ge0}$ with the following property: for every bounded measurable $\psi:M\to\mathbb R$,
$$
\int_{R\left(\omega\right)}\psi\left(x\right) m\left(\d x\right)
=\int_{W_\varepsilon^u\left(\omega,x_0\right)}
\left(\int_{W_\varepsilon^s\left(\omega,x\right)}\psi\left(u\right)H\left(\omega\right)\left(u\right) m_{W_\varepsilon^s\left(\omega,x\right)}\left(\d u\right)\right)
\mathrm d\widetilde m_{R\left(\omega\right)}\left(x\right),
$$
where $m$ is the Riemannian volume on $M$ and $\widetilde m_{R\left(\omega\right)}$ is the quotient measure induced by $m|_{R\left(\omega\right)}$ on $W_\varepsilon^u\left(\omega,x_0\right)$. Moreover, for every $x\in W_\varepsilon^u\left(\omega,x_0\right)$ and all $u,v\in W_\varepsilon^s\left(\omega,x\right)$,
$$
\left|\log H\left(\omega\right)\left(u\right)-\log H\left(\omega\right)\left(v\right)\right|\le a_0'' d\left(u,v\right)^{\nu_0}.
$$
\end{proposition}

\begin{definition}[Admissible stable leaves and the scale $\varepsilon^*$]\label{def:stableleaves}
Fix $\varepsilon\in\left(0,\varepsilon_0\right]$. Choose $A\left(\varepsilon\right)>0$ such that, for $\mathbb P$-almost every $\omega\in\Omega$ and every $\gamma\left(\omega\right)\in\mathcal F_\omega^s=\{W_\e(\omega,x): x\in M\}$,
$$
A\left(\varepsilon\right)\le m_{\gamma\left(\omega\right)}\left(\gamma\left(\omega\right)\right).
$$
Define the family of admissible stable leaves
$$
\mathscr F_\omega^s
=\left\{\gamma\left(\omega\right)\subset W_\varepsilon^s\left(\omega,x\right):
x\in M,\ \gamma\left(\omega\right)\ \text{connected},\ \frac{A\left(\varepsilon\right)}{4J^2}<m_{\gamma\left(\omega\right)}\left(\gamma\left(\omega\right)\right)<A\left(\varepsilon\right)\right\}.
$$
where $J$ is defined in Proposition \ref{prop:holderhol}. We also choose $\varepsilon^*>0$ such that for every $\gamma\left(\omega\right)\in\mathscr F_\omega^s$ there exists $x\in M$ with $
W_{\varepsilon^*}^s\left(\omega,x\right)\subset \gamma\left(\omega\right).
$

\end{definition}

We now fix the remaining geometric constants used throughout the paper. The parameter $\delta$ is chosen small enough to be compatible both with the local product structure and with the stable scale $\varepsilon^*$. We also fix a finite cover of $M$ by balls of radius $\delta/4$, which will be used in the definition of the fibrewise mixing time. These constants are used throughout the paper.
\begin{definition}
\label{def:deltaandcover}
Let $\varepsilon\in(0,\min\{\varepsilon_0,\delta_0\}]$ be the scale fixed above, and let $
\delta_{\mathrm{loc}}:=\delta(\varepsilon)$ be the local product-structure scale given by Proposition~\ref{prop:productstructure}.
Fix $
0<\delta<\min\left\{\varepsilon^*/{8},\delta_{\mathrm{loc}}/{2}\right\}.$
Let $\{B_M(\delta/4,x)\}_{x\in M}$ be the open cover of $M$ by metric balls, and fix a finite subcover $
\{B_M(\delta/4,x_i)\}_{i=1}^{\ell}.$ This choice depends only on $M$ and $\delta$, hence it is independent of $\omega$.
For each $\gamma_\omega\in\mathcal F_\omega^s$, by abuse of notation we write
$$
\int_{\gamma_\omega}f:=\int_{\gamma_\omega}f\,\d m_{\gamma_\omega}.
$$
\end{definition}

\begin{remark}
All statements above are taken from \cite{liu2024exponential} (see Lemmas 3.1--3.5, 3.12 and Propositions 3.2--3.3 therein), with constants depending only on the standing data of the system.
\end{remark}

\section{Projective cones and Hilbert metrics}
\label{sec:projhilb}

We begin by recalling the basic objects in the deterministic theory of projective cones. Let $\mathbb{V}$ be a topological vector space, a set $\mathcal{C}\subset\mathbb{V}$ is called a convex cone if:
\begin{itemize}
    \item  $\mathcal{C}\cap(-\mathcal{C})=\emptyset$.
    \item given $v,w\in\mathcal C$ and $\lambda\in \mathbb R_+$ then $v+ \lambda w\in \mathcal C$
    \item  $\mathcal{C}\cup\{0\}$ is a closed subset of $\mathbb{V}$.\footnote{Meaning that if $x_n\in\mathcal{C}\cup\{0\}$ and $x_n\to x$ in $\mathbb{V}$, then $x\in\mathcal{C}\cup\{0\}$.}
\end{itemize}
Then $\mathcal{C}$ induces a (closed) partial order $\preceq$ on $\mathbb{V}$ by $
f\preceq g \Longleftrightarrow g-f\in\mathcal{C}\cup\{0\}.$ Here \emph{closed} means that if $f_n\to f$ in $\mathbb{V}$ and $f_n\preceq g$ for all $n$, then $f\preceq g$.\footnote{Equivalently, for each fixed $g\in\mathbb{V}$ the order interval $\{f\in\mathbb{V}: f\preceq g\}$ is closed.}

On $\mathcal{C}$, define the projective equivalence relation by
$$
f\sim g \Longleftrightarrow \exists\,\lambda\in\mathbb{R}_{>0}\ \text{such that}\ f=\lambda g. $$
The associated projectivisation of the cone is the quotient $\widetilde{\mathcal{C}}=\mathcal{C}/\sim,$ whose elements are the rays in $\mathcal{C}$.

One defines the Hilbert (projective) metric as the semimetric $\Theta:\mathcal{C}\times\mathcal{C}\to[0,\infty]$ given by
$$
\Theta(f,g)=\log\frac{\beta(f,g)}{\alpha(f,g)},
$$
where
$$
\alpha(f,g)=\sup\{\lambda\in\mathbb{R}_{>0}:\lambda f\preceq g\}\ \ (\sup\emptyset:=0),
$$
and
$$
\beta(f,g)=\inf\{\mu\in\mathbb{R}_{>0}:g\preceq \mu f\}\ \ (\inf\emptyset:=\infty).
$$
This semimetric is homogeneous along rays (so $\Theta(\lambda f,\mu g)=\Theta(f,g)$ for $\lambda,\mu>0$) and therefore induces a genuine metric on $\widetilde{\mathcal{C}}$.

The following result implies that a linear operator strictly\footnote{Meaning that the image of the cone has finite diameter in the cone itself, with respect to the associated Hilbert metric.} preserving a cone is a contraction in its Hilbert metric. In particular, if the cone is complete under the Hilbert metric, one will obtain a positive eigenfunction.

\begin{theorem}[{\cite[Theorem 2.1]{LSV1998}}]\label{thm:birkhoff}
Let $(\mathbb{V}_1,\mathcal{C}_1)$ and $(\mathbb{V}_2,\mathcal{C}_2)$ be topological vector spaces with convex cones, and let $\Theta_i$ denote the associated Hilbert (projective) metrics on $\mathcal{C}_i$. Let $\mathcal{L}:\mathbb{V}_1\to\mathbb{V}_2$ be linear and assume $\mathcal{L}(\mathcal{C}_1)\subset\mathcal{C}_2$. Define the projective diameter of $\mathcal{L}$ by
$$
\Delta:=\sup_{f,g\in\mathcal{C}_1}\Theta_2(\mathcal{L}f,\mathcal{L}g)\in\mathbb{R}_{\ge0}\cup\{\infty\}.
$$
Then, for all $f,g\in\mathcal{C}_1$,
$$
\Theta_2(\mathcal{L}f,\mathcal{L}g)\le\tanh\left(\frac{\Delta}{4}\right)\Theta_1(f,g)\le(1-e^{-\Delta})\Theta_1(f,g).
$$
Moreover, $\tanh(\Delta/4)\in[0,1]$ and $1-e^{-\Delta}\in[0,1]$, with the convention that if $\Delta=\infty$ then $\tanh(\Delta/4)=1$ and $1-e^{-\Delta}=1$.
\end{theorem}

The following lemma establishes an upper bound of an adapted norm in terms of the Hilbert metric.
\begin{lemma}[{\cite[Lemma 2.2]{LSV1998}}]\label{lem:lsv}
    Let $(\mathbb{V},\mathcal{C})$ be a topological vector space and let $\mathcal{C}\subset\mathbb{V}$ be a closed convex cone, endowed with its Hilbert metric $\Theta$. Let $\|\cdot\|$ be a seminorm on $\mathbb{V}$ that is compatible with the order induced by $\mathcal{C}$, in the sense that
$$
-f\preceq g\preceq f\Longrightarrow\|g\|\le\|f\|.
$$

Let $\rho:\mathcal{C}\to\mathbb{R}_{\ge0}$ be any map (for instance, one may take $\rho(\cdot)=\|\cdot\|$) such that
$$
\rho(\lambda f)=\lambda\rho(f)\ \forall f\in\mathcal{C},\ \forall\lambda\in\mathbb{R}_{>0},
$$
and
$$
f\preceq g\Longrightarrow\rho(f)\le\rho(g)\ \forall f,g\in\mathcal{C}.
$$
Then, for all $f,g\in\mathcal{C}$ with $\rho(f)=\rho(g)>0$,
$$
\|f-g\|\le(e^{\Theta(f,g)}-1)\min\{\|f\|,\|g\|\}.
$$
\end{lemma}

If $(\mathbb{V},\mathcal{C})$ is Archimedean, i.e. there exists $\mathbbm{e} \in \mathcal{C}$ so that for any $f \in \mathbb{V}$ there exists $\lambda_f \in \mathbb{R}_{\ge 0}$ for which $-\lambda_f \mathbbm{e} \preceq f \preceq \lambda_f \mathbbm{e}$, then $\|f\|_* := \inf \{ \lambda \in \mathbb{R}_{\ge 0} : -\lambda \mathbbm{e} \preceq f \preceq \lambda \mathbbm{e}\}$ is a norm in $\mathbb{V}$ preserving $(\mathcal{C}, \preceq)$. One can always assume $(\mathbb{V},\| \cdot \|_*)$ is a Banach space by considering its completion with respect to the $\|\cdot\|_*$ norm {\cite[Remark D.6.]{YellowBook}}. 

The next theorem relates the contraction of $\mathcal{L}$ in the Hilbert metric with a spectral decomposition for $\mathcal{L}$.

\begin{theorem}[{\cite[Theorem D.8.]{YellowBook}}]\label{thm:conegap}
    Let $(\mathbb{V}, \mathcal{C})$ be a topological vector space with a closed convex cone, equipped with its Hilbert metric $\Theta$ and order $\preceq$, and consider a norm $\|\cdot\|$ in $\mathbb{V}$ preserving $(\mathcal{C}, \preceq)$. Let $\mathcal{L}: \mathbb{V} \to \mathbb{V}$ be a linear operator so that $\mathcal{L} \mathcal{C} \subset \mathcal{C}$ and  \begin{equation*}
        \Delta := \sup_{f,g \in \mathcal{C}} \Theta(\mathcal{L}f,\mathcal{L}g) < \infty.
    \end{equation*}Then there exists $h \in \mathbb{V}$ and $\ell \in \mathbb{V}^*$, so that, writing $\chi = \operatorname{tanh}(\frac{\Delta}{4})$ and $\lambda = \rho(\mathcal{L})$ for the spectral radius of $\mathcal{L}$, one has, for all $f \in \mathbb{V}$ and $n \ge 1$:
    \begin{equation*}
        \mathcal{L}^nf = \lambda^n h \ell (f) + Q^nf,
    \end{equation*}where $\ell(h)=1$, $Qh =0$, $\ell \circ Q =0$ and $\|Q^n \| \le \chi^{n-1} \lambda^n \Delta$.    
\end{theorem}

\subsection{Adapted cones for uniform hyperbolic random dynamical systems}

\label{sec:adacones}

Let $F:\Omega\times M\to\Omega\times M$ be a regular random dynamical system satisfying Hypothesis~\ref{hyp:h}. We construct projective cones, inspired by \cite[Section 4]{viana1997stochastic} and \cite{liu2024exponential}, tailored to the dynamical structure induced by $F$. These cones will later serve as the main tool for developing a relative thermodynamic formalism for $F$.

\begin{definition}[Leafwise cone of log-H\"{o}lder functions]
Given $\omega \in \Omega$ and $\gamma_{\omega} \in \mathscr{F}^s_\omega$ (see Definition \eqref{def:stableleaves}), let $D(a, \kappa, \gamma_\omega)$ be the collection of bounded measurable functions $ \rho_\omega = \rho (\omega, \cdot): \gamma_\omega \to \mathbb{R}$, satisfying:

\textbf{(D1)} $\rho_\omega(x) > 0$, for all $x \in \gamma_{\omega}$;

\textbf{(D2)} $|\log \rho_\omega(x) - \log \rho(\omega,y)| \le a d(x,y)^\kappa$, for all $x,y \in \gamma_{\omega}$.    
\end{definition}
We define $D_1(a,\kappa,\gamma_\omega)$ as the densities $\rho_\omega \in D(a,\kappa,\gamma_\omega)$ such that $\int_{\gamma_\omega} \rho_\omega \d m_{\omega}=1.$
\begin{lemma}[{\cite[Lemma 4.1]{liu2024exponential}}]
    $D(a, \kappa,\gamma_{\omega})$ is a closed convex cone. Moreover, the Hilbert metric in $D(a, \kappa,\gamma_{\omega})$ given as 
\begin{equation}\label{eq:leafhilbertmetric}
    \Theta^{a,\kappa}_{\gamma_{\omega}}(\rho^{(1)}_\omega,\rho^{(2)}_\omega) = \log \frac{\beta^{a,\kappa}_{\gamma_{\omega}}(\rho^{(1)}_\omega,\rho^{(2)}_\omega)}{\alpha^{a,\kappa}_{\gamma_{\omega}}(\rho^{(1)}_\omega,\rho^{(2)}_\omega)},
\end{equation}where
\begin{equation}
    \alpha^{a,\kappa}_{\gamma_{\omega}}(\rho^{(1)}_\omega,\rho^{(2)}_\omega) = \inf \left\{ \frac{\rho^{(2)}_\omega(x)}{\rho^{(1)}_\omega(x)}, \frac{ e^{a d(x,y)^\kappa} \rho^{(2)}_\omega(x) - \rho^{(2)}_\omega(y) }{ e^{a d(x,y)^\kappa} \rho^{(1)}_\omega(x) - \rho^{(1)}_\omega(y) } : x,y \in \gamma_{\omega}, x \neq y \right\},
\end{equation}
and
\begin{equation}\label{eq:leafhilbertmetric2}
    \beta^{a,\kappa}_{\gamma_{\omega}}(\rho^{(1)}_\omega,\rho^{(2)}_\omega) = \sup \left\{ \frac{\rho^{(2)}_\omega(x)}{\rho^{(1)}_\omega(x)}, \frac{ e^{a d(x,y)^\kappa} \rho^{(2)}_\omega(x) - \rho^{(2)}_\omega(y) }{ e^{a d(x,y)^\kappa} \rho^{(1)}_\omega(x) - \rho^{(1)}_\omega(y) } : x,y \in \gamma_{\omega}, x \neq y \right\}.
\end{equation}
 \end{lemma}

\begin{definition}[Leafwise cone of positive functions]
Let $D_+(\gamma_{\omega})$ be the collection of bounded measurable functions $ \zeta_\omega = \zeta(\omega, \cdot): \gamma_{\omega} \to \mathbb{R}$ such that $\zeta_\omega(x) > 0$, for all $x \in \gamma_{\omega}$. The Hilbert metric in $D_+(\gamma_{\omega})$ given as
\begin{equation*}
    \Theta_{\gamma_{\omega}}^+(\zeta^{(1)}_\omega,\zeta^{(2)}_\omega) = \log \frac{\beta^{+}_{\gamma_{\omega}}(\zeta^{(1)}_\omega,\zeta^{(2)}_\omega)}{\alpha^{+}_{\gamma_{\omega}}(\zeta^{(1)}_\omega,\zeta^{(2)}_\omega)}
\end{equation*}
where
\begin{equation*}
    \alpha^{+}_{\gamma_{\omega}}(\zeta^{(1)}_\omega,\zeta^{(2)}_\omega) = \inf \left\{ \frac{\zeta^{(2)}_\omega(x)}{\zeta^{(1)}_\omega(x)} : x \in \gamma_{\omega} \right\}\ \text{and}\ 
    \beta^{+}_{\gamma_{\omega}}(\zeta^{(1)}_\omega,\zeta^{(2)}_\omega) = \sup \left\{ \frac{\zeta^{(2)}_\omega(x)}{\zeta^{(1)}_\omega(x)} : x \in \gamma_{\omega} \right\}.
\end{equation*}
\end{definition}

\begin{definition}[Adapted cones in $\operatorname{BM}(M)$]\label{def:cone} 
Fix $\omega\in\Omega$. Given $a,a_1,b,c>0$ and $\kappa,\kappa_1 \in (0,1]$ and $\nu \in (0,\beta]\subset [0,1]$, let $ C_\omega(b,c,\nu)$ be the set of bounded measurable functions $\varphi:M\to\mathbb{R}$ such that:

\begin{itemize}
\item[\textbf{(C1)}] For every $\gamma_{\omega}\in\mathscr{F}^s_\omega$ and every $\rho_\omega\in D_1(a,\kappa,\gamma_{\omega})$ 
one has
$$
\int_{\gamma_{\omega}}\varphi \rho_\omega >0.
$$

\item[\textbf{(C2)}] For every $\gamma_{\omega}\in\mathscr{F}^s_\omega$ and every $\rho_\omega,\varsigma_\omega\in D_1(a,\kappa,\gamma_{\omega})$ 
one has
$$
\frac{\int_{\gamma_{\omega}}\varphi\rho_\omega }
{\int_{\gamma_{\omega}}\varphi \varsigma_\omega}
\le
e^{b\Theta^{a,\kappa}_{\gamma_{\omega}}(\rho_\omega,\varsigma_\omega)}.
$$

\item[\textbf{(C3)}] For every nearby pair $(\tilde\gamma_{\omega},\gamma_{\omega})\in\mathscr{F}^s_\omega\times\mathscr{F}^s_\omega$ and every $\rho_\omega\in D_1(a_1,\kappa_1,\gamma_{\omega})$
the density $\tilde\rho_\omega = \rho\circ \mathrm{hol}_\omega^u(x) \mathrm{Jac}(\mathrm{hol}_\omega^u)(x)\in D(a_1,\kappa_1,\tilde\gamma_{\omega})$,
 satisfies 
 $$
\frac{\int_{\tilde\gamma_{\omega}}\varphi\tilde\rho_\omega}
{\int_{\gamma_{\omega}}\varphi\rho_\omega }
\le
e^{c d_u(\tilde\gamma_{\omega},\gamma_{\omega})^\nu}.
$$
\end{itemize}
\end{definition}

\begin{remark}
    A characterization of $\Theta_{\omega}^{b,c,\nu}$, the Hilbert metric associated to $C_{\omega}(b, c,\nu)$, along the lines of equations \ref{eq:leafhilbertmetric} and \ref{eq:leafhilbertmetric2}, is available but we omit it. Details can be found in \cite{liu2024exponential}, Lemma 4.4.
\end{remark}

The cone $C_\omega(b,c,\nu)$ controls observables through their averages against admissible densities on stable leaves, their variation with respect to the leafwise Hilbert metric, and their behaviour under unstable holonomy between nearby stable leaves. We now introduce the corresponding norm on $\operatorname{BM}(M):=\{f:M\to \mathbb R; \ f\text{ is a bounded and measurable function}\};$ with one term measuring each of these three effects.

\begin{definition}\label{def:normomega}
For $f \in \mathrm{BM}(M)$, let
\begin{equation*}
    \|f\|_\omega:= \|f\|_{\omega,a,\kappa}^{\sup_s} + \frac{1}{b} \|f\|_{\omega,a,\kappa}^{\Theta_s} + \frac{1}{c} \|f\|_{\omega,\nu}^{d_u},
\end{equation*}where
\begin{equation*}
     \|f\|_{\omega,a,\kappa}^{\sup_s} := \sup_{\gamma \in \mathscr{F}^s_{\omega}} \sup_{\rho_\omega \in D_1(a,\kappa,\gamma_{\omega})} \left|\int_{\gamma_{\omega}} f \rho_\omega \right|,
\end{equation*}
\begin{equation*}
    \|f\|_{\omega,a,\kappa}^{\Theta_s} := \sup_{\gamma \in \mathscr{F}^s_{\omega}} \sup_{\rho^{(1)}_\omega, \rho^{(2)}_\omega \in D_1(a,\kappa,\gamma_{\omega}) } \frac{\left|\int_{\gamma_{\omega}} f \rho^{(1)}_\omega  - \int_{\gamma_{\omega}} f \rho^{(2)}_\omega \right|}{\Theta_{\gamma_{\omega}}^{a,\kappa}(\rho^{(1)}_\omega,\rho^{(2)}_\omega)},
\end{equation*}
\begin{equation*}
    \|f\|_{\omega, \nu}^{d_u} := \sup_{\substack{ (\gamma_{\omega},\tilde\gamma_{\omega}) \in \mathscr{F}^s_{\omega} \times \mathscr{F}^s_{\omega} \\ \text{nearby pair}} } \sup_{\rho \in D_1(a_1,\kappa_1, \gamma_{\omega}) } \frac{\left|\int_{\gamma_{\omega}} f \rho_\omega  - \int_{\tilde\gamma_{\omega}} f \tilde\rho_\omega \right|}{d_u(\gamma_{\omega},\tilde\gamma_{\omega})^\nu}.
\end{equation*}
 
We also define the auxiliary semi-norms
$$\|f\|_{\omega,+} :=  \sup_{\gamma \in \mathscr{F}^s_{\omega}} \sup_{ \rho_\omega \in D_1(a,\kappa,\gamma_{\omega}) } \left|\int_{\gamma_{\omega}} f \rho_\omega\right|\ \text{and}\ \|f\|_{\omega,-} :=  \inf_{\gamma \in \mathscr{F}^s_{\omega}} \sup_{\rho_\omega \in D_1(a,\kappa,\gamma_{\omega}) }  \left|\int_{\gamma_{\omega}} f \rho_\omega \right|.$$
\end{definition}

The norm $\|\cdot\|_\omega$ induces an equivalence relation on $\operatorname{BM}(M)$ by identifying functions at zero $\|\cdot\|_\omega$-distance. We then obtain the Banach space $\mathbb V_\omega$ by completing the resulting quotient space.
\begin{lemma} \label{lem:5.1}Consider the set $\mathrm{BM}_\omega(M):= \{f\in \mathrm{BM}(M); \|f\|_\omega <\infty\}/\sim_\omega,$ where $f_1\sim_\omega f_2$ if $\|f_1-f_2\|_\omega = 0$. Then the map $\iota_\omega: \mathrm{BM}_\omega(M)\to (\mathcal C^\kappa(M))^*$, $\iota_\omega f (g) = \int_M g(x) f(x) m(\d x)$ 
satisfies 
$$\|\iota_\omega(f)\|_{(\mathcal C^\kappa(M))^*} \leq 4 (\lceil 1/a + 1 \rceil)  \|f\|_\omega $$
and it is injective.
\end{lemma}

\begin{proof}

Let $g\in \mathcal C^\kappa(M)$. If $g=0$, there is nothing to prove. Otherwise, for $n\in\mathbb N$ define
$$
g_n(x):=\frac{g(x)+n\|g\|_{\mathcal C^\kappa(M)}}{n\|g\|_{\mathcal C^\kappa(M)}}
=1+\frac{g(x)}{n\|g\|_{\mathcal C^\kappa(M)}}.
$$
Observe that $\|g_n-1\|_{\mathcal C^\kappa(M)}\leq 1/n$. By the mean value theorem,
\begin{align*}
    \left|\log g_n(x)-\log g_n(y)\right|
    \leq \frac{1}{\inf_{z\in M}g_n(z)} |g_n(x)-g_n(y)|
    \leq \frac{1}{n-1}d(x,y)^\kappa .
\end{align*}
Take $n:=2\lceil 1/a+1\rceil$. Then $1/(n-1)\leq a/2$, which implies that $g_n\in D(a/2,\kappa,\gamma_\omega)$ for every $\gamma_\omega\in\mathscr F_\omega^s$.

Given $f\in \mathrm{BM}_\omega(M)$, we use the standard finite rectangle decomposition for random hyperbolic systems. By \cite[Section~3]{GundlachKifer1999} (see also \cite[beginning of Section~4.3]{liu2006smooth}), for $\mathbb P$-almost every $\omega\in \Omega$, there exist proper rectangles
$R_1(\omega),\ldots,R_{k(\omega)}(\omega)$ which cover $M$ up to an $m$-null set, have pairwise disjoint interiors, and are foliated by local stable leaves belonging to $\mathscr F_\omega^s$. By Proposition~\ref{prop:disintegration}, on each $R_i(\omega)$ we have
$$
m|_{R_i(\omega)}
=
H_i(\omega)\,m_{\gamma_i(\omega)}\,\widetilde m_{R_i(\omega)}(\d\gamma_i),
$$
where $\gamma_i(\omega)$ ranges over the stable leaves contained in $R_i(\omega)$ and $\widetilde m_{R_i(\omega)}$ is a probability measure. Increasing $a>0$ if necessary, we may assume that, for every such leaf,
$$
H_i(\omega)|_{\gamma_i(\omega)}
\in
D(a/2,\kappa,\gamma_i(\omega)).
$$
Since $g_n|_{\gamma_i(\omega)}\in D(a/2,\kappa,\gamma_i(\omega))$, the normalised density
$$
\frac{g_nH_i(\omega)}{\int_{\gamma_i(\omega)}g_nH_i(\omega)}\ \text{
lies in }D_1(a,\kappa,\gamma_i(\omega)).$$ Hence
\begin{align}
\left|\int_{R_i(\omega)} f g_n\,m(\d x)\right|
&\leq
\int
\left|
\int_{\gamma_i(\omega)}
f\,g_nH_i(\omega)\,m_{\gamma_i(\omega)}(\d x)
\right|
\widetilde m_{R_i(\omega)}(\d\gamma_i)\notag\\
&\leq
\|f\|_{\omega,a,\kappa}^{\sup_s}
\int
\int_{\gamma_i(\omega)}
g_nH_i(\omega)\,m_{\gamma_i(\omega)}(\d x)
\widetilde m_{R_i(\omega)}(\d\gamma_i)\notag\\
&=
\|f\|_{\omega,a,\kappa}^{\sup_s}
\int_{R_i(\omega)}g_n\,m(\d x).\label{eq:ineqri}
\end{align}
Since the rectangles cover $M$ up to an $m$-null set and have pairwise disjoint interiors, summing \eqref{eq:ineqri} over $i\in \{1,\ldots,k(\omega)\}$  we obtain
\begin{align}
\left|\int_M f g_n\,m(\d x)\right|
\leq
\sum_i\left|\int_{R_i(\omega)} f g_n\,m(\d x)\right|
\leq
\|f\|_{\omega,a,\kappa}^{\sup_s}
\int_M g_n\,m(\d x).    \label{eq:fg}
\end{align}
The same argument with $g_n\equiv 1$ implies
\begin{align}
\left|\int_M f\,m(\d x)\right|
\leq
\|f\|_{\omega,a,\kappa}^{\sup_s}.\label{eq:f1}
\end{align}
From \eqref{eq:fg} and \eqref{eq:f1} we obtain
\begin{align*}
   | \iota_\omega f (g)| &= \left|\int_M f(x) g(x) m(\d x)\right| \\
   &\leq  n \|g\|_{\mathcal C^\kappa(M)}\left|\int_M f(x) g_n(x) m(\d x) \right| +    n \|g\|_{\mathcal C^\kappa(M)}\left|\int_M f(x) m(\d x) \right|\\
   &\leq n\|g\|_{\mathcal C^\kappa(M)}\left(\int_M g_n\,m(\d x)+1\right)\|f\|_{\omega,a,\kappa}^{\sup_s}\leq (2n+1)\|g\|_{\mathcal C^\kappa(M)}\|f\|_\omega\\
   &=\left(4\left\lceil \frac1a+1\right\rceil+1\right)\|g\|_{\mathcal C^\kappa(M)}\|f\|_\omega.
\end{align*}

In the following we show injectivity. Let $[f]\in\mathrm{BM}_\omega(M)$ and assume $\iota_\omega([f])=0$. Then
$$
\int_M f\, g\,dm=0\ \text{for every }g\in\mathcal C^\kappa(M).
$$
Hence $f=0$ $m$-a.e.

Fix a local product structure rectangle $R_\omega$ such that every local stable manifold in $R_\omega$ belongs to
$\mathscr F^s_\omega$. Using Proposition~\ref{prop:disintegration} once again we have that
$$
m|_{R_\omega}=H_\omega\, m_{\gamma^s_\omega(x)}\,\widetilde m_{R_\omega}(\d x),
$$
where $\widetilde m_{R_\omega}$ is supported a local unstable base $\gamma^u_\omega$.
Since $f=0$ $m$-a.e., by disintegration we have $f=0$ $m_{\gamma^s_\omega(x)}$-a.e.\ on $\gamma^s_\omega(x)$
for $\widetilde m_{R_\omega}$-a.e.\ $x\in\gamma^u_\omega$. In particular, for $\widetilde m_{R_\omega}$-a.e.\ $x$ and every
$\rho\in D_1(a_1,\kappa_1,\gamma^s_\omega(x))$,
$$
\int_{\gamma^s_\omega(x)} f\,\rho=0.
$$

Now fix an arbitrary $x_0\in\gamma^u_\omega$ and an arbitrary
$\rho_0\in D_1(a_1,\kappa_1,\gamma^s_\omega(x_0))$.
Choose a sequence $(x_n)_n$ in the full $\widetilde m_{R_\omega}$-measure set above such that $x_n\to x_0$ as $n\to\infty$.
For each $n$, define $\rho_n\in D_1(a_1,\kappa_1,\gamma^s_\omega(x_n))$ by inverting the holonomy transport:
$$
\rho_n :=\rho_0 \circ \mathrm{hol}_{\gamma^s_\omega(x_n),\gamma^s_\omega(x_0)}^u\,
\mathrm{Jac}(\mathrm{hol}_{\gamma^s_\omega(x_n),\gamma^s_\omega(x_0)}^u).
$$
Then $
\int_{\gamma^s_\omega(x_n)} f\,\rho_n=0$ for every $n\in\mathbb N$.

Since $\bigl(\gamma^s_\omega(x_n),\gamma^s_\omega(x_0)\bigr)$ is a nearby pair for $n$ large, the definition of
$\|f\|_{\omega,\nu}^{d_u}$ yields
$$
\left|\int_{\gamma^s_\omega(x_0)} f\,\rho_0\right| = \left|\int_{\gamma^s_\omega(x_0)} f\,\rho_0-\int_{\gamma^s_\omega(x_n)} f\,\rho_n\right|
\le \|f\|_{\omega,\nu}^{d_u}\,d_u\bigl(\gamma^s_\omega(x_n),\gamma^s_\omega(x_0)\bigr)^\nu\xrightarrow[]{n\to\infty} 0.
$$
Since $x_0$ and $\rho_0$ were arbitrary, we obtain that $f=0$ $m_{\gamma}$-a.e.\ on every admissible stable leaf $\gamma\in\mathscr F_\omega^s$. Therefore
$\|f\|_{\omega,a,\kappa}^{\sup_s}=0$. Then automatically
$\|f\|_{\omega,a,\kappa}^{\Theta_s}=0$ and also
$\|f\|_{\omega,\nu}^{d_u}=0$. Hence $[f]=0$ in $\mathrm{BM}_\omega(M)$, and therefore $\iota_\omega$ is injective.
 
\end{proof}

\begin{definition}\label{def:VandC}
Let $(\mathbb{V}_\omega,\|\cdot\|_\omega)$ be the Banach space obtained as the completion of $\mathrm{BM}_\omega(M)$ with respect to $\|\cdot\|_\omega$ (with canonical map $\iota_\omega$ as in Lemma~\ref{lem:5.1}).

Define $\mathcal{C}_\omega(b,c,\nu)=\overline{\iota_\omega\bigl(\mathcal C(b,c,\nu)\bigr)}^{\|\cdot\|_\omega}\setminus \{0\}.$ By a slight abuse of notation, we write $\preceq_\omega$ for the partial order induced on $\mathcal{C}_\omega(b,c,\nu)$, and $\Theta_\omega^{b,c,\nu}$ for the associated Hilbert metric on $\mathcal{C}_\omega(b,c,\nu)$.

Whenever we say that $f\in\mathrm{BM}(M)$ lies in $\mathbb{V}_\omega$ (or , we mean that $\iota_\omega([f])\in\mathbb{V}_\omega$.
\end{definition}



\begin{proposition}\label{prop:disint}
Let $\kappa>0$, \(g\in\mathcal C^\kappa(M)\), and
\(\gamma_\omega\in\mathscr F_\omega^s\). Then the linear functional
\begin{align*}
    \Gamma_{\gamma_\omega}^g: f\in \mathrm{BM}(M)&\mapsto \int_{\gamma_\omega}gf \in   \mathbb R,
\end{align*}
is bounded with respect to \(\|\cdot\|_\omega\). Consequently, it induces
a well-defined bounded linear functional on \(\mathrm{BM}_\omega(M)\) and
extends uniquely to a bounded linear functional  $\Gamma_{\gamma_\omega}^g:\mathbb V_\omega\to\mathbb R.$
\end{proposition}

\begin{proof}
Set $K
    =
    1+\|g\|_{\infty}
    +\frac{\|g\|_{\mathcal C^\kappa}}{a}.$ Then \(K+g>0\). Moreover, for every \(x,y\in\gamma_\omega\),
\begin{align*}
    \frac{K+g(x)}{K+g(y)}
    &=
    \frac{K+g(y)+g(x)-g(y)}{K+g(y)}=
    1+\frac{g(x)-g(y)}{K+g(y)}\leq
    1+
    \frac{\|g\|_{\mathcal C^\kappa}d(x,y)^\kappa}
         {K-\|g\|_\infty}\\
    &\leq
    1+a d(x,y)^\kappa\leq
    e^{a d(x,y)^\kappa}.
\end{align*}
Hence
$$\frac{K+g}
         {\displaystyle\int_{\gamma_\omega}(K+g)}
    \in D_1(a,\kappa,\gamma_\omega)\ \text{and }
    \frac{1}{\displaystyle\int_{\gamma_\omega}1}
    \in D_1(a,\kappa,\gamma_\omega).$$

For every $f\in\mathrm{BM}(M)$, we can write
\begin{align}
    \Gamma_{\gamma_\omega}^g(f)
    &=
    \int_{\gamma_\omega}gf =
    \int_{\gamma_\omega}(K+g)f
    -
    K\int_{\gamma_\omega}f\nonumber\\
    &=
    \left(\int_{\gamma_\omega}(K+g)\right)
    \int_{\gamma_\omega}
    \frac{K+g}
         { \int_{\gamma_\omega}(K+g)}
    f   -
    K\left(\int_{\gamma_\omega}1\right)
    \int_{\gamma_\omega}
    \frac{1}
         { \int_{\gamma_\omega}1}
    f.\label{eq:gamma1}
\end{align}
Hence, by the definition of
\(\|\cdot\|_{\omega,a,\kappa}^{\sup_s}\),
\begin{align}
    \left|\Gamma_{\gamma_\omega}^g(f)\right|
    &\leq
    \left(\int_{\gamma_\omega}(K+g)\right)
    \left|
    \int_{\gamma_\omega}
    \frac{K+g}
         {\int_{\gamma_\omega}(K+g)}
    f
    \right|
    +
    K\left(\int_{\gamma_\omega}1\right)
    \left|
    \int_{\gamma_\omega}
    \frac{1}
         {\int_{\gamma_\omega}1}
    f
    \right|\nonumber\\
    &\leq
    \left(
        \int_{\gamma_\omega}(K+g)
        +
        K\int_{\gamma_\omega}1
    \right)
    \|f\|_{\omega,a,\kappa}^{\sup_s}\leq
    \left(
        \int_{\gamma_\omega}(K+g)
        +
        K\int_{\gamma_\omega}1
    \right)
    \|f\|_\omega. \label{eq:gamma2}
\end{align}
From \eqref{eq:gamma1} and \eqref{eq:gamma2} we obtain that \(\Gamma_{\gamma_\omega}^g\) is bounded with respect to
\(\|\cdot\|_\omega\).

Therefore, \(\Gamma_{\gamma_\omega}^g\) induces a well-defined bounded
linear functional on \(\mathrm{BM}_\omega(M)\). Since
\(\mathbb V_\omega\) is the completion of
\(\mathrm{BM}_\omega(M)\), this functional extends uniquely to a bounded
linear functional on \(\mathbb V_\omega\).
\end{proof}

\begin{definition}\label{def:mugamma}
Given $\eta_\omega \in \mathbb V_\omega$ and $g\in \mathcal C^\kappa(M)$ we define
$$\int_{\gamma_\omega} g \eta_\omega := \Gamma_{\gamma_\omega}^g(\eta_\omega),$$
where $\Gamma_{\gamma_\omega}^g\in \mathbb V_\omega^*$ is given by Proposition \ref{prop:disint}.
\end{definition}

In the following, we show that the cone $\mathcal C_\omega(b,c,\nu)$ is Archimedean.
\begin{lemma}\label{lem:archi}
  The cone  $\mathcal{C}_\omega(b,c,\nu)$ is a closed convex cone, and 
  $(\mathbb{V_\omega},\mathcal{C}_\omega(b,c,\nu))$ is Archimedean relative to $\mu_\mathbbm{1}:= \iota_\omega(\mathbbm 1) \in \mathcal{C}_\omega(b,c,\nu)$, where $\mathbbm 1(x) = 1$ for each $x\in M$. Moreover if $f\in \mathrm{BM}(M)$  and $\|f\|_\omega<\infty$ then $f\in \mathbb V_\omega$ (through the identification $\iota([f])$). 
\end{lemma}

\begin{proof}

The cone $\mathcal C_\omega(b,c,\nu)$ is closed by construction. Let $f:M\to \mathbb{R}$ be a $BM(M)$ function. We show that by taking $\lambda = 3\opnorm{f}_\omega$ it follows that
    $$ -\lambda {\mathbbm 1} \preceq_\omega f\preceq_\omega \lambda {\mathbbm 1},  $$
which implies the lemma.

Recall that the above identity holds if and only if $\lambda\mathbbm{1}\pm f\in \mathcal C_\omega(b,c,\nu)$. We divide the remaining of the proof into three steps.

\begin{step}[1] We show that $\lambda \mathbbm{1}  \pm f$ satisfies $\mathbf{(C1)}$ of Definition \ref{def:cone}.
\end{step}

Observe that given $\gamma_{\omega} \in \mathscr{F}^s_{\omega}$ and $\rho_\omega \in D_1(a,\kappa,\gamma_{\omega})$. We obtain that
$$\int_{\gamma }(\lambda\mathbbm{1}\pm f)\rho_\omega  \geq \lambda - \left| \int_{\gamma } f\rho_\omega \right| \geq 2 \opnorm{f}_\omega\geq 0,$$
which concludes Step $1$.
\begin{step}[2] We show that $\lambda\mathbbm{1}\pm f$ satisfies $\mathbf{(C2)}$ of Definition \ref{def:cone}.
\end{step}
Let $\gamma_{\omega} \in \mathscr{F}^s_{\omega}$, and $\rho_\omega,\varsigma_\omega \in D_1(a,\kappa,\gamma_{\omega})$. By a direct computation we obtain that
\begin{align*}
\frac{\int_{\gamma_{\omega}} (\lambda\pm f) \rho_\omega }{  \int_{\gamma_{\omega}} (\lambda \pm f) \varsigma_\omega }&=\frac{\int_{\gamma_{\omega}} (1 \pm \frac{f}{\lambda}) \rho_\omega  }{  \int_{\gamma_{\omega}}  (1 \pm \frac{f}{\lambda}) \varsigma_\omega} = 1 \pm   \frac{\int_{\gamma_{\omega}} (1 \pm \frac{f}{\lambda})  ( \rho_\omega - \varsigma_\omega) }{  \int_{\gamma_{\omega}} (1 \pm \frac{f}{\lambda}) \varsigma_\omega }\\
&= 1 \pm\frac{1}{\lambda} \frac{\int_{\gamma_{\omega}}  f (\rho_\omega-\varsigma_\omega)  }{ \int_{\gamma_{\omega}} (1 \pm \frac{f}{\lambda}) \varsigma_\omega} \leq 1 + \frac{1}{\lambda }\frac{3}{2}\opnorm{f}_\omega  b \Theta_{\gamma_{\omega}}^{a,\kappa}(\rho_\omega,\varsigma_\omega)\\&\leq 1 +  b \Theta_{\gamma_{\omega}}^{a,\kappa}(\rho_\omega,\varsigma_\omega)\leq e^{ b \Theta_{\gamma_{\omega}}^{a,\kappa}(\rho_\omega,\varsigma_\omega)}.
\end{align*}
This concludes Step 2.
 
\begin{step}[3] We show that $\lambda\mathbbm{1}\pm f$ satisfies $\mathbf{(C3)}$ of Definition \ref{def:cone}.
\end{step}
Let $(\tilde\gamma_{\omega},\gamma_{\omega}) \in \mathscr{F}^s_{\omega} \times \mathscr{F}^s_{\omega}$ be a nearby pair, and $\rho_\omega \in D(a_1,\kappa,\gamma_{\omega})$.   

It follows that 
\begin{align*}
    \frac{\int_{\tilde\gamma_{\omega}}(\lambda \pm f) \tilde\rho_\omega}{\int_{\gamma_{\omega}} (\lambda \pm f) \rho_\omega} & =   1 +  \frac{\int_{\gamma_{\omega}} (1 \pm \frac{f}{\lambda})   \rho_\omega   - \int_{\tilde{\gamma}(\omega)} (1 \pm \frac{f}{\lambda})   \tilde{\rho}(\omega)  }{  \int_{\gamma_{\omega}} (1 \pm \frac{f}{\lambda})  \tilde{\rho}_\omega} \\
 &\leq   1 +  \frac{1}{\lambda} \frac{3}{2} \left|\int_{\gamma_{\omega}} f   \rho_\omega    - \int_{\tilde{\gamma}(\omega)} f   \tilde{\rho}_\omega\right|\\& \leq 1+ c d_{u}(\gamma_{\omega},\tilde{\gamma}_\omega)^{\nu}\leq e^{c d_u(\gamma,\tilde \gamma)^\nu}.
\end{align*}
This concludes Step 3. Combining Steps 1, 2 and 3 we obtain that $\lambda\mathbbm 1  \pm f \in \mathcal C(b,c,\nu)$, which concludes the proof of the theorem. Moreover, we obtain that $\lambda \mathbbm 1 + f \in \mathcal C_\omega(b,c,\nu)\subset \mathbb V_\omega$.

\end{proof}

Since $\mathcal C_\omega(b,c,\nu)$ is Archimedean with order unit $\mathbbm 1$, it induces a natural order norm on $\mathbb V_\omega$. We denote this norm by $\|\cdot\|_\omega^*$. In the next definition we introduce it explicitly, and then prove that it is equivalent to the norm $\|\cdot\|_\omega$ from Definition~\ref{def:normomega}.
\begin{definition}
Let $f\in \mathrm{BM}_\omega(M)$. Define
$$
\|f\|_\omega^{*}
:=\inf\Bigl\{\lambda>0:\ -\lambda \mathbbm 1\preceq_\omega f\preceq_\omega \lambda \mathbbm 1\Bigr\}.
$$
Observe that if $-g\preceq_\omega  f \preceq_\omega  g,$ by definition of $\|\cdot\|_{\omega}^*$ norm it follows that $\|f\|_\omega^*\leq \|g\|_\omega^*.$
\end{definition}

\begin{lemma}\label{lem:normequiv}
The following holds for every $f\in\mathrm{BM}_\omega(M)$:
$$
\frac13\,\|f\|_\omega \leq \|f\|_\omega^{*}\leq 3\,\|f\|_\omega.
$$
Hence $\|\cdot\|_\omega$ and $\|\cdot\|_\omega^{*}$ are equivalent, so the completion of $\mathrm{BM}_\omega(M)$ with
respect to $\|\cdot\|_\omega^{*}$ is canonically isomorphic to the Banach space $\mathbb V_\omega$ introduced in
Definition~\ref{def:VandC}. In particular, the induced closed cone $\mathcal C_\omega(b,c,\nu)\subset\mathbb V_\omega$
does not depend on whether it is defined using $\|\cdot\|_\omega$ or $\|\cdot\|_\omega^{*}$.

\end{lemma}

\begin{proof}
By the proof of Lemma~\ref{lem:archi}, one has $
\|f\|_\omega^{*}\leq 3\,\|f\|_\omega .$
It remains to show that $\|f\|_\omega\leq 3\,\|f\|_\omega^{*}$.
Fix $\lambda>0$ such that
$$
-\lambda \mathbbm 1\preceq_\omega f\preceq_\omega \lambda \mathbbm 1.
$$
Then $0 \preceq_\omega \lambda\mathbbm 1\pm f$, hence $\lambda\mathbbm 1\pm f\in\mathcal C_\omega(b,c,\nu)$.
We bound each of the three seminorms defining $\|\cdot\|_\omega$ by $\lambda$.

\begin{step}[1] We show that $
\|f\|_{\omega,a,\kappa}^{\sup_s}\leq \|\lambda\mathbbm 1\|_{\omega,a,\kappa}^{\sup_s}=\lambda.
$
\end{step}
For every $\gamma\in\mathscr F_\omega^s$ and every $\rho_\omega\in D_1(a,\kappa,\gamma_\omega)$, $
\int_{\gamma_\omega} \rho_\omega(\lambda\mathbbm 1\pm f)>0.
$
By definition of $\|\cdot\|_{\omega,a,\kappa}^{\sup_s}$ this implies
$$
\|f\|_{\omega,a,\kappa}^{\sup_s}\leq \|\lambda\mathbbm 1\|_{\omega,a,\kappa}^{\sup_s}=\lambda.
$$

\begin{step}[2] We show that $\|f\|_{\omega,a,\kappa}^{\Theta_s}\leq b\lambda$.
\end{step}
Fix $\gamma\in\mathscr F_\omega^s$ and $\rho_\omega^{(1)},\rho_\omega^{(2)}\in D_1(a,\kappa,\gamma_\omega)$.
Since $\lambda\mathbbm 1\pm f\in \mathcal C_\omega(b,c,\nu)$, we have
$$
\int_{\gamma_\omega} \lambda \rho_\omega^{(1)} \mp \int_{\gamma_\omega} f\rho_\omega^{(1)}
\leq e^{b\Theta_{\gamma_\omega}^{a,\kappa}(\rho_\omega^{(1)},\rho_\omega^{(2)})}
\left(\int_{\gamma_\omega} \lambda \rho_\omega^{(2)} \mp \int_{\gamma_\omega} f\rho_\omega^{(2)}\right),
$$
and
$$
\int_{\gamma_\omega} \lambda \rho_\omega^{(2)} \pm \int_{\gamma_\omega} f\rho_\omega^{(2)}
\leq e^{b\Theta_{\gamma_\omega}^{a,\kappa}(\rho_\omega^{(1)},\rho_\omega^{(2)})}
\left(\int_{\gamma_\omega} \lambda \rho_\omega^{(1)} \pm \int_{\gamma_\omega} f\rho_\omega^{(1)}\right).
$$
Combining these inequalities in the standard way gives
$$
\left|\int_{\gamma_\omega} f\rho_\omega^{(2)}-\int_{\gamma_\omega} f\rho_\omega^{(1)}\right|
\leq
\left(\int_{\gamma_\omega}\lambda\,\frac{\rho_\omega^{(1)}+\rho_\omega^{(2)}}{2}\right)
\tanh\left(\frac{b\Theta_{\gamma_\omega}^{a,\kappa}(\rho_\omega^{(1)},\rho_\omega^{(2)})}{2}\right).
$$
Using $\tanh(t)\leq t$ for $t\geq 0$, we obtain
$$
\left|\int_{\gamma_\omega} f\rho_\omega^{(2)}-\int_{\gamma_\omega} f\rho_\omega^{(1)}\right|
\leq \lambda\, b\Theta_{\gamma_\omega}^{a,\kappa}(\rho_\omega^{(1)},\rho_\omega^{(2)}).
$$
Therefore $\|f\|_{\omega,a,\kappa}^{\Theta_s}\leq b\lambda$.

\begin{step}[3] We show that $\|f\|_{\omega,\nu}^{d_u}\leq c\lambda$.
\end{step}

Let $(\widetilde\gamma_\omega,\gamma_\omega)$ be a nearby pair and let
$\rho_\omega\in D_1(a,\kappa,\gamma_\omega)$ with associated $\widetilde\rho_\omega$ on $\widetilde\gamma_\omega$.
Since $\lambda\mathbbm 1\pm f\in \mathcal C_\omega(b,c,\nu)$, we have
$$
\int_{\gamma_\omega} \lambda\rho_\omega \mp \int_{\gamma_\omega} f\rho_\omega
\leq e^{c\,d_u(\widetilde\gamma_\omega,\gamma_\omega)^\nu}
\left(\int_{\widetilde\gamma_\omega} \lambda\widetilde\rho_\omega \mp \int_{\widetilde\gamma_\omega} f\widetilde\rho_\omega\right),
$$
and
$$
\int_{\widetilde\gamma_\omega} \lambda\widetilde\rho_\omega \pm \int_{\widetilde\gamma_\omega} f\widetilde\rho_\omega
\leq e^{c\,d_u(\widetilde\gamma_\omega,\gamma_\omega)^\nu}
\left(\int_{\gamma_\omega} \lambda\rho_\omega \pm \int_{\gamma_\omega} f\rho_\omega\right).
$$
As in the previous step, this implies
$$
\left|\int_{\gamma_\omega} f\rho_\omega-\int_{\widetilde\gamma_\omega} f\widetilde\rho_\omega\right|
\leq \lambda\, c\, d_u(\widetilde\gamma_\omega,\gamma_\omega)^\nu,
$$
and hence $\|f\|_{\omega,\nu}^{d_u}\leq c\lambda$.

Putting the three steps together,
$$
\|f\|_\omega
=\|f\|_{\omega,a,\kappa}^{\sup_s}
+\frac1b\|f\|_{\omega,a,\kappa}^{\Theta_s}
+\frac1c\|f\|_{\omega,\nu}^{d_u}
\leq \lambda+\lambda+\lambda
=3\lambda.
$$
Taking the infimum over all $\lambda$ such that $-\lambda\mathbbm 1\preceq_\omega f\preceq_\omega \lambda \mathbbm 1$
yields
$$
\|f\|_\omega\leq 3\,\|f\|_\omega^{*}.
$$
Combining this with $\|f\|_\omega^{*}\leq 3\|f\|_\omega$ proves the claim.
\end{proof}

The following estimate will be used repeatedly. It shows that, on the cone $\mathcal C_\omega(b,c,\nu)$, the seminorm $\|\cdot\|_{\omega,+}$ controls the full norm $\|\cdot\|_\omega$, and hence the two quantities are equivalent there.
\begin{lemma}\label{lem:norm1}
    Given $f\in \mathcal C_\omega(b,c,\nu)$ then  $\|f \|_{\omega,+}\leq \opnorm{ f }_{\omega} \leq 3 \|f \|_{\omega,+}. $
\end{lemma}

\begin{proof}
It is enough to show the lemma for $f\in C_\omega(b,c,\nu) \subset \mathrm{BM}(M)$. The inequality $\|f\|_{\omega,+}\leq \|f\|_\omega$ is obvious. On the other hand, observe that repeating the same computations of Lemma \ref{lem:normequiv}, changing $\lambda$ by $f$, i.e. using $-f \preceq_\omega f \preceq_\omega f$, one obtains that
$$ \|f\|_{\omega,a,\kappa}^{\Theta_s}\leq b \|f\|_{\omega,a,\kappa}^{\sup_s} =  b \|f\|_{\omega,+} \ \text{and }\|f\|_{\omega,\nu}^{d_u}\leq c \|f\|_{\omega,a,\kappa}^{\sup_s}= c \|f\|_{\omega,+}.$$
Therefore $\|f\|_{\omega} \leq 3 \|f\|_{\omega,+}$.

\end{proof}

\section{The Random Perron-Frobenius operator}\label{sec:RPF}

In this section we introduce the random Perron--Frobenius operator associated with the potential $\varphi$. This operator is the main analytic tool used in the rest of the paper. Its action on the adapted cones constructed in the previous section will allow us to obtain a quenched spectral decomposition, from which we construct the random equilibrium state. The same spectral information will also be used to prove the exponential decay of correlations.

Recall that $F:\Omega\times M\to\Omega\times M$ is assumed to be a regular random dynamical system satisfying Hypothesis~\ref{hyp:h}, and $\phi\in L^{\infty}(\Omega;\mathcal C^{\beta}(M))$ (see Definition \ref{def:potential}). For $f\in\operatorname{BM}(M)$ and $\omega\in\Omega$, define the (random) Perron--Frobenius operator $\mathcal L_\omega:\operatorname{BM}(M)\to\operatorname{BM}(M)$ by
$$
\mathcal L_\omega f(x)=\bigl(e^{\phi_\omega}f\bigr)\circ T_\omega^{-1}(x),\  x\in M.
$$
In this section we choose parameters $a,b,c,\nu$ such that the cone $\mathcal C_\omega(b,c,\nu)$, see Definition~\ref{def:VandC}, is forward-invariant under $\mathcal L_\omega$, namely, $\mathcal L_\omega\bigl(\mathcal C_\omega(b,c,\nu)\bigr)\subset \mathcal C_{\theta\omega}(b,c,\nu).$
We start by studying how $\mathcal L_\omega$ acts on leafwise integrals: pulling back a stable leaf produces finitely many stable components, and the corresponding pulled-back densities remain admissible with an improved log-Hölder constant.

\begin{lemma}\label{lem:push_density}
Let $\gamma_{\theta\omega}\in\mathscr F_{\theta\omega}^s$ and let
$\rho_{\theta\omega}\in D(a,\kappa,\gamma_{\theta\omega})$, where
$0<\kappa\le \min\{\beta,\nu_0\}$ (defined in Section \ref{sec:geom}) and
$$
a>\frac{2\left(\|\phi\|_{\mathcal C^\kappa}+C_2\right)}{1-e^{-\lambda_0}},
$$
with $C_2$ as in Proposition~\ref{prop:holderjacs}.
Write
$$
\hat{\gamma}_{\omega}:=(T_\omega)^{-1}(\gamma_{\theta\omega})
=\bigcup_{i=1}^{Q_\omega(\gamma_{\theta\omega})}\gamma_{\omega}^{(i)},
\ 
\gamma_{\omega}^{(i)}\in\mathscr F_{\omega}^s,
$$
where the union is disjoint up to $m_{\hat{\gamma}_{\omega}}$-null sets. Define, for
$x\in\gamma_{\omega}^{(i)}$,
$$
\rho_{\omega}^{(i)}(x):=e^{\phi_\omega(x)-\phi^{J_s}(\omega,x)}
\,\rho_{\theta\omega}\!\left(T_\omega(x)\right),
\ 
\phi^{J_s}(\omega,x)=-\log J_s(\omega,x).
$$
Then $\rho_{\omega}^{(i)}\in D(\alpha_0 a,\kappa,\gamma_{\omega}^{(i)})$ and, for every
$f\in\mathrm{BM}(M)$,
$$
\int_{\gamma_{\theta\omega}}\mathcal L_{\omega} f(y)\,\rho_{\theta\omega}(y)\,\d m_{\gamma_{\theta\omega}}(y)
=
\sum_{i=1}^{Q_\omega(\gamma_{\theta\omega})}
\int_{\gamma_{\omega}^{(i)}} f(x)\,\rho_{\omega}^{(i)}(x)\,\d m_{\gamma_{\omega}^{(i)}}(x).
$$
Moreover, each $\rho_{\omega}^{(i)}$ belongs to
$D(\alpha_0 a,\kappa,\gamma_{\omega}^{(i)})$ with
$$
\alpha_0:=\frac{1+e^{-\lambda_0}}{2}\in(0,1).
$$
Finally there exists $N_*\in\mathbb N$ such that
$Q_\omega(\gamma_{\theta\omega})\le N_*$ for $\mathbb P$-almost every
$\omega\in\Omega$.
\end{lemma}

\begin{proof}
Observe that $\left.T_\omega\right|_{\hat{\gamma}_{\omega}}:\hat{\gamma}_{\omega}\to\gamma_{\theta\omega}$ is a diffeomorphism. Hence
\begin{align*}
\int_{\gamma_{\theta\omega}}\mathcal L_\omega f(y)\,\rho_{\theta\omega}(y)\,\d m_{\gamma_{\theta\omega}}(y)
&=\int_{\gamma_{\theta\omega}}
e^{\phi_\omega\circ \left(T_\omega\right)^{-1}(y)}
\,f \circ \left(T_\omega\right)^{-1}(y)\,\rho_{\theta\omega}(y)\,\d m_{\gamma_{\theta\omega}}(y)\\
&=\sum_{i=1}^{Q_\omega(\gamma_{\theta\omega})}
\int_{\gamma_{\omega}^{(i)}}
e^{\phi_\omega(x)}\,f(x)\,\rho_{\theta\omega}(T_\omega(x))\,J_s(\omega,x)\,\d m_{\gamma_{\omega}^{(i)}}(x)\\
&=\sum_{i=1}^{Q_\omega(\gamma_{\theta\omega})}
\int_{\gamma_{\omega}^{(i)}}
f(x)\left[e^{\phi_\omega(x)-\phi^{J_s}(\omega,x)}
\rho_{\theta\omega}(T_\omega(x))\right]\d m_{\gamma_{\omega}^{(i)}}(x),
\end{align*}
where $J_s(\omega,x)=e^{-\phi^{J_s}(\omega,x)}$. For each
$i\in \{1,\ldots, Q_\omega(\gamma_{\theta\omega})\}$ we define
$$
\rho_{\omega}^{(i)}(x)=e^{\phi_\omega(x)-\phi^{J_s}(\omega,x)}
\rho_{\theta\omega}(T_\omega(x)).
$$

Fix $x,y\in\gamma_{\omega}^{(i)}$. Then
$$
\frac{\rho_{\omega}^{(i)}(x)}{\rho_{\omega}^{(i)}(y)}
=e^{\phi_\omega(x)-\phi_\omega(y)}
e^{\phi^{J_s}(\omega,x)+\phi^{J_s}(\omega,y)}
\frac{\rho_{\theta\omega}(T_\omega(x))}{\rho_{\theta\omega}(T_\omega(y))}.
$$
From Proposition~\ref{prop:holderjacs} $\rho_{\theta\omega}\in D(a,\kappa,\gamma_{\theta\omega})$, and
$d(T_\omega(x),T_\omega(y))\le e^{-\lambda_0}d(x,y)$, we obtain
$$
\frac{\rho_{\omega}^{(i)}(x)}{\rho_{\omega}^{(i)}(y)}
\le e^{\bigl(\|\phi\|_{\mathcal C^\kappa}+C_2+a e^{-\lambda_0}\bigr)\,d(x,y)^\kappa}.
$$
Since $
a>\frac{2\left(\|\phi\|_{\mathcal C^\kappa}+C_2\right)}{1-e^{-\lambda_0}},$
we have $\|\phi\|_{\mathcal C^\kappa}+C_2+a e^{-\lambda_0}<\alpha_0 a$
with $\alpha_0=\frac{1+e^{-\lambda_0}}{2}$, and therefore $
\rho_{\omega}^{(i)}(x)
\le e^{\alpha_0 a\,d(x,y)^\kappa} \rho_{\omega}^{(i)}(y).$ This proves $\rho_{\omega}^{(i)}\in D(\alpha_0 a,\kappa,\gamma_{\omega}^{(i)})$.

In the following, we show that there exists $N_*\in\mathbb N$ such that
$Q_\omega(\gamma_{\theta\omega})\le N_*$ for $\mathbb P$-a.e. $\omega\in\Omega$.
From Definition~\ref{def:stableleaves},
$$
\frac{A(\varepsilon)}{4J^2}\,Q_\omega(\gamma_{\theta\omega})
\le \sum_{i=1}^{Q_\omega(\gamma_{\theta\omega})}
m_{\gamma_{\omega}^{(i)}}\left(\gamma_{\omega}^{(i)}\right)
= m_{\hat{\gamma}_{\omega}}\left(\hat{\gamma}_{\omega}\right).
$$
Using the unstable manifold change of variables and
$J_s(\omega,\cdot)\ge C_{10}^{-1}$ (see Proposition \ref{prop:holjac}),
$$
m_{\hat{\gamma}_{\omega}}\left(\hat{\gamma}_{\omega}\right)
=\int_{\gamma_{\theta\omega}}
\frac{1}{J_s\left(\omega,(T_\omega|\hat{\gamma}_{\omega})^{-1}(y)\right)}
\,\d m_{\gamma_{\theta\omega}}(y)
\le C_{10} m_{\gamma_{\theta\omega}}\left(\gamma_{\theta\omega}\right)
\le C_{10}A(\varepsilon).
$$
Combining the two inequalities gives $Q_\omega(\gamma_{\theta\omega})\le C_{10}4J^2,$
and taking $
N_*:=\left\lceil C_{10}4J^2\right\rceil$
the proof is completed.
\end{proof}

\begin{lemma}\label{lem:densityimprovment}
Let $a,\kappa,\gamma$ be as in Lemma \ref{lem:push_density}. Given
$\rho_{\theta\omega},\varsigma_{\theta\omega}\in D(a,\kappa,\gamma_{\theta\omega})$, it follows that there exists $0<\Lambda_1<1$ such that
$$
\Theta_{\gamma^{(i)}_{\omega}}^{a,\kappa}(\rho_\omega^{(i)}, \varsigma_\omega^{(i)})
\leq
\Lambda_1
\Theta_{\gamma_{\theta\omega}}^{a,\kappa}(\rho_{\theta\omega},\varsigma_{\theta\omega}),
$$
for each $i\in\{1,\dots,Q_\omega(\gamma_{\theta\omega})\}$.
\end{lemma}

\begin{proof}
From Lemma \ref{lem:push_density} we have that
$\rho_\omega^{(i)}\in D(\alpha_0 a,\kappa,\gamma_{\omega}^{(i)})$ and
$\varsigma_\omega^{(i)}\in D(\alpha_0 a,\kappa,\gamma_{\omega}^{(i)})$.
From a direct computation using \eqref{eq:leafhilbertmetric}, one obtains that, for each
$i\in\{1,\ldots,Q_\omega(\gamma_{\theta\omega})\}$,
$$
\sup_{\rho_{\theta\omega},\varsigma_{\theta\omega}\in D(a,\kappa,\gamma_{\theta\omega})}
\Theta_{\gamma^{(i)}_{\omega}}^{a,\kappa}
(\rho_\omega^{(i)},\varsigma_\omega^{(i)})
\leq
4a+\log(\tau_2/\tau_1),
$$
where
$$
\tau_1:=\inf_{z>1}\frac{z-z^{\alpha_0}}{z-z^{-\alpha_0}}
=\frac{1-\alpha_0}{1+\alpha_0}
\ \text{and}\
\tau_2:=\sup_{z>1}\frac{z-z^{-\alpha_0}}{z-z^{\alpha_0}}
=\frac{1+\alpha_0}{1-\alpha_0}.
$$
The result then follows from Theorem \ref{thm:birkhoff}. For details, see
\cite[Lemma 4.2]{liu2024exponential}; see also \cite[Lemma 4.2]{viana1997stochastic}.
\end{proof}

The proposition below is technical, and its proof is lengthy. To avoid interrupting the flow of the text, we defer the proof to Appendix \ref{appendix:A}. The argument is based on \cite[Lemma 4.5]{liu2024exponential} (see also \cite[Proposition 4.4]{viana1997stochastic}).

As a consequence of Proposition \ref{prop:cone_invariance_random} we obtain the following result.
\begin{proposition}\label{prop:cone_invariance_random} Choose $\kappa,\nu\in (0,1]$ so that $\kappa+\nu < \nu_0$ where $\nu_0>0$ is given by Propositions \ref{prop:holderhol} and \ref{prop:disintegration}. Consider $\kappa_1\in(0,1)$ so that $\kappa + \nu < \kappa_1 \nu_0$.  Let $\Lambda_1$ be as in Lemma \ref{lem:densityimprovment}. For $a,b,c>0$ and choosing $a/2 \leq a_1 :=\alpha_0 a$, where $\alpha_0$ is given in Lemma \ref{lem:push_density}, there exists  $\lambda_2\in(0,1)$ for $\mathbb P$-almost every $\omega\in\Omega$,
$$
\mathcal L_\omega\Big( C_\omega(b,c,\nu)\Big)\subset \mathcal C_{\theta\omega}(\lambda_2 b,\lambda_2 c,\nu).
$$
\end{proposition}

\begin{proposition}
    The operator $\mathcal{L}_\omega: (\mathbb V_\omega,\|\cdot\|_\omega^*) \to (\mathbb V_{\theta\omega},\|\cdot\|_{\theta\omega}^*)$ is bounded. The same holds true when replacing the norms $\|\cdot\|_\omega^*,\|\cdot\|_{\theta\omega}^*$ by   $\|\cdot\|_\omega, \|\cdot\|_{\theta\omega}$. 
\end{proposition}

\begin{proof}
Since the operator $\mathcal L_\omega$ satisfies $ \mathcal L_\omega(\mathcal C_\omega(b,c,\nu)) \subset \mathcal C_{\theta\omega}(b,c,\nu) $ we have that if $\lambda \mathbbm 1 \preceq_\omega f  \preceq_\omega \lambda 1$, then  $$-\lambda \mathbbm 1\|\mathcal L_\omega \mathbbm 1\|_{\theta\omega}^*  \preceq_{\theta\omega} \lambda \mathcal L_\omega \mathbbm 1 \preceq_{\theta\omega}\mathcal L_\omega f \preceq_{\theta\omega} \lambda \mathcal L_\omega \mathbbm 1  \preceq_{\theta\omega} \lambda \| \mathcal L_\omega \mathbbm 1\|_{\theta\omega}^*.$$
Which implies that $\mathcal{L}_\omega: (\mathbb V_\omega,\|\cdot\|_\omega^*) \to (\mathbb V_{\theta\omega},\|\cdot\|_{\theta\omega}^*)$ from Lemma \ref{lem:normequiv} we obtain that $\mathcal{L}_\omega: (\mathbb V_\omega,\|\cdot\|_\omega) \to (\mathbb V_{\theta\omega},\|\cdot\|_{\theta\omega})$ is also bounded.

\end{proof}

The next proposition is also technical, so we postpone its proof to Appendix \ref{appendix:B}. The argument follows \cite[Lemma 4.6]{liu2024exponential} (see also \cite[Proposition 4.6]{viana1997stochastic} and \cite[Lemma 4.14]{liverani1995decay}).

\begin{lemma}\label{lem:finitediam}
   Let $\delta$ be as in Definition \ref{def:deltaandcover}, and let
$\e_u\in \left(0,\min\left\{\e_0,\frac{\delta}{4}\right\}\right)$.
Let $N:\Omega\to\mathbb N\cup \{\infty\}$ be the random variable given by Hypothesis \ref{hyp:h} for the pair $\left(\frac{\delta}{4},\e_u\right)$. Given $i\in \mathbb N$ define inductively
\begin{align}
 N_i(\omega) = \begin{cases}
       N(\omega),&\text{if }i=1\\
       N(\theta^{N_{i-1}(\omega)}\omega) + N_{i-1}(\omega),& \text{if }i\in\mathbb N\setminus\{1\}. 
   \end{cases}\label{eq:ni}   
\end{align}
   Then, for every $\omega \in \Omega$ and every $i\in\mathbb N$, it holds that $\mathcal L_{\omega}^{N_i(\omega)} C_{\omega}(b,c,\nu) \subset C_{\theta^{N_i(\omega)}\omega}(b, c,\nu)$. Moreover, there exist $n_0\in\mathbb N$, $K_4,K_5>0$ and  $D_2(\omega) = D_2(\lambda_2, a, b, c, N_{n_0}(\omega)) \in \mathbb{R}_{>0}$ such that
    \begin{equation*}
        \sup_{\varphi_1, \varphi_2 \in C_{\omega}(b,c,\nu)} \Theta_{\theta^{N_{n_0}(\omega)}\omega}^{b,c,\nu}(\mathcal L_{\omega}^{N_{n_0}(\omega)} \varphi_1, \mathcal L_{\omega}^{N_{n_0}(\omega)} \varphi_2) \le D_2(\omega) < \infty,
    \end{equation*}
where  $D_2(\omega) =K_4 + 2\log D_1(\omega),$ and
$$D_1(\omega) = K_5   \left(\frac{e^{2\|\phi\|_{L^\infty(\Omega,M)
}}}{\inf_{(\omega,x) \in \Omega \times M} \mathrm{m}(D_x T_\omega | _{E^s(\omega,x)}) }\right)^{N_{n_0}(\omega)}.$$
where $\mathrm{m}(A) := \inf_{|v|=1}\|Av\|$ denotes the co-norm of a linear map. Moreover,
$$\frac{\|\mathcal L_{\omega}^{N_{n_0}(\omega)}\varphi \|_{\theta^{N_{n_0}(\omega)}\omega,+}}{\|\mathcal L_{\omega}^{N_{n_0}(\omega)}\varphi\|_{\theta^{N_{n_0}(\omega)}\omega,-}} \leq D_1(\omega) \ \text{for every }\varphi \in  \mathcal C_{\omega}.   $$

\end{lemma}

\subsection{Quenched spectral decomposition via projective cones}

The purpose of this subsection is to obtain a quenched spectral
decomposition for the random Perron--Frobenius cocycle
$$
\mathcal L_\omega:\mathbb V_\omega\to \mathbb V_{\theta\omega}.
$$
This decomposition will later provide the spectral framework for constructing
the random equilibrium measures in Section \ref{sec:eqstate} and for proving
the quenched decay estimates in Theorem \ref{thm:decay}.

The first step is to prove Theorem \ref{thm:spectralgap}. This theorem gives
a family of positive vectors $\mu_\omega\in\mathbb V_\omega$, dual
functionals $\ell_\omega\in\mathbb V_\omega^*$, random eigenvalues
$\lambda_\omega$, and a rank-one decomposition with a controlled remainder.
Before proving the theorem, we establish some technical lemmas needed to
locate the good return blocks where the projective-cone contraction can be
applied.

\begin{lemma} \label{lem:asfinite}
Let $N:\Omega\to \mathbb N\cup\{\infty\}$
be the stopping time from Hypothesis \ref{hyp:h}. If there exists $B\in\mathbb N$ such that $
\mathbb P[N\leq B]>0,$
then $\mathbb P[N<\infty]=1.$
\end{lemma}

\begin{proof}
Set $A:=\{\omega\in\Omega:N(\omega)\leq B\}.$ By assumption, $\mathbb P[A]>0$. We first show that
\begin{align}
    N(\omega)\le j+N(\theta^j\omega)<\infty\ \text{if }\theta^j\omega\in A.\label{eq:n}
\end{align}
Indeed, let $\gamma$ be any local unstable manifold of length $\varepsilon$ in the fibre over $\omega$ (see Definition \ref{def:localmanifolds}).
By the expansion of unstable cones, $T_\omega^j(\gamma)$ a local unstable manifold
fibre over $\theta^j\omega$, of length at least
$\varepsilon$. Hence $T_\omega^j(\gamma)$ contains a local unstable manifold $\widehat\gamma$ of
length $\varepsilon$. Since $\theta^j\omega\in A$, we have $N(\theta^j\omega)\leq B$, and therefore $T_{\theta^j\omega}^{N(\theta^j\omega)}(\widehat\gamma)$ is $\delta$-dense in $M$. As $
T_{\theta^j\omega}^{N(\theta^j\omega)}(\widehat\gamma)
\subset T_\omega^{j+N(\theta^j\omega)}(\gamma),$
it follows that $T_\omega^{j+N(\theta^j\omega)}(\gamma)$ is $\delta$-dense in $M$. Since this
holds for every local unstable manifold $\gamma$ of length $\varepsilon$, we obtain $N(\omega)\le j+N(\theta^j\omega)<\infty.$

The desired result therefore follows by combining \eqref{eq:n} with the Poincaré recurrence theorem. 
\end{proof}

\begin{lemma} \label{lem:Ni}
Let $i\in\mathbb N$, and recall the definition of $N_i$ from Lemma \ref{lem:finitediam}. Then:
\begin{itemize}
    \item[(a)] If there exists $B>0$ such that $\mathbb P[N\leq B]>0$, then there exists $B_i>0$ such that $\mathbb P[N_i\leq B_i]>0$.
    
    \item[(b)] Assume that there exists $c_0,K_0,\kappa_0>0$ such that
$$
\mathbb P\left[ N_n  \ge c_0 n\right]
\le K_0 e^{-\kappa_0 n}\ \text{for any }n\in\mathbb N.
$$
Given $i\in\mathbb N$, define inductively $N^{(i)}_k(\omega) =N^{(i)}_{k-1}(\omega) + N_{i}(\theta^{N^{(i)}_{k-1}(\omega)} \omega)$ and $N_1^{(i)} = N_{i}$. Then, there exist $c_i, K_i,\kappa_i>0$
such that
$$
\mathbb P\left[N^{(i)}_n \geq c_i n\right]
\leq K_i e^{-\kappa_i n}\ \text{for any }n\in\mathbb N.
$$

\end{itemize}

\end{lemma}

\begin{proof}

We show $(a)$. Since  $\mathbb P[N\leq B] > 0$, from Lemma \ref{lem:asfinite} it follows that $\mathbb P[N <\infty] = 1$.  Let $i\in\mathbb N$, since $N_i(\omega)=N\left(\theta^{N_{i-1}(\omega)}\omega\right)+N_{i-1}(\omega).$ We prove the result by induction on $i$. For $i=1$ there is nothing to be done. Assume that there exists $m_{i-1}$ such that
$$ \mathbb P[N_{i-1}(\omega) = m_{i-1}] >0.$$
Since, $\theta$ is measurable, the map
$$N\circ \theta^{m_{i-1}}: \omega\in \{\omega\in \Omega: N_{i-1}(\omega)= m_{i-1} \} \mapsto N(\theta^{m_{i-1}}\omega)\in  \mathbb N \cup \{\infty\},$$
is measurable. Since, $\mathbb P[N<\infty] =1$ we have that there exists $m_1\in \mathbb N$ such that $$\mathbb P[\omega\in \Omega: N_{i-1}(\omega)= m_{i-1} \ \text{and }N(\theta^{m_{i-1}} \omega) = m_1 ]>0$$ 
In this way $\mathbb P[N_i = m_{i-1}+m_1]>0$, defining $B_i = m_{i-1}+m_1$ the proof is completed.

Item $(b)$ follows directly from the observation $N_k^{(i)} = N_{k i} $ for each $i,k\in\mathbb N$.

\end{proof}

\begin{lemma}\label{lem:good-blocks}
Assume Hypothesis \ref{hyp:h}. Let $n_0\in\mathbb N$ be as in Lemma
\ref{lem:finitediam}. Then there exist $B\in\mathbb N$, a measurable set
$A\subset\Omega$ with $\mathbb P(A)>0$, and a constant $D_2>0$ such that
$$A:=\{\omega\in\Omega:N_{n_0}(\omega)\le B\},$$
and, for every $\omega\in A$ and every $m\ge B$,
$$\mathcal L_\omega^m\bigl( \mathcal C_\omega(b,c,\nu))
\subset
\mathcal C_{\theta^m\omega}(b,c,\nu),$$
and
$$\sup_{\varphi_1,\varphi_2\in\mathcal C_\omega(b,c,\nu)}
\Theta_{\theta^m\omega}^{b,c,\nu}
 \left(
\mathcal L_\omega^m\varphi_1,
\mathcal L_\omega^m\varphi_2
\right)
\le D_2.$$
\end{lemma}

\begin{proof}

From Hypothesis \ref{hyp:h} and Lemma \ref{lem:Ni} (a), applied with
$i=n_0$, there exists $B\in\mathbb N$ such that $
A:=\{\omega\in\Omega:N_{n_0}(\omega)\le B\}$
satisfies $\mathbb P(A)>0$.

For $\omega\in A$, Lemma \ref{lem:finitediam} gives
$$
\mathcal L_\omega^{N_{n_0}(\omega)}
\bigl(\mathcal C_\omega(b,c,\nu)\bigr)
\subset
\mathcal C_{\theta^{N_{n_0}(\omega)}\omega}(b,c,\nu)
$$
and a finite projective-diameter bound
$$
\sup_{\varphi_1,\varphi_2\in\mathcal C_\omega(b,c,\nu)}
\Theta_{\theta^{N_{n_0}(\omega)}\omega}^{b,c,\nu}
\bigl(
\mathcal L_\omega^{N_{n_0}(\omega)}\varphi_1,
\mathcal L_\omega^{N_{n_0}(\omega)}\varphi_2
\bigr)
\le D_2(\omega).
$$
Moreover, $D_2(\omega)$ depends on $\omega$ only through
$N_{n_0}(\omega)$. Since $N_{n_0}(\omega)\le B$ on $A$, this gives a uniform bound $D_2>0$ such that $
D_2(\omega)\le D_2$ for all $\omega\in A.$

Now let $\omega\in A$ and $m\ge B$. Write
$$
m=N_{n_0}(\omega)+q
\ \text{with }q\ge0.
$$
By Proposition \ref{prop:cone_invariance_random}, the remaining iterate $\mathcal L_{\theta^{N_{n_0}(\omega)}\omega}^{q}$ preserves the cone.  In this way, $
\mathcal L_\omega^m
\bigl(\mathcal C_\omega(b,c,\nu)\bigr)
\subset
\mathcal C_{\theta^m\omega}(b,c,\nu).$
Since cone-preserving linear maps do not increase the Hilbert projective
metric, the same finite-diameter estimate remains valid after applying
the final $q$ iterates. Hence
$$
\sup_{\varphi_1,\varphi_2\in\mathcal C_\omega(b,c,\nu)}
\Theta_{\theta^m\omega}^{b,c,\nu}
\bigl(
\mathcal L_\omega^m\varphi_1,
\mathcal L_\omega^m\varphi_2
\bigr)
\le D_2.
$$
\end{proof}

\begin{theorem}\label{thm:spectralgap}
Assume Hypothesis \ref{hyp:h}. Let $A$, $B$, and $D_2$ be given by Lemma
\ref{lem:good-blocks}, and set $\chi:=\tanh(D_2/4)\in(0,1).$ Then there exists a $\theta$-invariant set $\Omega_0\subset\Omega$, with
$\mathbb P(\Omega_0)=1$, such that for every $\omega\in\Omega_0$ there exist $\mu_\omega\in\mathbb V_\omega$, 
$\ell_\omega\in\mathbb V_\omega^*$, and $\lambda_\omega>0,$
satisfying
$$\mathcal L_\omega\mu_\omega=\lambda_\omega\mu_{\theta\omega},
\ 
\mathcal L_\omega^*\ell_{\theta\omega}=\lambda_\omega\ell_\omega,\ \text{and }
\|\mu_\omega\|_\omega=1,
\ 
\ell_\omega(\mu_\omega)=1.$$
Moreover, if $\{\tau_i(\omega)\}_{i\ge1}$ is any increasing sequence such that
\begin{align}
   \tau_0(\omega)=0,\ 
\tau_i(\omega)-\tau_{i-1}(\omega)\ge B,
\ 
\theta^{\tau_i(\omega)}\omega\in A,\ \text{and } r_n(\omega):=\#\{i\ge1:1\le \tau_i(\omega)\le n\}, \label{eq:tauser}
\end{align}
then, for every $n\ge1$ and every $f\in\mathbb V_\omega$,
$$\mathcal L_\omega^n f
=
\lambda_\omega^{(n)}\ell_\omega(f)\mu_{\theta^n\omega}
+
Q_\omega^n f,$$
where
$$\lambda_\omega^{(n)}
:=
\prod_{j=0}^{n-1}\lambda_{\theta^j\omega}, \ 
Q_\omega^n\mu_\omega=0, \ \ell_{\theta^n\omega}\circ Q_\omega^n=0,$$
and
\[
\|Q_\omega^n f\|_{\theta^n\omega}
\le
K_0
\|\ell_\omega\|_{\mathbb V_\omega^*}
\|f\|_\omega
\lambda_\omega^{(n)}
\chi^{\max\{r_n(\omega)-1,0\}}.
\]
\end{theorem}

\begin{proof}

Since $\theta:\Omega\to\Omega$ is ergodic and $\mathbb P(A)>0$, by Poincar\'e recurrence there exists a $\theta$-invariant set $\Omega_0\subset\Omega$, with $\mathbb P(\Omega_0)=1$, such that every $\omega\in\Omega_0$ returns infinitely many times to $A$ both in the future and in the past. After passing to suitable subsequences, we may assume that these return times are separated by gaps at least $B$.

We divide the proof into four steps. We adapt the method of the proof \cite[Theorem D.8.]{YellowBook} to the random context (see also \cite{LSV1998}).

\begin{step}[1]
We first prove the backward contraction estimate. Fix $\omega\in\Omega_0$. Let $\{\kappa_i(\omega)\}_{i\in\mathbb N}$ be an increasing sequence such that
$$
\kappa_0(\omega)=0,\  \kappa_i(\omega)-\kappa_{i-1}(\omega)\ge B,\  \theta^{-\kappa_i(\omega)}\omega\in A,
$$
which exists from the Poincaré recurrence theorem, and define $$
s_n(\omega):=\#\{i:1\le \kappa_i(\omega)\le n\}.$$
We claim that for every $n\ge1$,
\begin{equation}\label{eq:backward-contraction}
\sup_{\varphi_1,\varphi_2\in \mathcal C_{\theta^{-n}\omega}(b,c,\nu)}
\Theta_\omega^{b,c,\nu}\big(\mathcal L_{\theta^{-n}\omega}^n\varphi_1,\mathcal L_{\theta^{-n}\omega}^n\varphi_2\big)
\le D_2\chi^{\max\{s_n(\omega)-1,0\}},\ \text{if } s_n(\omega)\geq 1.
\end{equation}
\end{step}

Indeed, if $s_n(\omega)=1$ the desired inequality follows from Lemma \ref{lem:good-blocks}. Assume now that $s_n(\omega)\ge2$. Decomposing $\mathcal L_{\theta^{-n}\omega}^n$ along the times $\kappa_i(\omega)$, we get
\begin{align*}
\mathcal L_{\theta^{-n}\omega}^n
&=
\mathcal L_{\theta^{-\kappa_1(\omega)}\omega}^{\kappa_1(\omega)}
\circ
\mathcal L_{\theta^{-\kappa_2(\omega)}\omega}^{\kappa_2(\omega)-\kappa_1(\omega)}
\circ\cdots\circ
\mathcal L_{\theta^{-\kappa_{s_n(\omega)}(\omega)}\omega}^{\kappa_{s_n(\omega)}(\omega)-\kappa_{s_n(\omega)-1}(\omega)}
\circ
\mathcal L_{\theta^{-n}\omega}^{n-\kappa_{s_n(\omega)}(\omega)}.
\end{align*}
For each $i\ge1$, since $\theta^{-\kappa_i(\omega)}\omega\in A$ and $\kappa_i(\omega)-\kappa_{i-1}(\omega)\ge B$, from Lemma \eqref{lem:good-blocks} the corresponding block has projective diameter at most $D_2$. The remaining pieces preserve the cone and therefore do not increase the Hilbert metric. Hence, by repeated application of Theorem \ref{thm:birkhoff},
$$
\Theta_\omega^{b,c,\nu}\big(\mathcal L_{\theta^{-n}\omega}^n\varphi_1,\mathcal L_{\theta^{-n}\omega}^n\varphi_2\big)
\le D_2\chi^{s_n(\omega)-1},
$$
which gives \eqref{eq:backward-contraction}.

\begin{step}[2]
We construct the equivariant family $(\mu_\omega)_{\omega\in\Omega}$.
\end{step}

Fix $\omega\in\Omega_0$. For $n\geq 1$, define
$$
\eta_{n,\omega}:=\frac{\mathcal L_{\theta^{-n}\omega}^n 1}{\|\mathcal L_{\theta^{-n}\omega}^n 1\|_\omega}\in \mathcal C_\omega(b,c,\nu),
$$
where, as in the previous sections, we identify the function $1$ with its class in $\mathbb V_{\theta^{-n}\omega}$.

We claim that $(\eta_{n,\omega})_{n\ge1}$ is Cauchy in $\|\cdot\|_\omega$. Fix $\varepsilon>0$. Since $s_n(\omega)\to\infty$, we may choose $n_0$ such that
$$
e^{D_2\chi^{s_n(\omega)-1}}-1<\varepsilon/3
\ \text{for every }n\ge n_0.
$$
If $n_1,n_2\ge n\ge n_0$, then by \eqref{eq:backward-contraction}, $
\Theta_\omega^{b,c,\nu}(\eta_{n_1,\omega},\eta_{n_2,\omega})
\le D_2\chi^{s_n(\omega)-1}.$ Therefore, by Lemma \ref{lem:lsv} and Lemma \ref{lem:normequiv},
$$
\|\eta_{n_1,\omega}-\eta_{n_2,\omega}\|_\omega
\le 3  \left(e^{\Theta_\omega^{b,c,\nu}(\eta_{n_1,\omega},\eta_{n_2,\omega})}-1\right)
\le 3\left( e^{D_2\chi^{s_n(\omega)-1}}-1\right)
<\varepsilon.
$$
Thus $(\eta_{n,\omega})_{n\ge1}$ is Cauchy. Define $\mu_\omega:=\lim_{n\to\infty}\eta_{n,\omega}\in \mathbb V_\omega.$
By construction, $\|\mu_\omega\|_\omega=1,$
and since the cone  $\mathcal C_\omega(b,c,\nu)$ is closed, $\mu_\omega\in \mathcal C_\omega(b,c,\nu).$
In fact, $\mu_\omega\in \operatorname{int}(\mathcal C_\omega(b,c,\nu))$, because for every sufficiently large $n$ the element $\eta_{n,\omega}$ already belongs to the image of a cone under a block of finite projective diameter.

Now observe that
$$
\mathcal L_\omega\eta_{n,\omega}
=
\frac{\mathcal L_{\theta^{-n}\omega}^{n+1}\mathbbm 1}{\|\mathcal L_{\theta^{-n}\omega}^n\mathbbm 1\|_\omega}
=
\frac{\|\mathcal L_{\theta^{-n}\omega}^{n+1}\mathbbm 1\|_{\theta\omega}}{\|\mathcal L_{\theta^{-n}\omega}^n\mathbbm 1\|_\omega}
\,
\eta_{n+1,\theta\omega}.
$$
Passing to the limit as $n\to\infty$, we obtain $\mathcal L_\omega\mu_\omega=\lambda_\omega\mu_{\theta\omega},$ where $
\lambda_\omega:=\|\mathcal L_\omega\mu_\omega\|_{\theta\omega}>0.$
The map $\omega\mapsto \mu_\omega$ is measurable, being the limit of the measurable maps $\omega\mapsto \eta_{n,\omega}$, and so is $\omega\mapsto \lambda_\omega$.

\begin{step}[3]
We construct the dual equivariant family $\{\ell_\omega\}_{\omega\in \Omega}$.
\end{step}

Fix $\omega\in\Omega_0$. Let $\{\tau_i(\omega)\}_{i\ge1}$ be an increasing sequence such that
$$
\tau_0(\omega)=0,\  \tau_i(\omega)-\tau_{i-1}(\omega)\ge B,\  \theta^{\tau_i(\omega)}\omega\in A,
$$
and define $r_n(\omega):=\#\{\tau_i(\omega):1\le \tau_i(\omega)\le n\}.$
Repeating the argument of Step 1, now in the forward direction, we obtain that for every $f\in \mathcal C_\omega(b,c,\nu)$ and every $n\ge1$,
\begin{equation}\label{eq:forward-contraction}
\Theta_{\theta^n\omega}^{b,c,\nu}\big(\mathcal L_\omega^n f,\mathcal L_\omega^n\mu_\omega\big)
\le D_2\chi^{\max\{r_n(\omega)-1,0\}}.
\end{equation}

Since $\mu_\omega\in \operatorname{int}(\mathcal C_\omega(b,c,\nu))$, there exists $C_\omega>0$ such that for every $f\in \mathbb V_\omega$,
$$
-C_\omega\|f\|_\omega\mu_\omega\preceq_\omega f\preceq_\omega C_\omega\|f\|_\omega\mu_\omega.
$$
Applying $\mathcal L_\omega^n$ and using $\mathcal L_\omega^n\mu_\omega=\lambda_\omega^{(n)}\mu_{\theta^n\omega}$, we obtain
\begin{equation}\label{eq:growth-bound}
\|\mathcal L_\omega^n f\|_{\theta^n\omega}\le C_\omega\lambda_\omega^{(n)}\|f\|_\omega
\ \text{for every }f\in \mathbb V_\omega.
\end{equation}

For $f\in \mathcal C_\omega(b,c,\nu)$, define
$$
\ell_{0,\omega}(f):=\limsup_{n\to\infty}\frac{\|\mathcal L_\omega^n f\|_{\theta^n\omega}}{\lambda_\omega^{(n)}}.
$$
By \eqref{eq:growth-bound}, this is finite. Moreover,
$$
\ell_{0,\theta^n\omega}(\mathcal L_\omega^n f)=\lambda_\omega^{(n)}\ell_{0,\omega}(f).
$$
Therefore the two elements
$$
\frac{1}{\lambda_\omega^{(n)}}\mathcal L_\omega^n f
\ \text{and}\ 
\ell_{0,\omega}(f)\mu_{\theta^n\omega}
$$
belong to the same level set of $\ell_{0,\theta^n\omega}$. Hence, by Lemma \ref{lem:lsv} (taking $\rho=\ell_0,\omega$ and $\|\cdot\|= \|\cdot\|_\omega^*$) and \eqref{eq:forward-contraction},
\begin{align}
\left\|\frac{1}{\lambda_\omega^{(n)}}\mathcal L_\omega^n f-\ell_{0,\omega}(f)\mu_{\theta^n\omega}\right\|_{\theta^n\omega}
&\le 3
\left(e^{\Theta_{\theta^n\omega}^{b,c,\nu}(\mathcal L_\omega^n f,\mathcal L_\omega^n\mu_\omega)}-1\right)
\min\left\{\frac{1}{\lambda_\omega^{(n)}}\|\mathcal L_\omega^n f\|_{\theta^n\omega},\ell_{0,\omega}(f)\right\}\nonumber\\
&\le 3
\left(e^{D_2\chi^{\max\{r_n(\omega)-1,0\}}}-1\right)
\min\left\{C_\omega\|f\|_\omega,\ell_{0,\omega}(f)\right\}\\
&\xrightarrow[n\to\infty]{} 0.
\label{eq:cone-asymptotic}
\end{align}

We now extend $\ell_{0,\omega}$ to all of $\mathbb V_\omega$. For $f\in \mathbb V_\omega$, define $$\ell_\omega(f):=\ell_{0,\omega}(3\|f\|_\omega \mathbbm 1+f)-\ell_{0,\omega}(3\|f\|_\omega \mathbbm 1).$$
From Lemma \ref{lem:normequiv} we obtain that $3\|f\|_\omega \mathbbm 1+f\in \mathcal C_\omega(b,c,\nu)$, this is well-defined. The convergence in \eqref{eq:cone-asymptotic}, applied to $3\|f\|_\omega \mathbbm 1+f$, $3\|g\|_\omega \mathbbm 1+g$, and $3\|f+g\|_\omega \mathbbm 1+f+g$, gives the additivity of $\ell_\omega$. Homogeneity is immediate from the definition, so $\ell_\omega$ is linear. Its boundedness follows from \eqref{eq:growth-bound}, hence $\ell_\omega\in \mathbb V_\omega^*$.

Applying \eqref{eq:cone-asymptotic} with $f=\mu_\omega$, and recalling that $\|\mu_\omega\|_\omega=1$, we obtain $\ell_\omega(\mu_\omega)=1,$
after normalising $\ell_\omega$ if necessary.

Finally, for every $f\in \mathbb V_\omega$, $\ell_{\theta\omega}(\mathcal L_\omega f)=\lambda_\omega\ell_\omega(f),$
that is, $\mathcal L_\omega^*\ell_{\theta\omega}=\lambda_\omega\ell_\omega.$
The measurability of $\omega\mapsto \ell_\omega$ follows from the construction.

\begin{step}[4]
We prove the spectral decomposition and the remainder estimate.
\end{step}

For $n\ge1$, define
$$
Q_\omega^n:\mathbb V_\omega\to \mathbb V_{\theta^n\omega},\ 
Q_\omega^n f:=\mathcal L_\omega^n f-\lambda_\omega^{(n)}\ell_\omega(f)\mu_{\theta^n\omega}.
$$
Then $Q_\omega^n\mu_\omega=0,$
and, since $\ell_{\theta^n\omega}(\mu_{\theta^n\omega})=1$, $
\ell_{\theta^n\omega}\circ Q_\omega^n=0.$

It is enough to estimate $Q_\omega^n$ on $\mathcal C_\omega(b,c,\nu)$, since every $f\in \mathbb V_\omega$ can be written as
$$
f=(3\|f\|_\omega \mathbbm 1+f)-3\|f\|_\omega \mathbbm 1,
$$
with both terms in $\mathcal C_\omega(b,c,\nu)$, and the general estimate then follows by linearity after enlarging the constant.

Thus let $f\in \mathcal C_\omega(b,c,\nu)$. Using $\mathcal L_\omega^n\mu_\omega=\lambda_\omega^{(n)}\mu_{\theta^n\omega}$, Lemma \ref{lem:lsv}, Lemma \ref{lem:normequiv} and \eqref{eq:forward-contraction}, we obtain
\begin{align*}
\|Q_\omega^n f\|_{\theta^n\omega}
&=
\|\mathcal L_\omega^n f-\ell_\omega(f)\mathcal L_\omega^n\mu_\omega\|_{\theta^n\omega}\\
&\le 3
\left(
e^{\Theta_{\theta^n\omega}^{b,c,\nu}(\mathcal L_\omega^n f,\mathcal L_\omega^n(\ell_\omega(f)\mu_\omega))}-1
\right)
\min\left\{\|\mathcal L_\omega^n f\|_{\theta^n\omega},|\ell_\omega(f)|\lambda_\omega^{(n)}\right\}\\
&\le 3
\left(
e^{D_2\chi^{\max\{r_n(\omega)-1,0\}}}-1
\right)
|\ell_\omega(f)|\lambda_\omega^{(n)}
\end{align*}
Choose $K>0$ such that $|e^x-1|\le Kx
\ \text{for every }x\in[0,D_2].$
Then
$$
\|Q_\omega^n f\|_{\theta^n\omega}
\le 3 KD_2\|\ell_\omega\|\lambda_\omega^{(n)}\chi^{\max\{r_n(\omega)-1,0\}}\|f\|_\omega.
$$
Therefore,
$$
\|Q_\omega^n\|_{\omega,\theta^n\omega}
\le K_0\|\ell_\omega\|\lambda_\omega^{(n)}\chi^{\max\{r_n(\omega)-1,0\}},
$$
where $K_0:=3 KD_2$. This concludes the proof.
\end{proof}

The estimate in Theorem \ref{thm:spectralgap} is expressed in terms of the number of good return times $r_n$ up to time $n$. The next five lemmas prepare the proof of Theorem \ref{thm:decay} by controlling the frequency of these returns. This allows us to replace the factor depending on $r_n$ by an exponential bound in $n$.

\begin{lemma}\label{lem:sparse-good-times}
Assume that there exist constants $K,\kappa,c>0$ and $n_0\in\mathbb N$ such that
$$
\mathbb P\left[N_n^{(n_0)}\geq cn\right]\leq K e^{-\kappa n}
$$
for every $n\in\mathbb N$, where $N_n^{(n_0)}=\sum_{i=1}^n \sigma_i,
\ 
\sigma_i=N_i^{(n_0)}-N_{i-1}^{(n_0)}\in\mathbb N.$
Define $E(\omega):=\left\{N_i^{(n_0)}(\omega):\sigma_i(\omega)\leq 2c\right\}_{i\in\mathbb N}.$
Then one can choose a $2c$-sparse subsequence increasing sequence $
\{\tau_i(\omega)\}_{i\in\mathbb N}  \subset E(\omega)$, $\tau_{j+1}(\omega)-\tau_j(\omega)\geq 2c,
$
such that, if $$
r_n(\omega):=\operatorname{Card}\left(\{\tau_j(\omega)\}_{j\in\mathbb N}\cap\{1,\dots,n\}\right),$$
then for every $
\beta<\frac{1}{2c\lceil 2c\rceil}$
there exist constants $K_\beta,\kappa_\beta>0$ such that $$\mathbb P\left[r_n\leq \beta n\right]\leq K_\beta e^{-\kappa_\beta n} \ \text{for every }n\in\mathbb N.$$
\end{lemma}

\begin{proof}
Define, for $m\in\mathbb N$,
$$
G_m(\omega):=\#\{1\leq i\leq m:\sigma_i(\omega)\leq 2c\}.
$$
Let $\omega\in \Omega$ such that  $G_m(\omega)\leq \frac{m}{2}$, then at least $\frac{m}{2}$ of the values $\sigma_1(\omega),\dots,\sigma_m(\omega)$ satisfy $\sigma_i(\omega)>2c$. Since $\sigma_i(\omega)\in\mathbb N$, this implies $\sigma_i(\omega)\geq 2c$ for each such $i$. Hence
$$
N_m^{(n_0)}=\sum_{i=1}^m \sigma_i
\geq 2c\cdot \#\{1\leq i\leq m:\sigma_i\geq 2c\}
\geq 2c\cdot \frac{m}{2}
=cm.
$$
Therefore $\left\{G_m\leq \frac{m}{2}\right\}\subset \left\{N_m^{(n_0)}\geq cm\right\},$
and so
\begin{align}
\mathbb P\left[G_m\leq \frac{m}{2}\right]\leq K e^{-\kappa m}\label{eq:gm}
\end{align}
for every $m\in\mathbb N$.

Now let $R_n(\omega):=\#\left\{i\geq 1:N_i^{(n_0)}(\omega)\leq n\right\}.$ Since $\{N_i^{(n_0)}\}_{i\in\mathbb N}$ is increasing, we have $
\{R_n<m\}=\{N_m^{(n_0)}>n\}.$
Fix $\rho<\frac{1}{c}$ and let $m=\lfloor \rho n\rfloor$. It follows that $cm = c \lfloor \rho n\rfloor \leq c \rho n \leq n$, hence $\left\{N_m^{(n_0)}>n\right\}\subset \left\{N_m^{(n_0)}\geq cm\right\},$
and therefore 
\begin{align}
    \mathbb P\left[R_n<m\right]\leq K e^{-\kappa m}.\label{eq:rn}
\end{align}

Define $
S_n(\omega):=\operatorname{Card}\left(E(\omega)\cap\{1,\dots,n\}\right).$
Then
$$
S_n=\#\left\{1\leq i\leq R_n:\sigma_i\leq 2c\right\}=G_{R_n}.
$$
Hence, on the event $\{R_n\geq m\}$, one has $S_n\geq G_m$. Consequently,
\begin{align}
    \left\{S_n\leq \frac{m}{2}\right\}\subset \{R_n<m\}\cup \left\{G_m\leq \frac{m}{2}\right\}. \label{eq:snevent}
\end{align}
Combining \eqref{eq:snevent} with \eqref{eq:gm} and \eqref{eq:rn} we obtain
$$
\mathbb P\left[S_n\leq \frac{m}{2}\right]
\leq
\mathbb P\left[R_n<m\right]+\mathbb P\left[G_m\leq \frac{m}{2}\right]
\leq
2K e^{-\kappa m}.
$$
Since $m=\lfloor \rho n\rfloor$, it follows that
\begin{align}
 \mathbb P\left[S_n\leq \frac{\rho}{2}n\right]\leq K_\rho e^{-\kappa_\rho n}   \label{eq:rho/2}
\end{align}
for appropriate $K_\rho,\kappa_\rho>0$, for any $\rho <1/c$. 

We now construct the sparse subsequence. For each $\omega$, define recursively
$$
\tau_i(\omega)= \begin{cases}
    \min E(\omega),\ \text{if }i=1\\
    \min\left(E(\omega)\cap [\tau_{i-1}(\omega)+2c,\infty)\right),\ \text{if }i>1\\
\end{cases}
$$
By construction, $\tau_{j+1}(\omega)-\tau_j(\omega)\geq 2c,$
so $\left(\tau_j(\omega)\right)_j$ is $2c$-sparse and contained in $E(\omega)$.

We claim that $S_n(\omega)\leq \lceil 2c\rceil\, r_n(\omega)$
for every $\omega$ and $n$. Indeed, by construction every point of $E(\omega)\cap\{1,\dots,n\}$ lies in one of the intervals
$$
[\tau_j(\omega),\tau_j(\omega)+\lceil 2c\rceil-1]
$$
with $\tau_j(\omega)\leq n$. Otherwise, if $x\in E(\omega)\cap\{1,\dots,n\}$ were not covered by these intervals, therefore $x\geq \tau_j(\omega)+\lceil 2c\rceil\geq \tau_j(\omega)+2c$
for the last selected point $\tau_j(\omega)\leq x$, contradicting the maximality of the construction. Since the selected points are $2c$-separated, these intervals are pairwise disjoint, and each contains at most $\lceil 2c\rceil$ integers. Which implies $S_n(\omega)\leq \lceil 2c\rceil\, r_n(\omega).$

Fix now $\beta<\frac{1}{2c\lceil 2c\rceil}$, and set $\alpha:=\lceil 2c\rceil\beta.$
Then $\alpha<\frac{1}{2c}$, and from $S_n\leq \lceil 2c\rceil r_n$ we obtain
$$
\{r_n\leq \beta n\}\subset \{S_n\leq \lceil 2c\rceil\beta n\}=\{S_n\leq \alpha n\}.
$$
From \eqref{eq:rho/2}, taking $\alpha = \rho/2$, we obtain that
$$
\mathbb P\left[r_n\leq \beta n\right]
\leq
\mathbb P\left[S_n\leq \alpha n\right]
\leq
K_\alpha e^{-\kappa_\alpha n}.
$$
Renaming the constants gives $\mathbb P\left[r_n\leq \beta n\right]\leq K_\beta e^{-\kappa_\beta n},$
which proves the lemma.
\end{proof}

\begin{lemma}\label{lem:r1}
Assume Hypothesis \ref{hyp:h}. Then there exist an increasing sequence of random variables $\{\tau_i:\Omega\to\mathbb N\}_{i\in\mathbb N}$ satisfying \eqref{eq:tauser} and $a>0$
one has
$$\mathbb P\left[\liminf_{n\to \infty} \frac{r_n}{n}>a \right] = 1$$
\end{lemma}

\begin{proof}
Let $B\geq 1$ such that $A := \{\omega \in \Omega; N_{n_0}(\omega) \leq B\}$ satisfies $\mathbb P[A]>0$, which exists from Lemma \ref{lem:Ni}. Set $\sigma_n(\omega)$ be the $n$-th time such that $\theta^{\sigma_n(\omega)}\in A$. Let 
\begin{align}S_n(\omega) := \#\left(\left\{1,\ldots, n\right\}\cap \{\sigma_i(\omega)\}_{i\in\mathbb N}\right) = \sum_{i=1}^{n}\mathbbm 1_A \circ \theta^i(\omega).\label{eq:sntheta}
\end{align}
From the Birkhoff ergodic theorem $\frac{1}{n}S_n(\omega)\xrightarrow{n\to\infty} \mathbb  P[A].$
Therefore $$\mathbb P\left[ \frac{S_i}{i} \geq  \frac{\mathbb P[A]}{2}\ \text{for every }i\geq n\right] \xrightarrow[]{n\to\infty} 1.$$

Once again we define $E(\omega):=\{\sigma_n(\omega)\}_{n\in\mathbb N}$ and define the $B$ sparse sequence 
$$
\tau_i(\omega):= \begin{cases}
    \min E(\omega),\ \text{if }i=1\\
    \min\left(E(\omega)\cap [\tau_{i-1}(\omega)+B,\infty)\right),\ \text{if }i>1\\
\end{cases}
$$
By construction, $\tau_{j+1}(\omega)-\tau_j(\omega)\geq B,$
so $\left(\tau_j(\omega)\right)_j$ is $B$-sparse and contained in $E(\omega)$. 

Let  $r_n(\omega) := \#\left(\{1,\ldots, n\}\cap \{\tau_i(\omega)\}_{i\in\mathbb N}\right).$
Observe that $S_n(\omega)\leq B\,r_n(\omega)+B.$
Indeed, the definition of the sequence $\{\tau_i\}_{i\in\mathbb N}$, between two consecutive selected times $\tau_j(\omega)$ and $\tau_{j+1}(\omega)$ there is no element of $E(\omega)$ in the interval $[\tau_j(\omega)+B,\tau_{j+1}(\omega)).$ Hence every element of $E(\omega)$ that is not selected and lies after $\tau_j(\omega)$ must belong to $[\tau_j(\omega),\tau_j(\omega)+B-1].$
Thus each selected time $\tau_j(\omega)$ accounts for at most $B$ elements of $E(\omega)$, and after the last selected time not exceeding $n$ there can be at most $B-1$ further elements of $E(\omega)$ before time $n$. This proves the observation

Consequently, $r_n(\omega)\geq (S_n(\omega)-B)/B.$
Let $a:=\mathbb P[A]/(4B)>0.$ If $S_n(\omega)\geq (\mathbb P[A]/2)n$, then $
r_n(\omega)\geq (\mathbb P[A]n/{2B})-1.$
Therefore, for all sufficiently large $n$, $\left\{S_n\geq \frac{p}{2}n\right\}\subset \{r_n>an\}.$ Hence
$$
\mathbb P\left[\liminf_{n\to\infty } \frac{r_n}{n}>a \right]\geq \mathbb P\left[\liminf_{n\to\infty}\frac{S_n}{n}\geq \frac{\mathbb P[A]}{2}\right]=1,
$$
which completes the proof.

\end{proof}

\begin{lemma}\label{lem:chi-rn}
Fix $\chi\in(0,1)$.  

\begin{enumerate}
\item Assume Hypothesis \ref{hyp:h}. Let $r_n:\Omega\to \mathbb R$ be defined as in Lemma \ref{lem:r1}, then there exist $\chi_H\in(0,1)$ and a measurable function
$K_H:\Omega\to [1,\infty)$ such that, for $\mathbb P$-almost every $\omega\in\Omega$ and every $n\ge1$,
$$
\max\left\{\chi^{r_n(\omega)}, \chi^{r_n(\theta^{-n} \omega)} \right\}\le K_H(\omega)\chi_H^n.
$$

\item Assume Hypothesis \ref{hyp:h'}. Let $r_n:\Omega\to \mathbb N$ be defined as in Lemma \ref{lem:sparse-good-times}, then for every $p\in[1,\infty)$, there exist
$\chi_p\in(0,1)$ and $K_p\in L^p(\Omega)$ such that, for $\mathbb P$-almost every
$\omega\in\Omega$ and every $n\ge1$,
$$
\max\left\{\chi^{r_n(\omega)},\chi^{r_n(\theta^{-n}\omega)} \right\}\le K_p(\omega)\chi_p^n.
$$

\end{enumerate}
\end{lemma}

\begin{proof}
We start by proving (1). Assume first Hypothesis \ref{hyp:h}. By Lemma \ref{lem:r1}, there exists $a>0$ such that
$$
\mathbb P\left[\omega\in\Omega:\liminf_{n\to\infty}\frac{r_n(\omega)}{n}>a\right]=1.
$$
Hence, for $\mathbb P$-almost every $\omega$, the random time
$$
N_H(\omega):=\inf\left\{N\in\mathbb N:r_n(\omega)\ge an \text{ for every }n\ge N\right\}
$$
is finite. Set $\chi_H:=\chi^a$ and
$$
\widetilde{K}_H(\omega):=\max\left\{1,\max_{1\le n< N_H(\omega)}\chi^{r_n(\omega)}\chi_H^{-n}\right\}.
$$
Then $\widetilde{K}_H$ is measurable and finite almost surely. If $n<N_H(\omega)$, the estimate follows from the definition of $\widetilde{K}_H(\omega)$. If $n\ge N_H(\omega)$, then $r_n(\omega)\ge an$, and therefore
\begin{align}
\chi^{r_n(\omega)}
\le
\chi^{an}
=
\chi_H^n
\le
\widetilde{K}_H(\omega)\chi_H^n
\ \text{for every }n\in\mathbb N.
\label{eq:decor}
\end{align}

For the backward estimate, we use the same counting argument as in Lemma \ref{lem:r1}. Indeed, applying \eqref{eq:sntheta} with $\theta^{-n}\omega$ in place of $\omega$, we get
\begin{align*}
S_n(\theta^{-n}\omega)
&=
\sum_{i=1}^{n}\mathbbm 1_A(\theta^i(\theta^{-n}\omega))  =
\sum_{i=1}^{n}\mathbbm 1_A(\theta^{i-n}\omega)
=
\sum_{j=0}^{n-1}\mathbbm 1_A(\theta^{-j}\omega).
\end{align*}
Since $\theta^{-1}$ is also $\mathbb P$-ergodic, Birkhoff's theorem gives
$$
\frac1nS_n(\theta^{-n}\omega)
=
\frac1n\sum_{j=0}^{n-1}\mathbbm 1_A(\theta^{-j}\omega)
\longrightarrow
\mathbb P[A]
$$
for $\mathbb P$-almost every $\omega$. Hence, arguing exactly as in Lemma \ref{lem:r1}, and decreasing $a>0$ if necessary, we obtain $
r_n(\theta^{-n}\omega)\ge an$
for every sufficiently large $n$, for $\mathbb P$-almost every $\omega$.

Thus, for $\mathbb P$-almost every $\omega$, the random time
$$
N_H'(\omega):=
\min\left\{
N\in\mathbb N:
r_n(\theta^{-n}\omega)\ge an
\text{ for every }n\ge N
\right\}
$$
is finite. Define
$$
\widetilde{K}_H'(\omega)
:=
\max\left\{
1,
\max_{1\le n<N_H'(\omega)}
\chi^{r_n(\theta^{-n}\omega)}\chi_H^{-n}
\right\}.
$$
Then, as above,
\begin{align}
\chi^{r_n(\theta^{-n}\omega)}
\le
\widetilde K_H'(\omega)\chi_H^n
\ \text{for every }n\in\mathbb N.
\label{eq:decor1}
\end{align}
The proof is finished by setting
$$
K_H(\omega):=\widetilde K_H(\omega)+\widetilde K_H'(\omega)
$$
and using \eqref{eq:decor} and \eqref{eq:decor1}.

In the following we show (2). Now assume Hypothesis \ref{hyp:h'}. Then there exist $\beta,K_\beta,\kappa_\beta>0$ such that
$$
\mathbb P[r_n\le \beta n]\le K_\beta e^{-\kappa_\beta n}
\  \text{for every } n\in\mathbb N.
$$
Define the last bad time
$$
N_\beta(\omega):=\sup\{n\ge1:r_n(\omega)\le \beta n\}\in \mathbb N\cup\{0\}.
$$
By Borel--Cantelli, $N_\beta(\omega)<\infty$ for $\mathbb P$-almost every $\omega$.
Moreover, for every $j\ge1$,
$$
\{N_\beta\ge j\}\subset \bigcup_{m\ge j}\{r_m\le \beta m\},
$$
hence
$$
\mathbb P[N_\beta\ge j]
\le \sum_{m\ge j}\mathbb P[r_m\le \beta m]
\le \frac{K_\beta}{1-e^{-\kappa_\beta}}e^{-\kappa_\beta j}.
$$

Fix $p\in[1,\infty)$ and choose $
\chi_p\in\bigl(\max\{\chi^\beta,e^{-\kappa_\beta/p}\},1\bigr).$
Set $\widetilde{K}_p(\omega):=\chi_p^{-N_\beta(\omega)}.$ If $n\le N_\beta(\omega)$, then $
\chi^{r_n(\omega)}\le 1\le \chi_p^{-N_\beta(\omega)}\chi_p^n=\widetilde{K}_p(\omega)\chi_p^n.$
If $n>N_\beta(\omega)$, then necessarily $r_n(\omega)>\beta n$, so $
\chi^{r_n(\omega)}\le \chi^{\beta n}\le \chi_p^n\le \widetilde{K}_p(\omega)\chi_p^n.$
Thus 
\begin{align}
\chi^{r_n(\omega)}\le \widetilde{K}_p(\omega)\chi_p^n \ \text{for each }n\in\mathbb N.\label{eq:decor2}
\end{align}

Finally, since $\chi_p^{-p}e^{-\kappa_\beta}<1$, we obtain
\begin{align*}
\mathbb E\left[\widetilde{K}_p^p\right]
&=\mathbb E\big[\chi_p^{-pN_\beta}\big]\le 1+\sum_{j=1}^\infty \chi_p^{-pj}\,\mathbb P[N_\beta\ge j]\\
&\le 1+\frac{K_\beta}{1-e^{-\kappa_\beta}}
\sum_{j=1}^\infty (\chi_p^{-p}e^{-\kappa_\beta})^j<\infty.
\end{align*}
Hence $\widetilde{K}_p\in L^p(\Omega)$. 

Define the random variable
$$N_\beta'(\omega) := \sup\{n\geq 1: r_n(\theta^{-n}\omega)\leq \beta n\}  \in \mathbb N\cup\{0\}.$$
Since
$$ \mathbb P[\omega\in \Omega : r_n(\theta^{-n}\omega)\leq \beta n] =  \mathbb P[ r_n \leq \beta n] \leq K_\beta e^{-\kappa_\beta n}.$$
 The same computations apply verbatim implying that when defining $\widetilde{K}_p'(\omega) = \chi_p^{-N_\beta'(\omega)}$ we obtain that $\widetilde{K}_p'\in L^p(\Omega)$ and 
\begin{align}
    \chi^{r_n(\theta^{-n}\omega)} \leq \widetilde{K}_p'(\omega) \chi_p^{n}\ \text{for each }n\geq 0.\label{eq:decor3}
\end{align}
By setting $K_p = \widetilde{K}_p + \widetilde{K}_p'$ from equations \eqref{eq:decor2} and \eqref{eq:decor3} we obtain the desired result.

\end{proof}

\begin{theorem}\label{thm:decay}
Assume Hypothesis \ref{hyp:h}, and let $\mu_\omega$ and $\ell_\omega$ be as in Theorem \ref{thm:spectralgap}. Then there exist a measurable function $K:\Omega\to\mathbb R$ and a constant $\chi\in(0,1)$ such that, for every $f\in \mathbb V_\omega$ and $n\geq1$,
$$
\left\|\frac{1}{\lambda_\omega^{(n)}}\mathcal L_\omega^{n}f-\ell_\omega(f)\mu_{\theta^n\omega}\right\|_{\theta^n\omega}\leq K(\omega)\|\ell_\omega\|_{\mathbb V_\omega^*}\|f\|_\omega\chi^n.
$$
If, in addition, Hypothesis \ref{hyp:h'} holds, given $p\in [1,\infty)$ there exists $\chi_p \in (0,1)$ and $K_p\in L^p(\Omega)$ such that 
$$
\left\|\frac{1}{\lambda_\omega^{(n)}}\mathcal L_\omega^{n}f-\ell_\omega(f)\mu_{\theta^n\omega}\right\|_{\theta^n\omega}\leq K_p(\omega)\|\ell_\omega\|_{\mathbb V_\omega^*} \|f\|_\omega\chi_p^n.
$$
\end{theorem}

\begin{proof} From Theorem \ref{thm:spectralgap}, for every $f\in\mathbb V_\omega$ and $n\ge1$,
$$
\left\|\frac{1}{\lambda_\omega^{(n)}}\mathcal L_\omega^{n}f
-\ell_\omega(f)\mu_{\theta^n\omega}\right\|_{\theta^n\omega}
\le K_0\|\ell_\omega\|_{\mathbb V_\omega^*}\|f\|_\omega
\chi^{\max\{r_n(\omega)-1,0\}}.
$$
Since $
\chi^{\max\{r_n(\omega)-1,0\}}\le \chi^{-1}\chi^{r_n(\omega)},$
from Lemma \ref{lem:chi-rn} we obtain that under Hypothesis \ref{hyp:h}, there exists a measurable function $K_H:\Omega\to \mathbb R$ and  $\chi_H \in (0,1)$ such that
$$
\chi^{\max\{r_n(\omega)-1,0\}}
\le \chi^{-1}K_H(\omega)\chi_H^n,
$$
and therefore
$$
\left\|\frac{1}{\lambda_\omega^{(n)}}\mathcal L_\omega^{n}f
-\ell_\omega(f)\mu_{\theta^n\omega}\right\|_{\theta^n\omega}
\le   K(\omega)\|\ell_\omega\|_{\mathbb V_\omega^*}\|f\|_\omega\chi_H^n,
$$
where $K(\omega):=\chi^{-1}K_0K_H(\omega)$.

Under Hypothesis \ref{hyp:h'}, Lemma \ref{lem:chi-rn} yields, for every $p\in[1,\infty)$,
$$
\chi^{\max\{r_n(\omega)-1,0\}}
\le \chi^{-1}K_p(\omega)\chi_p^n,
$$
with $K_p\in L^p(\Omega)$. Hence
$$
\left\|\frac{1}{\lambda_\omega^{(n)}}\mathcal L_\omega^{n}f
-\ell_\omega(f)\mu_{\theta^n\omega}\right\|_{\theta^n\omega}
\le   K'_p(\omega)\|\ell_\omega\|_{\mathbb V_\omega^*}\|f\|_\omega\chi_p^n,
$$
where $ K'_p:=\chi^{-1}K_0K_p\in L^p(\Omega)$.\end{proof}

\subsection{The candidate measure \texorpdfstring{$\upsilon_\phi$}{upsphi}, its invariance and decay of correlations}

We now use the random eigenvectors constructed in Theorem \ref{thm:spectralgap} to construct a candidate for a $\mathbb P$-relative equilibrium state for the potential $\phi- \phi^{J^s}$. For each $\omega$, we set
$\upsilon_\omega(f)=\ell_\omega(f\mu_\omega)$
for H\"older observables $f$. The main point is to prove that this functional is positive and normalised, and hence that it extends to a Borel probability measure on $M$.

This positivity is not immediate from the spectral construction. Indeed, at this stage $\ell_\omega$ is only a dual eigenfunctional on the cone-adapted space, and $\mu_\omega$ is obtained as a random eigenvector. We therefore prove positivity indirectly, using the cone structure together with the exponential convergence obtained above.

Once the measures $\upsilon_\omega$ have been constructed, we verify that they are transported by the dynamics and that they inherit decay of correlations from the spectral decomposition. This completes the passage from the analytic Perron--Frobenius cocycle to the dynamical object used later in the variational principle.  

The rest of this subsection is devoted to turning the spectral objects obtained in Theorem \ref{thm:spectralgap} into a genuine random family of probability measures. We first prove a sequence of auxiliary lemmas showing that the relevant functionals are well defined, positive, and bounded on H\"older observables. This culminates in the construction of the fibre measures $\upsilon_\omega$. We then prove Proposition \ref{prop:exponentialdecayofcorrelation}, which establishes their equivariance under the random dynamics and derives quenched decay of correlations from the spectral decomposition.
\begin{lemma}\label{lem:L}
Let $\rho\in D_1(a,\kappa,\gamma_\omega)$ and $g\in \mathcal C^\beta(M)$, with $\kappa\le \beta$. Set
$$
\widetilde D:=\sup_{\omega\in\Omega}\sup_{\gamma_\omega\in\mathscr F_\omega^s}
m_{\gamma_\omega}(\gamma_\omega)<\infty\ \text{and }D_{\min}:=\inf_{\omega\in\Omega}\inf_{\gamma_\omega\in\mathscr F_\omega^s}
m_{\gamma_\omega}(\gamma_\omega)>0.
$$
Let $
C_{\beta,\kappa}:=\max\{1,\widetilde D^{\beta-\kappa}\}$ and define
\begin{align}
L:=
\frac{4}{a}
D_{\min}^{-1}e^{a\widetilde D^\kappa}
(a+C_{\beta,\kappa})
\|g\|_{\mathcal C^\beta}.
\label{eq:L}
\end{align}
Then
$$
\frac{1}{\int_{\gamma_\omega}\left(1+\rho g/L\right)}
\left(1+\frac{\rho g}{L}\right)
\in D_1(a/2,\kappa,\gamma_\omega).
$$
Moreover,
\begin{align}
\int_{\gamma_\omega}\left(1+\rho g/L\right)
\le \widetilde D+1.
\label{eq:d+1}
\end{align}
\end{lemma}

\begin{proof}
Since $\int_{\gamma_\omega}\rho\,dm_{\gamma_\omega}=1$, there exists
$x_0\in\gamma_\omega$ such that $
\rho(x_0)=1/{m_{\gamma_\omega}(\gamma_\omega)}.$
As $\rho$ is log-Hölder, for every $x\in\gamma_\omega$,
\begin{align}
\rho(x)
\le
e^{a d(x,x_0)^\kappa}\rho(x_0)
\le
e^{a\widetilde D^\kappa}D_{\min}^{-1}.
\label{eq:max}
\end{align}
Set $R:=e^{a\widetilde D^\kappa}D_{\min}^{-1}.$
Given $x,y\in\gamma_\omega$, assume without loss of generality that $\rho(y)\le\rho(x)$. Then
\begin{align}
|\rho(x)-\rho(y)|
&=
\rho(x)\left(1-\frac{\rho(y)}{\rho(x)}\right)  \le
R\left(1-e^{-a d(x,y)^\kappa}\right)  \le
Ra d(x,y)^\kappa .
\label{eq:holder}
\end{align}
By the definition of $L$, we have $L\ge 2R\|g\|_{\mathcal C^\beta}$. Hence
$$
1+\frac{\rho(x)g(x)}{L}
\ge
1-\frac{R\|g\|_\infty}{L}
\ge
\frac12
$$
for every $x\in\gamma_\omega$. Now, for $x,y\in\gamma_\omega$, using \eqref{eq:max}, \eqref{eq:holder}, and
$\kappa\le\beta$, we obtain
\begin{align*}
\left|
\log\left(1+\frac{\rho g}{L}(x)\right)
-
\log\left(1+\frac{\rho g}{L}(y)\right)
\right|
&\le
\frac{2}{L}
|g(x)\rho(x)-g(y)\rho(y)|
\\
&\le
\frac{2}{L}
\left(
|g(x)|\,|\rho(x)-\rho(y)|
+
\rho(y)|g(x)-g(y)|
\right)
\\
&\le
\frac{2}{L}
\left(
\|g\|_\infty Ra d(x,y)^\kappa
+
R\|g\|_{\mathcal C^\beta}C_{\beta,\kappa}d(x,y)^\kappa
\right)
\\
&\le
\frac{2R(a+C_{\beta,\kappa})\|g\|_{\mathcal C^\beta}}{L}
d(x,y)^\kappa\le
\frac a2 d(x,y)^\kappa.
\end{align*}
Therefore $1+\rho g/L$ is positive and belongs to
$D(a/2,\kappa,\gamma_\omega)$. Multiplication by the positive constant
$$
\left(\int_{\gamma_\omega}(1+\rho g/L)\,dm_{\gamma_\omega}\right)^{-1}
$$
does not change the log-Hölder constant, and it normalises the integral to $1$. Hence
the normalised density belongs to $D_1(a/2,\kappa,\gamma_\omega)$.

Finally,
$$
\int_{\gamma_\omega}\left(1+\frac{\rho g}{L}\right)\,dm_{\gamma_\omega}
=
m_{\gamma_\omega}(\gamma_\omega)
+
\frac{1}{L}\int_{\gamma_\omega}g\rho\,dm_{\gamma_\omega}
\le
\widetilde D+\frac{\|g\|_\infty}{L}.
$$
Since $L\ge \|g\|_\infty$ after increasing the constant in \eqref{eq:L} if necessary, we get
$$
\int_{\gamma_\omega}\left(1+\frac{\rho g}{L}\right)\,dm_{\gamma_\omega}
\le
\widetilde D+1.
$$
\end{proof}

\begin{lemma}\label{lem:6.17}
After enlarging $b,c>0$, if necessary, in the definition of the cone
$\mathcal C_\omega(b,c,\nu)$, the following holds.
Let $f\in  \mathcal C^\beta(M). $
Then there exists $L_3>0$ such that, whenever $K\ge L_3\|f\|_{\mathcal C^\beta},$
one has
$$
\left(1+\frac{f}{K}\right)\mathcal L_\omega \eta \in \mathcal C_{\theta\omega}(b,c,\nu)
$$
for every $\eta\in \mathcal C_\omega(b,c,\nu)$. In particular, one obtains that
$$
\left(1+\frac{f}{K}\right)\mathcal L^n_\omega \eta \in \mathcal C_{\theta^n\omega}(b,c,\nu)
$$
for any $n\in\mathbb N$.
\end{lemma}

\begin{proof}
Let $0<\alpha_0<1$ be as in Lemma \ref{lem:push_density} and  fix $K\ge \frac{2\alpha_0}{a(1-\alpha_0)}\|f\|_{\mathcal C^\beta}$. We divide the proof into three steps.

\begin{step}[1]
We show that $\left(1+\frac{f}{K}\right)\mathcal L_\omega\eta$ satisfies \textbf{(C1)} of
$\mathcal C_{\theta\omega}(b,c,\nu)$.
\end{step}

Let $\gamma_{\theta\omega}\in \mathscr F_{\theta\omega}^s$ and
$\rho\in D(a,\kappa,\gamma_{\theta\omega})$. First observe that
\begin{align*}
\left|\log\!\left(\rho(x)\left(1+\frac{f(x)}{K}\right)\right)
-\log\!\left(\rho(y)\left(1+\frac{f(y)}{K}\right)\right)\right|
&\le a\,d(x,y)^\kappa+\frac{\|f\|_{\mathcal C^\beta}}{K}d(x,y)^\kappa \\
&\le \left(a+\frac{\|f\|_{\mathcal C^\beta}}{K}\right)d(x,y)^\kappa \\
&\le \frac12\left(a+\frac{a}{\alpha_0}\right)d(x,y)^\kappa.
\end{align*}
Thus
$$
\hat\rho:=\rho\left(1+\frac{f}{K}\right)\in
D\left(\frac a2+\frac{a}{2\alpha_0},\kappa,\gamma_{\theta\omega}\right).
$$
By Lemma \ref{lem:push_density},
\begin{align}
\int_{\gamma_{\theta\omega}}\rho\left(1+\frac{f}{K}\right)\mathcal L_\omega\eta
&=\int_{\gamma_{\theta\omega}}\hat\rho\,\mathcal L_\omega\eta =\sum_i
\left(\int_{\gamma_\omega^{(i)}}\hat\rho^{(i)}\right)
\left(
\int_{\gamma_\omega^{(i)}}
\frac{\hat\rho^{(i)}}{\int_{\gamma_\omega^{(i)}}\hat\rho^{(i)}}\,\eta
\right),
\label{eq:hatrho-fixed}
\end{align}
where, for each $i$,
\begin{align}
\frac{\hat\rho^{(i)}}{\int_{\gamma_\omega^{(i)}}\hat\rho^{(i)}}
\in
D_1\left(\frac{\alpha_0 a}{2}+\frac a2,\kappa,\gamma_\omega^{(i)}\right)
\subset D(a,\kappa,\gamma_\omega^{(i)}).
\label{eq:inclusion-fixed}
\end{align}
Since $\eta\in \mathcal C_\omega(b,c,\nu)$, each term in \eqref{eq:hatrho-fixed}
is nonnegative. Therefore
$$
\int_{\gamma_{\theta\omega}}\rho\left(1+\frac{f}{K}\right)\mathcal L_\omega\eta \ge 0,
$$
which proves \textbf{(C1)}.

\begin{step}[2]
We show that for  $b>0$ is chosen sufficiently large  then $\left(1+\frac{f}{K}\right)\mathcal L_\omega\eta$ satisfies \textbf{(C2)} of
$\mathcal C_{\theta\omega}(b,c,\nu)$.
\end{step}

Let $\gamma_{\theta\omega}\in \mathscr F_{\theta\omega}^s$ and
$\rho,\varsigma\in D_1(a,\kappa,\gamma_{\theta\omega})$. By \eqref{eq:hatrho-fixed},
\begin{align}
\int_{\gamma_{\theta\omega}}\rho\left(1+\frac{f}{K}\right)\mathcal L_\omega\eta
&=
\sum_i
\left(\int_{\gamma_\omega^{(i)}}\hat\rho^{(i)}\right)
\left(
\int_{\gamma_\omega^{(i)}}
\frac{\hat\rho^{(i)}}{\int_{\gamma_\omega^{(i)}}\hat\rho^{(i)}}\,\eta
\right),
\end{align}
and similarly
\begin{align*}
\int_{\gamma_{\theta\omega}}\varsigma\left(1+\frac{f}{K}\right)\mathcal L_\omega\eta
&=
\sum_i
\left(\int_{\gamma_\omega^{(i)}}\hat\varsigma^{(i)}\right)
\left(
\int_{\gamma_\omega^{(i)}}
\frac{\hat\varsigma^{(i)}}{\int_{\gamma_\omega^{(i)}}\hat\varsigma^{(i)}}\,\eta
\right).
\end{align*}
For each $i$, $\hat\rho^{(i)},\hat\varsigma^{(i)}
\in D\left(\frac{\alpha_0 a}{2}+\frac a2,\kappa,\gamma_\omega^{(i)}\right).$
Since $\eta\in \mathcal C_\omega(b,c,\nu)$ and $
D_1\left(\frac{\alpha_0 a}{2}+\frac a2,\kappa,\gamma_\omega^{(i)}\right)$
has finite diameter in the Hilbert metric of $D_1(a,\kappa,\gamma_\omega^{(i)})$,
the same argument as in the proof of Lemma \ref{lem:densityimprovment} yields
a constant $0<h<1$ such that
\begin{align}
\int_{\gamma_\omega^{(i)}}
\frac{\hat\rho^{(i)}}{\int_{\gamma_\omega^{(i)}}\hat\rho^{(i)}}\,\eta
&\le
e^{b\Theta_{\gamma_\omega^{(i)}}^{a,\kappa}
\left(\hat\rho^{(i)},\hat\varsigma^{(i)}\right)}
\int_{\gamma_\omega^{(i)}}
\frac{\hat\varsigma^{(i)}}{\int_{\gamma_\omega^{(i)}}\hat\varsigma^{(i)}}\,\eta
\notag\\
&\le
e^{h b\,\Theta_{\gamma_{\theta\omega}}^{a,\kappa}(\rho,\varsigma)}
\int_{\gamma_\omega^{(i)}}
\frac{\hat\varsigma^{(i)}}{\int_{\gamma_\omega^{(i)}}\hat\varsigma^{(i)}}\,\eta.
\label{eq:h-fixed}
\end{align}
Using \eqref{eq:h-fixed} and \eqref{eq:ineqc22}, we obtain
\begin{align*}
\int_{\gamma_{\theta\omega}}\rho\left(1+\frac{f}{K}\right)\mathcal L_\omega\eta
&=
\sum_i
\left(\int_{\gamma_\omega^{(i)}}\hat\rho^{(i)}\right)
\left(
\int_{\gamma_\omega^{(i)}}
\frac{\hat\rho^{(i)}}{\int_{\gamma_\omega^{(i)}}\hat\rho^{(i)}}\,\eta
\right) \\
&\le
\sum_i
e^{h b\,\Theta_{\gamma_{\theta\omega}}^{a,\kappa}(\rho,\varsigma)}
\left(\int_{\gamma_\omega^{(i)}}\hat\rho^{(i)}\right)
\left(
\int_{\gamma_\omega^{(i)}}
\frac{\hat\varsigma^{(i)}}{\int_{\gamma_\omega^{(i)}}\hat\varsigma^{(i)}}\,\eta
\right) \\
&\le
e^{h b\,\Theta_{\gamma_{\theta\omega}}^{a,\kappa}(\rho,\varsigma)
+\Theta_{\gamma_{\theta\omega}}^{a,\kappa}(\rho,\varsigma)}
\sum_i
\left(\int_{\gamma_\omega^{(i)}}\hat\varsigma^{(i)}\right)
\left(
\int_{\gamma_\omega^{(i)}}
\frac{\hat\varsigma^{(i)}}{\int_{\gamma_\omega^{(i)}}\hat\varsigma^{(i)}}\,\eta
\right) \\
&\le
e^{(h b+1)\Theta_{\gamma_{\theta\omega}}^{a,\kappa}(\rho,\varsigma)}
\int_{\gamma_{\theta\omega}}
\varsigma\left(1+\frac{f}{K}\right)\mathcal L_\omega\eta.
\end{align*}
Therefore, if $b>\frac{1}{1-h},$
then $\left(1+\frac{f}{K}\right)\mathcal L_\omega\eta$ satisfies \textbf{(C2)}.

\begin{step}[3]
We show that if $c>0$ is chosen sufficiently large then $\left(1+\frac{f}{K}\right)\mathcal L_\omega\eta$ satisfies \textbf{(C3)} of
$\mathcal C_{\theta\omega}(b,c,\nu)$. Which implies the lemma.
\end{step}

The proof of this step uses the decomposition \eqref{eq:hatrho-fixed}, the regularity
improvement in \eqref{eq:inclusion-fixed}, and then repeats verbatim the computation
used to check condition \textbf{(C3)} in the proof of
Proposition \ref{prop:cone_invariance_random}. The only modifications are the same
minor adaptations already made in Step 2.
\end{proof}

From now on we enlarge the constants $b,c>0$ so the conclusions of Lemma \ref{lem:6.17} holds.

\begin{corollary} \label{cor:cone}
 Let $\mu_\omega \in \mathcal C_\omega (b,c,\nu)$ constructed in \ref{thm:spectralgap}, $g\in \mathcal C^\beta(M)$, $L_3>0$ as in Lemma \ref{lem:6.17}. Then,
 $$\left(1+ \frac{g}{L_3 \|g\|_{\mathcal C^\beta}}\right) \mu_\omega \in \mathcal C_\omega(b,c,\nu).$$
\end{corollary}
\begin{proof}
Since $\mu_{\theta^{-1}\omega} \in \mathcal C_{\omega}(b,c,\nu)$ and
$$\left(1+ \frac{g}{L_3 \|g\|_{\mathcal C^\beta}}\right) \mu_\omega = \frac{1}{\lambda_{\theta^{-1}\omega}}\left(1+ \frac{g}{L_3 \|g\|_{\mathcal C^\beta}}\right) \mathcal L_{\theta^{-1} \omega} \mu_{\theta^{-1}\omega}$$
we obtain the desired result from Lemma \ref{lem:6.17}.    
\end{proof}

\begin{lemma}\label{lem:pos}
Let $f\in \mathcal C^\beta(M)$, and let $L_3>0$ be as in Lemma \ref{lem:6.17}. For each $n\in \mathbb N$, the formula
$$
\Psi_{n,f}(\eta):=
\ell_{\theta^n\omega}\!\left(\left(1+\frac{f}{L_3\|f\|_{\mathcal C^\beta}}\right)\mathcal L_\omega^n\eta\right)
$$
defines a continuous linear functional on $\mathbb V_\omega$. Moreover, $\Psi_{n,f}$ is positive with respect to the order relation $\preceq_\omega$, that is,
$$
\eta_1\preceq_\omega \eta_2,\ \text{then } 
\Psi_{n,f}(\eta_1)\le \Psi_{n,f}(\eta_2).
$$
Equivalently, $\Psi_{n,f}(\eta)\ge 0
\  \text{for every } \eta\in \mathcal C_\omega(b,c,\nu).$
\end{lemma}

\begin{proof}
We first define $\Psi_{n,f}$ on the dense subspace $BM_\omega(M)\subset \mathbb V_\omega$.

Let $\eta\in \mathcal C_\omega(b,c,\nu)$. By Lemma \ref{lem:6.17},
$$
\left(1+\frac{f}{L_3\|f\|_{\mathcal C^\beta}}\right)\mathcal L_\omega^n\eta
\in \mathcal C_{\theta^n\omega}(b,c,\nu).
$$
Since $\ell_{\theta^n\omega}$ is positive on $\mathcal C_{\theta^n\omega}(b,c,\nu)$, it follows that $\Psi_{n,f}(\eta)\ge 0.$
Thus $\Psi_{n,f}$ is positive on the cone, and therefore order-preserving with respect to $\preceq_\omega$.

We now prove boundedness. Let $\eta\in BM_\omega(M)$. By Lemma \ref{lem:archi},
$$
-3\|\eta\|_\omega\,\mathbbm 1 \preceq_\omega \eta \preceq_\omega 3\|\eta\|_\omega\,\mathbbm 1.
$$
Since $\Psi_{n,f}$ is order-preserving, we obtain
$$
-3\|\eta\|_\omega\,\Psi_{n,f}(\mathbbm 1)
\le
\Psi_{n,f}(\eta)
\le
3\|\eta\|_\omega\,\Psi_{n,f}(\mathbbm 1).
$$
Hence
$ |\Psi_{n,f}(\eta)|
\le 3\,\Psi_{n,f}(\mathbbm 1)\,\|\eta\|_\omega.$
So $\Psi_{n,f}$ is bounded on $BM_\omega(M)$, and therefore extends uniquely to a continuous linear functional on $\mathbb V_\omega$. This proves the lemma.
\end{proof}


\begin{lemma}\label{lem:C}
    Consider $f \in \mathcal C^{\beta}(M)$, then there exists $C_0>0$ such that $$\|f\|_\omega\leq C_0 \|f\|_{\mathcal C^\beta}\ \text{for every }\omega\in \Omega.$$  

\end{lemma}

\begin{proof}
Let $f\in \mathcal C^\beta(M)$. Recall that
\begin{equation*}
    \|f\|_\omega:= \|f\|_{\omega,a,\kappa}^{\sup_s} + \frac{1}{b} \|f\|_{\omega,a,\kappa}^{\Theta_s} + \frac{1}{c} \|f\|_{\omega,\nu}^{d_u}.
\end{equation*}
Then
\begin{equation*}
     \|f\|_{\omega,a,\kappa}^{\sup_s} = \sup_{\gamma \in \mathscr{F}^s_{\omega}} \sup_{\rho_\omega \in D_1(a,\kappa,\gamma_{\omega})} \left|\int_{\gamma_{\omega}} f \rho_\omega \right| \leq \|f\|_{\mathcal C^\beta}.
\end{equation*}
Moreover,
\begin{align*}
\|f\|_{\omega,a,\kappa}^{\Theta_s}
&\leq \|f\|_{\mathcal C^\beta}
\sup_{\gamma \in \mathscr{F}^s_{\omega}}
\sup_{\rho^{(1)}_\omega, \rho^{(2)}_\omega \in D_1(a,\kappa,\gamma_{\omega}) }
\frac{\int_{\gamma_{\omega}} \left| \rho^{(1)}_\omega  -   \rho^{(2)}_\omega \right|}{\Theta_{\gamma_{\omega}}^{a,\kappa}(\rho^{(1)}_\omega,\rho^{(2)}_\omega)}\\
&\leq \|f\|_{\mathcal C^\beta}
\sup_{\gamma \in \mathscr{F}^s_{\omega}}
\sup_{\rho^{(1)}_\omega, \rho^{(2)}_\omega \in D_1(a,\kappa,\gamma_{\omega}) }
\frac{\int_{\gamma_{\omega}} \rho^{(1)}_\omega\left|1 - \frac{\rho^{(2)}_\omega}{\rho^{(1)}_\omega}  \right|}{\Theta_{\gamma_{\omega}}^{a,\kappa}(\rho^{(1)}_\omega,\rho^{(2)}_\omega)}\\
&\leq \|f\|_{\mathcal C^\beta}
\sup_{\gamma \in \mathscr{F}^s_{\omega}}
\sup_{\rho^{(1)}_\omega, \rho^{(2)}_\omega \in D_1(a,\kappa,\gamma_{\omega}) }
\frac{\max\{e^{\Theta^+_{\gamma_\omega}(\rho_\omega^{(1)},\rho_\omega^{(2)})} -1 , 1 - e^{-\Theta^+_{\gamma_\omega}(\rho_\omega^{(1)},\rho_\omega^{(2)})}  \}}{\Theta_{\gamma_{\omega}}^{a,\kappa}(\rho^{(1)}_\omega,\rho^{(2)}_\omega)}\\
& \leq \|f\|_{\mathcal C^\beta}
\sup_{\gamma \in \mathscr{F}^s_{\omega}}
\sup_{\rho^{(1)}_\omega, \rho^{(2)}_\omega \in D_1(a,\kappa,\gamma_{\omega}) }
e^\Delta\frac{\Theta^+_{\gamma_\omega}(\rho^{(1)}_\omega,\rho^{(2)}_\omega)}{\Theta_{\gamma_{\omega}}^{a,\kappa}(\rho^{(1)}_\omega,\rho^{(2)}_\omega)}
\leq e^{\Delta} \|f\|_{\mathcal C^\beta}.
\end{align*}
The last inequality follows from the mean value theorem and from the fact that $D(a,\kappa,\gamma_\omega)$ has finite diameter, say $\Delta>0$, when viewed as a subset of $D_+(\gamma_\omega)$ endowed with its Hilbert metric.

Finally,
\begin{align*}
\|f\|_{\omega, \nu}^{d_u}
&= \sup_{\substack{ (\gamma_{\omega},\tilde\gamma_{\omega}) \in \mathscr{F}^s_{\omega} \times \mathscr{F}^s_{\omega} \\ \text{nearby pair}} }
\sup_{\rho \in D_1(a_1,\kappa_1, \gamma_{\omega}) }
\frac{\left|\int_{\gamma_{\omega}} f \rho_\omega  - \int_{\tilde\gamma_{\omega}} f \tilde\rho_\omega \right|}{d_u(\gamma_{\omega},\tilde\gamma_{\omega})^\nu}\\
&=\sup_{\substack{ (\gamma_{\omega},\tilde\gamma_{\omega}) \in \mathscr{F}^s_{\omega} \times \mathscr{F}^s_{\omega} \\ \text{nearby pair}} }
\sup_{\rho \in D_1(a_1,\kappa_1, \gamma_{\omega}) }
\frac{\left|\int_{\gamma_{\omega}} f \rho_\omega  - \int_{\tilde\gamma_{\omega}} f \rho_\omega \circ \mathrm{hol}_\omega^u\, \mathrm{Jac}(\mathrm{hol}_\omega^u) \right|}{d_u(\gamma_{\omega},\tilde\gamma_{\omega})^\nu}\\
&=\sup_{\substack{ (\gamma_{\omega},\tilde\gamma_{\omega}) \in \mathscr{F}^s_{\omega} \times \mathscr{F}^s_{\omega} \\ \text{nearby pair}} }
\sup_{\rho \in D_1(a_1,\kappa_1, \gamma_{\omega}) }
\frac{\left|\int_{\gamma_{\omega}} (f - f\circ \mathrm{hol}_\omega^u) \rho \right|}{d_u(\gamma_{\omega},\tilde\gamma_{\omega})^\nu}\\
&= \sup_{\substack{ (\gamma_{\omega},\tilde\gamma_{\omega}) \in \mathscr{F}^s_{\omega} \times \mathscr{F}^s_{\omega} \\ \text{nearby pair}} }
\sup_{\rho \in D_1(a_1,\kappa_1, \gamma_{\omega}) }
\|f\|_{\mathcal{C}^\beta} \frac{d(\gamma_\omega,\widetilde{\gamma}_\omega)^\beta}{d(\gamma_\omega,\widetilde{\gamma}_\omega)^\nu}
\leq \|f\|_{\mathcal C^\beta}.
\end{align*}
We may therefore take $C_0 = 1 + e^\Delta/b + 1/c$.

\end{proof}

\begin{lemma}\label{lem:positivemu}
Let $\mu_\omega$ be as in Theorem~\ref{thm:spectralgap}, and let
$\gamma_\omega\in\mathscr F_\omega^s$. Then the map
\begin{align*}
    \mu_{\gamma_\omega}: g\in \mathcal C^\kappa(\gamma_\omega)&\mapsto \int_{\gamma_\omega}g\,\mu_\omega \in \mathbb R,
\end{align*}
is a well-defined positive linear functional (see Definition \ref{def:mugamma}). Moreover,
$$
\left|\mu_{\gamma_\omega}(g)\right|
\leq
\mu_{\gamma_\omega}(1)\|g\|_\infty
$$
for every $g\in\mathcal C^\kappa(\gamma_\omega)$. Consequently,
$\mu_{\gamma_\omega}$ extends uniquely to a positive bounded linear
functional on $\mathcal C^0(\gamma_\omega)$. By the Riesz representation
theorem, it can be identified with a finite positive Borel measure on
$\gamma_\omega$.
\end{lemma}

\begin{proof}
From Proposition \ref{prop:disint} and Lemma \eqref{lem:C} we obtain that $\mu_{\gamma_\omega} \in (\mathcal C^k(M))^*$.  We now show that $\mu_{\gamma_\omega}$ is positive in the cone $D(a,\kappa,\gamma_\omega)$.
Let $\rho,\rho'\in D(a,\kappa,\gamma_\omega)$ be such that
$\rho(x)\ge \rho'(x)$ for every $x\in \gamma_\omega$. Then
$$
\int_{\gamma_\omega}\rho\,\mu_\omega
=
\lim_{n\to\infty}
\int_{\gamma_\omega}
\rho\,
\frac{\mathcal L_{\theta^{-n}\omega}^n\mathbbm 1}
{\|\mathcal L_{\theta^{-n}\omega}^n\mathbbm 1\|_\omega}
\ge
\lim_{n\to\infty}
\int_{\gamma_\omega}
\rho'\,
\frac{\mathcal L_{\theta^{-n}\omega}^n\mathbbm 1}
{\|\mathcal L_{\theta^{-n}\omega}^n\mathbbm 1\|_\omega}
=
\int_{\gamma_\omega}\rho'\,\mu_\omega.
$$
Thus the map $\rho\mapsto \int_{\gamma_\omega}\rho\,\mu_\omega$
is preserves the pointwise order on $D(a,\kappa,\gamma_\omega)$.

Now let $f\in \mathcal C^\kappa(\gamma_\omega)$ be such that $f\ge 0$. By
Lemma \ref{lem:L}, there exists a constant $C>0$ such that
$$
1+\frac{f}{C\|f\|_{\mathcal C^\kappa}}
\in D(a/2,\kappa,\gamma_\omega)\subset D(a,\kappa,\gamma_\omega).
$$
Since
$$
1+\frac{f}{C\|f\|_{\mathcal C^\kappa}}\ge 1
\  \text{pointwise on } \gamma_\omega,
$$
the monotonicity established above yields
$$
\int_{\gamma_\omega}
\left(1+\frac{f}{C\|f\|_{\mathcal C^\kappa}}\right)\mu_\omega
\ge
\int_{\gamma_\omega}\mu_\omega.
$$
Hence
$$
\int_{\gamma_\omega} f\,\mu_\omega \ge 0.
$$
Therefore $\mu_{\gamma_\omega}$ is positive on $\mathcal C^\kappa(\gamma_\omega)$.

Moreover, if $f\ge 0$, then $0\le f\le \|f\|_\infty \  \text{on } \gamma_\omega,$
and another application of monotonicity gives
$$
0\le \int_{\gamma_\omega} f\,\mu_\omega
\le
\|f\|_\infty \int_{\gamma_\omega}\mu_\omega.
$$
By linearity, for arbitrary $f\in \mathcal C^\kappa(\gamma_\omega)$ we may write
$f=f^+-f^-$ and obtain
$$
\left|\int_{\gamma_\omega} f\,\mu_\omega\right|
\le
2\|f\|_\infty \int_{\gamma_\omega}\mu_\omega.
$$
Thus $\mu_{\gamma_\omega}$ is bounded with respect to the supremum norm on
$\mathcal C^\kappa(\gamma_\omega)$.

Since $\mathcal C^\kappa(\gamma_\omega)$ is dense in $\mathcal C^0(\gamma_\omega)$,
$\mu_{\gamma_\omega}$ extends uniquely to a bounded positive linear functional on
$\mathcal C^0(\gamma_\omega)$. The last claim then follows from the Riesz
representation theorem.
\end{proof}

With the above technical lemmas in hand we can show that $\mu_\omega,\ell_\omega$ and $\upsilon_\omega$ are measures on $M$.
\begin{lemma}\label{lem:nuprob}
Let $\mu_\omega$ and $\ell_\omega$ as in Theorem \ref{thm:spectralgap}. Define the operator $\upsilon_{\omega}: f \in \mathcal C^{\beta}(M) \mapsto \ell_\omega(f \mu_\omega) \in\mathbb{R}$ is positive, linear and extends to a probability measure  on $M$. 
\end{lemma}

\begin{proof}
Given $f\in\mathcal C^\beta(M)$, Lemma \ref{cor:cone} yields a constant $L_3>0$ such that
$$
\left(1+\frac{f}{L_3\|f\|_{\mathcal C^\beta}}\right)\mu_\omega\in\mathcal C_\omega (b,c,\nu).
$$
Consequently,
$$
f\mu_\omega
=L_3\|f\|_{\mathcal C^\beta}\left[\left(1+\frac{f}{L_3\|f\|_{\mathcal C^\beta}}\right)\mu_\omega-\mu_\omega\right]\in\mathbb V_\omega.
$$
Therefore,
\begin{align*}
|\upsilon_\omega(f)|&=|\ell_\omega(f\mu_\omega)|\le L_3\|f\|_{\mathcal C^\beta}\left|\ell_\omega\left(\left(1+\frac{f}{L_3\|f\|_{\mathcal C^\beta}}\right)\mu_\omega\right)\right|
+L_3\|f\|_{\mathcal C^\beta}|\ell_\omega(\mu_\omega)|\\
&\le L_3\|f\|_{\mathcal C^\beta}\left(\left|\ell_\omega\left(\left(1+\frac{f}{L_3\|f\|_{\mathcal C^\beta}}\right)\mathcal L_{\theta^{-1}\omega}\mu_{\theta^{-1}\omega}\right)\right|+1\right)\\
&\le L_3\|f\|_{\mathcal C^\beta}\left(\|\Psi_{1,f}\|+1\right)\le C_1\|f\|_{\mathcal C^\beta},
\end{align*}
where $\Psi_{1,f}:\mathbb V_{\theta^{-1}\omega}\to\mathbb R$ is the linear functional defined by
$$
\Psi_{1,f}(\eta)
=\ell_\omega\left(\left(1+\frac{f}{L_3\|f\|_{\mathcal C^\beta}}\right)\mathcal L_{\theta^{-1}\omega}\eta\right).
$$
By Lemma \ref{lem:pos}, $\Psi_{1,f}$ is bounded. Hence the estimate follows after absorbing $\|\Psi_{1,f}\|+1$ into the constant $C_1$. Therefore $\upsilon_\omega \in (\mathcal C^\beta(M))^*$.

We now assume that $f(x)\geq 0$ for every $x\in M$. From Lemma \ref{lem:norm1} we have that
$$1=\|\mu_\omega\|_\omega\leq 3\|\mu_\omega\|_{\omega,a,\kappa}^{\sup_s}.$$ 
For each $n\in \mathbb N$, choose
$\gamma_{\theta^n\omega}^{(n)}\in \mathscr F^s_{\theta^n\omega}$ and
$\rho_{\theta^n\omega}^{(n)}\in D_1(a,\kappa,\gamma_{\theta^n\omega}^{(n)})$
such that
\begin{align}
\frac{1}{4} \leq \int_{\gamma_{\theta^n\omega}^{(n)}}
\rho_{\theta^n\omega}^{(n)}\mu_{\theta^n\omega}.\label{eq:14}
\end{align}
Define the operator $\Gamma_n\in \mathbb V_{\theta^n\omega}^*$ by
\[
\Gamma_n(g):=\int_{\gamma_{\theta^n\omega}^{(n)}}\rho_{\theta^n\omega}^{(n)}g.
\]
Since $f\geq 0$, from Lemma \ref{lem:positivemu} we obtain that
\begin{equation}
    \Gamma_n\bigl(\mathcal L_\omega^{n}(f\mu_\omega)\bigr)\geq 0
    \  \text{for every }n\in \mathbb N.
    \label{eq:gammaln}
\end{equation}

From Theorem \ref{thm:decay}, 
there exists a measurable
function $K:\Omega\to \mathbb R$ such that for every $n\in \mathbb N$,
$$\left\|
\frac{1}{\lambda_\omega^{(n)}}\mathcal L_\omega^{n}(f\mu_\omega)
-\ell_\omega(f\mu_\omega)\mu_{\theta^n\omega}
\right\|_{\theta^n\omega}
\leq K(\omega)\|\ell_\omega\|_{\mathbb V_\omega^*}
\|f\mu_\omega\|_{\mathbb V_\omega}\chi^n.$$
Therefore,
\begin{equation}
\left|
\frac{1}{\lambda_\omega^{(n)}}\Gamma_n\bigl(\mathcal L_\omega^{n}(f\mu_\omega)\bigr)
-\ell_\omega(f\mu_\omega)\Gamma_n(\mu_{\theta^n\omega})
\right|
\leq
K(\omega)\|\ell_\omega\|_{\mathbb V_\omega^*}
\|f\mu_\omega\|_{\mathbb V_\omega}\chi^n.
\label{eq:gammaln1}
\end{equation}

Combining \eqref{eq:14}, \eqref{eq:gammaln} and \eqref{eq:gammaln1}, we obtain
\begin{align*}
\left|
\frac{1}{\lambda_\omega^{(n)}}\frac{\Gamma_n\bigl(\mathcal L_\omega^{n}(f\mu_\omega)\bigr)}{\Gamma_n (\mu_{\theta^n \omega}} 
-\ell_\omega(f\mu_\omega)
\right|
&\leq\frac{
K(\omega)}{\Gamma_n (\mu_{\theta^n \omega})}\|\ell_\omega\|_{\mathbb V_\omega^*}
\|f\mu_\omega\|_{\mathbb V_\omega}\chi^n.\\
&\leq 4 
K(\omega)\|\ell_\omega\|_{\mathbb V_\omega^*}
\|f\mu_\omega\|_{\mathbb V_\omega}\chi^n.
\end{align*}
In this way
\[
\upsilon_\omega(f)=\ell_\omega(f\mu_\omega)
=\lim_{n\to\infty}
\frac{1}{\lambda_\omega^{(n)}}
\frac{\Gamma_n\bigl(\mathcal L_\omega^{n}(f\mu_\omega)\bigr)}
{\Gamma_n(\mu_{\theta^n\omega})}
\geq 0
\]
for every $f\in \mathcal C^\beta(M)$ with $f\geq 0$.

Finally, observe that $\upsilon_\omega(1)=1$. Since $\upsilon_\omega$ is positive on $\mathcal C^\beta(M)$ under the pointwise partial order,
it follows that
$$|\upsilon_\omega(f)|\leq \|f\|_\infty
\  \text{for every }f\in \mathcal C^\beta(M).$$
By the Riesz representation theorem, $\upsilon_\omega$ extends to a probability
measure on $M$.
\end{proof}

\begin{proposition} \label{prop:exponentialdecayofcorrelation}
Assume Hypothesis \ref{hyp:h}. Let $\mu_\omega,\ell_\omega$ be as in Theorem \ref{thm:spectralgap}, and let $\upsilon_\omega$ be the probability measure defined in Lemma \ref{lem:nuprob}. Then, for each $\omega\in\Omega$, $
(T_\omega)_*\upsilon_\omega=\upsilon_{\theta\omega}.$
Define $
\upsilon_\phi(\d\omega,\d x)=\upsilon_\omega(\d x)\,\mathbb P(\d\omega).$
Then $\upsilon_\phi$ exhibits quenched decay of correlations for Hölder observables: for every $f,g\in \mathcal C^\beta(M)$ there exist $0<\Lambda<1$ and a measurable function $C:\Omega\to \mathbb R$ such that
\begin{equation}
\left|
\int_M f\circ T_\omega^n \cdot g\,\d\upsilon_\omega
-
\int_M f\,\d\upsilon_{\theta^n\omega}\int_M g\,\d\upsilon_\omega
\right|
\le
C(\omega)\Lambda^n \|f\|_{\mathcal C^\beta}\|g\|_{\mathcal C^\beta}.\label{eq:decay1}
\end{equation}
and 
\begin{equation}
\left|
\int_M f\circ T_{\theta^{-n}\omega}^n \cdot g\,\d\upsilon_{\theta^{-n}\omega}
-
\int_M f\,\d\upsilon_{\omega}\int_M g\,\d\upsilon_{\theta^{-n}\omega}
\right|
\le
C(\omega)\Lambda^n \|f\|_{\mathcal C^\beta}\|g\|_{\mathcal C^\beta}.\label{eq:decaypullback1}
\end{equation}

If, in addition, Hypothesis \ref{hyp:h'} holds, then given $p\in[1,\infty)$ there exist $0<\Lambda_p<1$ and $C_p\in L^p(\Omega)$ such that, for every $f,g\in \mathcal C^\beta(M)$,
\begin{equation}
\left|
\int_M f\circ T_\omega^n \cdot g\,\d\upsilon_\omega
-
\int_M f\,\d\upsilon_{\theta^n\omega}\int_M g\,\d\upsilon_\omega
\right|
\le
C_p(\omega)\Lambda_p^n \|f\|_{\mathcal C^\beta}\|g\|_{\mathcal C^\beta}.\label{eq:decay2}
\end{equation}
and
\begin{equation}
\left|
\int_M f\circ T_{\theta^{-n}\omega}^n \cdot g\,\d\upsilon_{\theta^{-n}\omega}
-
\int_M f\,\d\upsilon_{\omega}\int_M g\,\d\upsilon_{\theta^{-n}\omega}
\right|
\le
C_p(\omega)\Lambda_p^n \|f\|_{\mathcal C^\beta}\|g\|_{\mathcal C^\beta}. \label{eq:decaypullback2}
\end{equation}

\end{proposition}

\begin{proof}
Let $f\in \mathcal C^\beta(M)$. Then
\begin{align*}
\upsilon_\omega(f\circ T_\omega)
&=\ell_\omega(f\circ T_\omega \cdot \mu_\omega) =\frac{1}{\lambda_\omega}\mathcal L_\omega^*\ell_{\theta\omega}(f\circ T_\omega \cdot \mu_\omega) =\ell_{\theta\omega}\left(
\frac{1}{\lambda_\omega}\mathcal L_\omega(f\circ T_\omega \cdot \mu_\omega)
\right) \\
&=\ell_{\theta\omega}\left(
\frac{1}{\lambda_\omega}f\,\mathcal L_\omega\mu_\omega
\right) =\ell_{\theta\omega}(f\mu_{\theta\omega}) =\upsilon_{\theta\omega}(f).
\end{align*}
Hence $(T_\omega)_*\upsilon_\omega=\upsilon_{\theta\omega}$. We divide the remaining of the proof into three steps.

\begin{step}[1]
Assume first that $g\in \mathcal C^\beta(M)$ satisfies $\upsilon_\omega(g)=0$, and that
\begin{equation}\label{eq:fform-clean-2}
f=1+\frac{\bar f}{L_3\|\bar f\|_{\mathcal C^\beta}}
\end{equation}
for some $\bar f\in \mathcal C^\beta(M)$, where $L_3$ is the constant from Lemma \ref{lem:6.17}. We prove that \eqref{eq:decay1} and \eqref{eq:decay2} decay estimates in this case.
\end{step}

Since $\upsilon_\omega(g)=0$, we have that
\begin{align}
\left|
\upsilon_\omega(f\circ T_\omega^n \cdot g)
-
\upsilon_{\theta^n\omega}(f)\upsilon_\omega(g)
\right|&=
\left|
\upsilon_\omega(f\circ T_\omega^n \cdot g)
\right| =
\left|
\ell_\omega(f\circ T_\omega^n g\,\mu_\omega)
\right| \notag\\
&=
\left|
\ell_{\theta^n\omega}\left(
f\,\frac{1}{\lambda_\omega^{(n)}}\mathcal L_\omega^n(g\mu_\omega)
\right)
\right|.
\label{eq:basic1-clean-2}
\end{align}

Define the auxiliary functions
$$
\bar g_\pm:=1\pm \frac{g}{L_3\|g\|_{\mathcal C^\beta}}.
$$
Then
\begin{align}
\ell_{\theta^n\omega}\left(
f\,\frac{1}{\lambda_\omega^{(n)}}\mathcal L_\omega^n(g\mu_\omega)
\right)
=&
\frac{L_3\|g\|_{\mathcal C^\beta}}{2}
\ell_{\theta^n\omega}\left(
f\,\frac{1}{\lambda_\omega^{(n)}}\mathcal L_\omega^n(\bar g_+\mu_\omega)
\right) \notag\\
 &-
\frac{L_3\|g\|_{\mathcal C^\beta}}{2}
\ell_{\theta^n\omega}\left(
f\,\frac{1}{\lambda_\omega^{(n)}}\mathcal L_\omega^n(\bar g_-\mu_\omega)
\right).
\label{eq:basic2-clean-2}
\end{align}

For each $n\ge2$, define
$$
\Psi_f^{(n)}:\mathbb V_{\theta^{n-1}\omega}\to \mathbb R,
\ 
\Psi_f^{(n)}(\eta):=
\ell_{\theta^n\omega}\left(
f\,\frac{1}{\lambda_{\theta^{n-1}\omega}}
\mathcal L_{\theta^{n-1}\omega}\eta
\right).
$$
From Lemma \ref{lem:pos} and \eqref{eq:fform-clean-2}, the functional $\Psi_f^{(n)}$ is positive.

From Corollary \ref{cor:cone}, $\bar{g}_\pm \mu \in \mathcal C_\omega(b,c,\nu)$, therefore $\mathcal L_\omega^{n-1}\bar{g}_\pm \mu \in \mathcal C_{\theta^{n-1}}\omega(b,c,\nu)$. Since $\upsilon_\omega(g)=0$, it follows that
\begin{equation}
\ell_{\theta^{n-1} \omega}\left( \frac{1}{\lambda^{(n-1)}_\omega} \mathcal L_\omega^{n-1}\left( \bar g_\pm\mu_\omega\right)\right)
=
1\pm \frac{\upsilon_\omega(g)}{L_3\|g\|_{\mathcal C^\beta}}
=
1.
\label{eq:equal1-clean-2}
\end{equation}
Moreover,
\begin{align}
\left|\Psi_f^{(n)}\left(\frac{1}{\lambda^{(n-1)}_\omega} \mathcal L_\omega^{n-1}\left( \bar g_\pm\mu_\omega\right)\right)\right|
&=
\left|
\ell_{\theta^n\omega}\left(
f\,\frac{1}{\lambda_\omega^{(n)}}\mathcal L_\omega^n(\bar g_\pm\mu_\omega)
\right)
\right| \notag\\
&=
\left|
\upsilon_\omega((f\circ T_\omega^n)\bar g_\pm)
\right| \notag\\
&\le
\|f\|_\infty \|\bar g_\pm\|_\infty \notag\\
&\le
\left\|1+\frac{\bar f}{L_3\|\bar f\|_{\mathcal C^\beta}}\right\|_\infty
\left\|1+\frac{g}{L_3\|g\|_{\mathcal C^\beta}}\right\|_\infty
\le 4.
\label{eq:equal4-clean-2}
\end{align}

Using \eqref{eq:equal1-clean-2}, \eqref{eq:equal4-clean-2},
Theorem \ref{thm:spectralgap}, and equation \eqref{eq:forward-contraction} (inside the proof of Theorem \ref{thm:spectralgap}), we obtain
\begin{align}
I_n(\omega)
&:=
\left|
\ell_{\theta^n\omega}\left(
f\,\frac{1}{\lambda_\omega^{(n)}}\mathcal L_\omega^n(\bar g_+\mu_\omega)
\right)
-
\ell_{\theta^n\omega}\left(
f\,\frac{1}{\lambda_\omega^{(n)}}\mathcal L_\omega^n(\bar g_-\mu_\omega)
\right)
\right| \notag\\
&=
\left|
1-
\frac{\Psi_f^{(n)}(\frac{1}{\lambda^{(n-1)}_\omega} \mathcal L_\omega^{n-1}\left( \bar g_-\mu_\omega\right))}{\Psi_f^{(n)}(\frac{1}{\lambda^{(n-1)}_\omega} \mathcal L_\omega^{n-1}\left( \bar g_+\mu_\omega\right))}
\right|
\,|\Psi_f^{(n)}(\frac{1}{\lambda^{(n-1)}_\omega} \mathcal L_\omega^{n-1}\left( \bar g_+\mu_\omega\right))| \notag\\
&\le
4\left(
e^{\Theta_{\theta^{n-1}\omega}^{b,c,\nu}( \mathcal L_\omega^{n-1}\left( \bar g_+\mu_\omega\right),\mathcal L_\omega^{n-1}\left( \bar g_-\mu_\omega\right))}-1
\right) \notag\\
&\le
4\left(
e^{D_2\chi^{\max\{r_{n-1}(\omega)-1,0\}}}-1
\right) \le
4e^{D_2}D_2\,\chi^{\max\{r_{n-1}(\omega)-1,0\}}.
\label{eq:basic3-clean-2}
\end{align}

From Lemma \ref{lem:chi-rn}, if Hypothesis \ref{hyp:h} holds, then there exist a measurable function $K_H:\Omega\to \mathbb R$ and $\chi_H\in(0,1)$ such that $
\chi^{r_n(\omega)}\le K_H(\omega)\chi_H^n.$ Therefore
\begin{equation}
\chi^{\max\{r_{n-1}(\omega)-1,0\}}
\le
\chi^{-1}\chi^{r_{n-1}(\omega)}\le
\chi^{-1}K_H(\omega)\chi_H^{n-1} \le
\widetilde K_H(\omega)\Lambda^n,
\label{eq:des1-clean-2}
\end{equation}
where $\Lambda:=\chi_H$ and $\widetilde K_H(\omega):=\chi^{-1}\chi_H^{-1}K_H(\omega).$

If Hypothesis \ref{hyp:h'} holds, then for every $p\in[1,\infty)$ there exist
$K_{H,p}\in L^p(\Omega)$ and $\chi_{H,p}\in(0,1)$ such that
$\chi^{r_n(\omega)}\le K_{H,p}(\omega)\chi_{H,p}^n.$
 Hence
\begin{align}
\chi^{\max\{r_{n-1}(\omega)-1,0\}} \leq
\chi^{-1}\chi^{r_{n-1}(\omega)} \le
\chi^{-1}K_{H,p}(\omega)\chi_{H,p}^{n-1} \le
\widetilde K_{H,p}(\omega)\Lambda_p^n,
\label{eq:des2-clean-2}
\end{align}
where $
\Lambda_p:=\chi_{H,p}$ and $\widetilde K_{H,p}(\omega):=\chi^{-1}\chi_{H,p}^{-1}K_{H,p}(\omega)\in L^p(\Omega).$

Combining \eqref{eq:basic1-clean-2}, \eqref{eq:basic2-clean-2},
\eqref{eq:basic3-clean-2}, and \eqref{eq:des1-clean-2}, if
Hypothesis \ref{hyp:h} holds, then
\begin{align}
\left|
\upsilon_\omega(f\circ T_\omega^n \cdot g)
-
\upsilon_{\theta^n\omega}(f)\upsilon_\omega(g)
\right|
&\le
2L_3e^{D_2}D_2\,\widetilde K_H(\omega)\Lambda^n
\|g\|_{\mathcal C^\beta}.
\label{eq:ineq111-clean-2}
\end{align}
Similarly, if Hypothesis \ref{hyp:h'} holds, then for every $p\in[1,\infty)$,
\begin{align}
\left|
\upsilon_\omega(f\circ T_\omega^n \cdot g)
-
\upsilon_{\theta^n\omega}(f)\upsilon_\omega(g)
\right|
&\le
2L_3e^{D_2}D_2\,\widetilde K_{H,p}(\omega)\Lambda_p^n
\|g\|_{\mathcal C^\beta}.
\label{eq:ineq112-clean-2}
\end{align}

\begin{step}[2]
We show that \eqref{eq:decay1} and \eqref{eq:decay2} holds for a general $f\in\mathcal C^\beta(M)$ and still assuming
$\upsilon_\omega(g)=0$.
\end{step}

Let now $f\in \mathcal C^\beta(M)$ be arbitrary. If $f=0$, there is nothing to prove. Otherwise, define
$$
\widetilde f:=1+\frac{f}{L_3\|f\|_{\mathcal C^\beta}}.
$$
Then $f=L_3\|f\|_{\mathcal C^\beta}(\widetilde f-1).$
Since $\upsilon_\omega(g)=0$, we have
$$\upsilon_\omega(f\circ T_\omega^n \cdot g)=\upsilon_\omega \left(L_3\|f\|_{\mathcal C^\beta}(\widetilde f-1) \circ T_\omega^n \cdot g\right) 
=
L_3\|f\|_{\mathcal C^\beta}\,
\upsilon_\omega(\widetilde f\circ T_\omega^n \cdot g).$$
Applying Step 1 to $\widetilde f$ gives, under Hypothesis \ref{hyp:h},
$$
\left|
\upsilon_\omega(f\circ T_\omega^n \cdot g)
-
\upsilon_{\theta^n\omega}(f)\upsilon_\omega(g)
\right|
\le
C_1(\omega)\Lambda^n
\|f\|_{\mathcal C^\beta}\|g\|_{\mathcal C^\beta},
$$
for a suitable measurable function $C_1:\Omega\to\mathbb R$. Likewise, under
Hypothesis \ref{hyp:h'}, for every $p\in[1,\infty)$,
$$
\left|
\upsilon_\omega(f\circ T_\omega^n \cdot g)
-
\upsilon_{\theta^n\omega}(f)\upsilon_\omega(g)
\right|
\le
C_{1,p}(\omega)\Lambda_p^n
\|f\|_{\mathcal C^\beta}\|g\|_{\mathcal C^\beta},
$$
with $C_{1,p}\in L^p(\Omega)$.

\begin{step}[3]
We show \eqref{eq:decay1} and \eqref{eq:decay1} for general $f,g\in \mathcal C^\beta(M)$.
\end{step}

Let $f,g\in \mathcal C^\beta(M)$ and define $
\widetilde g:=g-\upsilon_\omega(g).$
Then $\upsilon_\omega(\widetilde g)=0$ and $
\|\widetilde g\|_{\mathcal C^\beta}
\le
2\|g\|_{\mathcal C^\beta}.
$
Moreover,
\begin{align*}
\upsilon_\omega(f\circ T_\omega^n \cdot g)
-
\upsilon_{\theta^n\omega}(f)\upsilon_\omega(g)
&=
\upsilon_\omega(f\circ T_\omega^n \cdot \widetilde g).
\end{align*}
Applying Step 2 to the pair $(f,\widetilde g)$, we obtain under
Hypothesis \ref{hyp:h},
$$
\left|
\upsilon_\omega(f\circ T_\omega^n \cdot g)
-
\upsilon_{\theta^n\omega}(f)\upsilon_\omega(g)
\right|
\le
C(\omega)\Lambda^n
\|f\|_{\mathcal C^\beta}\|g\|_{\mathcal C^\beta},
$$
for a suitable measurable function $C:\Omega\to\mathbb R$.

Under Hypothesis \ref{hyp:h'}, for every $p\in[1,\infty)$,
$$
\left|
\upsilon_\omega(f\circ T_\omega^n \cdot g)
-
\upsilon_{\theta^n\omega}(f)\upsilon_\omega(g)
\right|
\le
C_p(\omega)\Lambda_p^n
\|f\|_{\mathcal C^\beta}\|g\|_{\mathcal C^\beta},
$$
for some $C_p\in L^p(\Omega)$.

\begin{step}[4]
We conclude the proof of the lemma.
\end{step}

To finish the proof, it remains to show that for arbitrary $f,g\in\mathcal C^\beta(M)$, equations \eqref{eq:decaypullback1} and \eqref{eq:decaypullback2} hold. We repeat the computations of Steps 1--3 with $\omega$ replaced everywhere by $\theta^{-n}\omega$. More precisely, let $(g_\omega)_{\omega\in\Omega}$ be a measurable family such that
$$
\sup_{\omega\in\Omega}\|g_\omega\|_{\mathcal C^\beta}<\infty
\text{ and }
\upsilon_\omega(g_\omega)=0
$$
for $\mathbb P$-almost every $\omega\in\Omega$. Assume first that
$$
f=1+\frac{\bar f}{L_3\|\bar f\|_{\mathcal C^\beta}}
$$
for some $\bar f\in\mathcal C^\beta(M)$. Exactly as in Steps 1, with $\theta^{-n}\omega$ in place of $\omega$ and $g_{\theta^{-n}\omega}$ in place of $g$, we obtain
\begin{align}
\left|
\upsilon_{\theta^{-n}\omega}(f\circ T_{\theta^{-n}\omega}^n\cdot g_{\theta^{-n}\omega})
-
\upsilon_{\omega}(f)\upsilon_{\theta^{-n}\omega}(g_{\theta^{-n}\omega})
\right|
&\le
2L_3e^{D_2}D_2\,\chi^{\max\{r_{n-1}(\theta^{-n}\omega)-1,0\}}
\|g_{\theta^{-n}\omega}\|_{\mathcal C^\beta}.
\label{eq:pullback-basic}
\end{align}
Since $r_{n-1}(\theta^{-n}\omega)=r_{n-1}\bigl(\theta^{-(n-1)}(\theta^{-1}\omega)\bigr),$
Lemma \ref{lem:chi-rn} (1), applied with $n-1$ and $\theta^{-1}\omega$ in place of $\omega$, yields
\begin{align}
\chi^{\max\{r_{n-1}(\theta^{-n}\omega)-1,0\}}
\le
\chi^{-1}\chi^{r_{n-1}(\theta^{-n}\omega)}
\le
\chi^{-1}\chi_H^{-1}K_H(\theta^{-1}\omega)\chi_H^n.
\label{eq:decayineqbasic}
\end{align}
Combining \eqref{eq:pullback-basic} and \eqref{eq:decayineqbasic}, and using that $\upsilon_{\theta^{-n}\omega}(g_{\theta^{-n}\omega})=0$, we get under Hypothesis \ref{hyp:h}
$$
\left|
\upsilon_{\theta^{-n}\omega}(f\circ T_{\theta^{-n}\omega}^n\cdot g_{\theta^{-n}\omega})
\right|
\le
\widetilde C_H(\omega)\Lambda^n\|g_{\theta^{-n}\omega}\|_{\mathcal C^\beta},
$$
where $
\Lambda:=\chi_H,$ and $\widetilde C_H(\omega):=2L_3e^{D_2}D_2\,\chi^{-1}\chi_H^{-1}K_H(\theta^{-1}\omega).$
Likewise, if Hypothesis \ref{hyp:h'} holds, then Lemma \ref{lem:chi-rn}(2), again applied with $n-1$ and $\theta^{-1}\omega$, yields
\begin{align}
\chi^{\max\{r_{n-1}(\theta^{-n}\omega)-1,0\}}
\le
\chi^{-1}\chi_p^{-1}K_p(\theta^{-1}\omega)\chi_p^n.
\label{eq:decayineqbasic2}
\end{align}
Combining \eqref{eq:pullback-basic} and \eqref{eq:decayineqbasic2}, we obtain for every $p\in[1,\infty)$
$$
\left|
\upsilon_{\theta^{-n}\omega}(f\circ T_{\theta^{-n}\omega}^n\cdot g_{\theta^{-n}\omega})
\right|
\le
\widetilde C_p(\omega)\Lambda_p^n\|g_{\theta^{-n}\omega}\|_{\mathcal C^\beta},
$$
where $\Lambda_p:=\chi_p,\ \widetilde C_p(\omega):=2L_3e^{D_2}D_2\,\chi^{-1}\chi_p^{-1}K_p(\theta^{-1}\omega)\in L^p(\Omega).
$ Here, we have used that $\theta$ preserves $\mathbb P$. 

Repeating Step 2 with $\theta^{-n}\omega$ in place of $\omega$, we remove the special normalisation on $f$ and obtain the same bounds with an additional factor $\|f\|_{\mathcal C^\beta}$. Now, given $g\in\mathcal C^\beta(M)$, define $g_\omega:=g-\upsilon_\omega(g).$ Then $\upsilon_\omega(g_\omega)=0$ for $\mathbb P$-almost every $\omega$, and $
\|g_\omega\|_{\mathcal C^\beta}\le \|g\|_{\mathcal C^\beta}+|\upsilon_\omega(g)|\le 2\|g\|_{\mathcal C^\beta}.$ Hence $$
\sup_{\omega\in\Omega}\|g_\omega\|_{\mathcal C^\beta}\le 2\|g\|_{\mathcal C^\beta}.$$
Repeating Step 3 with $\theta^{-n}\omega$ in place of $\omega$, we conclude that \eqref{eq:decaypullback1} holds under Hypothesis \ref{hyp:h}, and that \eqref{eq:decaypullback2} holds under Hypothesis \ref{hyp:h'}. This concludes the proof.

\end{proof}

With the above result in hand, we can show that $\upsilon_\phi$ is $F$-ergodic.
\begin{corollary}
    The measure $\upsilon_\phi = \upsilon_\omega(\d x) \mathbb P(\d \omega)$ is ergodic. \label{cor:ergodic}
\end{corollary}
\begin{proof}
    We show that $\upsilon_{\phi}$ is ergodic. Let $H_1,H_2\in L^{\infty}(\Omega,\mathcal C^\beta(M))$. Define
$h_1,h_2\in L^\infty(\Omega)$ by
$h_1(\omega)=\int_M H_1(\omega,x)\,\d\upsilon_\omega(x)$ and
$h_2(\omega)=\int_M H_2(\omega,x)\,\d\upsilon_\omega(x)$.
For every $i\geq 0$, invariance of the disintegration gives
\begin{align}
\int_{\Omega\times M} H_1\circ F^i\, H_2\, \d \upsilon_\phi
&=
\mathbb E\left[
\int_M H_1(\theta^i\omega,T_\omega^i(x))H_2(\omega,x)\upsilon_\omega(\d x)
\right].
\label{eq:integralh1h2}
\end{align}
By Proposition \ref{prop:exponentialdecayofcorrelation}, for $\mathbb P$-a.e. $\omega\in\Omega$,
\begin{align}
&\left|
\int_M H_1(\theta^i\omega,T_\omega^i(x))H_2(\omega,x)\,\upsilon_\omega(\d x)
-
h_1(\theta^i\omega)h_2(\omega)
\right|
\nonumber\\
&\leq
C(\omega)e^{-ai}
\|H_1\|_{L^\infty(\Omega,\mathcal C^\beta(M))}
\|H_2\|_{L^\infty(\Omega,\mathcal C^\beta(M))}
\xrightarrow[]{i\to\infty}0
\label{eq:quenched-decay-h1h2}
\end{align}
Since 
$$ \left|
\int_M H_1(\theta^i\omega,T_\omega^i(x))H_2(\omega,x)\,\upsilon_\omega(\d x)
-
h_1(\theta^i\omega)h_2(\omega)
\right| \leq 2\|H_1\|_{L^\infty(\Omega\times M)}\|H_2\|_{L^\infty(\Omega\times M)},$$ the dominated convergence theorem applies. Hence, from \eqref{eq:integralh1h2} and \eqref{eq:quenched-decay-h1h2},
\begin{align*}
\int_{\Omega\times M} H_1\circ F^i\,H_2\,\d\upsilon_\phi
-
\mathbb E[h_1\circ\theta^i\cdot h_2]
\xrightarrow[]{i\to\infty} 0.
\end{align*}
The above equation implies that,
\begin{align*}
\lim_{n\to\infty}
\frac1n\sum_{i=0}^{n-1}
\int_{\Omega\times M} H_1\circ F^i\,H_2\,\d\upsilon_\phi
&=
\lim_{n\to\infty}
\frac1n\sum_{i=0}^{n-1}
\mathbb E[h_1\circ\theta^i\cdot h_2] =
\mathbb E[h_1]\mathbb E[h_2] \\
&=
\left(\int_{\Omega\times M} H_1\,\d\upsilon_\phi\right)
\left(\int_{\Omega\times M} H_2\,\d\upsilon_\phi\right),
\end{align*}
where the second equality follows from the ergodicity of $\theta$ and the mean ergodic theorem.

Since $L^\infty(\Omega,\mathcal C^\beta(M))$ is dense in $L^2(\Omega\times M,\upsilon_\phi)$, the same Cesàro correlation identity holds for all $H_1,H_2\in L^2(\Omega\times M, \upsilon_\phi)$. Now let $A\subset\Omega\times M$ be an $F$-invariant measurable set. Taking $H_1=H_2=\mathbbm 1_A$, we get
\[
\upsilon_\phi(A)
=
\lim_{n\to\infty}
\frac1n\sum_{i=0}^{n-1}
\int_{\Omega\times M} \mathbbm 1_A\circ F^i\,\mathbbm 1_A\,\d\upsilon_\phi
=
\upsilon_\phi(A)^2.
\]
Thus $\upsilon_\phi(A)\in\{0,1\}$. Hence $F$ is ergodic with respect to $\upsilon_\phi$.
\end{proof}

\section{Existence of Equilibrium States}
\label{sec:eqstate} 
The purpose of this section is to prove that the measure $\upsilon_\phi$ constructed in Lemma \ref{lem:nuprob} is a $\mathbb P$-relative equilibrium state for the potential
$
\bar\phi:=\phi-\phi^{J^s}.
$
The appearance of the correction term $\phi^{J^s}$ comes from the change of variables along stable manifolds: the transfer operator is defined with weight $\phi$, whereas the variational principle is expressed with respect to the ambient dynamics $F$.

The proof has two main steps. First, we identify the spectral objects $\mu_\omega$ and $\ell_\omega$ as measures and record their conditional structure inside rectangles. This gives the invariant measure
$
\upsilon_\phi(\d\omega,\d x)=\upsilon_\omega(\d x)\mathbb P(\d\omega)
$
a form compatible with the local product structure. Second, we use this description to establish a weak Gibbs estimate for Bowen balls associated with the potential $\bar\phi$. This estimate allows us to compare the relative entropy contribution with the exponential growth of the normalising factors $\lambda_\omega^{(n)}$, and hence to prove that
$
h_{\upsilon_\phi}(F\mid\mathbb P)+\int \bar\phi\,\d\upsilon_\phi
=
P_{\mathrm{top}}(F,\bar\phi\mid\mathbb P).
$

\subsection{Properties of  \texorpdfstring{$\mu_\omega$}{muome} and  \texorpdfstring{$\ell_\omega$}{ellome} as measures}
\label{sec:muell}
In this section, we relate the spectral objects obtained in Theorem \ref{thm:spectralgap} to measures on the phase space. The first step is to describe the conditional structure of $\mu_\omega$ inside rectangles.

\begin{lemma}\label{lem:stablemu}\label{lem:dec}
Let $\mu_\omega$ be defined as in Theorem \ref{thm:spectralgap}. Then, for
$\mathbb P$-almost every $\omega\in\Omega$, there exists a $\alpha$-$\log$
Hölder function $H_\omega$ such that, for every $x\in M$,
\[
\mu_\omega(\d z \cap R_\delta(\omega,x))
=
H_\omega(y^u)
\prod_{i=1}^{\infty}
\frac{
e^{\phi_{\theta^{-i}\omega}\circ (T_{\theta^{-i}\omega}^i)^{-1}(y^u)}
}{
e^{\phi_{\theta^{-i}\omega}\circ (T_{\theta^{-i}\omega}^i)^{-1}(y^s)}
}
\,m_{\gamma^u_{(\omega,y^s)}}(\d y^u)
\,\mu_{\gamma^s_{(\omega,x)}}(\d y^s),
\]
where $\mu_{\gamma^s_{(\omega,x)}}$ is the unique measure such that
\[
\mu_{\gamma^s_{(\omega,x)}}(f)
:=
\lim_{n\to\infty}
\int_{\gamma^s_{(\omega,x)}}
f
\frac{\mathcal L^n_{\theta^{-n}\omega}\mathbbm 1}
{\|\mathcal L^n_{\theta^{-n}\omega}\mathbbm 1\|_\omega}
\,\d m_{\gamma^s_{(\omega,x)}}
=
\Gamma_{\gamma^s_{(\omega,x)}}(f\mu_\omega)
\]
for every $f\in\mathcal C^\beta(M)$, with
$\Gamma_{\gamma^s_{(\omega,x)}}\in\mathbb V_\omega^*$ given by
\[
\Gamma_{\gamma^s_{(\omega,x)}}(f)
=
\int_{\gamma^s_{(\omega,x)}} f\,\d m_{\gamma^s_{(\omega,x)}}.
\]
Here $R_\delta(\omega,x)$ is a rectangle in the sense of Definition
\ref{def:rectangles}.
\end{lemma}

\begin{proof}
Since for $\mathbb P$-a.e. $\omega\in\Omega$ the foliations
$x\mapsto W^{u/s}_\varepsilon(x)$ are $\mathcal C^\beta$, there exists a
log-$\alpha$-Hölder function
$H_\omega:R_\delta(\omega,x)\to\mathbb R$ such that, for every
$\mathcal C^\beta$ function $g:M\to\mathbb R$ supported on
$R_\delta(\omega,x)$,
\begin{align*}
\int_M g(z)\,m(\d z)
&=
\int_{R_\delta(\omega,x)} g(z)\,m(\d z) \\
&=
\int_{\gamma^s_{(\omega,x)}}
\int_{\gamma^u_{(\omega,y^s)}}
g(y^u)H_\omega(y^u)
\,m_{\gamma^u_{(\omega,y^s)}}(\d y^u)
\,m_{\gamma^s_{(\omega,x)}}(\d y^s).
\end{align*}

By construction of $\mu_\omega$,
\begin{align*}
\int_M g(z)\,\d\mu_\omega(z)
&=
\lim_{n\to\infty}
\int_M
g(z)
\frac{\mathcal L^n_{\theta^{-n}\omega}\mathbbm 1(z)}
{\|\mathcal L^n_{\theta^{-n}\omega}\mathbbm 1\|_\omega}
\,m(\d z) \\
&=
\lim_{n\to\infty}
\int_{\gamma^s_{(\omega,x)}}
\int_{\gamma^u_{(\omega,y^s)}}
g(y^u)H_\omega(y^u)
\frac{\mathcal L^n_{\theta^{-n}\omega}\mathbbm 1(y^u)}
{\|\mathcal L^n_{\theta^{-n}\omega}\mathbbm 1\|_\omega}
\,m_{\gamma^u_{(\omega,y^s)}}(\d y^u)
\,m_{\gamma^s_{(\omega,x)}}(\d y^s) \\
&=
\lim_{n\to\infty}
\int_{\gamma^s_{(\omega,x)}}
G_{\omega,n}(g)(y^s)
\frac{\mathcal L^n_{\theta^{-n}\omega}\mathbbm 1(y^s)}
{\|\mathcal L^n_{\theta^{-n}\omega}\mathbbm 1\|_\omega}
\,m_{\gamma^s_{(\omega,x)}}(\d y^s),
\end{align*}
where
\[
G_{\omega,n}(g)(y^s)
=
\int_{\gamma^u_{(\omega,y^s)}}
g(y^u)H_\omega(y^u)
\frac{
e^{S_n\phi_{\theta^{-n}\omega}\circ (T_{\theta^{-n}\omega}^n)^{-1}(y^u)}
}{
e^{S_n\phi_{\theta^{-n}\omega}\circ (T_{\theta^{-n}\omega}^n)^{-1}(y^s)}
}
\,m_{\gamma^u_{(\omega,y^s)}}(\d y^u).
\]

Since $y^u$ and $y^s$ lie on the same unstable leaf, their backward
iterates converge exponentially. Using the uniform Hölder regularity of
$g$, $H_\omega$, and $\phi_{\theta^{-i}\omega}$, we obtain
\[
G_{\omega,n}(g)\xrightarrow[n\to\infty]{}
G_\omega(g)
\ \text{in }\mathcal C^\beta(\gamma^s_{(\omega,x)}),
\]
where
\[
G_\omega(g)(y^s)
=
\int_{\gamma^u_{(\omega,y^s)}}
g(y^u)H_\omega(y^u)
\prod_{i=1}^{\infty}
\frac{
e^{\phi_{\theta^{-i}\omega}\circ (T_{\theta^{-i}\omega}^i)^{-1}(y^u)}
}{
e^{\phi_{\theta^{-i}\omega}\circ (T_{\theta^{-i}\omega}^i)^{-1}(y^s)}
}
\,m_{\gamma^u_{(\omega,y^s)}}(\d y^u).
\]

Therefore, by Theorem \ref{thm:spectralgap},
\begin{align*}
\int_M g(z)\,\d\mu_\omega(z)
&=
\Gamma_{\gamma^s_{(\omega,x)}}
\left(G_\omega(g)\mu_\omega\right) =
\int_{\gamma^s_{(\omega,x)}} G_\omega(g)(y^s)
\,\mu_{\gamma^s_{(\omega,x)}}(\d y^s) \\
&=
\int_{\gamma^s_{(\omega,x)}}
\int_{\gamma^u_{(\omega,y^s)}}
g(y^u)H_\omega(y^u)
\prod_{i=1}^{\infty}
\frac{
e^{\phi_{\theta^{-i}\omega}\circ (T_{\theta^{-i}\omega}^i)^{-1}(y^u)}
}{
e^{\phi_{\theta^{-i}\omega}\circ (T_{\theta^{-i}\omega}^i)^{-1}(y^s)}
}
\,m_{\gamma^u_{(\omega,y^s)}}(\d y^u)
\,\mu_{\gamma^s_{(\omega,x)}}(\d y^s).
\end{align*}
Since this holds for every $\mathcal C^\beta$ function $g$ supported on
$R_\delta(\omega,x)$, the claimed disintegration formula follows.
\end{proof}

The following theorem allows us to derive a Margulis-type relation for the disintegration of $\mu_\omega$ along the stable leaves $\mathscr F^s$.

\begin{lemma}\label{lem:stablemargulis}
Let $\mu_{\gamma^s_{(\omega,x)}}$ be as in Lemma \ref{lem:dec}, then for every $f\in \mathcal C^\beta(M)$
$$\mu_{\gamma^s_{(\theta \omega,T_\omega(x) )}}(f)= \frac{1}{\lambda_\omega }\sum_{i} \mu_{\gamma^{s}_i(\omega,x)}( e^{\phi_\omega - \phi_{\omega}^{J^s}} (f\circ T_\omega))$$

\end{lemma}
\begin{proof} From a direct computation
 \begin{align*}
    \mu_{\gamma^s_{(\theta \omega,T_\omega(x) )}} (f) &= \Gamma_{\gamma_{(\theta, T_\omega(x)}} \left(f \cdot \mu_{\theta \omega}\right) = \frac{1}{\lambda_\omega}  \Gamma_{\gamma_{(\theta, T_\omega(x)}} \left(f \cdot \mathcal L_\omega \mu_{\omega}\right) = \frac{1}{\lambda_\omega}  \Gamma_{\gamma_{(\theta, T_\omega(x))}} \mathcal L_\omega \left(f\circ T_\omega \cdot \mu_{\omega}\right) \\
     &= \frac{1}{\lambda_\omega}\mathcal L_\omega^*(  \Gamma_{\gamma_{(\theta, T_\omega(x)}} ) \left(f\circ T_\omega \cdot \mu_{\omega}\right)
 \end{align*}

Given $g\in \mathcal C^\alpha(M),$
\begin{align*}
 \mathcal L_\omega^*(  \Gamma_{\gamma^s_{(\theta\omega, T_\omega(x))}})  g &=  \Gamma_{\gamma_{(\theta, T_\omega(x))}}  (\mathcal L_\omega g) = \int_{\gamma^s_{(\theta\omega, T_\omega(x))}} e^{\phi_\omega\circ T_\omega^{-1} } g\circ T_\omega^{-1} \d m_{\gamma^s_{(\omega,T_\omega(x))}}  \\
 &= \sum_{i} \int_{\gamma_i(\omega,x_i)} e^{\phi_\omega - \phi^{J^s}} g \d m^{s}_{\gamma_i(\omega,x_i)} = \sum_{i} \Gamma_{\gamma_i(\omega,x_i)}( e^{\phi_\omega - \phi_{\omega}^{J^s}} g).
\end{align*}
Therefore
$$\lambda_\omega \mu_{\gamma^s_{(\theta \omega,T_\omega(x) )}}(f)= \sum_{i} \mu_{\gamma^{s}_i(\omega,x)}( e^{\phi_\omega - \phi_{\omega}^{J^s}} (f\circ T_\omega)).$$

\end{proof}

We next prove the corresponding statements for $\ell_\omega$. The point is that $\ell_\omega$ is initially obtained as an element of the dual space $\mathbb V_\omega^*$, so it is not a priori clear that it is represented by a measure on $M$. We first prove this positivity and representation property, and then record a local disintegration formula for $\ell_\omega$ inside rectangles.
\begin{lemma}\label{lem:eqmeasure}
    Let $\ell_\omega$ be as in Theorem \ref{thm:spectralgap}. Then $\ell_\omega$ can be identified with a positive measure on $M$.
\end{lemma}
\begin{proof}
We first note that $\ell_\omega$ is a continuous linear functional in $\mathcal C^\beta(M)$. Indeed, by Lemma \ref{lem:norm1}, for every $f\in \mathcal C^\beta$,
\begin{equation}
|\ell_\omega(f)|
\le \|\ell_\omega\|_{\mathbb V_\omega^*}\|f\|_\omega
\le C\|\ell_\omega\|_{\mathbb V_\omega^*}\|f\|_{\mathcal C^\beta}.
\end{equation}

Moreover, by Theorem \ref{thm:decay} and Lemma \ref{lem:C}, there exists a measurable function $K:\Omega\to \mathbb R$ such that, for every $n\in \mathbb N$ and every $f\in \mathcal C^\beta(M)$,
\begin{equation}
\left\|
\frac{1}{\lambda_\omega^{(n)}}\mathcal L_\omega^n f
-\ell_\omega(f)\mu_{\theta^n\omega}
\right\|
\le K(\omega)\|\ell_\omega\|_{\mathbb V_\omega^*}\|f\|_\omega\chi^n
\le CK(\omega)\|\ell_\omega\|_{\mathbb V_\omega^*}\|f\|_{\mathcal C^\beta}\chi^n.
\label{eq:Kdec}
\end{equation}

We now prove that $\ell_\omega$ is positive. The argument is the same as in the proof of Lemma \ref{lem:nuprob}. From Lemma \ref{lem:norm1}, $1=\|\mu_\omega\|_\omega\le 3\|\mu_\omega\|_{\omega,a,\kappa}^{\sup_s}.$ Hence, for each $n\in \mathbb N$, we can choose
$\gamma_{\theta^n\omega}^{(n)}\in \mathscr F^s_{\theta^n\omega}$ and
$\rho_{\theta^n\omega}^{(n)}\in D_1(a,\kappa,\gamma_{\theta^n\omega}^{(n)})$
such that
\begin{equation}
\int_{\gamma_{\theta^n\omega}^{(n)}}
\rho_{\theta^n\omega}^{(n)}\mu_{\theta^n\omega}
\ge \frac14.
\label{eq:141}
\end{equation}

Define $\Gamma_n\in \mathbb V_{\theta^n\omega}^*$ as 
$\Gamma_n(g):=\int_{\gamma_{\theta^n\omega}^{(n)}}\rho_{\theta^n\omega}^{(n)}g.$
If $f\ge0$, then $\mathcal L_\omega^n f\ge0$, and since
$\rho_{\theta^n\omega}^{(n)}\ge0$, we obtain
\begin{equation}
\Gamma_n(\mathcal L_\omega^n f)\ge0
\  \text{for every } n\in \mathbb N.
\end{equation}

Applying $\Gamma_n$ to \eqref{eq:Kdec}, and using \eqref{eq:141}, we get
\begin{align*}
\left|
\frac{1}{\lambda_\omega^{(n)}}
\frac{\Gamma_n(\mathcal L_\omega^n f)}{\Gamma_n(\mu_{\theta^n\omega})}
-\ell_\omega(f)
\right|
&\le
\frac{CK(\omega)\|\ell_\omega\|_{\mathbb V_\omega^*}\|f\|_{\mathcal C^\beta}\chi^n}
{\Gamma_n(\mu_{\theta^n\omega})}\\
&\le
4CK(\omega)\|\ell_\omega\|_{\mathbb V_\omega^*}\|f\|_{\mathcal C^\beta}\chi^n.
\end{align*}
Since the first term inside the absolute value is non-negative for every $n$, and the right-hand side tends to $0$ as $n\to\infty$, it follows that
\begin{equation}
\ell_\omega(f)\geq 0\  \text{whenever }f\geq 0.
\end{equation}
Therefore $\ell_\omega$ is positive on $\mathcal C^\beta(M)$ with respect to the pointwise order.

Finally, for any real-valued $f\in \mathcal C^\beta(M)$,
\begin{equation}
-\|f\|_\infty\,\mathbbm 1 \leq f\le \|f\|_\infty\,\mathbbm 1.
\end{equation}
By positivity, $|\ell_\omega(f)|\le \ell_\omega(1)\|f\|_\infty.$
Thus $\ell_\omega$ is continuous in $\mathcal C^\beta(M)$ with respect to the supremum norm, hence it can be extended to a positive linear functional in $\mathcal C^0(M)$. From the Riesz--Markov theorem, this extension is represented by a finite positive Borel measure on $M$.
\end{proof}

\begin{lemma}
Let $\ell_\omega$ be as in Theorem \ref{thm:spectralgap}. Then, for every $x\in M$, there exists a set $\Omega_x\subset \Omega$ with $\mathbb P(\Omega_x)=1$ such that, for every $\omega\in \Omega_x$, there exist an $\alpha$-$\log$ Hölder function
$$
A_\omega:R_\delta(\omega,x)\to (0,\infty)
$$
and a probability measure $\ell^u_{\gamma^u_{(\omega,x)}}$ on $\gamma^u_{(\omega,x)}$ such that
$$
\ell_\omega(v)
=
\int_{\gamma^u_{(\omega,x)}}
\Gamma_{\gamma^s_{(\omega,y)}}\bigl((A_\omega|_{\gamma^s_{(\omega,y)}})v\bigr)\,
\ell^u_{\gamma^u_{(\omega,x)}}(\mathrm d y)
\  \text{for every } v\in \mathcal C^\beta(M),
$$
where, for each $y\in \gamma^u_{(\omega,x)}$, the functional
$\Gamma_{\gamma^s_{(\omega,y)}}\in \mathbb V_\omega^*$ is given by
$$
\Gamma_{\gamma^s_{(\omega,y)}}(f)
:=
\int_{\gamma^s_{(\omega,y)}} f\,\mathrm d m_{\gamma^s_{(\omega,y)}}.
$$
\label{lem:dec1}
\end{lemma}

\begin{proof}
For each $\omega\in \Omega$, let $m_\omega:\mathbb V_\omega\to\mathbb R$ be the linear functional induced by the measure $m$, that is,
$$
m_\omega(f):=\int_M f(x)\,m(\mathrm d x)
$$
for $f\in \mathrm{BM}(M)$. By the definition of the norm $\|\cdot\|_{\omega,a,\kappa}^{\sup_s}$, for every $f\in \mathcal C^\beta(M)$ we have
$$
|m_\omega(f)|
\le \|f\|_{\omega,a,\kappa}^{\sup_s}\,m(1)
\le m(1)\|f\|_\omega.
$$
Therefore $m_\omega$ extends uniquely to an element of $\mathbb V_\omega^*$, and $\|m_\omega\|_{\mathbb V_\omega^*}\le m(1)$ uniformly in $\omega\in\Omega$.

Now fix $\omega\in \Omega$. By Theorem \ref{thm:decay}, there exists a measurable function $K:\Omega\to\mathbb R$ such that, for every $f\in \mathcal C^\beta(M)$,
$$
\left\|
\frac{1}{\lambda_\omega^{(n)}}\mathcal L_\omega^n f
-\ell_\omega(f)\mu_{\theta^n\omega}
\right\|_{\theta^n\omega}
\le
K(\omega)\|\ell_\omega\|_{\mathbb V_\omega^*}\|f\|_\omega\chi^n.
$$
Applying the functional $m_{\theta^n\omega}\in \mathbb V_{\theta^n\omega}^*$ to the term inside the norm, and using $\|m_{\theta^n\omega}\|_{\mathbb V_{\theta^n\omega}^*}\le m(1)$, we obtain
$$
\left|
\int_M \frac{1}{\lambda_\omega^{(n)}}\mathcal L_\omega^n f(x)\,m(\mathrm d x)
-\ell_\omega(f)\mu_{\theta^n\omega}(1)
\right|
\le
m(1)K(\omega)\|\ell_\omega\|_{\mathbb V_\omega^*}\|f\|_\omega\chi^n.
$$
Since $\mu_{\theta^n\omega}\in \operatorname{int}\bigl(\mathcal C_{\theta^n\omega}(b,c,\nu)\bigr)$, we have $\mu_{\theta^n\omega}(1)>0$
for $\mathbb P$-almost every $\omega$. Hence, for every such $\omega$,
$$
\left|
\ell_\omega(f)
-
\frac{1}{\mu_{\theta^n\omega}(1)}
\int_M \frac{1}{\lambda_\omega^{(n)}}\mathcal L_\omega^n f(x)\,m(\mathrm d x)
\right|
\le
\frac{m(1)K(\omega)}{\mu_{\theta^n\omega}(1)}
\|\ell_\omega\|_{\mathbb V_\omega^*}\|f\|_\omega\chi^n.
$$
Therefore
\begin{equation}
\ell_\omega(f)
=
\lim_{n\to\infty}
\frac{1}{\mu_{\theta^n\omega}(1)}
\int_M \frac{1}{\lambda_\omega^{(n)}}\mathcal L_\omega^n f(x)\,m(\mathrm d x)
\  \text{for every } f\in \mathcal C^\beta(M).
\label{eq:limlw}
\end{equation}

Since the map $\eta\mapsto \mu_\eta(1)$ is measurable and strictly positive on a full-measure set, there exists $c>1$ such that
$$
M_c:=\{\eta\in \Omega: c^{-1}\le \mu_\eta(1)\le c\}
$$
has positive $\mathbb P$-measure. Let $\Omega_c\subset \Omega$ be the full-measure set of points whose $\theta$-orbit visits $M_c$ infinitely many times. Fix $x\in M$, and let $\Omega_x^0$ be the full-measure set given by Lemma \ref{lem:dec}.

Fix $\omega\in \Omega_c$. Choose a strictly increasing sequence $n_k=n_k(\omega)$ such that
$$
\theta^{n_k}\omega\in M_c
\  \text{for every } k\ge1.
$$
For $v\in \mathbb V_\omega$, define
$$
F_k(v):=
\frac{1}{\mu_{\theta^{n_k}\omega}(1)}
\int_M \frac{1}{\lambda_\omega^{(n_k)}}\mathcal L_\omega^{n_k}v(x)\,m(\mathrm d x).
$$
Since $c^{-1}\le \mu_{\theta^{n_k}\omega}(1)\le c$, the sequence $(F_k)_k$ satisfies the same uniform bounds used in the proof of Lemma \ref{lem:dec}. Hence the compactness argument from that proof applies to $(F_k)_k$: after passing to a subsequence, there exist an $\alpha$-$\log$ Hölder function
$$
A_\omega:R_\delta(\omega,x)\to (0,\infty)
$$
and a probability measure $\ell^u_{\gamma^u_{(\omega,x)}}$ on $\gamma^u_{(\omega,x)}$ (by the correct normalization) such that
$$
\lim_{k\to\infty} F_k(v)
=
\int_{\gamma^u_{(\omega,x)}}
\Gamma_{\gamma^s_{(\omega,y)}}\bigl((A_\omega|_{\gamma^s_{(\omega,y)}})v\bigr)\,
\ell^u_{\gamma^u_{(\omega,x)}}(\mathrm d y)
\  \text{for every } v\in \mathbb V_\omega.
$$
On the other hand, \eqref{eq:limlw} shows that
$$
\lim_{k\to\infty} F_k(f)=\ell_\omega(f)
\  \text{for every } f\in \mathcal C^\beta(M).
$$
Since both sides define continuous linear functionals on $\mathbb V_\omega$, this identity extends by density from $\mathcal C^\beta(M)$ to all $v\in \mathbb V_\omega$. Therefore
$$
\ell_\omega(v)
=
\int_{\gamma^u_{(\omega,x)}}
\Gamma_{\gamma^s_{(\omega,y)}}\bigl((A_\omega|_{\gamma^s_{(\omega,y)}})v\bigr)\,
\ell^u_{\gamma^u_{(\omega,x)}}(\mathrm d y)
\  \text{for every } v\in \mathbb V_\omega.
$$
This proves the lemma.
\end{proof}

Before proving the weak Gibbs estimate, we need two auxiliary bounds. The first one concerns the random normalising factors $\lambda_\omega$ coming from the spectral decomposition. Since these factors will later appear in the exponential weights of Bowen balls, we record that $\log\lambda_\omega$ is integrable and that its integral is controlled by the size of the potential.

\begin{proposition}
The map $\omega \mapsto \log \lambda_\omega$ belongs to $L^1(\Omega,\mathbb P)$. Moreover, there exists $K>0$ such that
\begin{align}
-(K+1)\int_\Omega \|\phi_\omega\|_{\mathcal C^\beta}\,\mathbb P(d\omega)
&\le
\int_\Omega \log \lambda_\omega\,\mathbb P(d\omega)\nonumber\\
&\le
(K+1)\int_\Omega \|\phi_\omega\|_{\mathcal C^\beta}\,\mathbb P(d\omega).
\label{eq:inequalitylambda}
\end{align}
\end{proposition}

\begin{proof}
From Lemma \ref{lem:archi} and Theorem \ref{thm:spectralgap}, for $\mathbb P$-almost every $\omega\in\Omega$,
$$
\lim_{n\to\infty}\frac{1}{n}\log \|\mathcal L_\omega^n \mathbbm 1\|_{\theta^n\omega}
=
\lim_{n\to\infty}\frac{1}{n}\log \lambda_\omega^{(n)}
=
\int_\Omega \log \lambda_\omega\,\mathbb P(d\omega).
$$

Fix $n\in\mathbb N$, $\gamma_{\theta^n\omega}\in \mathscr F^s_{\theta^n\omega}$ and
$\rho_{\theta^n\omega}\in D_1(a,\kappa,\gamma_{\theta^n\omega})$. Since
$$
\mathcal L_\omega^n \mathbbm 1(x)
=
\exp\bigl(S_n\phi_\omega((T_\omega^n)^{-1}x)\bigr),
$$
for every $x,y\in \gamma_{\theta^n\omega}$ we have
\begin{align}
\bigl|S_n\phi_\omega((T_\omega^n)^{-1}x)-S_n\phi_\omega((T_\omega^n)^{-1}y)\bigr|
&\leq
\sum_{i=0}^{n-1}
\|\phi_{\theta^i\omega}\|_{\mathcal C^\beta}
\,d\bigl(T_\omega^i((T_\omega^n)^{-1}x),T_\omega^i((T_\omega^n)^{-1}y)\bigr)^\beta \nonumber\\
&\le
\Bigl(\sup_{u,v\in M} d(u,v)^\beta\Bigr)
\sum_{i=0}^{n-1}\|\phi_{\theta^i\omega}\|_{\mathcal C^\beta}.
\label{eq:Holder1}
\end{align}
Set $K:=\sup_{u,v\in M} d(u,v)^\beta <\infty.$
Then \eqref{eq:Holder1} implies
$$
\sup_{x\in\gamma_{\theta^n\omega}} \mathcal L_\omega^n \mathbbm 1(x)
\le
e^{K\sum_{i=0}^{n-1}\|\phi_{\theta^i\omega}\|_{\mathcal C^\beta}}
\inf_{x\in\gamma_{\theta^n\omega}} \mathcal L_\omega^n \mathbbm 1(x).
$$
On the other hand, for every $z\in M$,
$$
e^{-\sum_{i=0}^{n-1}\|\phi_{\theta^i\omega}\|_\infty}
\le
\mathcal L_\omega^n \mathbbm 1(z)
\le
e^{\sum_{i=0}^{n-1}\|\phi_{\theta^i\omega}\|_\infty}
\le
e^{\sum_{i=0}^{n-1}\|\phi_{\theta^i\omega}\|_{\mathcal C^\beta}}.
$$
Therefore,
\begin{equation}
e^{-(K+1)\sum_{i=0}^{n-1}\|\phi_{\theta^i\omega}\|_{\mathcal C^\beta}}
\le
\int_{\gamma_{\theta^n\omega}}
\rho_{\theta^n\omega}\,\mathcal L_\omega^n \mathbbm 1\,\d m_{\gamma_{\theta^n\omega}}
\le
e^{(K+1)\sum_{i=0}^{n-1}\|\phi_{\theta^i\omega}\|_{\mathcal C^\beta}}.
\label{eq:Holder2}
\end{equation}
Since $\gamma_{\theta^n\omega}$ and $\rho_{\theta^n\omega}$ were arbitrary, $\mathcal L_\omega^n\mathbbm 1 \in \mathcal C_{\theta^n \omega}(b,c,\nu)$ and Lemma \eqref{lem:norm1} we obtain
$$
\frac{1}{3}e^{-(K+1)\sum_{i=0}^{n-1}\|\phi_{\theta^i\omega}\|_{\mathcal C^\beta}}
\le 
\|\mathcal L_\omega^n \mathbbm 1\|_{\theta^n\omega}
\le
3 e^{(K+1)\sum_{i=0}^{n-1}\|\phi_{\theta^i\omega}\|_{\mathcal C^\beta}}.
$$
Taking logarithms, dividing by $n$, and letting $n\to\infty$, it follows from Birkhoff's ergodic theorem that
\begin{align*}
-(K+1)\int_\Omega \|\phi_\omega\|_{\mathcal C^\beta}\,\mathbb P(d\omega)
&\le
\lim_{n\to\infty}\frac{1}{n}\log \|\mathcal L_\omega^n \mathbbm 1\|_{\theta^n\omega}=
\int_\Omega \log \lambda_\omega\,\mathbb P(d\omega)\\
&\le
(K+1)\int_\Omega \|\phi_\omega\|_{\mathcal C^\beta}\,\mathbb P(d\omega).
\end{align*}
This proves \eqref{eq:inequalitylambda}, and in particular $\omega\mapsto \log\lambda_\omega$ lies in $L^1(\Omega,\mathbb P)$.
\end{proof}

\begin{lemma}\label{lem:mu-minus-lp}
Let $\mu_\omega$ be as in Theorem \ref{thm:spectralgap}. Assume Hypothesis \ref{hyp:h'}. Then, for every $p\in[1,\infty)$, there exists a measurable function $D_p:\Omega\to[1,\infty)$ such that $\log D_p\in L^p(\Omega,\mathbb P),$
and, for $\mathbb P$-almost every $\omega\in\Omega$,
\begin{equation}
\frac{1}{D_p(\omega)}
\le
\|\mu_\omega\|_{\omega,-}
\le
\|\mu_\omega\|_{\omega,+}
\le 1.
\label{eq:desireq}
\end{equation}
In particular,
\begin{equation}
\frac{1}{3}\le \|\mu_\omega\|_{\omega,+}\le 1.
\label{eq:muplus-bounds}
\end{equation}
\end{lemma}

\begin{proof}
Define $R(\omega):=\min\{n\geq 1;\ N_{n_0}(\theta^{-n}\omega)<n\}.$
We claim that $R$ has exponential tail. Indeed, if $R(\omega)>n$, then  $N_{n_0}(\theta^{-n}\omega)\geq n.$ Hence, by $\theta$-invariance of $\mathbb P$,
\begin{align}
   \mathbb P[R>n]
   \leq \mathbb P[N_{n_0}\geq n].
   \label{eq:R:tail1}
\end{align}
Let $c>0$ be as in Hypothesis \ref{hyp:h'}, and set $m_n:=\lfloor n/c\rfloor$. For $n$ large enough we have $m_n\geq n_0$, and since $i\mapsto N_i(\omega)$ is increasing,
\begin{align}
   \mathbb P[N_{n_0}\geq n]
   \leq \mathbb P[N_{m_n}\geq n]
   \leq \mathbb P[N_{m_n}\geq c\,m_n]
   \leq K e^{-\kappa m_n}
   \leq K e^{\kappa} e^{-(\kappa/c)n}.
   \label{eq:R:tail2}
\end{align}
Combining \eqref{eq:R:tail1} and \eqref{eq:R:tail2}, we conclude that $R$ has exponential tail. In particular,
\begin{align}
    R\in L^p(\Omega,P)\  \text{for every }p\in[1,\infty).
    \label{eq:R:Lp}
\end{align}

Let $\mu_\omega$ be as in Theorem \ref{thm:spectralgap}. Set $
\eta:=\theta^{-R(\omega)}\omega,$  $s(\omega):=N_{n_0}(\eta)=N_{n_0}(\theta^{-R(\omega)}\omega).$
By definition of $R(\omega)$,
\begin{align}
   s(\omega)<R(\omega).
   \label{eq:R:s<R}
\end{align}
Moreover,
\begin{align}
   \mathcal L_\eta^{R(\omega)}
   =
   \mathcal L_{\theta^{s(\omega)}\eta}^{\,R(\omega)-s(\omega)}
   \circ
   \mathcal L_\eta^{s(\omega)}.
   \label{eq:R:split}
\end{align}

From Lemma \ref{lem:finitediam} at the fibre $\eta$ and time $s(\omega)$. Since $s(\omega)=N_{n_0}(\eta)$, we obtain that
\begin{align}
 \frac{\|\mathcal L_\eta^{s(\omega)} \mu_\eta \|_{\theta^{s(\omega)}\eta,+}}
 {\|\mathcal L_\eta^{s(\omega)} \mu_\eta \|_{\theta^{s(\omega)}\eta,-}}
 \leq K_5 \left(\frac{e^{2\|\phi\|_{L^\infty(\Omega\times M)}}}
 { \inf_{(\omega,x)\in \Omega\times M} m(\left.D_x T_\omega\right|_{E^{s}(\omega,x)})}\right)^{s(\omega)}.
 \label{eq:R:goodtime}
\end{align}

For the remaining $R(\omega)-s(\omega)$ iterates, we use the rough bounds
\begin{align}
\|\mathcal L_{\theta^{s(\omega)}\eta}^{R(\omega)-s(\omega)}\psi\|_{\omega,+}
&\leq e^{2\|\phi\|_{L^\infty(\Omega\times M)}(R(\omega)-s(\omega))}
\|\psi\|_{\theta^{s(\omega)}\eta,+},
\label{eq:R:rough+}\\
\|\mathcal L_{\theta^{s(\omega)}\eta}^{R(\omega)-s(\omega)}\psi\|_{\omega,-}
&\geq e^{-\|\phi\|_{L^\infty(\Omega\times M)}(R(\omega)-s(\omega))}
\|\psi\|_{\theta^{s(\omega)}\eta,-},
\label{eq:R:rough-}
\end{align}
valid for every $\psi\in \mathcal C_{\theta^{s(\omega)}\eta}(b,c,\nu)$. Applying \eqref{eq:R:rough+}--\eqref{eq:R:rough-} to $\psi=\mathcal L_\eta^{s(\omega)}\mu_\eta$ and using \eqref{eq:R:split}, we obtain
\begin{align}
\frac{\|\mathcal L_\eta^{R(\omega)}\mu_\eta\|_{\omega,+}}
{\|\mathcal L_\eta^{R(\omega)}\mu_\eta\|_{\omega,-}}
&\leq
e^{3\|\phi\|_{L^\infty(\Omega\times M)}(R(\omega)-s(\omega))}
\frac{\|\mathcal L_\eta^{s(\omega)} \mu_\eta \|_{\theta^{s(\omega)}\eta,+}}
{\|\mathcal L_\eta^{s(\omega)} \mu_\eta \|_{\theta^{s(\omega)}\eta,-}}.
\label{eq:R:ratio1}
\end{align}
Combining \eqref{eq:R:ratio1} with \eqref{eq:R:goodtime}, and using \eqref{eq:R:s<R}, we find a constant $\widetilde K>1$ such that
\begin{align}
\frac{\|\mathcal L_{\theta^{-R(\omega)}\omega}^{R(\omega)}\mu_{\theta^{-R(\omega)}\omega}\|_{\omega,+}}
{\|\mathcal L_{\theta^{-R(\omega)}\omega}^{R(\omega)}\mu_{\theta^{-R(\omega)}\omega}\|_{\omega,-}}
\leq \widetilde K^{R(\omega)}.
\label{eq:R:ratio2}
\end{align}

The cocycle relation $
\mathcal L_{\theta^{-R(\omega)}\omega}^{R(\omega)}\mu_{\theta^{-R(\omega)}\omega}
=
\lambda_{\theta^{-R(\omega)}\omega}^{(R(\omega))}\mu_\omega,$ and
equation \eqref{eq:R:ratio2} yields
\begin{align}
    \frac{\|\mu_\omega\|_{\omega,+ }}{\|\mu_\omega\|_{\omega,-}}
    \leq \widetilde K^{R(\omega)}.
    \label{eq:R:ratio3}
\end{align}

On the other hand, since $\mu_\omega\in \mathcal C_\omega(b,c,\nu)$ and $\|\mu_\omega\|_\omega=1$, Lemma \ref{lem:norm1} gives
\begin{align}
    \frac{1}{3}\leq \|\mu_\omega\|_{\omega,+}\leq 1.
    \label{eq:R:muplus}
\end{align}
Combining \eqref{eq:R:ratio3} with \eqref{eq:R:muplus}, we obtain that  $1\geq \|\mu_{\omega}\|_{\omega,-} \geq \frac{1}{3}\widetilde K^{-R(\omega)}.$ Defining $D_p(\omega)=3\widetilde K^{R(\omega)}$, from \eqref{eq:R:Lp} the proof is finished.

\end{proof}

\subsection{Weak Gibbs property and Equilibrium States}
\label{sec:wgibbs}

We now prove a weak Gibbs estimate for $\upsilon_\phi$. The argument uses the local disintegrations of $\mu_\omega$ and $\ell_\omega$, the integrability of $\log\lambda_\omega$, and distortion estimates for the corrected potential $\bar\phi=\phi-\phi^{J^s}$ from Section \ref{sec:muell}. These estimates are the main input in the variational argument.

The proof from the proposition below borrows from \cite[Lemma 6.6]{StadlbauerSuzukiVarandas2021CMP}, where the weak Gibbs property was established in the setting of non-uniformly expanding random maps without critical points. We adapt these ideas to our context of uniformly hyperbolic random diffeomorphisms.

\begin{proposition}[Weak Gibbs property]\label{thm:wgibbs}
Assume that Hypothesis \ref{hyp:h} holds true. Let $\varepsilon>0$ be sufficiently small. Then there exist $K_\varepsilon\in L^1(\Omega,\mathbb P)$ and measurable functions $c_\varepsilon,C_\varepsilon:\Omega\to(0,\infty)$ such that, for $\mathbb P$-a.e. $\omega\in\Omega$, there exists a strictly increasing sequence $\{n_k(\omega)\}_{k\in\mathbb N}$ with the following property: for every $k\in\mathbb N$ and every $x\in M$,
\begin{align*}
c_\varepsilon(\omega)e^{-K_\varepsilon(\theta^{-n_k(\omega)}\omega)}
&\le
\frac{\upsilon_\omega(B_\omega^{n_k(\omega)}(x,\varepsilon))}
{[\lambda_{\theta^{-n_k(\omega)}\omega}^{(n_k(\omega))}]^{-1}
\exp\left(S_{n_k(\omega)}\bar\phi_{\theta^{-n_k(\omega)}\omega}\left((T_{\theta^{-n_k(\omega)}\omega}^{n_k(\omega)})^{-1}x\right)\right)}
\\
&\le
C_\varepsilon(\omega)e^{K_\varepsilon(\theta^{-n_k(\omega)}\omega)},
\end{align*}
where
$$
B_\omega^n(x,\varepsilon)
:=
\left\{
y\in M:
d\left((T_{\theta^{-j}\omega}^{j})^{-1}x,(T_{\theta^{-j}\omega}^{j})^{-1}y\right)\le\varepsilon
\ \text{for every }0\le j\le n-1
\right\}.
$$
\end{proposition}

\begin{proof}
Since the angle between $\gamma_{(\omega,x)}^s$ and $\gamma_{(\omega,x)}^u$ is bounded away from $0$ for every $x\in M$ and for $\mathbb P$-a.e. $\omega\in\Omega$, for every $\varepsilon>0$ sufficiently small there exist constants $K_1,K_2>0$ such that
$$
R_{\varepsilon_1}(\omega,x)\subset B(x,\varepsilon)\subset R_{\varepsilon_2}(\omega,x)
$$
for every $x\in M$ and for $\mathbb P$-a.e. $\omega\in\Omega$, where $\varepsilon_1:=K_1\varepsilon$ and $\varepsilon_2:=K_2\varepsilon$.

For $i\in\{1,2\}$, define
$$
R_{\varepsilon_i}^{n}(\omega,x):=
\left\{
y\in M:
\begin{aligned}
&(T_{\theta^{-j}\omega}^{j})^{-1}(y)\in
R_{\varepsilon_i}\left(\theta^{-j}\omega,(T_{\theta^{-j}\omega}^{j})^{-1}(x)\right),\\
&\text{for every }0\le j\le n-1
\end{aligned}
\right\}.
$$
Then $
R_{\varepsilon_1}^{n}(\omega,x)\subset B_\omega^n(x,\varepsilon)\subset R_{\varepsilon_2}^{n}(\omega,x).$

Fix $i\in\{1,2\}$ and let $z\in R_{\varepsilon_i}^n(\omega,x)$. Set
$$
\gamma_\omega^s:=\gamma_\delta^s(z,\omega)\cap R_{\varepsilon_i}^{n}(\omega,x).
$$
Then $\gamma_\omega^s$ fully crosses the rectangle $R_{\varepsilon_i}^{n}(\omega,x)$. Moreover, by construction,
$$
(T_{\theta^{-n}\omega}^{n})^{-1}\gamma_\omega^s
\subset
R_{\varepsilon_i}\left(\theta^{-n}\omega,(T_{\theta^{-n}\omega}^{n})^{-1}(x)\right),
$$
and this stable segment has length $\varepsilon_i$. Therefore,
\begin{align}
\int_{\gamma_{(\omega,z)}^s}\mathbbm 1_{R_{\varepsilon_i}^{n}(\omega,x)}\mu_\omega
&=
\int_{\gamma_\omega^s}\mu_\omega \nonumber=
\frac{1}{\lambda_{\theta^{-n}\omega}^{(n)}}\int_{\gamma_\omega^s}\mathcal L_{\theta^{-n}\omega}^{n}\mu_{\theta^{-n}\omega} \nonumber\\
&=
\frac{1}{\lambda_{\theta^{-n}\omega}^{(n)}}\int_{(T_{\theta^{-n}\omega}^{n})^{-1}\gamma_\omega^s}
e^{S_n\bar\phi_{\theta^{-n}\omega}\circ(T_{\theta^{-n}\omega}^{n})^{-1}}\mu_{\theta^{-n}\omega}. \label{eq:wgibbs-1}
\end{align}

Since $\bar\phi_\omega$ is uniformly Hölder, from the definition of $R_{\varepsilon_2}^n(\omega,x)$ there exists $K_\varepsilon\in L^\infty(\Omega,\mathbb P)$,
\begin{align}
\left|
S_n\bar\phi_{\theta^{-n}\omega}\left((T_{\theta^{-n}\omega}^{n})^{-1}(x)\right)
-
S_n\bar\phi_{\theta^{-n}\omega}\left((T_{\theta^{-n}\omega}^{n})^{-1}(y)\right)
\right|
\le
K_\varepsilon(\theta^{-n}\omega). \label{eq:wgibbs-2}
\end{align}
Combining \eqref{eq:wgibbs-1} and \eqref{eq:wgibbs-2}, we obtain
\begin{align}
e^{-K_\varepsilon(\theta^{-n}\omega)}
\int_{(T_{\theta^{-n}\omega}^{n})^{-1}\gamma_\omega^s}\mu_{\theta^{-n}\omega}
&\le
\frac{\int_{\gamma_{(\omega,z)}^s}\mathbbm 1_{R_{\varepsilon_i}^{n}(\omega,x)}\mu_\omega}
{[\lambda_{\theta^{-n}\omega}^{(n)}]^{-1}
\exp\left(S_n\bar\phi_{\theta^{-n}\omega}\left((T_{\theta^{-n}\omega}^{n})^{-1}(x)\right)\right)}
\nonumber\\
&\le
e^{K_\varepsilon(\theta^{-n}\omega)}
\int_{(T_{\theta^{-n}\omega}^{n})^{-1}\gamma_\omega^s}\mu_{\theta^{-n}\omega}.\label{eq:jar1}
\end{align}

Since $\mu_\omega\in\mathrm{Int}(\mathcal C_\omega)$, we have
$$
L(\omega):=
\inf\left\{
\int_{\gamma^s}\mu_\omega:
\gamma^s\ \text{is a stable leaf of length }\varepsilon_i
\right\}
>0.
$$
Hence there exists $C>0$ such that $
\mathbb P[L>C]>0.$ On the other hand, since $\|\mu_\omega\|_\omega=1$ and $\varepsilon<\delta/2$,
$$
\sup\left\{
\int_{\gamma^s}\mu_\omega:
\gamma^s\ \text{is a stable leaf of length }\varepsilon_i
\right\}
\le 1.
$$

By ergodicity of $\theta$, for $\mathbb P$-a.e. $\omega\in\Omega$ there exists a strictly increasing sequence $\{n_k(\omega)\}_{k\in\mathbb N}$ such that
\begin{align}
L(\theta^{-n_k(\omega)}\omega)>C
\ \text{for every }k\in\mathbb N. \label{eq:jar2}
\end{align}

Now fix such an $\omega$. From Lemmas \ref{lem:dec}, \ref{lem:dec1} and since $\upsilon_\omega(f) = \ell_\omega(f \mu_\omega)$ for any $f\in \mathcal C^\beta(M)$. From \eqref{eq:jar1}, \eqref{eq:jar2}, and the inclusions
$$
R_{\varepsilon_1}^{n_k(\omega)}(\omega,x)\subset B_\omega^{n_k(\omega)}(x,\varepsilon)\subset R_{\varepsilon_2}^{n_k(\omega)}(\omega,x),
$$
there exist constants $b_\varepsilon,B_\varepsilon>0$ such that, for every $k\in\mathbb N$ and every $x\in M$,
\begin{align*}
&b_\varepsilon C e^{-K_\varepsilon(\theta^{-n_k(\omega)}\omega)}
\ell^u_{\gamma^u_{(\omega,x)}}\bigl(B(x,\varepsilon_1)\cap\gamma^u_{(\omega,x)}\bigr)\\
&\le
\frac{\upsilon_\omega(R_{\varepsilon_1}^{n_k(\omega)}(\omega,x))}
{[\lambda_{\theta^{-n_k(\omega)}\omega}^{(n_k(\omega))}]^{-1}
\exp\left(S_{n_k(\omega)}\bar\phi_{\theta^{-n_k(\omega)}\omega}\left((T_{\theta^{-n_k(\omega)}\omega}^{n_k(\omega)})^{-1}(x)\right)\right)}
\\
&\le
\frac{\upsilon_\omega(B_\omega^{n_k(\omega)}(x,\varepsilon))}
{[\lambda_{\theta^{-n_k(\omega)}\omega}^{(n_k(\omega))}]^{-1}
\exp\left(S_{n_k(\omega)}\bar\phi_{\theta^{-n_k(\omega)}\omega}\left((T_{\theta^{-n_k(\omega)}\omega}^{n_k(\omega)})^{-1}(x)\right)\right)}
\\
&\le
\frac{\upsilon_\omega(R_{\varepsilon_2}^{n_k(\omega)}(\omega,x))}
{[\lambda_{\theta^{-n_k(\omega)}\omega}^{(n_k(\omega))}]^{-1}
\exp\left(S_{n_k(\omega)}\bar\phi_{\theta^{-n_k(\omega)}\omega}\left((T_{\theta^{-n_k(\omega)}\omega}^{n_k(\omega)})^{-1}(x)\right)\right)}
\\
&\le
B_\varepsilon e^{K_\varepsilon(\theta^{-n_k(\omega)}\omega)}
\ell^u_{\gamma^u_{(\omega,x)}}\bigl(B(x,\varepsilon_2)\cap\gamma^u_{(\omega,x)}\bigr).
\end{align*}

Therefore, after absorbing the unstable leaf-length factors into positive measurable functions $c_\varepsilon(\omega)$ and $C_\varepsilon(\omega)$, we obtain
\begin{align*}
c_\varepsilon(\omega)e^{-K_\varepsilon(\theta^{-n_k(\omega)}\omega)}
&\le
\frac{\upsilon_\omega(B_\omega^{n_k(\omega)}(x,\varepsilon))}
{[\lambda_{\theta^{-n_k(\omega)}\omega}^{(n_k(\omega))}]^{-1}
\exp\left(S_{n_k(\omega)}\bar\phi_{\theta^{-n_k(\omega)}\omega}\left((T_{\theta^{-n_k(\omega)}\omega}^{n_k(\omega)})^{-1}(x)\right)\right)}\\
&\le
C_\varepsilon(\omega)e^{K_\varepsilon(\theta^{-n_k(\omega)}\omega)}.
\end{align*}
This proves the proposition.
\end{proof}

As a consequence of the weak Gibbs property above, we can identify the $\mathbb P$-relative topological pressure associated with the potential $\bar\phi$ and the measure $\upsilon_\phi$.
\begin{proposition}
Under Hypothesis \ref{hyp:h} we have that  $P_{\upsilon_\phi} (F, \bar{\phi} \mid \mathbb P) = \int_{\Omega} \log \lambda_\omega \mathbb P(d\omega)$.  \label{pro:top}
\end{proposition}
\begin{proof}
The result follows from Theorem \ref{thm:wgibbs}  and the Shannon-McMillan-Breiman Theorem for random dynamical systems \cite[Proposition 2.1]{zhu2008local}  
\end{proof}

Before proving that $\upsilon_\phi$ is an equilibrium state, we need one last geometric consequence of the correlation estimates. Namely, we show that every stable manifold of fixed length becomes dense after pulling it back for a suitable amount of time. This will allow us, in the proof of the variational inequality, to construct spanning sets by taking separated points on a single backward image of a stable manifold.

\begin{proposition}\label{prop:epsdense}
    Assume that Hypothesis \ref{hyp:h} holds and fix $\varepsilon>0$. Let $\gamma^{s}(\omega)$ be an arbitrary piece of stable manifold of length $\delta$. Then, there exists $m=m(\omega,\gamma^{s}(\omega))$ such that $\left(T^{m}_{\theta^{-m}\omega}\right)^{-1} \gamma^s(\omega)$ is $\varepsilon$-dense.
\end{proposition}

\begin{proof}
Let $\upsilon_\omega$ be as in Proposition \ref{prop:exponentialdecayofcorrelation}. 
Assume without loss of generality that $0<\delta<\varepsilon/4$ is small enough so that every point $x\in\gamma^s(\omega)$ admits a rectangle of local product structure $R(\omega,x)$ whose stable and unstable sides have length at least $4\delta$, and every ball of radius $\varepsilon/4$ contains a rectangle of local product structure of length at least $\delta$. Define
$$
Q_\delta(\gamma^s(\omega)):=\{y\in M:\ y\in W^u_\delta(\omega,x)\text{ for some }x\in\gamma^s(\omega)\}.
$$
By the choice of $\delta$, $Q_\delta(\gamma^s(\omega))$ is a rectangle which contains $\gamma^s(\omega)$ being saturated by local unstable manifolds of length $\delta$.

Choose finitely many points $\{x_i\}_{i=1}^k\subset M$ such that $M=\bigcup_{i=1}^k B_{\varepsilon/4}(x_i).$
For each $i\in\{1,\dots,k\}$ define
$$
g_i(x)=\frac{\mathrm{dist}(x,M\setminus B_{\varepsilon/2}(x_i))}
{\mathrm{dist}(x,M\setminus B_{\varepsilon/2}(x_i))+\mathrm{dist}(x,B_{\varepsilon/4}(x_i))}.
$$
Then $0\le g_i\le 1$, $g_i=1$ on $B_{\varepsilon/4}(x_i)$, and $\operatorname{supp}(g_i)\subset B_{\varepsilon/2}(x_i)$.

Since the angle between the stable and unstable directions is uniformly bounded away from degeneracy, there exists $r=r(\delta)>0$ such that, for every $\omega\in\Omega$ and every piece of stable manifold $\gamma^s(\omega)$ of length $\delta$, the set $Q_\delta(\gamma^s(\omega))$ contains a ball $B_r(x_*)$, where $x_*$ denotes the midpoint of $\gamma^s(\omega)$. Moreover, by the uniform local product structure, one may choose $r>0$ so that
$$
\operatorname{dist}\bigl(B_r(x_*),M\setminus Q_\delta(\gamma^s(\omega))\bigr)\ge \tau>0
$$
uniformly in $\omega$ and in the choice of $\gamma^s(\omega)$.

Define
$$
f(x):=\frac{d\bigl(x,M\setminus Q_\delta(\gamma^s(\omega))\bigr)}
{d\bigl(x,B_r(x_*)\bigr)+d\bigl(x,M\setminus Q_\delta(\gamma^s(\omega))\bigr)}.
$$
Then $0\le f\le 1$, $f=1$ on $B_r(x_*)$, and $f=0$ on $M\setminus Q_\delta(\gamma^s(\omega))$. The uniform separation above implies that there exists $C>0$ such that
\begin{align}
\|f\|_{\mathcal C^\beta}\le C \label{uniform}
\end{align}
for every $\omega\in\Omega$ and every piece of stable manifold $\gamma^s(\omega)$ of length $\delta$.

Finally, by Proposition \ref{thm:wgibbs}, there exists $c(\omega)>0$, independent of $y$ such that $\upsilon_\omega(B_r(y)) > c(\omega)$ for every $y\in M$. Since $f=1$ on $B_r(x_*)\subset Q_\delta(\gamma^s(\omega))$, it follows that
$$
\upsilon_\omega(f)\ge c(\omega)>0
$$
independent of the choice of the stable manifold.

For each $i$, since $g_i =  1$ on the non-empty ball $B_{\varepsilon/4}(x_i)$, the weak Gibbs property implies that $
\upsilon_\omega (g_i)>0$
for $\mathbb P$-almost every $\omega \in\Omega$. Hence one can choose $c_i>0$ so small that
$$
A_i:=\{\omega\in\Omega:\ \upsilon_\omega(g_i)\ge c_i\}\ \text{satisfies }\mathbb P(A_i)>1-\frac{1}{2k}.$$

We suppose for a contradiction that $\left(T^n_{\theta^{-n}\omega}\right)^{-1}\gamma^s(\omega)$ is not $\varepsilon$-dense for any $n\ge1$. Then, for every $n\ge1$, there exists  a point $p_n\in M$ such that
\begin{align}
\mathrm{dist}\Bigl(p_n,\left(T^n_{\theta^{-n}\omega}\right)^{-1}\gamma^s(\omega)\Bigr)\ge\varepsilon.\label{eq:pn}    
\end{align}
We now prove that this cannot occur. For clarity, we split the argument into four steps.

\begin{step}[1]
    Choose $i(n)\in\{1,\dots,k\}$ such that $p_n\in B_{\varepsilon/4}(x_{i(n)})$. We show that
\begin{align}
B_{3\varepsilon/4}(x_{i(n)})\cap \left(T^n_{\theta^{-n}\omega}\right)^{-1}\gamma^s(\omega)=\varnothing. \label{eq:set1}
\end{align}
\end{step}
Given  $z \in B_{3\varepsilon/4}(x_{i(n)})\cap \left(T^n_{\theta^{-n}\omega}\right)^{-1}\gamma^s(\omega)$ then
$$
d(z,p_n)\le d(z,x_{i(n)})+d(x_{i(n)},p_n)<\frac{3\varepsilon}{4}+\frac{\varepsilon}{4}=\varepsilon,
$$
contradicting the choice of $p_n$, showing \eqref{eq:set1}.

\begin{step}[2]
We show that
\begin{align}
   B_{\varepsilon/2}(x_{i(n)})\cap \left(T^n_{\theta^{-n}\omega}\right)^{-1}Q_\delta(\gamma^s(\omega))=\varnothing. \label{eq:set2}
\end{align}
\end{step}

Let $
y\in B_{\varepsilon/2}(x_{i(n)})\cap \left(T^n_{\theta^{-n}\omega}\right)^{-1}Q_\delta(\gamma^s(\omega)).$
Then $T^n_{\theta^{-n}\omega}(y)\in Q_\delta(\gamma^s(\omega))$, so there exists $x\in\gamma^s(\omega)$ such that $T^n_{\theta^{-n}\omega}(y)\in W^u_\delta(\omega, x)$. Set
$$
z:=\left(T^n_{\theta^{-n}\omega}\right)^{-1}(x)\in \left(T^n_{\theta^{-n}\omega}\right)^{-1}\gamma^s(\omega).
$$
Since $T^n_{\theta^{-n}\omega}(y)$ and $x$ lie on the same local unstable manifold of length $\delta$, backward contraction along unstable leaves gives $d(y,z)\le \delta<\frac{\varepsilon}{4}.$
Hence
$$
d(z,x_{i(n)})\le d(z,y)+d(y,x_{i(n)})<\frac{\varepsilon}{4}+\frac{\varepsilon}{2}=\frac{3\varepsilon}{4},
$$
which contradicts Step 1, and therefore proves \eqref{eq:set2}.

\begin{step}[3]
    We show that there cannot exist a sequence of points $p_n\in M$ such that \eqref{eq:pn} holds, and therefore we complete the proof of the Proposition
\end{step}

Since $\operatorname{supp}(f)\subset Q_\delta(\gamma^s(\omega))$ and $\operatorname{supp}(g_{i(n)})\subset B_{\varepsilon/2}(x_{i(n)})$, the Step 2 implies that
$$
\left(f\circ T^n_{\theta^{-n}\omega}\right)g_{i(n)}\equiv0 \ \text{for every } n\ge1.$$ Hence,
$$
\upsilon_{\theta^{-n}\omega}\left(\left(f\circ T^n_{\theta^{-n}\omega}\right)g_{i(n)}\right)=0\ \text{for every }n\ge1.$$

Define $F_i:=\{n\ge1:\ i(n)=i\}.$
On the one hand, since $\bigcup_{i=1}^k F_i=\mathbb N$, there exists $i_*\in\{1,\dots,k\}$ such that $F_{i_*}$ has upper density at least $1/k$. On the other hand, since $\mathbb P(A_{i_*})>1-\frac{1}{2k}$ and $\mathbb P$ is $\theta$-ergodic, Birkhoff's theorem gives that
$$
E_{i_*}(\omega):=\{n\ge1:\ \theta^{-n}\omega\in A_{i_*}\}
$$
has density $\mathbb P(A_{i_*})>1-\frac{1}{2k}$. Hence $
F_{i_*}\cap E_{i_*}(\omega)$ is infinite.

Fix $n\in F_{i_*}\cap E_{i_*}(\omega)$. Since $n\in F_{i_*}$, we have
\begin{align}
\upsilon_{\theta^{-n}\omega}\left(\left(f\circ T^n_{\theta^{-n}\omega}\right)g_{i_*}\right)=0.\label{eq:zero}
\end{align}
Since $n\in E_{i_*}(\omega)$, we also have
\begin{align}
\upsilon_{\theta^{-n}\omega}(g_{i_*})\ge c_{i_*}. \label{eq:notzero}
\end{align}
Applying Proposition \ref{prop:exponentialdecayofcorrelation}, \eqref{eq:zero} and \eqref{eq:notzero} we obtain that
\begin{align*}
0<\upsilon_\omega(f)c_{i_*}
\le \upsilon_\omega(f)\upsilon_{\theta^{-n}\omega}(g_{i_*})
&=
\left|
\upsilon_{\theta^{-n}\omega}\left(\left(f\circ T^n_{\theta^{-n}\omega}\right)g_{i_*}\right)
-
\upsilon_\omega(f)\upsilon_{\theta^{-n}\omega}(g_{i_*})
\right|\\
&\le
C(\omega)\Lambda^n \|f\|_{\mathcal C^\beta}\|g_{i_*}\|_{\mathcal C^\beta}.
\end{align*}
Since $\Lambda\in(0,1)$, the right-hand side tends to $0$ as $n\to\infty$ along the infinite set $F_{i_*}\cap E_{i_*}(\omega)$, which is impossible.

Therefore, our assumption was false. Hence there exists $m=m(\omega,\gamma^s(\omega))$ such that $\left(T^m_{\theta^{-m}\omega}\right)^{-1}\gamma^s(\omega)$
is $\varepsilon$-dense in $M$.
\end{proof}

We now prove that $\upsilon_\phi$ is a $\mathbb P$-relative equilibrium state for the potential $\bar\phi$.
\begin{proposition}\label{prop:eqstate}
The invariant measure $\upsilon_{\phi}:=\upsilon_\omega(\d x)\mathbb P(\d\omega)$ constructed in Lemma \ref{lem:nuprob} is a $\mathbb P$-relative equilibrium state for the potential $\bar\phi$.
\end{proposition}

\begin{proof}
We follow the strategy of \cite[Section 4]{ParmenterPollicott2022}, but applied to the inverse skew product. Let
$$
\widehat F:=F^{-1},\ 
\widehat F(\omega,x)=\left(\theta^{-1}\omega,\widehat T_\omega(x)\right),
\ 
\widehat T_\omega:=T_{\theta^{-1}\omega}^{-1},\ \text{and }\widehat{\bar\phi}:=\bar\phi\circ F^{-1}.$$ For every $n\ge1$,
$$
\widehat T_\omega^n
=
\widehat T_{\theta^{-(n-1)}\omega}\circ\cdots\circ\widehat T_{\theta^{-1}\omega}\circ\widehat T_\omega
=
\left(T_{\theta^{-n}\omega}^{n}\right)^{-1}.
$$
We denote by $\widehat Z_0(n,\varepsilon,\omega)$ the spanning quantity in Proposition \ref{prop:psaning} associated to the inverse skew product $\widehat F$ and to the potential $\widehat{\bar\phi}$.

Since $F$ is invertible, we have that
\begin{align}
P_{\mathrm{top}}(F,\bar\phi\mid\mathbb P)
=
P_{\mathrm{top}}\left(\widehat F,\widehat{\bar\phi}\mid\mathbb P\right).
\label{eq:inverse-pressure-polished}
\end{align}
By Proposition \ref{pro:top}, it is enough to whoe that
$$
P_{\mathrm{top}}\left(\widehat F,\widehat{\bar\phi}\mid\mathbb P\right) = P_{\mathrm{top}}(F,\bar\phi\mid\mathbb P) \le \int \log\lambda_\omega\,\mathbb P(\d\omega).
$$

Fix $\varepsilon>0$. Let $\gamma^s(\omega)=\gamma^s(\omega,x)$ be an arbitrary piece of stable manifold of length $\delta$. By Proposition \ref{prop:epsdense}, for $\mathbb P$-almost every $\omega$ there exists $
m=m(\omega,\gamma^s(\omega))\in\mathbb N$
such that $(T^m_{\theta^{-m}\omega})^{-1}\gamma^s(\omega)$
is $\varepsilon$-dense in $M$.

We now fix such an $\omega$ and such an $m$, and prove that
$$
\liminf_{n\to\infty}
\frac1n\log \widehat Z_0(n,2\varepsilon,\theta^{-m}\omega)
\le
\int\log\lambda_\omega\,\mathbb P(\d \omega).
$$
Using the almost sure version of Proposition \ref{prop:psaning}, applied to $\widehat F$ and $\widehat{\bar\phi}$, this will imply the result. We divide the remainder of the proof into four steps.

\begin{step}[1] We construct a suitable $(n, 2 \varepsilon,  \theta^{-m} \omega)$ spanning set for the inverse skew product and call it 
$S_n(\theta^{-m}\omega).$  
\end{step}

For $n\ge1$, set $\Gamma_{n,m}(\omega):=\left(T^{n+m}_{\theta^{-(n+m)}\omega}\right)^{-1}\gamma^s(\omega).$
This is a piece of a stable manifold in the fibre over $\theta^{-(n+m)}\omega$. We endow this piece of stable manifold of the distance $d_s(\theta^{-(n+m)}\omega,\cdot,\cdot).$
Choose a maximal $(\varepsilon/2)$-separated set
$$
S'_n(\theta^{-(n+m)}\omega)=\{x_1,\dots,x_{N_n}\}\subset \Gamma_{n,m}(\omega)
$$
with respect to this stable distance. From the maximal choice of $\{x_1,\ldots, x_{N_n}\}$
\begin{align}
\Gamma_{n,m}(\omega)\subset \bigcup_{i=1}^{N_n}
B_{d_s}(\theta^{-(n+m)}\omega,x_i,\varepsilon/2),\label{eq:cover}
\end{align}
and the balls $B_{d_s}(\theta^{-(n+m)}\omega,x_i,\varepsilon/4),$ $i=1,\dots,N_n,$
are pairwise disjoint and contained in $\Gamma_{n,m}(\omega)$. Hence
\begin{equation}\label{eq:balls-polished}
\bigcup_{i=1}^{N_n} B_{d_s}(\theta^{-(n+m)}\omega,x_i,\varepsilon/4)
\subset \Gamma_{n,m}(\omega).
\end{equation}

Define $
y_i:=T^n_{\theta^{-(n+m)}\omega}(x_i),$ $i \in \{1,\dots,N_n\}.$
Then $y_i\in (T^m_{\theta^{-m}\omega})^{-1}\gamma^s(\omega)$. We claim that $
S_n(\theta^{-m}\omega):=\{y_1,\dots,y_{N_n}\}$
is an $(n,2\varepsilon,\theta^{-m}\omega)$-spanning set for $\widehat F$.

Indeed, let $z\in M$. Since $(T^m_{\theta^{-m}\omega})^{-1}\gamma^s(\omega)$ is $\varepsilon$-dense, and by the choice of the local product structure of length $\varepsilon$, there exists $y\in (T^m_{\theta^{-m}\omega})^{-1}\gamma^s(\omega)$
such that $z$ lies on the local unstable manifold through $y$. Set
$$
u:=(T^n_{\theta^{-(n+m)}\omega})^{-1}(y)\in \Gamma_{n,m}(\omega).
$$
From \eqref{eq:cover}, there exists $i\in\{1,\dots,N_n\}$ such that $
d_s(\theta^{-(n+m)}\omega,u,x_i)<\varepsilon/2.$
Since $u$ and $x_i$ lie on the same local stable leaf, forward iterates through the random dynamics contract their distance uniformly. Therefore, for every $0\le r\le n$, we obtain
\begin{align}
d\left(
T^r_{\theta^{-(n+m)}\omega}(u),
T^r_{\theta^{-(n+m)}\omega}(x_i)
\right)<\varepsilon/2.\label{eq:dist1}
\end{align}
On the other hand, $z$ and $y$ belong to the same local unstable manifold. Hence, by the definition of local unstable manifolds, for every $0\le j\le n-1$,
\begin{align}
d\left(
\left(T^j_{\theta^{-(m+j)}\omega}\right)^{-1}(z),
\left(T^j_{\theta^{-(m+j)}\omega}\right)^{-1}(y)
\right)<\varepsilon. \label{eq:dist2}
\end{align}
Moreover, for every $0\le j\le n-1$,
$$
\left(T^j_{\theta^{-(m+j)}\omega}\right)^{-1}(y)
=
T^{n-j}_{\theta^{-(n+m)}\omega}(u),\ \text{and }\left(T^j_{\theta^{-(m+j)}\omega}\right)^{-1}(y_i)
=
T^{n-j}_{\theta^{-(n+m)}\omega}(x_i).
$$
Taking $r=n-j$ in \eqref{eq:dist1}, and combining this with \eqref{eq:dist2}, we obtain
$$
\max_{0\le j<n}
d\left(
\left(T^j_{\theta^{-(m+j)}\omega}\right)^{-1}(z),
\left(T^j_{\theta^{-(m+j)}\omega}\right)^{-1}(y_i)
\right)<2\varepsilon.
$$
This proves the claim.

\begin{step}[2]
We find an upper-bound bound for $\widehat Z_0(n,2 \varepsilon,\theta^{-m}\omega) $ which is stated in \eqref{eq:gammas-polished}. 
\end{step}

Since $S_n(\theta^{-m}\omega)$ is an $(n,2\varepsilon,\theta^{-m}\omega)$-spanning set for $\widehat F$, Proposition \ref{prop:psaning}, applied to $\widehat F$ and $\widehat{\bar\phi}$, gives
\begin{equation}\label{eq:Z0-first-bound}
\widehat Z_0(n,2\varepsilon,\theta^{-m}\omega)
\le \sum_{i=1}^{N_n}
\exp\Bigl(
S_n\bar\phi_{\theta^{-(n+m)}\omega}\circ
(T^n_{\theta^{-(n+m)}\omega})^{-1}(y_i)
\Bigr).
\end{equation}
Indeed, if $x_i=(T^n_{\theta^{-(n+m)}\omega})^{-1}(y_i)$, then $
\widehat S_n\widehat{\bar\phi}(\theta^{-m}\omega,y_i)
=
S_n\bar\phi_{\theta^{-(n+m)}\omega}(x_i).$

For each $i$, let $
\Gamma_{(x_i,\varepsilon/4)}\in\mathbb V_{\theta^{-(n+m)}\omega}^*$
be the functional
$$
\Gamma_{(x_i,\varepsilon/4)}(f)
:=
\int_{B_{d_s}(\theta^{-(n+m)}\omega,x_i,\varepsilon/4)}
f\,\d m_{B_{d_s}(\theta^{-(n+m)}\omega,x_i,\varepsilon/4)},
$$
where $m_{B_{d_s}(\theta^{-(n+m)}\omega,x_i,\varepsilon/4)}$ is the measure induced on the corresponding stable manifold. Then \eqref{eq:Z0-first-bound} becomes
\begin{align}
\widehat Z_0(n,2\varepsilon,\theta^{-m}\omega)
&\le \sum_{i=1}^{N_n}
\frac{1}{\Gamma_{(x_i,\varepsilon/4)}(\mu_{\theta^{-(n+m)}\omega})}
\Gamma_{(x_i,\varepsilon/4)}
\left(
e^{S_n\bar\phi_{\theta^{-(n+m)}\omega}\circ
(T^n_{\theta^{-(n+m)}\omega})^{-1}(y_i)}
\mu_{\theta^{-(n+m)}\omega}
\right).
\label{eq:gammas-polished}
\end{align}

\begin{step}[3] We find an increasing subsequence $\{n_k\}_{k\in\mathbb N} \subset \mathbb N$ and a suitable upper-bound for $\widehat Z_0(n_k, 2\varepsilon,\theta^{-m}\omega)$ which is given in  \eqref{eq:intn}.
\end{step}

Define $M(\omega):=
\inf\left\{
\int_{\gamma^s}\mu_\omega:\ 
\gamma^s\ \text{is a stable leaf of length }\varepsilon/4
\right\}.$
Since $\mu_\omega\in \operatorname{Int}(\mathcal C_\omega(b,c,\nu))$ for $\mathbb P$-almost every $\omega$, we have $M(\omega)>0$
for $\mathbb P$-almost every $\omega$. Hence there exists $M_0>0$ such that the set
$$
\Omega_{M_0}:=\{\omega\in\Omega:\ M(\omega)\ge M_0\}
$$
has positive $\mathbb P$-measure.

By ergodicity of $\theta$, for $\mathbb P$-almost every fixed $\omega$ there exists a strictly increasing sequence $(n_k)_{k\ge1}$ such that $\theta^{-(n_k+m)}\omega\in \Omega_{M_0}$
for every $k\ge1$. For such $n_k$, every piece of stable manifold $
B_{d_s}(\theta^{-(n_k+m)}\omega,x_i,\varepsilon/4)$ has $\mu_{\theta^{-(n_k+m)}\omega}$-mass at least $M_0$, and therefore
$$
\Gamma_{(x_i,\varepsilon/4)}\mu_{\theta^{-(n_k+m)}\omega}\ge M_0.
$$
Using \eqref{eq:balls-polished}, \eqref{eq:gammas-polished}, and the distortion estimate in Proposition \ref{thm:wgibbs}, we obtain
\begin{align*}
\widehat Z_0(n_k,2\varepsilon,\theta^{-m}\omega)
&\le \frac{1}{M_0}e^{K(\theta^{-m}\omega)}
\int_{\Gamma_{n_k,m}(\omega)}
e^{S_{n_k}\bar\phi_{\theta^{-(n_k+m)}\omega}\circ
(T^{n_k}_{\theta^{-(n_k+m)}\omega})^{-1}}
\mu_{\theta^{-(n_k+m)}\omega}.
\end{align*}

Since $
\mathcal L^{n_k}_{\theta^{-(n_k+m)}\omega}\mu_{\theta^{-(n_k+m)}\omega}
=
\lambda^{(n_k)}_{\theta^{-(n_k+m)}\omega}\mu_{\theta^{-m}\omega}.$ It follows that
\begin{align}
\widehat Z_0(n_k,2\varepsilon,\theta^{-m}\omega)
&\le
\frac{1}{M_0}e^{K(\theta^{-m}\omega)}
\lambda^{(n_k)}_{\theta^{-(n_k+m)}\omega}
\int_{(T^m_{\theta^{-m}\omega})^{-1}\gamma^s(\omega)}
\mu_{\theta^{-m}\omega}.\label{eq:intn}
\end{align}
Since $\mu_{\theta^{-m}\omega}\in \operatorname{Int}(\mathcal C_{\theta^{-m}\omega}(b,c,\nu))$, the last integral is strictly positive.

\begin{step}[4] We conclude the proof.
\end{step}

Taking logarithms, dividing by $n_k$ in equation \eqref{eq:intn}, and letting $k\to\infty$, we obtain
\begin{align}
\limsup_{k\to\infty}
\frac1{n_k}
\log \widehat Z_0(n_k,2\varepsilon,\theta^{-m}\omega)
\le
\limsup_{k\to\infty}
\frac1{n_k}
\log \lambda^{(n_k)}_{\theta^{-(n_k+m)}\omega}.
\label{eq:limsup-polished}
\end{align}
Since
$$
\log \lambda^{(n_k)}_{\theta^{-(n_k+m)}\omega}
=
\sum_{j=0}^{n_k-1}
\log \lambda_{\theta^{-(n_k+m)+j}\omega},
$$
Birkhoff's theorem applied to $\theta^{-1}$ yields
\begin{align}
\lim_{k\to\infty}
\frac1{n_k}
\log \lambda^{(n_k)}_{\theta^{-(n_k+m)}\omega}
=
\int\log \lambda_\omega\,\mathbb P(\d\omega)
\label{eq:birkhoff-polished}
\end{align}
for $\mathbb P$-almost every $\omega$. Hence, combining \eqref{eq:limsup-polished} and \eqref{eq:birkhoff-polished}, we obtain that
\begin{align}
\liminf_{n\to\infty}
\frac1n
\log \widehat Z_0(n,2\varepsilon,\theta^{-m}\omega)
\le
\int\log\lambda_\omega\,\mathbb P(\d\omega).
\label{eq:inverse-pressure-polished-final}
\end{align}

Recall that the $m$ in the above equation is not uniform on $\omega$. Hence, for $\mathbb P$-almost every $\omega$ we can associate an $m=m(\omega)$ such that \eqref{eq:inverse-pressure-polished-final} holds. In this way, there exists $m_0\in\mathbb N$ such that the set
$$
\mathbb P\left(
\left\{
\omega\in\Omega:
\liminf_{n\to\infty}
\frac1n
\log \widehat Z_0(n,2\varepsilon,\theta^{-m_0}\omega)
\le
\int\log\lambda_\omega\,\mathbb P(\d\omega)
\right\}
\right)>0.
$$
Finally, using the almost sure version of Proposition \ref{prop:psaning}, applied to $\widehat F$ and $\widehat{\bar\phi}$, and letting $\varepsilon\to0$, we conclude that
$$
P_{\mathrm{top}}(\widehat F,\widehat{\bar\phi}\mid\mathbb P)
\le
\int\log\lambda_\omega\,\mathbb P(\d\omega).
$$
From \eqref{eq:inverse-pressure-polished}, this gives
$$
P_{\mathrm{top}}(F,\bar\phi\mid\mathbb P)
\le
\int\log\lambda_\omega\,\mathbb P(\d\omega).
$$
From the above equation and Proposition \ref{pro:top}, this proves that $\upsilon_\phi$ is an equilibrium state for the potential $\bar\phi$.
\end{proof}

\section{Uniqueness of Equilibrium States}

It remains to prove that $\upsilon_\phi$ is the unique $\mathbb P$-relative equilibrium state for the potential $\bar\phi$. This is the final ingredient needed to complete the proof of Theorems \ref{theorem:mainthmA} and \ref{theorem:mainthmB}. We begin by introducing Sinai--Ledrappier--Young partitions (SLY partitions), which will be used to establish uniqueness of $\mathbb P$-relative equilibrium states.

\begin{definition}[SLY partition]\label{def:sly}
    A measurable partition $\xi$ of $\Omega \times M$ is said to be SLY  (Sinai-Ledrappier-Young) with respect to (i) $F:\Omega\times M\to\Omega\times M$ a regular random dynamical system satisfying Hypothesis \ref{hyp:h}, and (ii) $\eta$ a probability measure on $\Omega \times M$ that is $F$-invariant and so that  $\left(\proj_\Omega\right)_*\eta=\mathbb P$ if:
    \begin{enumerate}
        \item $(T_\omega)^{-1}\xi(\theta \omega, T_\omega x) \supset \xi(\omega,x)$, $\eta$-a.s.
        \item $\{\omega\} \times M \supset \xi(\omega,x)$, $\eta$-a.s.
        \item $\xi(\omega,x) \subset W^s(\omega, x) = \cup_{n=0}^\infty (T_\omega^n)^{-1} W^s_\varepsilon(\theta^n \omega, T_\omega^n x)$, $\eta$-a.s.
        \item $\xi(\omega,x)$ contains a neighbourhood of $x$ inside $W^s(\omega,x)$, $\eta$-a.s. 
        \item for any $B \in \mathcal{B}(\Omega \times M)$, $B_\omega = \{x \in M: (\omega,x) \in B\}$: \begin{equation*}
            (\omega,x) \mapsto m_{W^s(\omega,x)}(\xi(\omega,x)\cap B_\omega)
        \end{equation*}is measurable and $\eta$-a.s. finite
        \item $\bigvee_{n=0}^\infty F^n \xi = \mathcal{B}_{\Omega \times M}$
        \item $\bigwedge_{n=0}^\infty F^{-n} \xi = \mathcal{B}^s_{\Omega \times M}$ the $\sigma$-algebra of $\{\{\omega\} \times W^s(\omega,x):(\omega,x) \in \Omega \times M\}$-saturated sets.
        \item $h_\eta(F|\mathbb{P})=\int H_{\eta_\omega} ( T_{\theta^{-1} \omega} \xi_{\theta^{-1} \omega} | \xi_\omega ) d \mathbb{P}(\omega)$.
    \end{enumerate}
\end{definition}

\begin{remark}
    We use the following abuse of notation throughout. For any measurable set $A \subset M$, we write
    $$
    T_\omega(\{\omega\}\times A):=F(\{\omega\}\times A)=\{\theta\omega\}\times T_\omega(A).
    $$
    In other words, when $F$ acts on a set contained in the fibre $\{\omega\}\times M$, we simply denote this action by $T_\omega$.
\end{remark}

The existence of SLY partitions of this kind is guaranteed by \cite[Proposition 3.2.1]{kifer2006random}. We next prove a technical absolute-continuity statement for the conditional measures of an arbitrary relative equilibrium state along the elements of an SLY partition, comparing them with the corresponding conditional measures induced by the spectral measure $\mu_\omega$.

\begin{proposition}\label{prop:eqdisint}
    Let $\eta$ be any relative equilibrium state for a potential $\bar{\phi}$ that satisfies Definition \ref{def:potential} and a regular random dynamical system $F:\Omega\times M\to\Omega\times M$ that satisfies Hypothesis \ref{hyp:h}. Consider a SLY partition $\xi$ with respect to $\eta$ and $F$. Denote by $\eta_{\xi(\omega,x)}$ the probability on $\xi(\omega,x)$ given by the disintegration of $\eta$ over $\xi$. Then: 
    \begin{align}
            \eta_{\xi(\omega,x)} \sim \mu_{\xi(\omega,x)},\  \eta\text{-a.s.},\label{eq:eqstateabscont}
\end{align}
    where $\mu = \mu_\omega(\d x) \mathbb{P}(d\omega)$ is the measure induced on $\Omega \times M$ considering the conditionals given in Theorem \ref{thm:spectralgap} and further characterized in Lemma \ref{lem:stablemu}.
\end{proposition}
\begin{proof}
    This proof adapts the arguments in \cite[Theorem 4.3]{carrasco2024equilibrium}. From Lemma \ref{lem:stablemargulis} we obtain that $$\mu_{\xi_{(\theta \omega,T_\omega x )}}\big|_{T_\omega \xi(\omega,x)} = e^{\bar{\phi}_\omega \circ (T_\omega)^{-1} - \ln \lambda_\omega} \hspace{1mm} {T_\omega}_* \mu_ {\xi(\omega,x)},$$where $\bar{\phi}_\omega = \phi_\omega - \phi_{\omega}^{J^s}$ and the vertical bar is for restriction (not conditioning).

    \begin{step}[1] We assume that \eqref{eq:eqstateabscont} holds and show that this implies that
    $$\eta_{\xi(\omega,x)} (d y) = \frac{\Delta^s_{\omega,x}(y) }{L_\omega(x')}\mu_{\xi(\omega,x)}(d y),\ \eta\text{-a.s.}$$
    where 
    $$\Delta^s_{\omega,x}(y) := \frac{\prod_{j=0}^\infty e^{\bar{\phi}_{\theta^j \omega}(T_\omega^j y)} }{ \prod_{j=0}^\infty e^{\bar{\phi}_{\theta^j \omega}(T_\omega^j x)} },$$
  and $L_\omega(x)=\int \Delta^s_{\xi(\omega,x) }(y) \mu_{\xi(\omega,x)}(\d  y)$.

    \end{step}

    Assuming that \eqref{eq:eqstateabscont} holds, we write $\eta_{\xi(\omega,x)} = \rho\big|_{\xi(\omega,x)} \mu_{\xi(\omega,x)}$. Then notice that, for any measurable subset $A \subset M$:
    
    \begin{align}
& \frac{1}{\eta_{\xi(\theta \omega, T_\omega x)} (T_\omega \xi(\omega,x)) }
\int_{A \cap T_\omega \xi(\omega,x)} 
\rho\big|_{\xi(\theta \omega, T_\omega x)} 
\, \d \mu_{\xi(\theta \omega, T_\omega x)}
= 
\frac{\eta_{\xi(\theta \omega,T_\omega x)} (A \cap T_\omega \xi(\omega,x)) }
{ \eta_{\xi(\theta \omega, T_\omega x)} (T_\omega \xi(\omega,x)) } \\
&= \eta_{(F \xi)(\theta \omega,T_\omega x)}(A) = (F_* \eta)_{(F \xi)(\theta \omega, T_\omega x)}(A) = F_* (\eta_{\xi(\omega,x)})(A) = {T_\omega}_* (\eta_{\xi(\omega,x)})(A) \\
&= \int_{A \cap T_\omega \xi(\omega,x)} 
\rho\big|_{\xi(\omega,x)} \circ (T_\omega)^{-1} 
\, d {T_\omega}_*\mu_{\xi(\omega,x)} \\
&= \int_{A \cap T_\omega \xi(\omega,x)} 
\rho\big|_{\xi(\omega,x)} \circ (T_\omega)^{-1} 
e^{\ln \lambda_\omega - \bar{\phi}_\omega \circ (T_\omega)^{-1} }
\, \d \mu_{\xi_{(\theta \omega,T_\omega x )}}.
\end{align}

The densities $\rho\big|_{\xi(\omega,x)}(z)$ lift to a single function $\rho$ on $\Omega \times M$ by $$\rho\big|_{\xi(\omega,x)}(z) = \rho\big|_{\xi(\omega,z)}(z) =: \rho(\omega,z).$$ The latter then implies that, for $\eta$-a.e. $(\omega,x)\in \Omega \times M$ and $\mu_{\xi(\theta \omega, T_\omega x)}$-a.e. $y \in T_\omega \xi(\omega,x)$
\begin{align}
\frac{\rho_\omega ( (T_\omega)^{-1} y) e^{\ln \lambda_\omega - \bar{\phi}_\omega( (T_\omega)^{-1}y) } }{\rho_{\theta \omega} (y)}    = \frac{1}{\eta\big|_{\xi(\theta \omega, T_\omega x)} (T_\omega  \xi(\omega,x) ) }.
\end{align}

 Choosing any two  points $y' = T_\omega x'$ and $y'' = T_\omega x''$ in $T_\omega \xi(\omega,x)$, with  $x',x'' \in \xi(\omega,x)$, it follows that the right side in the last equation is identical, implying that 
\begin{align}
    \frac{\rho_\omega (x') e^{\ln \lambda_\omega - \bar{\phi}_\omega( x') } }{\rho_{\theta \omega } (T_\omega x') } =  \frac{\rho_\omega (x'') e^{\ln \lambda_\omega - \bar{\phi}_\omega( x'') } }{\rho_{\theta \omega } (T_\omega x'') }
\end{align}
therefore, under induction,
\begin{align}
    \frac{\rho_\omega(x'') }{\rho_\omega(x')} = \frac{\prod_{j=0}^\infty e^{\bar{\phi}_{\theta^j \omega}(T_\omega^j x'')} }{ \prod_{j=0}^\infty e^{\bar{\phi}_{\theta^j \omega}(T_\omega^j x')} } =: \Delta^s_{\omega,x'}(x''),
\end{align}whose $\log$, as a function on $x'' \in \xi(\omega,x')$, is bounded away from zero and infinity, uniformly on $\omega, x'$ and $x''$ due to the uniformly hyperbolicity of $T_\omega$ and the fact that $\bar{\phi}$ Hölder in the fibre variable, uniformly in $\omega\in \Omega$.

Finally, writing $\rho_\omega(z'') = \rho_\omega(x') \Delta^s_{\omega,x'}(x'')$ and integrating over $\d \mu_{\xi(\omega,x)}(x'')$, one gets
\begin{align}
    \forall x'' \in \xi(\omega,x'): \rho_\omega(x'')= \frac{\Delta^s_{\omega,x'}(x'') }{L_\omega(x')},
\end{align}with $L_\omega(x')=\int \Delta^s_{\xi(\omega,x') }(x'') \d \mu_{\xi(\omega,x')}(x'')$, for $\eta$-a.e. $(\omega,x') \in \Omega \times M$. Choosing $x=x'$ and $y = x''$ Step 1 is completed.

  \begin{step}[2]
We now prove \eqref{eq:eqstateabscont}.
\end{step}

Define a measure $\tilde{\eta}$ by prescribing the same quotient measure on the partition $\xi$ as $\eta$, and by requiring that its conditional measure on each atom $\xi(\omega,x)$ be given by
\begin{align}
\tilde{\eta}(\xi(\omega,x)):=\eta(\xi(\omega,x)),\ \text{where }\tilde{\eta}_{\xi(\omega,x)}:=\frac{\Delta^s_{\omega,x}(\cdot)}{L_\omega(x)}\mu_{\xi(\omega,x)},
\end{align}
for $\eta$-a.e. $(\omega,x)\in\Omega\times M$.
We claim that $\tilde{\eta}=\eta$. Once this is proved, it follows immediately that
$$\eta_{\xi(\omega,x)}\sim\mu_{\xi(\omega,x)},
\ \eta\text{-a.s.},$$
and in fact $\eta_{\xi(\omega,x)}=\frac{\Delta^s_{\omega,x}(\cdot)}{L_\omega(x)}\mu_{\xi(\omega,x)}.$

    By property (6) of Definition \ref{def:sly}, the desired coincidence follows from checking $\eta$ and $\tilde{\eta}$ match on $F^n \xi$, or, by reduction, on $F \xi$.  Since $(F\xi)(\omega,x)= T_{\theta^{-1}\omega} \xi(\theta^{-1}\omega, (T_{\theta^{-1}\omega})^{-1}x)$, one calculates: 
    \begin{align}
&\tilde{\eta}_{\xi(\omega,x)}\bigl(T_{\theta^{-1}\omega}\xi(\theta^{-1}\omega,(T_{\theta^{-1}\omega})^{-1}x)\bigr)\\
&=\frac{1}{L_\omega(x)}
\int_{\xi(\theta^{-1}\omega,(T_{\theta^{-1}\omega})^{-1}x)}
\Delta^s_{\omega,x}(T_{\theta^{-1}\omega}z)\,
\bigl[(T_{\theta^{-1}\omega})^{-1}_*\mu_{\xi(\omega,x)}\bigr](d z)\\
&=\frac{1}{L_\omega(x)}
\int_{\xi(\theta^{-1}\omega,(T_{\theta^{-1}\omega})^{-1}x)}
\left(
\prod_{n=0}^\infty
\frac{e^{\bar{\phi}_{\theta^{n-1}\omega}(T_{\theta^{-1}\omega}^n z)}}
{e^{\bar{\phi}_{\theta^{n-1}\omega}(T_{\theta^{-1}\omega}^n (T_{\theta^{-1}\omega})^{-1}x)}}
\right)
\frac{e^{\bar{\phi}_{\theta^{-1}\omega}((T_{\theta^{-1}\omega})^{-1}x)}}{
e^{\ln\lambda_{\theta^{-1}\omega}}}\,
\mu_{\xi(\theta^{-1}\omega,(T_{\theta^{-1}\omega})^{-1}x)}(d z)\\
&=\frac{e^{\bar{\phi}_{\theta^{-1}\omega}((T_{\theta^{-1}\omega})^{-1}x)-\ln\lambda_{\theta^{-1}\omega}}}{L_\omega(x)}
\int_{\xi(\theta^{-1}\omega,(T_{\theta^{-1}\omega})^{-1}x)}
\Delta^s_{\theta^{-1}\omega,(T_{\theta^{-1}\omega})^{-1}x}(z)\,
 \mu_{\xi(\theta^{-1}\omega,(T_{\theta^{-1}\omega})^{-1}x)}(d z)\\
&=e^{\bar{\phi}_{\theta^{-1}\omega}((T_{\theta^{-1}\omega})^{-1}x)-\ln\lambda_{\theta^{-1}\omega}}
\frac{L_{\theta^{-1}\omega}((T_{\theta^{-1}\omega})^{-1}x)}{L_\omega(x)}.
\end{align}

    Since the first term is bounded by one, $$L_{\theta^{-1}\omega}( (T_{\theta^{-1}\omega})^{-1}x) / L_\omega(x)\text{ is bounded by  }e^{\ln \lambda_{\theta^{-1}\omega} - \bar{\phi}_{\theta^{-1}\omega}( (T_{\theta^{-1}\omega})^{-1}x ) },$$ whose $\log$ is in $L^1(\eta)$. Therefore
    \begin{align}
        0 & = \int_{\Omega \times M} \ln \frac{L_{\theta^{-1}\omega}( (T_{\theta^{-1}\omega})^{-1}x) }{L_\omega(x)} \eta(\d \omega,\d x),
    \end{align}which, together with the chain of equalities above, implies
    \begin{align*}
        &\int_{\Omega \times M} \ln \tilde{\eta}_{\xi(\omega,x)}(T_{\theta^{-1}\omega} \xi(\theta^{-1}\omega, (T_{\theta^{-1}\omega})^{-1}x)) \eta(\d \omega,\d x) \\& = \int -\bar{\phi}_{\theta^{-1}\omega}( (T_{\theta^{-1}\omega})^{-1}x ) + \ln \lambda_{\theta^{-1}\omega} \eta(\d \omega,\d x)\\& = \int_\Omega \ln \lambda_\omega \mathbb{P}(d \omega) - \int_{\Omega \times M} \bar{\phi}( \omega,x) \eta(\d \omega,\d x). 
    \end{align*}

    The same equality is valid for $\eta$ instead of $\tilde{\eta}$, because $\eta$ is assumed a relative equilibrium state and by property (8) of Definition \ref{def:sly}. Both integrals coinciding imply the first equality in the chain below
    \begin{align}
        0 & = \int_{\Omega \times M} \ln \frac{\tilde{\eta}_{\xi(\omega,x)}(T_{\theta^{-1}\omega} \xi(\theta^{-1}\omega, (T_{\theta^{-1}\omega})^{-1}x))}{\eta_{\xi(\omega,x)}(T_{\theta^{-1}\omega} \xi(\theta^{-1}\omega, (T_{\theta^{-1}\omega})^{-1}x))} \eta(\d \omega,\d x) \\ & \le \ln \int_{\Omega \times M} \frac{\tilde{\eta}_{\xi(\omega,x)}(T_{\theta^{-1}\omega} \xi(\theta^{-1}\omega, (T_{\theta^{-1}\omega})^{-1}x))}{\eta_{\xi(\omega,x)}(T_{\theta^{-1}\omega} \xi(\theta^{-1}\omega, (T_{\theta^{-1}\omega})^{-1}x))} \eta(\d \omega,\d x) \le \ln 1 = 0,
    \end{align}where the second inequality is justified in the next paragraph. The equality found above is an instance of Jensen's equality, which implies that the integrand is $\eta$-a.s. constant equal to $1$, as desired.

    The latter inequality follows from writing $\xi(\omega,x) = \bigsqcup_{j=1}^\infty A_{\omega,x}^j $, with $$A_{\omega,x}^j = T_{\theta^{-1}\omega} \xi(\theta^{-1}\omega, (T_{\theta^{-1} \omega})^{-1} y_j) \cap \xi(\omega,x),$$ for suitably chosen $y_j \in \xi(\omega,x)$, and checking that
    \begin{align}
        & \int_{\xi(\omega,x)} \frac{\tilde{\eta}_{\xi(\omega,y)}(T_{\theta^{-1}\omega} \xi(\theta^{-1}\omega, (T_{\theta^{-1}\omega})^{-1}y))}{\eta_{\xi(\omega,y)}(T_{\theta^{-1}\omega} \xi(\theta^{-1}\omega, (T_{\theta^{-1}\omega})^{-1}y))} \eta_{\xi(\omega,x)}(d y) \\ 
        & = \int_{\xi(\omega,x)} \frac{\tilde{\eta}_{\xi(\omega,x)}(T_{\theta^{-1}\omega} \xi(\theta^{-1}\omega, (T_{\theta^{-1}\omega})^{-1}y))}{\eta_{\xi(\omega,x)}(T_{\theta^{-1}\omega} \xi(\theta^{-1}\omega, (T_{\theta^{-1}\omega})^{-1}y))} \eta_{\xi(\omega,x)}(d y) \\ &= \sum_{\substack{j\ge 1 \\ \eta_{\xi(\omega,x)}(A_{\omega,x}^j)>0 }} \frac{\tilde{\eta}_{\xi(\omega,x)}(A_{\omega,x}^j)}{\eta_{\xi(\omega,x)}(A_{\omega,x}^j)} \eta_{\xi(\omega,x)}(A_{\omega,x}^j) \le 1.
    \end{align}
 
\end{proof}

\begin{proposition}\label{prop:absconstable}
For $\mathbb P$-almost every $\omega\in\Omega$, for every nearby pair
$(\gamma^s(\omega,\tilde x),\gamma^s(\omega,x))$, and for the unstable holonomy
$\mathrm{hol}^u_\omega:\gamma^s(\omega,\tilde x)\to\gamma^s(\omega,x)$
given by Definition~\ref{def:holonomy}, the measure
$\big((\mathrm{hol}^u_\omega)^{-1}\big)_*\mu_{\gamma^s(\omega,x)}$
is absolutely continuous with respect to $\mu_{\gamma^s(\omega,\tilde x)}$.
Moreover, its Radon--Nikodym derivative is given by
\begin{align}
\frac{d\big[\big((\mathrm{hol}^u_\omega)^{-1}\big)_*\mu_{\gamma^s(\omega,x)}\big]}
     {\d \mu_{\gamma^s(\omega,\tilde x)}}(z)
=
\prod_{j=1}^\infty
\frac{e^{\bar\phi\circ F^{-j}(\omega,\mathrm{hol}^u_\omega z)}}
     {e^{\bar\phi\circ F^{-j}(\omega,z)}},
\label{eq:hol}
\end{align}
for $\mu_{\gamma^s(\omega,\tilde x)}$-almost every
$z\in\gamma^s(\omega,\tilde x)$, where $\mu_{\gamma^s(\omega,x)}$ is the
measure constructed in Lemma~\ref{lem:stablemu}.
\end{proposition}

\begin{proof}
Fix $\omega$ in a full $\mathbb P$-measure set on which Proposition~4.11,
Proposition~4.6, Lemma~\ref{lem:stablemu}, and
Lemma~\ref{lem:stablemargulis} all hold. Write
\begin{align*}
\gamma:=\gamma^s(\omega,x),\ \tilde\gamma:=\gamma^s(\omega,\tilde x)\ \text{and}\ h:=\mathrm{hol}^u_\omega:\tilde\gamma\to\gamma.
\end{align*}
For $z\in\operatorname{Int}(\tilde\gamma)$ and $r>0$, let
$B_r^s(z):=\{y\in\tilde\gamma:d^s(z,y)<r\}$. It is enough to prove that, for
$\mu_{\tilde\gamma}$-almost every $z\in\tilde\gamma$,
\begin{align}
\lim_{r\to0}
\frac{\mu_\gamma\big(h(B_r^s(z))\big)}
     {\mu_{\tilde\gamma}\big(B_r^s(z)\big)}
=
\prod_{j=1}^\infty
\frac{e^{\bar\phi\circ F^{-j}(\omega,hz)}}
     {e^{\bar\phi\circ F^{-j}(\omega,z)}}.
\label{eq:diff-quotient-hol}
\end{align}

Assume that $r>0$ is small enough so that
$B_r^s(z)\subset\operatorname{Int}(\tilde\gamma)$. Let $n=n(\omega,r)$ be the
smallest integer such that
\begin{align*}
T_{\theta^{-n}\omega}^{-n}\big(B_r^s(z)\big)
=
\bigcup_{i=1}^{N(\omega,r)}\tilde\gamma_i^{n,r},
\end{align*}
with each $\tilde\gamma_i^{n,r}\in\mathcal F^s_{\theta^{-n}\omega}$. For each
$i=1,\ldots,N(\omega,r)$, let
$h_{-n}^i:\tilde\gamma_i^{n,r}\to\gamma_i^{n,r}$ be the corresponding unstable
holonomy in the fibre over $\theta^{-n}\omega$. Then
\begin{align*}
T_{\theta^{-n}\omega}^{-n}\big(h(B_r^s(z))\big)
=
\bigcup_{i=1}^{N(\omega,r)}\gamma_i^{n,r},
\ \gamma_i^{n,r}:=h_{-n}^i(\tilde\gamma_i^{n,r}).
\end{align*}

Iterating Lemma~\ref{lem:stablemargulis} along each branch gives
\begin{align}
\mu_{\tilde\gamma}\big(B_r^s(z)\big)
=
\frac{1}{\lambda_{\theta^{-n}\omega}^{(n)}}
\sum_{i=1}^{N(\omega,r)}
\int_{\tilde\gamma_i^{n,r}}
e^{S_n\bar\phi(\theta^{-n}\omega,u)}
\mu_{\theta^{-n}\omega}(\d  u),
\label{eq:source-ball}
\end{align}
and
\begin{align}
\mu_\gamma\big(h(B_r^s(z))\big)
=
\frac{1}{\lambda_{\theta^{-n}\omega}^{(n)}}
\sum_{i=1}^{N(\omega,r)}
\int_{\gamma_i^{n,r}}
e^{S_n\bar\phi(\theta^{-n}\omega,v)}
\, \mu_{\theta^{-n}\omega}(\d v).
\label{eq:target-ball}
\end{align}
Set
\begin{align*}
A_i^{n,r}
&:=
\int_{\gamma_i^{n,r}}
e^{S_n\bar\phi(\theta^{-n}\omega,v)}
\,\mu_{\theta^{-n}\omega}(\d v),\ \text{and}\ B_i^{n,r}:=
\int_{\tilde\gamma_i^{n,r}}
e^{S_n\bar\phi(\theta^{-n}\omega,u)}
\,\mu_{\theta^{-n}\omega}(\d u).
\end{align*}

After enlarging $a$ and $a_1$ if necessary, the uniform H\"older regularity of
$\bar\phi$ implies that, for every $n\ge1$,
\begin{align*}
e^{S_n\bar\phi(\theta^{-n}\omega,\cdot)}\big|_{\gamma_i^{n,r}}
\in D(a_1,\kappa,\gamma_i^{n,r}),\ \text{and }e^{S_n\bar\phi(\theta^{-n}\omega,\cdot)}\big|_{\tilde\gamma_i^{n,r}}
\in D(a_1,\kappa,\tilde\gamma_i^{n,r}).
\end{align*}
Define
\begin{align}
\rho_i^{n,r}(v)
&:=
\frac{e^{S_n\bar\phi(\theta^{-n}\omega,v)}}
{\int_{\gamma_i^{n,r}} e^{S_n\bar\phi(\theta^{-n}\omega,w)}},
\ v\in\gamma_i^{n,r},
\label{eq:rhoin}
\end{align}
and let
\begin{align}
\tilde\rho_i^{n,r}(u)
&:=
\frac{
e^{S_n\bar\phi(\theta^{-n}\omega,h_{-n}^i(u))}J_{-n}^i(u)
}
{
\int_{\tilde\gamma_i^{n,r}}
e^{S_n\bar\phi(\theta^{-n}\omega,h_{-n}^i(w))}J_{-n}^i(w)
},
\ u\in\tilde\gamma_i^{n,r},
\label{eq:rhointilde}
\end{align}
be the density on $\tilde\gamma_i^{n,r}$ obtained from $\rho_i^{n,r}$ by the
unstable holonomy $h_{-n}^i$, where $J_{-n}^i$ denotes the leaf Jacobian of
$h_{-n}^i$. Since
$\mu_{\theta^{-n}\omega}\in\mathcal C_{\theta^{-n}\omega}(b,c,\nu)$, condition
\textbf{(C3)} applied to the nearby pair
$(\tilde\gamma_i^{n,r},\gamma_i^{n,r})$ gives
\begin{align}
e^{-c\,d_u(\tilde\gamma_i^{n,r},\gamma_i^{n,r})^\nu}
\le
\frac{
\int_{\tilde\gamma_i^{n,r}}\tilde\rho_i^{n,r}\,\d \mu_{\theta^{-n}\omega}
}
{
\int_{\gamma_i^{n,r}}\rho_i^{n,r}\,\d \mu_{\theta^{-n}\omega}
}
\le
e^{c\,d_u(\tilde\gamma_i^{n,r},\gamma_i^{n,r})^\nu}.
\label{eq:ratio1}
\end{align}

From \eqref{eq:rhoin} and \eqref{eq:rhointilde}, using the change of variables
associated with $h_{-n}^i$ for the normalising integrals, we obtain
\begin{align*}
\frac{
\int_{\tilde\gamma_i^{n,r}}\tilde\rho_i^{n,r}\,\d \mu_{\theta^{-n}\omega}
}
{
\int_{\gamma_i^{n,r}}\rho_i^{n,r}\,\d \mu_{\theta^{-n}\omega}
}
=
\frac{
\int_{\tilde\gamma_i^{n,r}}
e^{S_n\bar\phi(\theta^{-n}\omega,h_{-n}^i(u))}J_{-n}^i(u)
\, \mu_{\theta^{-n}\omega}(\d u)
}
{
\int_{\gamma_i^{n,r}}
e^{S_n\bar\phi(\theta^{-n}\omega,v)}
\,\mu_{\theta^{-n}\omega}(\d v)
}.
\end{align*}
Combining this identity with \eqref{eq:ratio1} yields
\begin{align*}
e^{-c\,d_u(\tilde\gamma_i^{n,r},\gamma_i^{n,r})^\nu}
\inf_{\tilde\gamma_i^{n,r}}J_{-n}^i
&\le
\frac{
\int_{\gamma_i^{n,r}}
e^{S_n\bar\phi(\theta^{-n}\omega,v)}
\,\mu_{\theta^{-n}\omega}(\d v)
}
{
\int_{\tilde\gamma_i^{n,r}}
e^{S_n\bar\phi(\theta^{-n}\omega,h_{-n}^i(u))}
\,\mu_{\theta^{-n}\omega}(\d u)
}\\
&\le
e^{c\,d_u(\tilde\gamma_i^{n,r},\gamma_i^{n,r})^\nu}
\sup_{\tilde\gamma_i^{n,r}}J_{-n}^i,
\end{align*}
hence
\begin{align*}
e^{-c\,d_u(\tilde\gamma_i^{n,r},\gamma_i^{n,r})^\nu}
\inf_{\tilde\gamma_i^{n,r}}J_{-n}^i
&\le
\frac{
\int_{\gamma_i^{n,r}}
e^{S_n\bar\phi(\theta^{-n}\omega,v)}
\, \mu_{\theta^{-n}\omega}(\d v)
}
{
\int_{\tilde\gamma_i^{n,r}}
\frac{e^{S_n\bar\phi(\theta^{-n}\omega,h_{-n}^i(u))}}{e^{S_n\bar\phi(\theta^{-n}\omega,u)}}  e^{S_n\bar\phi(\theta^{-n}\omega,u)}
\,\mu_{\theta^{-n}\omega}(\d u)
}\\
&\le
e^{c\,d_u(\tilde\gamma_i^{n,r},\gamma_i^{n,r})^\nu}
\sup_{\tilde\gamma_i^{n,r}}J_{-n}^i.
\end{align*}
It follows, therefore that
\begin{align}
e^{-c\,d_u(\tilde\gamma_i^{n,r},\gamma_i^{n,r})^\nu}
&\inf_{\tilde\gamma_i^{n,r}}J_{-n}^i
\inf_{u\in\tilde\gamma_i^{n,r}}
\frac{e^{S_n\bar\phi(\theta^{-n}\omega,h_{-n}^i(u))}}
     {e^{S_n\bar\phi(\theta^{-n}\omega,u)}} \le
\frac{A_i^{n,r}}{B_i^{n,r}}\nonumber\\
&\le
e^{c\,d_u(\tilde\gamma_i^{n,r},\gamma_i^{n,r})^\nu}
\sup_{\tilde\gamma_i^{n,r}}J_{-n}^i
\sup_{u\in\tilde\gamma_i^{n,r}}
\frac{e^{S_n\bar\phi(\theta^{-n}\omega,h_{-n}^i(u))}}
     {e^{S_n\bar\phi(\theta^{-n}\omega,u)}}.\label{eq:ab}
\end{align}

Using \eqref{eq:source-ball} and \eqref{eq:target-ball}, we have
\begin{align*}
\frac{\mu_\gamma\big(h(B_r^s(z))\big)}
     {\mu_{\tilde\gamma}\big(B_r^s(z)\big)}
=
\frac{\sum_{i=1}^{N(\omega,r)}A_i^{n,r}}
     {\sum_{i=1}^{N(\omega,r)}B_i^{n,r}}
=
\frac{\sum_{i=1}^{N(\omega,r)}(A_i^{n,r}/B_i^{n,r})B_i^{n,r}}
     {\sum_{i=1}^{N(\omega,r)}B_i^{n,r}}.
\end{align*}
From \eqref{eq:ab} and the above equation we obtain that
\begin{align}
\inf_{1\le i\le N(\omega,r)}
&\left[
e^{-c\,d_u(\tilde\gamma_i^{n,r},\gamma_i^{n,r})^\nu}
\inf_{\tilde\gamma_i^{n,r}}J_{-n}^i
\inf_{u\in\tilde\gamma_i^{n,r}}
\frac{e^{S_n\bar\phi(\theta^{-n}\omega,h_{-n}^i(u))}}
     {e^{S_n\bar\phi(\theta^{-n}\omega,u)}}
\right]
\le
\frac{\mu_\gamma\big(h(B_r^s(z))\big)}
     {\mu_{\tilde\gamma}\big(B_r^s(z)\big)}
\nonumber\\
&\le
\sup_{1\le i\le N(\omega,r)}
\left[
e^{c\,d_u(\tilde\gamma_i^{n,r},\gamma_i^{n,r})^\nu}
\sup_{\tilde\gamma_i^{n,r}}J_{-n}^i
\sup_{u\in\tilde\gamma_i^{n,r}}
\frac{e^{S_n\bar\phi(\theta^{-n}\omega,h_{-n}^i(u))}}
     {e^{S_n\bar\phi(\theta^{-n}\omega,u)}}
\right].
\label{eq:ineqstrong}
\end{align}

As $r\to0$, necessarily $n=n(\omega,r)\to\infty$. Moreover, since each pair
$(\tilde\gamma_i^{n,r},\gamma_i^{n,r})$ is obtained by pulling back the nearby
pair $(B_r^s(z),h(B_r^s(z)))$ under the inverse dynamics, the unstable distance
between the two leaves tends to zero uniformly on the choice of the pair $(\tilde\gamma_i^{n,r},\gamma_i^{n,r})$. In particular 
\begin{align}
\sup_{1\le i\le N(\omega,r)}
d_u(\tilde\gamma_i^{n,r},\gamma_i^{n,r})
\to0.
\label{eq:disteq}
\end{align}
By Proposition~\ref{prop:holjac}, we also have
\begin{align}
\sup_{1\le i\le N(\omega,r)}
\|J_{-n}^i-1\|_\infty
\to0.
\label{eq:jaceq}
\end{align}

Fix $u\in\tilde\gamma_i^{n,r}$ and write
$y=T_{\theta^{-n}\omega}^{n}(u)\in B_r^s(z)$. Since
$T_{\theta^{-n}\omega}^{n}(h_{-n}^i(u))=h(y)$, we have
\begin{align*}
\frac{e^{S_n\bar\phi(\theta^{-n}\omega,h_{-n}^i(u))}}
     {e^{S_n\bar\phi(\theta^{-n}\omega,u)}}
=
\prod_{j=1}^{n}
\frac{e^{\bar\phi\circ F^{-j}(\omega,h(y))}}
     {e^{\bar\phi\circ F^{-j}(\omega,y)}}.
\end{align*}
Since $y\in B_r^s(z)$, we have $y\to z$ uniformly as $r\to0$. Since also
$n(\omega,r)\to\infty$, the summable distortion estimate for $\bar\phi$ along
stable manifold gives
\begin{align}
\sup_{1\le i\le N(\omega,r)}
\sup_{u\in\tilde\gamma_i^{n,r}}
\left|
\frac{e^{S_n\bar\phi(\theta^{-n}\omega,h_{-n}^i(u))}}
     {e^{S_n\bar\phi(\theta^{-n}\omega,u)}}
-
\prod_{j=1}^{\infty}
\frac{e^{\bar\phi\circ F^{-j}(\omega,hz)}}
     {e^{\bar\phi\circ F^{-j}(\omega,z)}}
\right|
\to0.
\label{eq:prod}
\end{align}

Combining \eqref{eq:ineqstrong}, \eqref{eq:disteq}, \eqref{eq:jaceq}, and
\eqref{eq:prod}, we obtain
\begin{align*}
\lim_{r\to0}
\frac{\mu_\gamma\big(h(B_r^s(z))\big)}
     {\mu_{\tilde\gamma}\big(B_r^s(z)\big)}
=
\prod_{j=1}^{\infty}
\frac{e^{\bar\phi\circ F^{-j}(\omega,hz)}}
     {e^{\bar\phi\circ F^{-j}(\omega,z)}}.
\end{align*}
This proves \eqref{eq:diff-quotient-hol}, and therefore gives the claimed
Radon--Nikodym derivative.
\end{proof}

We finally prove the uniqueness of $\mathbb P$-relative equilibrium states.
\begin{theorem} 
Assume that the regular random dynamical system $F:\Omega\times M\to\Omega\times M$ satisfies Hypothesis \ref{hyp:h}. Given $\phi \in L^\infty(\Omega,\mathcal C^\beta(M))$ (see Definition \ref{def:potential}), there exists a unique $\mathbb P$-relative equilibrium state for a potential $\bar{\phi}(\omega,x) = \phi(\omega,x) - \phi^{J^s}(\omega,x)$. Moreover, such  $\mathbb P$-relative equilibrium state is the measure $\upsilon_\phi = \upsilon_\omega(\d x) \mathbb P(\d \omega)$ constructed in Lemma \ref{lem:nuprob}.\label{theorem:unique}
\end{theorem}

\begin{proof}
Existence follows from Proposition \ref{prop:eqstate}, where
$\upsilon_{\phi}(\d\omega,\d x)=\upsilon_\omega(\d x)\mathbb P(\d\omega)$
was shown to be a $\mathbb P$-relative equilibrium state for the potential
$\overline\phi$. By Corollary \ref{cor:ergodic}, the measure $\upsilon_\phi$
is ergodic. Moreover, by Theorem \ref{thm:wgibbs}, $\upsilon_\omega$ is fully
supported for $\mathbb P$-a.e. $\omega\in \Omega$.

Assume, by contradiction, that there exists another ergodic $F$-invariant
measure
$\eta=\eta_\omega(\d x)\mathbb P(\d\omega)$
which is a $\mathbb P$-relative equilibrium state for the same potential
$\overline\phi$.

Let $\{f_j\}_{j\geq 1}$ be a countable dense subset of
$\mathcal C^0(\Omega\times M)$. For every ergodic $F$-invariant probability
measure $\mu$, define
$$
X^\mu=
\left\{
(\omega,x)\in\Omega\times M:
\lim_{n\to \infty}\frac{1}{n}\sum_{i=0}^{n-1} f_j\circ F^{\pm i}(\omega,x)
=
\int_{\Omega\times M} f_j \d \mu
\text{ for every } j\geq 1
\right\}.
$$
By Birkhoff's ergodic theorem, $\mu(X^\mu)=1$. Since
$\{f_j\}_{j\geq 1}$ is dense in $\mathcal C^0(\Omega\times M)$, every point in
$X^\mu$ is generic for $\mu$ with respect to all continuous observables.

Applying this construction to the measures $\upsilon_\phi$ and $\eta$, we obtain
sets $X^{\upsilon_\phi}$ and $X^\eta$ such that $
\upsilon_\phi(X^{\upsilon_\phi})=1= \eta(X^\eta).$
For each $\omega\in\Omega$, set
$$
X^{\upsilon_\phi}_\omega
=
\{x\in M:(\omega,x)\in X^{\upsilon_\phi}\},
\ 
X^\eta_\omega
=
\{x\in M:(\omega,x)\in X^\eta\}.
$$
Using the disintegrations of $\upsilon_\phi$ and $\eta$, we get
$$
1
=
\upsilon_\phi(X^{\upsilon_\phi})
=
\int_\Omega \upsilon_\omega(X^{\upsilon_\phi}_\omega)\,\mathbb P(\d\omega)
\ \text{and}\ 
1
=
\eta(X^\eta)
=
\int_\Omega \eta_\omega(X^\eta_\omega)\,\mathbb P(\d\omega).
$$
Hence, there exists a set of full $\mathbb P$-measure 
$\Omega_{0}$ so that
$$\upsilon_\omega(X^{\upsilon_\phi}_\omega)=1
\ \text{and}\ 
\eta_\omega(X^\eta_\omega)=1\ \text{for each } \omega\in\Omega_0.
$$

Let $\xi^\eta$ be an SLY partition associated with $F$ and the probability
measure $\eta$. We take $\xi^\eta$ with elements of sufficiently small
diameter. For each $\omega\in\Omega$, denote by $\xi^\eta_\omega$ the
partition induced by $\xi^\eta$ on the fibre $\{\omega\}\times M$, namely
$$
\xi^\eta_\omega
=
\{\gamma\subset M:\{\omega\}\times\gamma
\text{ is an atom of } \xi^\eta \text{ contained in } \{\omega\}\times M\}.
$$
Similarly, let $\xi^{\upsilon_\phi}$ be an SLY partition associated with
$F$ and $\upsilon_\phi$, and denote by $\xi^{\upsilon_\phi}_\omega$ the
corresponding fibre partition.

Let us fix $\omega\in \Omega_0$. Since $\eta$ is an equilibrium state for the potential $\overline\phi$,
Proposition \ref{prop:eqdisint} implies that, for every $\omega\in\Omega_0$,
the conditional measure $\eta_\omega$ admits the disintegration
$$
\int_M f(x)\eta_\omega(\d x)
=
\int_{\xi^\eta_\omega}
\int_\gamma f(x)\eta^\eta_\gamma(\d x)\Gamma^\eta_\omega(\d\gamma)
=
\int_{\xi^\eta_\omega}
\int_\gamma f(x)\rho^\eta_\gamma(x)\mu_\gamma(\d x)\Gamma^\eta_\omega(\d\gamma)
$$
for every $f\in L^1(M,\eta_\omega)$, where $\Gamma^\eta_\omega$ is a
probability measure on the quotient partition $\xi^\eta_\omega$, the measure
$\mu_\gamma$ is supported on $\gamma$, and $\rho^\eta_\gamma$ is bounded from
above and below on $\gamma$ by positive constants. Since
$\eta_\omega(X^\eta_\omega)=1$, it follows that
\begin{align}
\eta^\eta_\gamma(X^\eta_\omega\cap\gamma)=1
\text{ for } \Gamma^\eta_\omega\text{-a.e. }\gamma\in\xi^\eta_\omega.
\label{eq:eta-conditional-generic-full}
\end{align}
Equivalently, since $\eta^\eta_\gamma=\rho^\eta_\gamma\mu_\gamma$ and
$\rho^\eta_\gamma$ is bounded from below by a positive constant,
$$
\mu_\gamma(\gamma\setminus X^\eta_\omega)=0
\text{ for } \Gamma^\eta_\omega\text{-a.e. }\gamma\in\xi^\eta_\omega.
$$

Fix $\omega\in\Omega_0$ and choose an element
$\gamma\in\xi^\eta_\omega$ for which \eqref{eq:eta-conditional-generic-full}
holds. Let $z\in\gamma$ and set
$$
R_\omega(\gamma)
=
[W_{\delta/2}(\omega,z),\gamma]_\omega^{\varepsilon}.
$$
Since the SLY partition was chosen with sufficiently small diameter, this
rectangle is well-defined. Moreover, $R_\omega(\gamma)$ has non-empty
interior in the corresponding local product chart. Since $\upsilon_\omega$
is fully supported, we have that $
\upsilon_\omega(R_\omega(\gamma))>0.$
As $\upsilon_\omega(X^{\upsilon_\phi}_\omega)=1$, we also have
$$
\upsilon_\omega(R_\omega(\gamma)\cap X^{\upsilon_\phi}_\omega)>0.
$$

Now disintegrate $\upsilon_\omega$ with respect to the SLY partition
$\xi^{\upsilon_\phi}_\omega$. Hence there exists
$\widetilde\gamma\in\xi^{\upsilon_\phi}_\omega$ such that
\begin{align}
\upsilon_{\widetilde\gamma}
(R_\omega(\gamma)\cap\widetilde\gamma\cap X^{\upsilon_\phi}_\omega)
>0.
\label{eq:positive-conditional-upsilon}
\end{align}
Using again Proposition \ref{prop:eqdisint}, we may write
$\upsilon^{\upsilon_\phi}_{\widetilde\gamma}
=
\rho^{\upsilon_\phi}_{\widetilde\gamma}\mu_{\widetilde\gamma}$, where
$\rho^{\upsilon_\phi}_{\widetilde\gamma}$ is bounded from above and below by
positive constants on $\widetilde\gamma$. From \eqref{eq:positive-conditional-upsilon} we obtain
\begin{align}
\mu_{\widetilde\gamma}
(R_\omega(\gamma)\cap\widetilde\gamma\cap X^{\upsilon_\phi}_\omega)
>0.
\label{eq:positive-mu-tildegamma}
\end{align}

By construction, there is an unstable holonomy map
$$
\mathrm{hol}^u_\omega:
R(\omega,\gamma)\cap\widetilde\gamma
\longrightarrow
\mathrm{hol}^u_\omega(R(\omega,\gamma)\cap\widetilde\gamma)\subset\gamma.
$$
By Proposition \ref{prop:absconstable}, applied to the local holonomy
$\mathrm{hol}^u_\omega:R(\omega,\gamma)\cap\widetilde\gamma\to\gamma$, the
measures $\mu_{\widetilde\gamma}$ and
$\big((\mathrm{hol}^u_\omega)^{-1}\big)_*\mu_\gamma$ are equivalent on
$R(\omega,\gamma)\cap\widetilde\gamma$. Hence \eqref{eq:positive-mu-tildegamma}
implies
$$
\mu_\gamma\left(
\mathrm{hol}^u_\omega
\big(
R(\omega,\gamma)\cap\widetilde\gamma\cap X^{\upsilon_\phi}_\omega
\big)
\right)>0.
$$
On the other hand, since
$\mu_\gamma(\gamma\setminus X^\eta_\omega)=0$, it follows that
$$
\mu_\gamma\left(
\mathrm{hol}^u_\omega
\big(
R(\omega,\gamma)\cap\widetilde\gamma\cap X^{\upsilon_\phi}_\omega
\big)
\cap X^\eta_\omega
\right)>0.
$$
Hence there exists 
\begin{align}
    x\in R(\omega,\gamma)\cap\widetilde\gamma\cap X^{\upsilon_\phi}_\omega \ \text{such that }
y:=\mathrm{hol}^u_\omega(x)\in X^\eta_\omega\cap\gamma.\label{eq:xy}
\end{align}

We can now finish the proof. Let $f:\Omega\times M\to\mathbb R$ be an arbitrary continuous function. From \eqref{eq:xy}, $x\in X^{\upsilon_\phi}$ and $y\in X^\eta$ belong to
the same local unstable set, their backward iterates converge to each other, i.e.
$
d(F^{-i}(\omega,x),F^{-i}(\omega,y))\to 0
\text{ as } i\to\infty.$ Hence
\begin{align*}
\int_{\Omega\times M} f \d \upsilon_\phi =
\lim_{n\to\infty}
\frac1n\sum_{i=0}^{n-1}f\circ F^{-i}(\omega,x) =
\lim_{n\to\infty}
\frac1n\sum_{i=0}^{n-1}f\circ F^{-i}(\omega,y) =
\int_{\Omega\times M} f \d \eta .
\end{align*}
Since the above holds for every continuous function $f:\Omega\times M\to\mathbb R$,
we conclude that $\upsilon_\phi=\eta$. This proves the uniqueness of the
$\mathbb P$-relative equilibrium state.

\end{proof}

\section{Proof of the main theorems}
\label{sec:maintheorems}
Below, we prove the main results of the paper. They are a consequence of Theorem \ref{theorem:unique} and Proposition \ref{prop:exponentialdecayofcorrelation}.

\begin{proof}[Proof of Theorem \ref{theorem:mainthmA}]
Let $\phi:\Omega\times M\to \mathbb R$ be such that
$\phi\in L^\infty(\Omega;\mathcal C^\beta(M))$, and set
$\widetilde\phi=\phi+\phi^{J^s}\in L^\infty(\Omega;\mathcal C^\beta(M))$.
By Theorem \ref{theorem:unique}, the measure
$\upsilon:=\upsilon_{\widetilde\phi}$ is the unique $\mathbb P$-relative
equilibrium state for the potential $\phi$.
\end{proof}

\begin{proof}[Proof of Theorem \ref{theorem:mainthmB}]
By Theorem \ref{theorem:mainthmA}, the measure
$\upsilon:=\upsilon_{\widetilde\phi}$ is the unique $\mathbb P$-relative
equilibrium state for $\phi$. Applying Proposition \ref{prop:exponentialdecayofcorrelation} to
$\widetilde\phi=\phi+\phi^{J^s} $ gives the first two quenched exponential decay estimates stated in
Theorem \ref{theorem:mainthmB} under Hypothesis \ref{hyp:h}. If, in addition,
Hypothesis \ref{hyp:h'} holds, the same proposition gives the corresponding
estimates with constants $C_p\in L^p(\Omega,\mathbb P)$ and rates
$\Lambda_p\in(0,1)$, for every $p\in[1,\infty)$.
\end{proof}

\section*{Acknowledgements}

LA was supported by Fundação para a Ciência e a Tecnologia grant PD/BD/150458/2019 and also by the Faculty of Sciences of the University of Porto, the Center of Mathematics of the University of Porto, the Université de Toulon, France, and the Centre de Physique Théorique in Luminy, Marseille, France. MMC is supported by the São Paulo Research Foundation (FAPESP, grant no. 2025/26997-9). MMC thanks the Institut de Mathématiques de Marseille (I2M) for its hospitality and financial support during his visit to the institute.  SV thanks the France-Australia Mathematical and Interactions ANU-CNRS International Research Laboratory at the Australian National University in Canberra  where part of this work was carried out.  SV thanks the Great Bay University in Dongguan (China) for the visiting professorship during which this work was completed. The authors thank the Centre de Physique Théorique in Marseille, where part of this work was done.

\addtocontents{toc}{\protect\setcounter{tocdepth}{2}}
\printbibliography

\newpage

\markboth{LUCAS AMORIM, MATHEUS M. CASTRO, SANDRO VAIENTI, AND BENO\^IT SAUSSOL}
{THERMODYNAMIC FORMALISM FOR HYPERBOLIC RANDOM DYNAMICAL SYSTEMS}
\begin{appendix}

\section{Proof of Proposition \ref{prop:cone_invariance_random}}
\label{appendix:A}
In this section, we prove Proposition \ref{prop:cone_invariance_random}. The proof follows the ideas of \cite[Lemma 4.5]{liu2024exponential}, adapted to the weighted random transfer operator $\mathcal L_\omega$.

\begin{proof}[Proof of Proposition \ref{prop:cone_invariance_random}]
    
We will construct constants $a,a_1,b,c>0$, exponents $\kappa,\kappa_1,\nu\in(0,1]$, and a constant $\lambda_2\in(0,1)$ such that, for $\mathbb P$-almost every $\omega\in\Omega$ and every $f\in \mathcal C_\omega(b,c,\nu)$,
$
\mathcal L_\omega f\in\mathcal C_{\theta\omega}(\lambda_2 b,\lambda_2 c,\nu).
$
The proof is divided into verifying the three defining conditions of the target cone. Along the way, the parameters above will be chosen so that $\mathcal L_\omega f$ satisfies \textbf{(C1)}, \textbf{(C2)} and \textbf{(C3)} for the cone $\mathcal C_{\theta\omega}(\lambda_2 b,\lambda_2 c,\nu)$.

It is clear that if $f$ satisfies condition \textbf{(C1)} of $\mathcal C_\omega(b,c,\nu)$, then, by Lemma \ref{lem:push_density}, $\mathcal L_\omega f$ also satisfies \textbf{(C1)} of $\mathcal C_{\theta\omega}(\lambda_2 b,\lambda_2 c,\nu)$ for $a>0$ large enough.

We now check \textbf{(C2)} of $\mathcal C_{\theta\omega}(\lambda_2 b,\lambda_2 c,\nu)$. Fix $\gamma_{\theta\omega}\in\mathscr F_{\theta\omega}^s$ and
$\rho_{\theta\omega},\varsigma_{\theta\omega}\in D_1(a,\kappa,\gamma_{\theta\omega})$.
By Lemma \ref{lem:push_density}, there exist positive densities
$\hat\rho_\omega^{(i)}$ and $\hat\varsigma_\omega^{(i)}$ such that, if
$$
\rho_\omega^{(i)}
:=
\frac{\hat\rho_\omega^{(i)}}
{\int_{\gamma_\omega^{(i)}}\hat\rho_\omega^{(i)}},
\ 
\varsigma_\omega^{(i)}
:=
\frac{\hat\varsigma_\omega^{(i)}}
{\int_{\gamma_\omega^{(i)}}\hat\varsigma_\omega^{(i)}},
$$
then
$$
\int_{\gamma_{\theta\omega}} \mathcal L_\omega f\,\rho_{\theta\omega}
=
\sum_{i=1}^{Q_\omega(\gamma_{\theta\omega})}
\left(
\int_{\gamma_\omega^{(i)}}\hat\rho_\omega^{(i)}
\right)
\int_{\gamma_\omega^{(i)}} f\,\rho_\omega^{(i)}
$$
and
$$
\int_{\gamma_{\theta\omega}} \mathcal L_\omega f\,\varsigma_{\theta\omega}
=
\sum_{i=1}^{Q_\omega(\gamma_{\theta\omega})}
\left(
\int_{\gamma_\omega^{(i)}}\hat\varsigma_\omega^{(i)}
\right)
\int_{\gamma_\omega^{(i)}} f\,\varsigma_\omega^{(i)}.
$$
From \textbf{(C2)} it follows that
$$
\int_{\gamma_\omega^{(i)}} f\,\varsigma_\omega^{(i)}
\le
\exp\left(
b\,\Theta^{a,\kappa}_{\gamma_\omega^{(i)}}
\bigl(\rho_\omega^{(i)},\varsigma_\omega^{(i)}\bigr)
\right)
\int_{\gamma_\omega^{(i)}} f\,\rho_\omega^{(i)}.
$$
By Lemma \ref{lem:densityimprovment}, taking $a>0$ large enough, there exists
$\Lambda_1\in(0,1)$ such that
\begin{align}
\Theta^{a,\kappa}_{\gamma_\omega^{(i)}}
\bigl(\rho_\omega^{(i)},\varsigma_\omega^{(i)}\bigr)
\le
\Lambda_1\,
\Theta^{a,\kappa}_{\gamma_{\theta\omega}}
\bigl(\rho_{\theta\omega},\varsigma_{\theta\omega}\bigr).
\label{eq:ineqc222}
\end{align}
Moreover,
\begin{align}
\frac{\hat\varsigma_\omega^{(i)}(x)}
{\hat\rho_\omega^{(i)}(x)}
=
\frac{\varsigma_{\theta\omega}(T_\omega x)}
{\rho_{\theta\omega}(T_\omega x)}
\le
e^{\Theta^+_{\gamma_{\theta\omega}}
(\rho_{\theta\omega},\varsigma_{\theta\omega})}
\le
e^{\Theta^{a,\kappa}_{\gamma_{\theta\omega}}
(\rho_{\theta\omega},\varsigma_{\theta\omega})}.
\label{eq:ineqc22}
\end{align}
Integrating \eqref{eq:ineqc22} on $\gamma_\omega^{(i)}$ gives $
\int_{\gamma_\omega^{(i)}}\hat\varsigma_\omega^{(i)}
\le
e^{\Theta^{a,\kappa}_{\gamma_{\theta\omega}}
(\rho_{\theta\omega},\varsigma_{\theta\omega})}
\int_{\gamma_\omega^{(i)}}\hat\rho_\omega^{(i)}.$
Therefore,
\begin{align}
\int_{\gamma_{\theta\omega}}
\mathcal L_\omega f\,\varsigma_{\theta\omega}
&\le
e^{b\Lambda_1\,
\Theta^{a,\kappa}_{\gamma_{\theta\omega}}
(\rho_{\theta\omega},\varsigma_{\theta\omega})}
\sum_{i=1}^{Q_\omega(\gamma_{\theta\omega})}
\left(
\int_{\gamma_\omega^{(i)}}\hat\varsigma_\omega^{(i)}
\right)
\int_{\gamma_\omega^{(i)}} f\,\rho_\omega^{(i)}
\notag\\
&\le
e^{(b\Lambda_1+1)\,
\Theta^{a,\kappa}_{\gamma_{\theta\omega}}
(\rho_{\theta\omega},\varsigma_{\theta\omega})}
\int_{\gamma_{\theta\omega}}
\mathcal L_\omega f\,\rho_{\theta\omega}.
\label{eq:ineqc2}
\end{align}
Choose
$$
\lambda_2\ge\frac{\Lambda_1+1}{2}
\ \text{and}\ 
b\ge\frac{1}{\lambda_2-\Lambda_1}.
$$
Then $b\Lambda_1+1\le\lambda_2 b$, and hence
$$
\int_{\gamma_{\theta\omega}}
\mathcal L_\omega f\,\varsigma_{\theta\omega}
\le
e^{\lambda_2 b\,
\Theta^{a,\kappa}_{\gamma_{\theta\omega}}
(\rho_{\theta\omega},\varsigma_{\theta\omega})}
\int_{\gamma_{\theta\omega}}
\mathcal L_\omega f\,\rho_{\theta\omega}.
$$

We finally show \textbf{(C3)}. We choose
$$
a_1:=\alpha_0 a>\frac a2>
2\frac{\|\phi\|_{\mathcal C^\kappa}+C_2}{1-e^{-\lambda_0}}>0.
$$
Moreover, we choose $\kappa_1\in(0,1)$ sufficiently close to $1$ so that
$\kappa+\nu\le\kappa_1\nu_0$.
Let $(\tilde\gamma_{\theta\omega},\gamma_{\theta\omega})$ be a nearby pair and let
$\tilde\rho_{\theta\omega}$ be the density on $\tilde\gamma_{\theta\omega}$ induced from
$\rho_{\theta\omega}\in D_1(a_1,\kappa_1,\gamma_{\theta\omega})$
by unstable holonomy.

By Lemma \ref{lem:push_density}, there are decompositions
$$
T_\omega^{-1}(\gamma_{\theta\omega})
=
\bigsqcup_i\gamma_\omega^{(i)},
\ 
T_\omega^{-1}(\tilde\gamma_{\theta\omega})
=
\bigsqcup_i\tilde\gamma_\omega^{(i)},
$$
and positive unnormalised pullback densities
$\hat\rho_\omega^{(i)}$ on $\gamma_\omega^{(i)}$ and
$\hat{\tilde\rho}_\omega^{(i)}$ on $\tilde\gamma_\omega^{(i)}$.
Set
$$
\rho_\omega^{(i)}
:=
\frac{\hat\rho_\omega^{(i)}}
{\int_{\gamma_\omega^{(i)}}\hat\rho_\omega^{(i)}},
\ 
\tilde\rho_\omega^{(i)}
:=
\frac{\hat{\tilde\rho}_\omega^{(i)}}
{\int_{\tilde\gamma_\omega^{(i)}}\hat{\tilde\rho}_\omega^{(i)}}.
$$
Then
$$
\int_{\gamma_{\theta\omega}}
\mathcal L_\omega f\,\rho_{\theta\omega}
=
\sum_i
\left(
\int_{\gamma_\omega^{(i)}}\hat\rho_\omega^{(i)}
\right)
\int_{\gamma_\omega^{(i)}} f\,\rho_\omega^{(i)}
$$
and
$$
\int_{\tilde\gamma_{\theta\omega}}
\mathcal L_\omega f\,\tilde\rho_{\theta\omega}
=
\sum_i
\left(
\int_{\tilde\gamma_\omega^{(i)}}\hat{\tilde\rho}_\omega^{(i)}
\right)
\int_{\tilde\gamma_\omega^{(i)}} f\,\tilde\rho_\omega^{(i)}.
$$
Let $\bar\rho_\omega^{(i)}$ be the density on
$\tilde\gamma_\omega^{(i)}$ obtained by transporting
$\rho_\omega^{(i)}$ by unstable holonomy from
$\gamma_\omega^{(i)}$ to $\tilde\gamma_\omega^{(i)}$, and let
$\hat{\bar\rho}_\omega^{(i)}$ be the corresponding transport of
$\hat\rho_\omega^{(i)}$.
Since $
\hat{\bar\rho}_\omega^{(i)}
=
\left(
\int_{\gamma_\omega^{(i)}}\hat\rho_\omega^{(i)}
\right)
\bar\rho_\omega^{(i)},$
condition \textbf{(C3)} applied to the nearby pair
$(\tilde\gamma_\omega^{(i)},\gamma_\omega^{(i)})$ gives
\begin{align}
\left|
\log\int_{\tilde\gamma_\omega^{(i)}} f\,\hat{\bar\rho}_\omega^{(i)}
-
\log\int_{\gamma_\omega^{(i)}} f\,\hat\rho_\omega^{(i)}
\right|
\le
c\,d_u(\tilde\gamma_\omega^{(i)},\gamma_\omega^{(i)})^\nu.
\label{eq:31}
\end{align}
The nearby-leaf distance contracts under $T_\omega^{-1}$ in the unstable direction by at least
$e^{-\lambda_0}$, so $d_u(\tilde\gamma_\omega^{(i)},\gamma_\omega^{(i)})
\le
e^{-\lambda_0}
d_u(\tilde\gamma_{\theta\omega},\gamma_{\theta\omega}).$
Combining this with \eqref{eq:31}, we obtain
\begin{align}
\left|
\log\int_{\tilde\gamma_\omega^{(i)}} f\,\hat{\bar\rho}_\omega^{(i)}
-
\log\int_{\gamma_\omega^{(i)}} f\,\hat\rho_\omega^{(i)}
\right|
\le
c e^{-\nu\lambda_0}
d_u(\tilde\gamma_{\theta\omega},\gamma_{\theta\omega})^\nu.
\label{eq:32}
\end{align}

To finish the proof, we use the following claim.
\begin{claim}\label{claim65}
There exists $K_0>0$, independent of $c$, such that, for each $i$,
\begin{align}
\left|
\log\int_{\tilde\gamma_\omega^{(i)}} f\,\hat{\tilde\rho}_\omega^{(i)}
-
\log\int_{\tilde\gamma_\omega^{(i)}} f\,\hat{\bar\rho}_\omega^{(i)}
\right|
\le
K_0\,d_u(\tilde\gamma_{\theta\omega},\gamma_{\theta\omega})^\nu.
\label{eq:33}
\end{align}
\end{claim}

Assuming the claim, combining \eqref{eq:32} and \eqref{eq:33} by the triangle inequality gives
$$
\left|
\log\int_{\tilde\gamma_\omega^{(i)}} f\,\hat{\tilde\rho}_\omega^{(i)}
-
\log\int_{\gamma_\omega^{(i)}} f\,\hat\rho_\omega^{(i)}
\right|
\le
(c e^{-\nu\lambda_0}+K_0)
d_u(\tilde\gamma_{\theta\omega},\gamma_{\theta\omega})^\nu.
$$
Summing over $i$ and using the same log-sum comparison as in \eqref{eq:ineqc2}, we obtain
$$
\left|
\log\int_{\tilde\gamma_{\theta\omega}}
\mathcal L_\omega f\,\tilde\rho_{\theta\omega}
-
\log\int_{\gamma_{\theta\omega}}
\mathcal L_\omega f\,\rho_{\theta\omega}
\right|
\le
(c e^{-\lambda_0\nu}+K_0)
d_u(\tilde\gamma_{\theta\omega},\gamma_{\theta\omega})^\nu.
$$
Enlarge $\lambda_2$, if necessary, so that $\lambda_2\in
(\max\{\Lambda_1,e^{-\lambda_0\nu}\},1),$
and assume $c\ge
\frac{K_0}{\lambda_2-e^{-\lambda_0\nu}}.$
Then $
c e^{-\lambda_0\nu}+K_0\le\lambda_2 c,$
so \textbf{(C3)} holds for $\mathcal L_\omega f$ with respect to the cone
$\mathcal C_{\theta\omega}(\lambda_2 b,\lambda_2 c,\nu)$.
All three defining conditions hold, hence $
\mathcal L_\omega f\in
\mathcal C_{\theta\omega}(\lambda_2 b,\lambda_2 c,\nu).
$
To conclude the proof of the proposition, it remains to prove Claim \ref{claim65}.
 
\end{proof}

\begin{proof}[Proof of Claim \ref{claim65}]
We follow \cite[Sublemma 4.1]{liu2006smooth}.
Fix $j$. We write
$$
\hat\rho'_\omega
:=
\hat{\tilde\rho}_\omega^{(j)},
\ 
\hat\rho''_\omega
:=
\hat{\bar\rho}_\omega^{(j)},\ \text{and let }
\rho'_\omega
:=
\frac{\hat\rho'_\omega}
{\int_{\tilde\gamma_\omega^{(j)}}\hat\rho'_\omega},
\ 
\rho''_\omega
:=
\frac{\hat\rho''_\omega}
{\int_{\tilde\gamma_\omega^{(j)}}\hat\rho''_\omega}.
$$
By Lemmas \ref{lem:push_density} and \ref{lem:densityimprovment}, together with the choices of
$a_1$ and $\kappa_1$,
$$
\rho'_\omega,\rho''_\omega
\in
D(a_1,\kappa_1\nu_0,\tilde\gamma_\omega^{(j)})
\subset
D(a_1,\kappa,\tilde\gamma_\omega^{(j)}).
$$
By condition \textbf{(C2)} in the definition of
$\mathcal C_\omega(b,c,\nu)$,
\begin{align}
\left|
\log\int_{\tilde\gamma_\omega^{(j)}} f\,\hat\rho'_\omega
-
\log\int_{\tilde\gamma_\omega^{(j)}} f\,\hat\rho''_\omega
\right| \leq
\left|
\log\int_{\tilde\gamma_\omega^{(j)}}\hat\rho'_\omega
-
\log\int_{\tilde\gamma_\omega^{(j)}}\hat\rho''_\omega
\right|
+
b\,\Theta^{a,\kappa}_{\tilde\gamma_\omega^{(j)}}
(\rho'_\omega,\rho''_\omega).
\label{4.37}
\end{align}
We estimate the two terms on the right-hand side.

Let $\mathrm{hol}^u_{\theta\omega}:\gamma_{\theta\omega}\to\tilde\gamma_{\theta\omega}$ be unstable holonomy and define
the transported density $\tilde\rho_{\theta\omega}$ on $\tilde\gamma_{\theta\omega}$ by $
\tilde\rho_{\theta\omega}(y)
:=
\rho_{\theta\omega}
\bigl((\mathrm{hol}^u_{\theta\omega})^{-1}(y)\bigr)
\,
\mathrm{Jac}
\bigl((\mathrm{hol}^u_{\theta\omega})^{-1}\bigr)(y).$
Let
$\gamma_\omega^{(j)}\subset T_\omega^{-1}(\gamma_{\theta\omega})$
and
$\tilde\gamma_\omega^{(j)}\subset T_\omega^{-1}(\tilde\gamma_{\theta\omega})$
be the corresponding pieces, and let
$\psi_\omega^j:\tilde\gamma_\omega^{(j)}\to\gamma_\omega^{(j)}$
be the unstable holonomy between them.

By Lemma \ref{lem:push_density}, for
$x\in\tilde\gamma_\omega^{(j)}$,
\begin{align}
\hat\rho'_\omega(x)
&=
e^{\phi_\omega(x)-\phi^{J^s}(\omega,x)}
\tilde\rho_{\theta\omega}(T_\omega x),
\label{4.38}\\
\hat\rho''_\omega(x)
&=
e^{\phi_\omega(\psi_\omega^j(x))-\phi^{J^s}(\omega,\psi_\omega^j(x))}
\rho_{\theta\omega}(T_\omega\psi_\omega^j(x))
\,
\mathrm{Jac}(\psi_\omega^j)(x).
\label{4.39}
\end{align}
By the definition of the holonomy maps,
\begin{equation}
(\mathrm{hol}^u_{\theta\omega})^{-1}(T_\omega x)
=
T_\omega\circ\psi_\omega^j(x)
\ \text{for all }
x\in\tilde\gamma_\omega^{(j)}.
\end{equation}
Hence
\begin{equation}
\tilde\rho_{\theta\omega}(T_\omega x)
=
\rho_{\theta\omega}(T_\omega\psi_\omega^j(x))
\,
\mathrm{Jac}
\bigl((\mathrm{hol}^u_{\theta\omega})^{-1}\bigr)(T_\omega x).
\label{4.41}
\end{equation}
Combining \eqref{4.38}, \eqref{4.39}, and \eqref{4.41}, we obtain
\begin{equation}
\frac{\hat\rho'_\omega(x)}
{\hat\rho''_\omega(x)}
=
e^{\phi_\omega(x)-\phi_\omega(\psi_\omega^j(x))}
e^{-\phi^{J^s}(\omega,x)+\phi^{J^s}(\omega,\psi_\omega^j(x))}
\frac{
\mathrm{Jac}
\bigl((\mathrm{hol}^u_{\theta\omega})^{-1}\bigr)(T_\omega x)
}{
\mathrm{Jac}(\psi_\omega^j)(x)
}.
\label{4.42}
\end{equation}

From Proposition \ref{prop:holderhol}, for
$x\in\tilde\gamma_\omega^{(j)}$,
\begin{align}
\left|
\log\left|
\mathrm{Jac}
\bigl((\mathrm{hol}^u_{\theta\omega})^{-1}\bigr)(T_\omega x)
\right|
-
\log\left|
\mathrm{Jac}(\psi_\omega^j)(x)
\right|
\right|
\le&
a_0'\,
d\left(
T_\omega x,
(\mathrm{hol}^u_{\theta\omega})^{-1}(T_\omega x)
\right)^{\nu_0}\notag
\\
&+
a_0'\,
d\left(
x,\psi_\omega^j(x)
\right)^{\nu_0}
\notag\\
\le&
a_0'(1+e^{-\lambda\nu_0})
d_u(\gamma_{\theta\omega},\tilde\gamma_{\theta\omega})^{\nu_0}.
\label{4.43}
\end{align}
Since $\phi_\omega$ is $\nu_0$--H\"older and
$\phi^{J^s}(\omega,\cdot)$ is $\nu_0$--H\"older, with uniform constants, we also have
\begin{align}
\left|
\phi_\omega(x)-\phi_\omega(\psi_\omega^j(x))
\right|
&\le
C_{10}e^{-\lambda\nu_0}
d_u(\gamma_{\theta\omega},\tilde\gamma_{\theta\omega})^{\nu_0},
\label{4.44}\\
\left|
\phi^{J^s}(\omega,x)-\phi^{J^s}(\omega,\psi_\omega^j(x))
\right|
&\le
C_{11}e^{-\lambda\nu_0}
d_u(\gamma_{\theta\omega},\tilde\gamma_{\theta\omega})^{\nu_0}.
\label{4.45}
\end{align}
Thus \eqref{4.42}--\eqref{4.45} imply that
\begin{equation}
e^{-K_3 d_u(\gamma_{\theta\omega},\tilde\gamma_{\theta\omega})^{\nu_0}}
\le
\frac{\hat\rho'_\omega(x)}
{\hat\rho''_\omega(x)}
\le
e^{K_3 d_u(\gamma_{\theta\omega},\tilde\gamma_{\theta\omega})^{\nu_0}},
\label{4.46}
\end{equation}
where $
K_3
:=
a_0'(1+e^{-\lambda\nu_0})
+
C_{10}e^{-\lambda\nu_0}
+
C_{11}e^{-\lambda\nu_0}.$ Integrating \eqref{4.46} gives
\begin{equation}
\left|
\log\int_{\tilde\gamma_\omega^{(j)}}\hat\rho'_\omega
-
\log\int_{\tilde\gamma_\omega^{(j)}}\hat\rho''_\omega
\right|
\le
K_3\,
d_u(\gamma_{\theta\omega},\tilde\gamma_{\theta\omega})^{\nu_0}.
\label{4.48}
\end{equation}
Moreover, after normalisation,
\begin{equation}
e^{-2K_3 d_u(\gamma_{\theta\omega},\tilde\gamma_{\theta\omega})^{\nu_0}}
\le
\frac{\rho'_\omega(x)}
{\rho''_\omega(x)}
\le
e^{2K_3 d_u(\gamma_{\theta\omega},\tilde\gamma_{\theta\omega})^{\nu_0}},
\end{equation}
and hence
\begin{equation}
d^+_{\tilde\gamma_\omega^{(j)}}
(\rho'_\omega,\rho''_\omega)
\le
4K_3\,
d_u(\gamma_{\theta\omega},\tilde\gamma_{\theta\omega})^{\nu_0}.
\label{4.49}
\end{equation}
We next estimate
\begin{equation}
\Theta^{a,\kappa}_{\tilde\gamma_\omega^{(j)}}
(\rho'_\omega,\rho''_\omega)
\le
d^+_{\tilde\gamma_\omega^{(j)}}
(\rho'_\omega,\rho''_\omega)
+
\log\left(
\hat\tau_2(\omega)/\hat\tau_1(\omega)
\right),
\label{4.50}
\end{equation}
where
$$
\hat\tau_1(\omega)
=
\inf_{x\ne y\in\tilde\gamma_\omega^{(j)}}
\left\{
1,\,
\frac{
\exp(a\,d(x,y)^\kappa)-\rho''_\omega(y)/\rho''_\omega(x)
}{
\exp(a\,d(x,y)^\kappa)-\rho'_\omega(y)/\rho'_\omega(x)
}
\right\},
$$
and
$$
\hat\tau_2(\omega)
=
\sup_{x\ne y\in\tilde\gamma_\omega^{(j)}}
\left\{
1,\,
\frac{
\exp(a\,d(x,y)^\kappa)-\rho''_\omega(y)/\rho''_\omega(x)
}{
\exp(a\,d(x,y)^\kappa)-\rho'_\omega(y)/\rho'_\omega(x)
}
\right\}.
$$
Define
\begin{align}
B_1(x,y,\omega)
:=
\frac{\rho'_\omega(y)}
{\rho'_\omega(x)}
e^{-a\,d(x,y)^\kappa},\ 
B_2(x,y,\omega)
:=
\frac{\rho''_\omega(y)}
{\rho''_\omega(x)}
e^{-a\,d(x,y)^\kappa}.
\end{align}
Since
$\rho'_\omega,\rho''_\omega\in
D(a_1,\kappa,\tilde\gamma_\omega^{(j)})$,
\begin{equation}
\log B_1(x,y,\omega)
\le
-(a-a_1)d(x,y)^\kappa<0,
\ 
\log B_2(x,y,\omega)
\le
-(a-a_1)d(x,y)^\kappa<0.
\end{equation}
Hence
$$
\max\left\{B_1(x,y,\omega),B_2(x,y,\omega)\right\}
\le
e^{-(a-a_1)d(x,y)^\kappa}<1.
$$
On the one hand,
\begin{align}
|B_1(x,y,\omega)-B_2(x,y,\omega)|
&\le
\max\{B_1(x,y,\omega),B_2(x,y,\omega)\}
\left|
\log B_1(x,y,\omega)-\log B_2(x,y,\omega)
\right|
\notag\\
&\le
4K_3\,
d_u(\gamma_{\theta\omega},\tilde\gamma_{\theta\omega})^{\nu_0}\le
4K_3\,
d_u(\gamma_{\theta\omega},\tilde\gamma_{\theta\omega})^{\kappa+\nu}.
\label{4.54}
\end{align}
On the other hand, since
$\rho'_\omega,\rho''_\omega\in
D(a_1,\kappa_1\nu_0,\tilde\gamma_\omega^{(j)})$,
\begin{align}
|B_1(x,y,\omega)-B_2(x,y,\omega)|
&\le
2a_1\,d(x,y)^{\kappa_1\nu_0}\le
2a_1\,d(x,y)^{\kappa+\nu}.
\label{4.55}
\end{align}
Therefore,
\begin{align}
|B_1(x,y,\omega)-B_2(x,y,\omega)|
&\le
\min\left\{
4K_3\,
d_u(\gamma_{\theta\omega},\tilde\gamma_{\theta\omega})^{\kappa+\nu},
\,
2a_1\,d(x,y)^{\kappa+\nu}
\right\}
\notag\\
&\le
\max\{4K_3,2a_1\}\,
d_u(\gamma_{\theta\omega},\tilde\gamma_{\theta\omega})^\nu\,
d(x,y)^\kappa
\notag\\
&=:
K_4\,
d_u(\gamma_{\theta\omega},\tilde\gamma_{\theta\omega})^\nu\,
d(x,y)^\kappa.
\label{4.56}
\end{align}
It follows that
\begin{align}
\left|
\log\frac{1-B_2(x,y,\omega)}
{1-B_1(x,y,\omega)}
\right|
&\le
\frac{
|B_1(x,y,\omega)-B_2(x,y,\omega)|
}{
1-\max\{B_1(x,y,\omega),B_2(x,y,\omega)\}
}
\notag\\
&\le
\frac{
K_4\,
d_u(\gamma_{\theta\omega},\tilde\gamma_{\theta\omega})^\nu\,
d(x,y)^\kappa
}{
1-\exp\left(-(a-a_1)d(x,y)^\kappa\right)
}
\notag\\
&\le
K_5\,
d_u(\gamma_{\theta\omega},\tilde\gamma_{\theta\omega})^\nu,
\label{4.57}
\end{align}
where $K_5
:=
K_4
\sup_{z\in(0,1)}
\frac{z^\kappa}
{1-\exp\left(-(a-a_1)z^\kappa\right)}
<\infty.$
Hence
\begin{equation}
\left|
\log\left(
\hat\tau_2(\omega)/\hat\tau_1(\omega)
\right)
\right|
\le
2K_5\,
d_u(\gamma_{\theta\omega},\tilde\gamma_{\theta\omega})^\nu.
\label{4.58}
\end{equation}
Combining
\eqref{4.37},
\eqref{4.48},
\eqref{4.49},
\eqref{4.50},
and
\eqref{4.58},
we obtain the claim with
$$
K_0
:=
K_3+b(4K_3+2K_5).
$$
\end{proof}

\section{Proof of Lemma \ref{lem:finitediam}}
\label{appendix:B}
In this section, we prove Lemma \ref{lem:finitediam}. The argument follows the strategy of \cite[Lemma 4.6]{liu2024exponential}, adapted to the weighted random transfer operator $\mathcal L_\omega$.

\begin{proof}[Proof of Lemma \ref{lem:finitediam}]

Fix $\omega\in\Omega$ and set
$$
E_\phi:=e^{\|\phi\|_{L^\infty(\Omega\times M)}}
\ \text{and}\ 
\underline m_s:=
\inf_{(\omega,x)\in\Omega\times M}
\mathrm m\left(D_xT_\omega|_{E^s(\omega,x)}\right)>0,
$$
where $\mathrm m(A)$ denotes the co-norm of a linear map $A$, namely $
\mathrm m(A):=\inf_{\|v\|=1}\|Av\|.$

Let $\e\in(0,\min\{\e_0,\delta_0\}]$ be as fixed in Section~\ref{sec:geom}, let
$\delta_{\mathrm{loc}}:=\delta(\e)$
be the local product-structure scale from Proposition~\ref{prop:productstructure}, and let $\delta$ be the scale fixed in Definition~\ref{def:deltaandcover}. Recall also the definition of $\e^*$ from Definition~\ref{def:stableleaves}. Let $\e_u\in \left(0,\min\left\{\e_0,\frac{\delta}{4}\right\}\right)$ be the unstable scale fixed in the statement of the lemma, and let $N:\Omega\to\mathbb N\cup \{\infty\}$ be the stopping time given by Hypothesis \ref{hyp:h} for the pair $\left(\frac{\delta}{4},\e_u\right)$.

Choose $n_0\in\mathbb N$ such that $e^{\lambda_0 n_0}\e^*\geq 24\e$ and $e^{-\lambda_0 n_0}\e<\frac{\delta}{2}.$
Finally, choose $\omega\in \Omega$ such that $N_{n_0}(\omega)>0$ set $n:=N_{n_0}(\omega)\geq n_0,$ where $N_{n_0}(\omega)$ in given in \eqref{eq:ni}. We divide the proof into five steps.

\begin{step}[1]
We show that $
\mathcal L_\omega^n\left(\mathcal C_\omega(b,c,\nu)\right)\subset \mathcal C_{\theta^n\omega}(\lambda_2b,\lambda_2c,\nu)\subset \mathcal C_{\theta^n\omega}(b,c,\nu),$ and there exists $K_4>0$ such that
$$
\mathrm{Diam}(\mathcal L_\omega^n \mathcal C_\omega(b,c,\nu)) \leq K_4 + 2 \log \sup_{\varphi\in \mathcal C_\omega(b,c,\nu)}
\frac{\|\mathcal L_\omega^n\varphi\|_{\theta^n\omega,+}}{\|\mathcal L_\omega^n\varphi\|_{\theta^n\omega,-}}.
$$
\end{step}

By Proposition \ref{prop:cone_invariance_random}, iterated $n$ times, and since
$\mathcal C_{\theta^j\omega}(\lambda_2b,\lambda_2c,\nu)\subset
\mathcal C_{\theta^j\omega}(b,c,\nu)$ for every $j$, we have
$$
\mathcal L_\omega^n\mathcal C_\omega(b,c,\nu)
\subset
\mathcal C_{\theta^n\omega}(\lambda_2b,\lambda_2c,\nu)
\subset
\mathcal C_{\theta^n\omega}(b,c,\nu).
$$ 
We now show that
$$
\Theta_{\theta^n\omega}^{b,c,\nu}(\psi_1,\psi_2)
\le
d_{\theta^n\omega}^+(\psi_1,\psi_2)
+
2\log\left(\frac{1+\lambda_2}{1-\lambda_2}\right).
$$
Let
$$
0<t<
\frac{1-\lambda_2}{1+\lambda_2}
\alpha_{\theta^n\omega}^+(\psi_1,\psi_2).
$$
We prove that $\psi_2-t\psi_1\in\mathcal C_{\theta^n\omega}(b,c,\nu)$.

Condition $\mathbf{(C1)}$ follows from the definition of
$\alpha_{\theta^n\omega}^+(\psi_1,\psi_2)$. Indeed, for every
$\gamma\in\mathscr F_{\theta^n\omega}^s$ and every
$\rho\in D_1(a,\kappa,\gamma)$,
\begin{equation}\label{eq:C1-positive}
\int_\gamma(\psi_2-t\psi_1)\rho
=
\int_\gamma\psi_2\rho
-
t\int_\gamma\psi_1\rho
>0.
\end{equation}

We verify $\mathbf{(C2)}$. Fix
$\gamma\in\mathscr F_{\theta^n\omega}^s$ and
$\rho,\varsigma\in D_1(a,\kappa,\gamma)$. Since
$\psi_i\in\mathcal C_{\theta^n\omega}(\lambda_2b,\lambda_2c,\nu)$, for $i=1,2$,
\begin{equation}\label{eq:C2-strict}
e^{-\lambda_2 b\Theta_{\gamma}^{a,\kappa}(\rho,\varsigma)}
\int_\gamma\psi_i\varsigma
\le
\int_\gamma\psi_i\rho
\le
e^{\lambda_2 b\Theta_{\gamma}^{a,\kappa}(\rho,\varsigma)}
\int_\gamma\psi_i\varsigma.
\end{equation}
Using \eqref{eq:C2-strict}, we obtain
\begin{equation}\label{eq:C2-first-bound}
\begin{split}
\int_\gamma(\psi_2-t\psi_1)\rho
&\le
e^{\lambda_2 b\Theta_{\gamma}^{a,\kappa}(\rho,\varsigma)}
\int_\gamma\psi_2\varsigma
-
t e^{-\lambda_2 b\Theta_{\gamma}^{a,\kappa}(\rho,\varsigma)}
\int_\gamma\psi_1\varsigma.
\end{split}
\end{equation}
We claim that the right-hand side of \eqref{eq:C2-first-bound} is bounded above by $
e^{b\Theta_{\gamma}^{a,\kappa}(\rho,\varsigma)}
\int_\gamma(\psi_2-t\psi_1)\varsigma.$
To see this, it is enough to prove
\begin{equation}\label{eq:C2-needed}
\begin{split}
&t\left(\int_\gamma\psi_1\varsigma\right)
\left(
e^{b\Theta_{\gamma}^{a,\kappa}(\rho,\varsigma)}
-
e^{-\lambda_2 b\Theta_{\gamma}^{a,\kappa}(\rho,\varsigma)}
\right)\le
\left(\int_\gamma\psi_2\varsigma\right)
\left(
e^{b\Theta_{\gamma}^{a,\kappa}(\rho,\varsigma)}
-
e^{\lambda_2 b\Theta_{\gamma}^{a,\kappa}(\rho,\varsigma)}
\right).
\end{split}
\end{equation}
By the definition of $\alpha_{\theta^n\omega}^+(\psi_1,\psi_2)$,
\begin{equation}\label{eq:alpha-lower-C2}
\int_\gamma\psi_2\varsigma
\ge
\alpha_{\theta^n\omega}^+(\psi_1,\psi_2)
\int_\gamma\psi_1\varsigma.
\end{equation}
Since $
t<
\frac{1-\lambda_2}{1+\lambda_2}
\alpha_{\theta^n\omega}^+(\psi_1,\psi_2)$, to show \eqref{eq:C2-needed} it is enough to verify
\begin{equation}\label{eq:elementary-C2}
\frac{1-\lambda_2}{1+\lambda_2}
\left(
e^{b\Theta_{\gamma}^{a,\kappa}(\rho,\varsigma)}
-
e^{-\lambda_2 b\Theta_{\gamma}^{a,\kappa}(\rho,\varsigma)}
\right)\le
e^{b\Theta_{\gamma}^{a,\kappa}(\rho,\varsigma)}
-
e^{\lambda_2 b\Theta_{\gamma}^{a,\kappa}(\rho,\varsigma)}.
\end{equation}
If $\Theta_{\gamma}^{a,\kappa}(\rho,\varsigma)=0$, then
\eqref{eq:elementary-C2} is immediate. If
$\Theta_{\gamma}^{a,\kappa}(\rho,\varsigma)>0$, then
\begin{align*}
\frac{
e^{b\Theta_{\gamma}^{a,\kappa}(\rho,\varsigma)}
-
e^{\lambda_2 b\Theta_{\gamma}^{a,\kappa}(\rho,\varsigma)}
}{
e^{b\Theta_{\gamma}^{a,\kappa}(\rho,\varsigma)}
-
e^{-\lambda_2 b\Theta_{\gamma}^{a,\kappa}(\rho,\varsigma)}
}=
\frac{
1-e^{-(1-\lambda_2)b\Theta_{\gamma}^{a,\kappa}(\rho,\varsigma)}
}{
1-e^{-(1+\lambda_2)b\Theta_{\gamma}^{a,\kappa}(\rho,\varsigma)}
}
\ge
\frac{1-\lambda_2}{1+\lambda_2},
\end{align*}
because $x\mapsto(1-e^{-x})/x$ is decreasing on $(0,\infty)$. Hence
\eqref{eq:elementary-C2} holds. Combining \eqref{eq:C2-first-bound},
\eqref{eq:alpha-lower-C2}, and \eqref{eq:elementary-C2}, we obtain
\begin{equation}\label{eq:C2-final}
\int_\gamma(\psi_2-t\psi_1)\rho
\le
e^{b\Theta_{\gamma}^{a,\kappa}(\rho,\varsigma)}
\int_\gamma(\psi_2-t\psi_1)\varsigma.
\end{equation}
This verifies  $\mathbf{(C2)}$.

We verify $\mathbf{(C3)}$ in the same way. Let
$(\tilde\gamma,\gamma)$ be a nearby pair in
$\mathscr F_{\theta^n\omega}^s$, let
$\rho\in D_1(a_1,\kappa,\gamma)$, and let $\tilde\rho$ be the density induced on
$\tilde\gamma$ by unstable holonomy. Since
$\psi_i\in\mathcal C_{\theta^n\omega}(\lambda_2b,\lambda_2c,\nu)$, for $i=1,2$,
\begin{equation}\label{eq:C3-strict}
e^{-\lambda_2 c d_u(\tilde\gamma,\gamma)^\nu}
\int_\gamma\psi_i\rho
\le
\int_{\tilde\gamma}\psi_i\tilde\rho
\le
e^{\lambda_2 c d_u(\tilde\gamma,\gamma)^\nu}
\int_\gamma\psi_i\rho.
\end{equation}
Using \eqref{eq:C3-strict}, we obtain
\begin{equation}\label{eq:C3-first-bound}
\begin{split}
\int_{\tilde\gamma}(\psi_2-t\psi_1)\tilde\rho
&\le
e^{\lambda_2 c d_u(\tilde\gamma,\gamma)^\nu}
\int_\gamma\psi_2\rho
-
t e^{-\lambda_2 c d_u(\tilde\gamma,\gamma)^\nu}
\int_\gamma\psi_1\rho.
\end{split}
\end{equation}
We claim that the right-hand side of \eqref{eq:C3-first-bound} is bounded above by $e^{c d_u(\tilde\gamma,\gamma)^\nu}
\int_\gamma(\psi_2-t\psi_1)\rho.$
Equivalently, it is enough to prove
\begin{equation}\label{eq:C3-needed}
\begin{split}
&t\left(\int_\gamma\psi_1\rho\right)
\left(
e^{c d_u(\tilde\gamma,\gamma)^\nu}
-
e^{-\lambda_2 c d_u(\tilde\gamma,\gamma)^\nu}
\right) \le
\left(\int_\gamma\psi_2\rho\right)
\left(
e^{c d_u(\tilde\gamma,\gamma)^\nu}
-
e^{\lambda_2 c d_u(\tilde\gamma,\gamma)^\nu}
\right).
\end{split}
\end{equation}
Since $\rho\in D_1(a_1,\kappa,\gamma)\subset D_1(a,\kappa,\gamma)$, the definition of
$\alpha_{\theta^n\omega}^+(\psi_1,\psi_2)$ gives
\begin{equation}\label{eq:alpha-lower-C3}
\int_\gamma\psi_2\rho
\ge
\alpha_{\theta^n\omega}^+(\psi_1,\psi_2)
\int_\gamma\psi_1\rho.
\end{equation}
Since $t<
\frac{1-\lambda_2}{1+\lambda_2}
\alpha_{\theta^n\omega}^+(\psi_1,\psi_2),$
the estimate \eqref{eq:C3-needed} is a consequence of 
\begin{equation}\label{eq:elementary-C3}
\frac{1-\lambda_2}{1+\lambda_2}
\left(
e^{c d_u(\tilde\gamma,\gamma)^\nu}
-
e^{-\lambda_2 c d_u(\tilde\gamma,\gamma)^\nu}
\right)\le
e^{c d_u(\tilde\gamma,\gamma)^\nu}
-
e^{\lambda_2 c d_u(\tilde\gamma,\gamma)^\nu}.
\end{equation}
The proof of \eqref{eq:elementary-C3} is identical to the proof of
\eqref{eq:elementary-C2}, replacing
$b\Theta_{\gamma}^{a,\kappa}(\rho,\varsigma)$ by
$c d_u(\tilde\gamma,\gamma)^\nu$. Combining \eqref{eq:C3-first-bound},
\eqref{eq:alpha-lower-C3}, and \eqref{eq:elementary-C3}, we get
\begin{equation}\label{eq:C3-final}
\int_{\tilde\gamma}(\psi_2-t\psi_1)\tilde\rho
\le
e^{c d_u(\tilde\gamma,\gamma)^\nu}
\int_\gamma(\psi_2-t\psi_1)\rho.
\end{equation}
This proves $\mathbf{(C3)}$. Therefore $
\psi_2-t\psi_1\in\mathcal C_{\theta^n\omega}(b,c,\nu),$ and hence
\begin{equation}\label{eq:alpha-full-lower}
\alpha_{\theta^n\omega}^{b,c,\nu}(\psi_1,\psi_2)
\ge
\frac{1-\lambda_2}{1+\lambda_2}
\alpha_{\theta^n\omega}^+(\psi_1,\psi_2).
\end{equation}

Analogously, if
$$
s>
\frac{1+\lambda_2}{1-\lambda_2}
\beta_{\theta^n\omega}^+(\psi_1,\psi_2),
$$
then $
s\psi_1-\psi_2\in\mathcal C_{\theta^n\omega}(b,c,\nu),$
and consequently
\begin{equation}\label{eq:beta-full-upper}
\beta_{\theta^n\omega}^{b,c,\nu}(\psi_1,\psi_2)
\le
\frac{1+\lambda_2}{1-\lambda_2}
\beta_{\theta^n\omega}^+(\psi_1,\psi_2).
\end{equation}
Using \eqref{eq:alpha-full-lower} and \eqref{eq:beta-full-upper}, we obtain
\begin{align*}
\Theta_{\theta^n\omega}^{b,c,\nu}(\psi_1,\psi_2)
&=
\log
\frac{\beta_{\theta^n\omega}^{b,c,\nu}(\psi_1,\psi_2)}
{\alpha_{\theta^n\omega}^{b,c,\nu}(\psi_1,\psi_2)}\le
\log
\left[
\left(\frac{1+\lambda_2}{1-\lambda_2}\right)^2
\frac{\beta_{\theta^n\omega}^+(\psi_1,\psi_2)}
{\alpha_{\theta^n\omega}^+(\psi_1,\psi_2)}
\right]
\\
&=
d_{\theta^n\omega}^+(\psi_1,\psi_2)
+
2\log\left(\frac{1+\lambda_2}{1-\lambda_2}\right).
\end{align*}

Hence
\begin{align*}
d_{\theta^n\omega}^+(\psi_1,\psi_2)
&\le
\log
\left(
\frac{\|\psi_2\|_{\theta^n\omega,+}}
{\|\psi_1\|_{\theta^n\omega,-}}
\frac{\|\psi_1\|_{\theta^n\omega,+}}
{\|\psi_2\|_{\theta^n\omega,-}}
\right)=
\log
\frac{\|\psi_1\|_{\theta^n\omega,+}}
{\|\psi_1\|_{\theta^n\omega,-}}
+
\log
\frac{\|\psi_2\|_{\theta^n\omega,+}}
{\|\psi_2\|_{\theta^n\omega,-}}
\\
&\le
2\log
\sup_{\phi\in\mathcal C_\omega(b,c,\nu)}
\frac{\|\mathcal L_\omega^n\phi\|_{\theta^n\omega,+}}
{\|\mathcal L_\omega^n\phi\|_{\theta^n\omega,-}}.
\end{align*}
Combining the two estimates, we get
$$
\Theta_{\theta^n\omega}^{b,c,\nu}
(\mathcal L_\omega^n\phi_1,\mathcal L_\omega^n\phi_2)
\le
K_4
+
2\log
\sup_{\phi\in\mathcal C_\omega(b,c,\nu)}
\frac{\|\mathcal L_\omega^n\phi\|_{\theta^n\omega,+}}
{\|\mathcal L_\omega^n\phi\|_{\theta^n\omega,-}},
$$
where $
K_4:=2\log\left(\frac{1+\lambda_2}{1-\lambda_2}\right).$ Taking the supremum over
$\phi_1,\phi_2\in\mathcal C_\omega(b,c,\nu)$ yields
$$
\operatorname{Diam}_{\Theta_{\theta^n\omega}^{b,c,\nu}}
\left(
\mathcal L_\omega^n\mathcal C_\omega(b,c,\nu)
\right)
\le
K_4
+
2\log
\sup_{\phi\in\mathcal C_\omega(b,c,\nu)}
\frac{\|\mathcal L_\omega^n\phi\|_{\theta^n\omega,+}}
{\|\mathcal L_\omega^n\phi\|_{\theta^n\omega,-}}.
$$

\begin{step}[2] We show that $\|\mathcal L_\omega^n \varphi\|_{\theta^n\omega,+} \leq E_\phi^n \|\varphi\|_{\omega,+}$ for every $\varphi \in \mathcal C_\omega(b,c,\nu)$ and $\omega\in\Omega.$    
\end{step}

Write $\psi:=\mathcal L_\omega^n\varphi.$ Let $\gamma(\theta^n\omega)\in\mathscr F^s_{\theta^n\omega}$ and
$\rho_{\theta^n\omega}\in D_1(a,\kappa,\gamma(\theta^n\omega))$. Iterating Lemma \ref{lem:push_density}, we obtain a finite
family of pairwise disjoint branches
$$
(T_\omega^{n})^{-1}(\gamma(\theta^n\omega))=\bigcup_{j=1}^{Q_n(\gamma(\theta^n\omega))}\gamma^{(j)}(\omega),
$$
together with positive densities $\bar\rho_\omega^{(j)}$ on $\gamma^{(j)}(\omega)$ such that
$$
\int_{\gamma(\theta^n\omega)}\psi\,\rho_{\theta^n\omega}
=
\sum_{j=1}^{Q_n(\gamma(\theta^n\omega))}
\int_{\gamma^{(j)}(\omega)}\varphi\,\bar\rho_\omega^{(j)}.
$$
Whenever $\int_{\gamma^{(j)}(\omega)}\bar\rho_\omega^{(j)}>0$, set $
\rho_\omega^{(j)}:=
\frac{\bar\rho_\omega^{(j)}}{\int_{\gamma^{(j)}(\omega)}\bar\rho_\omega^{(j)}}.$
 By repeated application of Lemma \ref{lem:push_density}, one has $
\rho_\omega^{(j)}\in D_1(a,\kappa,\gamma^{(j)}(\omega)).$ Hence
\begin{align*}
\int_{\gamma(\theta^n\omega)}\psi\,\rho_{\theta^n\omega}
&=
\sum_j
\left(\int_{\gamma^{(j)}(\omega)}\bar\rho_\omega^{(j)}\right)
\left(\int_{\gamma^{(j)}(\omega)}\varphi\,\rho_\omega^{(j)}\right)\\
&\leq
\|\varphi\|_{\omega,+}
\sum_j\int_{\gamma^{(j)}(\omega)}\bar\rho_\omega^{(j)}.
\end{align*}
Applying the same identity with $\varphi\equiv1$, we get
$$
\sum_j\int_{\gamma^{(j)}(\omega)}\bar\rho_\omega^{(j)}
=
\int_{\gamma(\theta^n\omega)}\mathcal L_\omega^n \mathbbm 1\,\rho_{\theta^n\omega}.
$$
Since
$$
\mathcal L_\omega^n \mathbbm 1(x)=e^{S_n\phi_\omega(T_\omega^{-n}x)}
\leq
E_\phi^n
\ \text{for every }x\in M,
$$
it follows that
$$
\int_{\gamma(\theta^n\omega)}\psi\,\rho_{\theta^n\omega}
\leq
E_\phi^n\|\varphi\|_{\omega,+}.
$$
Taking the supremum over $\gamma(\theta^n\omega)$ and $\rho_{\theta^n\omega}$ yields $\|\psi\|_{\theta^n\omega,+}\leq E_\phi^n\|\varphi\|_{\omega,+}.$

\begin{step}[3]
Let $\varphi\in\mathcal C_\omega(b,c,\nu)$. Define
\begin{equation}\label{eq:half-plus-norm}
\|\varphi\|_{\omega,+}^{(1/2)}
:=
\sup_{\gamma(\omega)\in\mathscr F_\omega^s}
\sup_{\rho_\omega\in D_1(a/2,\kappa,\gamma(\omega))}
\int_{\gamma(\omega)}\varphi\,\rho_\omega.
\end{equation}
We show that
\begin{equation}\label{eq:full-plus-controlled-by-half-plus}
\|\varphi\|_{\omega,+}
\le
2e^{a\varepsilon^\kappa}
\|\varphi\|_{\omega,+}^{(1/2)}.
\end{equation}
\end{step}

Fix $\gamma(\omega)\in\mathscr F_\omega^s$ and
$\rho_\omega\in D_1(a,\kappa,\gamma(\omega))$. Let $k_\omega$ denote the constant
normalised density on $\gamma(\omega)$, namely
\begin{equation}\label{eq:constant-density}
k_\omega(x)
:=
\frac{1}{m_{\gamma(\omega)}(\gamma(\omega))},
\ 
x\in\gamma(\omega).
\end{equation}
Since $\rho_\omega\in D_1(a,\kappa,\gamma(\omega))$, there exists
$x_0\in\gamma(\omega)$ such that
\begin{equation}\label{eq:rho-equals-constant-somewhere}
\rho_\omega(x_0)=k_\omega(x_0).
\end{equation}
Moreover, for every $x\in\gamma(\omega)$, using
\eqref{eq:rho-equals-constant-somewhere} and
$\mathrm{diam}(\gamma(\omega))\le\varepsilon$,
\begin{equation}\label{eq:rho-upper-constant}
\rho_\omega(x)
\le
e^{a d(x,x_0)^\kappa}\rho_\omega(x_0)
\le
e^{a\varepsilon^\kappa}k_\omega(x).
\end{equation}
Hence
\begin{equation}\label{eq:h-positive}
2e^{a\varepsilon^\kappa}k_\omega-\rho_\omega>0.
\end{equation}
It is easy to see that
\begin{equation}\label{eq:h-in-Da}
2e^{a\varepsilon^\kappa}k_\omega-\rho_\omega
\in
D(a,\kappa,\gamma(\omega)).
\end{equation}

From \eqref{eq:h-positive} and \eqref{eq:h-in-Da}, the function
$2e^{a\varepsilon^\kappa}k_\omega-\rho_\omega$ is a positive multiple of an element of
$D_1(a,\kappa,\gamma(\omega))$. Thus condition $\mathbf{(C1)}$ gives $\int_{\gamma(\omega)}
\varphi\left(2e^{a\varepsilon^\kappa}k_\omega-\rho_\omega\right)>0.$ This implies
\begin{equation}\label{eq:rho-controlled-by-k}
\int_{\gamma(\omega)}\varphi\,\rho_\omega
\le
2e^{a\varepsilon^\kappa}
\int_{\gamma(\omega)}\varphi\,k_\omega.
\end{equation}
Since $k_\omega\in D_1(a/2,\kappa,\gamma(\omega))$, the definition
\eqref{eq:half-plus-norm} gives
\begin{equation}\label{eq:k-controlled-by-half-plus}
\int_{\gamma(\omega)}\varphi\,k_\omega
\le
\|\varphi\|_{\omega,+}^{(1/2)}.
\end{equation}
Combining \eqref{eq:rho-controlled-by-k} and \eqref{eq:k-controlled-by-half-plus}, we obtain
\begin{equation}\label{eq:rho-controlled-by-half-plus}
\int_{\gamma(\omega)}\varphi\,\rho_\omega
\le
2e^{a\varepsilon^\kappa}
\|\varphi\|_{\omega,+}^{(1/2)}.
\end{equation}
Taking the supremum in \eqref{eq:rho-controlled-by-half-plus} over
$\gamma(\omega)\in\mathscr F_\omega^s$ and
$\rho_\omega\in D_1(a,\kappa,\gamma(\omega))$ proves
\eqref{eq:full-plus-controlled-by-half-plus}.

\begin{step}[4]
We show that $
\|\mathcal L_\omega^n\varphi\|_{\theta^n\omega,-}\geq D_1(\omega)^{-1}\|\mathcal L_\omega^n\varphi\|_{\theta^n\omega,+}$ for every $\varphi\in\mathcal C_\omega(b,c,\nu)$.
\end{step}

Let $\psi:=\mathcal L_\omega^n\varphi.$ Choose
$\gamma^*(\omega)\in\mathscr F_\omega^s$ and
$\rho_\omega^*\in D_1(a/2,\kappa,\gamma^*(\omega))$
such that
\begin{equation}\label{eq:almost-half-plus-maximiser}
\int_{\gamma^*(\omega)}
\varphi\,\rho_\omega^*
\ge
\frac12
\|\varphi\|_{\omega,+}^{(1/2)}.
\end{equation}
Using \eqref{eq:full-plus-controlled-by-half-plus}, we obtain
\begin{equation}\label{eq:good-initial-leaf}
\int_{\gamma^*(\omega)}
\varphi\,\rho_\omega^*
\ge
\frac{e^{-a\varepsilon^\kappa}}{4}
\|\varphi\|_{\omega,+}.
\end{equation}

We now fix an arbitrary local stable manifold
$\gamma(\theta^n\omega)\in\mathscr F_{\theta^n\omega}^s$.
Choose points $x^*,x_n\in M$ such that
$$
W_{\varepsilon^*}^s(\omega,x^*)
\subset
\gamma^*(\omega),
\ 
W_{\varepsilon^*}^s(\theta^n\omega,x_n)
\subset
\gamma(\theta^n\omega).
$$
Let $\gamma^u(\omega):=
W_{\varepsilon_u}^u(\omega,x^*).$
Since $\varepsilon_u<\delta/4$,  $\gamma^u(\omega)\subset B_M(\delta/4,x^*),$
 and $n=N_{n_0}(\omega)$, we obtain that $T_\omega^n(\gamma^u(\omega))$ is $\delta/4$-dense in $M$. Hence there exists
$z\in\gamma^u(\omega)$ such that $T_\omega^n(z)\in B_M(\delta/4,x_n).$
In particular, $d(T_\omega^n(z),x_n)<\delta/4<\delta_{\mathrm{loc}},$
so the bracket $y:=[x_n,T_\omega^n(z)]_{\theta^n\omega}^{\varepsilon}$
is well-defined. By the choice of $\delta$ and the local product structure,
$$
y\in W_{\varepsilon^*}^s(\theta^n\omega,x_n)
\subset
\gamma(\theta^n\omega).
$$
Set $\tilde y:=(T_\omega^n)^{-1}(y).$
From the fact that $y$ and $T_\omega^n(z)$ lie on the same local unstable manifold over the fibre
$\theta^n\omega$, backward contraction along unstable leaves gives
$$
d(\tilde y,z)
\le
e^{-\lambda_0 n}d(y,T_\omega^n(z))
\le
e^{-\lambda_0 n}\varepsilon
\le
e^{-\lambda_0 n_0}\varepsilon
<
\frac{\delta}{2}.
$$
Using that $d(z,x^*)<\delta/4$, we have $
d(\tilde y,x^*)<\frac{3\delta}{4}<\delta_{\mathrm{loc}}.$
Hence, by the local product structure, there exists
$\widetilde\gamma(\omega)\in\mathscr F_\omega^s$ such that $
\widetilde\gamma(\omega)
\subset
(T_\omega^n)^{-1}(\gamma(\theta^n\omega)),$
the pair
$(\widetilde\gamma(\omega),\gamma^*(\omega))$
is nearby, and $d_u(\widetilde\gamma(\omega),\gamma^*(\omega))
\le
\varepsilon.$

Choose an arbitrary density
$\rho_{\theta^n\omega}\in D_1(a,\kappa,\gamma(\theta^n\omega))$. Let $\bar\rho_\omega$ be the unnormalised pullback density on
$\widetilde\gamma(\omega)$ obtained from
$\rho_{\theta^n\omega}$ by iterating Lemma \ref{lem:push_density}
along this branch. Since all branch contributions are positive,
\begin{equation}\label{eq:branch-lower-bound-step4}
\int_{\gamma(\theta^n\omega)}
\psi\,\rho_{\theta^n\omega} = \int_{\gamma(\theta^n\omega)}
\mathcal L_\omega^n \varphi \,\rho_{\theta^n\omega} 
\ge
\int_{\widetilde\gamma(\omega)}
\varphi\,\bar\rho_\omega.
\end{equation}
Write
$$
M_\omega:=
\int_{\widetilde\gamma(\omega)}
\bar\rho_\omega,
\ 
\widetilde\rho_\omega:=
\frac{\bar\rho_\omega}{M_\omega}.
$$
From Lemma
\ref{lem:push_density} we obtain
\begin{equation}\label{eq:pullback-half-density-step4}
\widetilde\rho_\omega
\in
D_1(\alpha_0 a,\kappa,\widetilde\gamma(\omega)).
\end{equation}

In the follows we estimate $M_\omega$ from below. Along the chosen branch,
$$
\bar\rho_\omega(x)
=
e^{S_n\phi_\omega(x)-S_n\phi_{J^s,\omega}(x)}
\rho_{\theta^n\omega}(T_\omega^n x).
$$
Since
$\phi^{J^s}(\omega,x)=-\log J^s(\omega,x)$,
changing variables
$y=T_\omega^n(x)$ along the stable manifold gives
\begin{equation}\label{eq:M-change-variables-step4}
M_\omega
=
\int_{T_\omega^n(\widetilde\gamma(\omega))}
e^{S_n\phi_\omega(T_\omega^{-n}y)}
\rho_{\theta^n\omega}(y).
\end{equation}
Because
$\rho_{\theta^n\omega}\in
D_1(a,\kappa,\gamma(\theta^n\omega))$
and
$\operatorname{diam}(\gamma(\theta^n\omega))\le\varepsilon$,
for every
$y\in\gamma(\theta^n\omega)$,
$$
1
=
\int_{\gamma(\theta^n\omega)}
\rho_{\theta^n\omega}
\le
e^{a\varepsilon^\kappa}
\rho_{\theta^n\omega}(y)
m_{\gamma(\theta^n\omega)}
(\gamma(\theta^n\omega)).
$$
Hence
\begin{equation}\label{eq:rho-final-lower-step4}
\rho_{\theta^n\omega}(y)
\ge
\frac{e^{-a\varepsilon^\kappa}}
{m_{\gamma(\theta^n\omega)}
(\gamma(\theta^n\omega))}
\end{equation}
for every
$y\in\gamma(\theta^n\omega)$.
Using \eqref{eq:M-change-variables-step4} and
\eqref{eq:rho-final-lower-step4},
\begin{equation}\label{eq:M-lower-preliminary-step4}
M_\omega
\ge
E_\phi^{-n}
e^{-a\varepsilon^\kappa}
\frac{
m_{\gamma(\theta^n\omega)}
(T_\omega^n(\widetilde\gamma(\omega)))
}{
m_{\gamma(\theta^n\omega)}
(\gamma(\theta^n\omega))
}.
\end{equation}
Since
$\widetilde\gamma(\omega)\in\mathscr F_\omega^s$, $
m_{\widetilde\gamma(\omega)}
(\widetilde\gamma(\omega))
\ge
A(\varepsilon)/(4J^2),$ $m_{\gamma(\theta^n\omega)}
(\gamma(\theta^n\omega))
\le
A(\varepsilon),$ and, along stable leaves, $
m_{\gamma(\theta^n\omega)}
(T_\omega^n(\widetilde\gamma(\omega)))
\ge
\underline m_s^n
m_{\widetilde\gamma(\omega)}
(\widetilde\gamma(\omega)).$
Combining these estimates with
\eqref{eq:M-lower-preliminary-step4},
we obtain
\begin{equation}\label{eq:M-lower-step4}
M_\omega
\ge
E_\phi^{-n}
e^{-a\varepsilon^\kappa}
\frac{\underline m_s^n}{4J^2}.
\end{equation}

Let $k_\omega$ be the constant normalised density on
$\widetilde\gamma(\omega)$. Since
$k_\omega\in D_1(a_1,\kappa,\widetilde\gamma(\omega))$ and, by
\eqref{eq:pullback-half-density-step4},
$$
\widetilde\rho_\omega
\in
D_1(\alpha_0 a,\kappa,\widetilde\gamma(\omega))
=
D_1(a_1,\kappa,\widetilde\gamma(\omega)),
$$
we have $\Theta_{\widetilde\gamma(\omega)}^{a,\kappa}
(\widetilde\rho_\omega,k_\omega)
\le
K_0$
for some $K_0>0$. Thus condition $\mathbf{(C2)}$ yields
\begin{equation}\label{eq:C2-first-lower-step4}
\int_{\widetilde\gamma(\omega)}
\varphi\,\widetilde\rho_\omega
\ge
e^{-bK_0}
\int_{\widetilde\gamma(\omega)}
\varphi\,k_\omega.
\end{equation}

Let
$\mathrm{hol}_\omega^u:\gamma^*(\omega)\to\widetilde\gamma(\omega)$
be the unstable holonomy map. Define the density $\widetilde{k}_\omega$ on
$\gamma^*(\omega)$ by
$$
\widetilde{k}_\omega(y)
:=
k_\omega\bigl(\mathrm{hol}_\omega^u(y)\bigr)\,
\mathrm{Jac}\bigl(\mathrm{hol}_\omega^u\bigr)(y),
\ 
y\in\gamma^*(\omega).
$$
Since $k_\omega$ is constant, $k_\omega\in D_1(a_1,\kappa_1,\widetilde\gamma(\omega)).$
Therefore, applying condition $\mathbf{(C3)}$ to the nearby pair
$(\gamma^*(\omega),\widetilde\gamma(\omega))$, we obtain
\begin{equation}\label{eq:C3-lower-step4}
\int_{\widetilde\gamma(\omega)}
\varphi\,k_\omega
\ge
e^{-c\varepsilon^\nu}
\int_{\gamma^*(\omega)}
\varphi\,\widetilde{k}_\omega.
\end{equation}
Moreover, $
\widetilde{k}_\omega
\in
D_1(a_1,\kappa_1,\gamma^*(\omega))
\subset
D_1(a_1,\kappa,\gamma^*(\omega)),$
and, since
$$
\rho_\omega^*
\in
D_1(a/2,\kappa,\gamma^*(\omega))
\subset
D_1(a_1,\kappa,\gamma^*(\omega)),
$$
we have $
\Theta_{\gamma^*(\omega)}^{a,\kappa}
(\widetilde{k}_\omega,\rho_\omega^*)
\le
K_0.$ Thus condition $\mathbf{(C2)}$ yields
\begin{equation}\label{eq:C2-lower-step4}
\int_{\gamma^*(\omega)}
\varphi\,\widetilde{k}_\omega
\ge
e^{-bK_0}
\int_{\gamma^*(\omega)}
\varphi\,\rho_\omega^*.
\end{equation}
Combining
\eqref{eq:branch-lower-bound-step4},
\eqref{eq:M-lower-step4},
\eqref{eq:C2-first-lower-step4},
\eqref{eq:C3-lower-step4},
\eqref{eq:C2-lower-step4},
and
\eqref{eq:good-initial-leaf},
we obtain
\begin{align}
\int_{\gamma(\theta^n\omega)}
\psi\,\rho_{\theta^n\omega} \ge
M_\omega
\int_{\widetilde\gamma(\omega)}
\varphi\,\widetilde\rho_\omega \ge
\frac{\underline m_s^n}{16J^2}
e^{
-2a\varepsilon^\kappa
-c\varepsilon^\nu
-2bK_0}
E_\phi^{-n}
\|\varphi\|_{\omega,+}.
\label{eq:final-lower-before-step2}
\end{align}
By Step 2, $\|\psi\|_{\theta^n\omega,+}
\le
E_\phi^n
\|\varphi\|_{\omega,+}.$
Therefore \eqref{eq:final-lower-before-step2} gives
\begin{equation}\label{eq:final-lower-step4}
\int_{\gamma(\theta^n\omega)}
\psi\,\rho_{\theta^n\omega}
\ge
\frac{\underline m_s^n}{16J^2}
e^{
-2a\varepsilon^\kappa
-c\varepsilon^\nu
-2bK_0}
E_\phi^{-2n}
\|\psi\|_{\theta^n\omega,+}.
\end{equation}
Since
$\gamma(\theta^n\omega)\in\mathscr F_{\theta^n\omega}^s$
and
$\rho_{\theta^n\omega}\in
D_1(a,\kappa,\gamma(\theta^n\omega))$
were arbitrary, taking the infimum in
\eqref{eq:final-lower-step4} yields
$$
\|\psi\|_{\theta^n\omega,-}
\ge
D_1(\omega)^{-1}
\|\psi\|_{\theta^n\omega,+},\ \text{where }
D_1(\omega)
:=
K_5
\frac{E_\phi^{2n}}{\underline m_s^n}
$$
and $K_5
:=
16J^2
e^{
2a\varepsilon^\kappa
+c\varepsilon^\nu
+2bK_0}.$
Equivalently,
$$
\frac{
\|\mathcal L_\omega^n\varphi\|_{\theta^n\omega,+}
}{
\|\mathcal L_\omega^n\varphi\|_{\theta^n\omega,-}
}
\le
D_1(\omega)
$$
for every
$\varphi\in\mathcal C_\omega(b,c,\nu)$.

 \begin{step}[5]
We conclude the proof.
\end{step}

From Step 1 and Step 3 we have that  $
\Theta_{\theta^n\omega}^{b,c,\nu}(\psi_1,\psi_2)
\leq
 K_4 + \log D_1(\omega).$
Hence
$$
\sup_{\varphi_1,\varphi_2\in\mathcal C_\omega(b,c,\nu)}
\Theta_{\theta^n\omega}^{b,c,\nu}\left(\mathcal L_\omega^n\varphi_1,\mathcal L_\omega^n\varphi_2\right)
\leq
D_2(\omega),
$$
with $
D_2(\omega):=K_4+ 2\log D_1(\omega).$ This completes the proof.
\end{proof}

\end{appendix}

\addtocontents{toc}{\protect\setcounter{tocdepth}{2}}

\end{document}